\documentclass{memo-l}

\usepackage{latexsym,amscd,amssymb, graphicx, amsthm, bm, arydshln, tikz}
\usetikzlibrary{calc}
\usetikzlibrary{arrows}

\newtheorem{theorem}{Theorem}[chapter]
\newtheorem{proposition}[theorem]{Proposition}
\newtheorem{corollary}[theorem]{Corollary}
\newtheorem{lemma}[theorem]{Lemma}
\newtheorem{conjecture}[theorem]{Conjecture}
\newtheorem{observation}[theorem]{Observation}
\newtheorem{question}[theorem]{Question}
\newtheorem{problem}[theorem]{Problem}

\theoremstyle{definition}
\newtheorem{definition}[theorem]{Definition}
\newtheorem{defn}[theorem]{Definition}
\newtheorem{example}[theorem]{Example}

\theoremstyle{remark}
\newtheorem{remark}[theorem]{Remark}

\numberwithin{section}{chapter}
\numberwithin{equation}{chapter}

\newcommand{\stat}{{\mathrm {stat}}}
\newcommand{\maj}{{\mathrm {maj}}}
\newcommand{\inv}{{\mathrm {inv}}}
\newcommand{\exc}{{\mathrm {exc}}}
\newcommand{\sign}{{\mathrm {sign}}}
\newcommand{\peak}{{\mathrm {peak}}}

\newcommand{\mult}{{\mathrm {mult}}}
\newcommand{\Des}{{\mathrm {Des}}}
\newcommand{\End}{{\mathrm {End}}}
\newcommand{\path}{{\mathrm {path}}}
\newcommand{\id}{{\mathrm {id}}}
\newcommand{\TTT}{{\mathcal {T}}}
\newcommand{\OOO}{{\mathcal {O}}}
\newcommand{\PPP}{{\mathcal {P}}}
\newcommand{\HHH}{{\mathcal {H}}}

\newcommand{\des}{{\mathrm {des}}}
\newcommand{\asc}{{\mathrm {asc}}}
\newcommand{\GGG}{{\mathcal {G}}}
\newcommand{\VVV}{{\mathcal {V}}}
\newcommand{\content}{{\mathrm {cont}}}
\newcommand{\fix}{{\mathrm {fix}}}
\newcommand{\lis}{{\mathrm {lis}}}
\newcommand{\Irr}{{\mathrm {Irr}}}
\newcommand{\Mat}{{\mathrm {Mat}}}
\newcommand{\Res}{{\mathrm {Res}}}

\newcommand{\triv}{{\mathrm {triv}}}

\newcommand{\Mob}{{\mathrm {M\ddot{o}b}}}

\newcommand{\Ind}{{\mathrm {Ind}}}
\newcommand{\height}{{\mathrm {ht}}}

\newcommand{\shape}{{\mathrm {shape}}}

\newcommand{\symm}{{\mathfrak{S}}}

\newcommand{\wt}{{\mathrm{wt}}}

\newcommand{\YY}{{\mathbb {Y}}}
\newcommand{\FF}{{\mathbb {F}}}
\newcommand{\PP}{{\mathbb {P}}}
\newcommand{\EE}{{\mathbb {E}}}
\newcommand{\CC}{{\mathbb {C}}}
\newcommand{\QQ}{{\mathbb {Q}}}
\newcommand{\ZZ}{{\mathbb {Z}}}
\newcommand{\NN}{{\mathbb{N}}}
\newcommand{\RR}{{\mathbb{R}}}

\newcommand{\VV}{{\mathbb{V}}}
\newcommand{\WWW}{{\mathcal{W}}}

\newcommand{\UUU}{{\mathcal{U}}}

\newcommand{\NNN}{{\mathcal{N}}}

\newcommand{\AAA}{{\mathcal{A}}}
\newcommand{\MMM}{{\mathcal{M}}}
\newcommand{\BBB}{{\mathcal{B}}}
\newcommand{\one}{{\bm{1}}}
\newcommand{\III}{{\mathcal{I}}}
\newcommand{\gr}{{\mathrm{gr}}}
\newcommand{\Hilb}{{\mathrm{Hilb}}}

\newcommand{\domcyc}{{\mathrm {domcyc}}}
\newcommand{\cyc}{{\mathrm {cyc}}}

\newcommand{\xx}{{\mathbf {x}}}
\newcommand{\II}{{\mathbf {I}}}

\newcommand{\ch}{{\mathrm {ch}}}
\newcommand{\Fun}{{\mathrm{Fun}}}
\newcommand{\Class}{{\mathrm{Class}}}
\newcommand{\Loc}{{\mathrm{Loc}}}

\makeindex

\begin{document}

\frontmatter

\title[characters of local and regular permutation statistics]
{Characters of Local and Regular Permutation Statistics}


\author{Zachary Hamaker}
\address
{Department of Mathematics \newline \indent
University of Florida \newline \indent
Gainesville, FL, 32611, USA}
\email{zhamaker@ufl.edu}
\thanks{Z. Hamaker was partially supported by NSF Grant DMS-2054423.}

\author{Brendon Rhoades}
\address
{Department of Mathematics \newline \indent
University of California, San Diego \newline \indent
La Jolla, CA, 92093, USA}
\email{bprhoades@ucsd.edu}
\thanks{B. Rhoades was partially supported by NSF Grants DMS-1953781 and DMS-2246846.}

\date{}

\subjclass[2020]{Primary }

\keywords{}

\dedicatory{}

\begin{abstract}
The goal of this monograph is to study the indicator function for a set of permutations mapping one finite sequence of positive integers to another from a representation theoretic, combinatorial and probabilistic perspective.
The degree of a function of permutations is the size of the largest pair of sequences required when expressing it as a linear combination of these indicators.
This notion of degree, implicit in work of Diaconis, is critical for many applications of representation theory to extremal combinatorics, machine learning, probability and statistics.
We use the term \emph{local} to indicate bounded degree, as it is both easier to work with and less confusing given the pervasiveness of the term `degree'.
Most permutation statistics in enumerative and algebraic combinatorics are manifestly local.
Such functions can be decomposed into indicators, hence are amenable to analysis using techniques and machinery from representation theory.
We initiate the study of low degree class functions, which encode probabilistic data for permutation statistics on each cycle type simultaneously.

Our monograph splits into three parts.
Since the concept of degree is unstudied in the enumerative and algebraic combinatorics literatures, we first develop its theory from scratch in a language familiar to this audience.
Class functions correspond to symmetric functions, and the pattern counts used in our work lead naturally to a novel basis we call the \emph{path power sum symmetric functions}.
The second and most technically challenging part of this monograph is our \emph{path Murnaghan-Nakayama formula}, which expands path power sums into Schur functions.

By combining the the path Murnaghan-Nakayam formula with the classical theory of character polynomials, one obtains a structural characterization for moments for permutation statistics conditioning on cycle type.
The third part analyzes asymptotic properties of these moments.
In doing so, we introduce the novel family of \emph{regular permutation statistics}, which include almost all reasonable weighted pattern counting statistics.
We show a large family of regular statistics satisfy a law of large numbers on a given cycle type depending only on the proportion of fixed points and a have variances depending only on fixed points and two cycles.
As a consequence, we prove in a precise sense the heuristic that permutons, which are limiting objects of permutations, do not depend on cycle type.

\end{abstract}

\maketitle

\tableofcontents

\section*{Acknowledgements}
\label{Acknowledgements}

The authors are grateful to Michael Coopman, Yeor Hafouta, Mohammed Slim Kammoun, Gene Kim, Toby Johnson, Kevin Liu, 
Moxuan Liu,
 Arnaud Marsiglietti, 
James Pascoe, 
Bruce Sagan, John Stembridge,
and Yan Zhuang for helpful conversations, to Valentin F\'eray, Christian Gaetz, and Dan Rockmore for their guidance on the previous literature and to Tony Mendes for sharing his ribbon drawing code.
We are especially grateful to Eric Ramos, who helped us conceptualize the project at its inception.
Z. Hamaker was partially supported by NSF Grant DMS-2054423.
B. Rhoades was partially supported by NSF Grant DMS-1953781 and DMS-2246846.

\mainmatter
\chapter{Introduction}
\label{Introduction}

\section{Overview}
Viewed as smooth functions, polynomials are fundamental due to their ubiquity and ease of analysis.
When studying smooth functions $f: S^1 \to \CC$ instead, periodic functions play an analogous role.
Taylor series and the Fourier transform show these subfamilies can approximate all functions in their reference class with arbitrary precision.
More generally, mathematicians encountering a family of functions must first aim to identify a subfamily that strikes this balance between simplicity and expressiveness.
This is a necessary prelude to analysis.
When accomplished successfully, great rewards follow as demonstrated by the classical theories mentioned above and, as a more recent success, the numerous beautiful consequences obtained by approximating Boolean functions with multivariate square--free polynomials (for a textbook treatment see~\cite{ODonnell} and for examples see, e.g., ~\cite{FB,KKL}).

We study functions $f: \symm_n \to \CC$ on permutations, which are ubiquitous in mathematics, statistics, computer science and social choice theory.
Our work makes critical use of two subfamilies of functions, class functions and low-degree functions.
Class functions, also known as characters, are invariant under conjugation of the input.
They date back to work of Frobenius and are perhaps the most fundamental objects of study in the representation theory of groups.
The notion of degree for functions of permutations is a comparatively recent development, having first been applied implicitly in work of Diaconis from the 1980's~\cite{Diaconis1,Diaconis2}.
While the original definition of degree is representation theoretic, it can also be interpreted combinatorially~\cite{DFLLV,EFP}:

\begin{definition}
\label{k-local-definition}
For $i,j \in [n] = \{1,2,\dots,n\}$, define $\one_{ij}: \symm_n \to \CC$ by $\one_{ij}(w) = 1$ if $w(i) = j$ and $0$ if $w(i) \neq j$.
The \emph{degree} of $f: \symm_n \to \CC$ is the minimal degree needed to express $f$ as a polynomial in the $\one_{ij}$'s.
\end{definition}

Equivalently, for positive integer sequences $I= (i_1,\dots,i_k)$ and $J=(j_1,\dots,j_k)$ in $[n]^k$, define $\one_{IJ}: \symm_n \to \CC$ by
\[
\one_{IJ} = \one_{i_1 j_1} \cdot \one_{i_2j_2} \cdot \ldots \cdot \one_{i_k j_k}
\]
where $\cdot$ is pointwise multiplication of functions.
This expression can be abbreviated or will vanish  when $I$ or $J$ has repeated entries.
We call pairs $(I,J)$ with no repetitions \emph{partial permutations}, and let $\symm_{n,k}$ denote the set of partial permutations with $k$ entries in $[n]$. 
It turns out (see Observation~\ref{lower-closure}) $f:\symm_n \to \CC$ has degree at most $k$ if
\[f \in \mbox{span}(\one_{IJ}:(I,J) \in \symm_{n,k}).
\]

Identifying $f$ as a linear combinations of $\one_{IJ}$'s does not determine degree, so we introduce the term \emph{$k$-local} for functions of degree at most $k$.
This nomenclature is also useful since the term `degree' occurs in several other contexts in our work.

Beginning with Ellis, Friedgut and Pilpel's groundbreaking proof of the Deza--Frankl conjecture, an Erd\H{o}s--Ko--Rado  property for sets of permutations~\cite{EFP}, analysis by degree has been a key tool in the solution of further Erd\H{o}s--Ko--Rado type problems~\cite{EL,Lindzey,KLM}.
Structural properties of approximately local $\{0,1\}$--valued functions and local almost $\{0,1\}$--valued functions are investigated in~\cite{EFF1,EFF2,Filmus}.
Degree is related to other notions of complexity for functions of permutations in~\cite{DFLLV}.
Local functions occur organically in probability and social choice theory as demonstrated in Diaconis's aforementioned foundational work, with a recent application to Markov chain mixing~\cite{Dubail}.
Approximation using local functions has become an invaluable tool for dimension reduction in the machine learning literature, with early examples~\cite{HGG,KHJ}.

The vast majority of permutation statistics studied in the combinatorics and probability literatures are manifestly local.
Motivated by this observation, we introduce a powerful new approach for analyzing probabilistic and asymptotic properties of local statistics.
Every fucntion $\stat: \symm_n \to \CC$ has an associated class function $R\, \stat:\symm_n \to \CC$ defined by
\begin{equation}
R\, \stat(w) = \frac{1}{n!}\sum_{v \in \symm_n} \stat(v^{-1} w v),
\label{eq:intro-reynolds}	
\end{equation}
which computes conditional expectations of $\stat$ on all cycle types simultaneously.
Here $R$ is the linear \emph{Reynolds operator}.
Therefore when $\stat$ is $k$-local, we have 
\begin{equation}
	\label{eq:intro-stat}
	\stat = \sum_{(I,J) \in \symm_{n,k}} c_{IJ} \one_{IJ} \quad \mbox{so} \quad R\,\stat = \sum_{(I,J) \in \symm_{n,k}} c_{IJ} R\,\one_{IJ}.
\end{equation}
We show $R\, \one_{IJ}$ 
is determined by the isomorphism type
of the directed graph on 
$[n] := \{1, \dots, n \}$
corresponding to the partial permutation $(I,J)$, 
so 
many of the $R \, \one_{IJ}$ are equal.
Therefore, by grouping like terms and tallying 
$c_{IJ}$'s, we reduce the problem of computing the 
conditional expectation of $\stat$ on all cycle types 
to understanding $R\, \one_{IJ}$ for unlabeled 
graphs associated to partial permutations.
Such graphs are disjoint unions of directed paths and cycles.

Our analysis of $R \, \one_{IJ}$ uses
the machinery of symmetric functions.
The Frobenius isomorphism
$\ch_n: \Class(\symm_n,\CC) \xrightarrow{\, \sim \,} \Lambda_n$
identifies complex-valued class functions on $\symm_n$
with homogeneous symmetric functions of degree $n$.
Since $R\, \one_{IJ}$ is a class function, 
we have a corresponding symmetric function 
$\ch_n(R \, \one_{IJ})$.
Up to a scalar, we call $\ch_n(R \, \one_{IJ})$ 
the {\em atomic symmetric function}
$A_{n,I,J}$.
Like $R \, \one_{IJ}$, the atomic symmetric function
$A_{n,I,J}$ depends only on the unlabeled graph
of the partial permutation $(I,J)$.
We introduce a novel basis for the ring of symmetric functions called the \emph{path power sums} $\{\vec{p}_\lambda\}$ and show $A_{n,I,J}$ is the product of a path power sum and a usual power sum;
this corresponds precisely to the 
decomposition of the graph of $(I,J)$ into 
paths and cycles.
The most technically challenging contribution of our work is the \emph{path Murnaghan--Nakayama formula}, which computes the Schur expansion of a path power sum and has further interpretation in the character theory of $\symm_n$.
The path Murnaghan--Nakayama formula combines with the
usual Murnaghan--Nakayama formula to give 
the Schur expansion of the atomic function
$A_{n,I,J}.$

Recall that irreducible characters $\chi^\lambda:\symm_n \to \CC$ of $\symm_n$ are in one-to-one correspondence with partitions $\lambda$ of $n$.
By linear extension, we may regard $\chi^\lambda$ as a function on the group algebra $\CC[\symm_n]$.
Given any group algebra element
$ \sum_{w \in \symm_n} \alpha_w \cdot w \in \CC[\symm_n]$, one can request a `formula' for
\[
\chi^\lambda\left(\sum_{w \in \symm_n} \alpha_w \cdot w\right) = \sum_{w \in \symm_n} \alpha_w \cdot \chi^\lambda(w).
\]
In general, such character evaluation problems are intractable.
However, our Schur expansion of the atomic function
$A_{n,I,J}$
solves this problem for group algebra elements of the form $\sum_{w \in \symm_n}\one_{IJ}(w) \cdot w$ as a signed count of ribbon tilings for $\lambda$ reminiscent of but fundamentally distinct from the classical Murnaghan--Nakayama rule.
This expression is sufficiently tractable to understand stability phenomena as $n \to \infty$.

In the setting of class functions the path Murnaghan--Nakayma formula allows us to expand $R\, \one_{IJ}$ into the basis of irreducible characters for any partial permutation $(I,J)$.
The aforementioned stability phenomena as $n \to \infty$ extend to the expansion of $R\,\one_{IJ}$.
The theory of character polynomials, introduced by Specht in~\cite{Specht} but previously applied by Murnaghan~\cite{Murnaghan}\footnote{Our chronology follows that of~\cite{GG}.}, shows the $\chi^\lambda$'s appearing in this expansion are polynomials in the cycle counts for cycles of length at most $k$ and do not depend on $n$ for $n$ sufficiently large.

Returning to~\eqref{eq:intro-stat}, for $w \in \symm_n$ we see $R\, \stat(w)$ only depends on the counts for cycles of lengths at most $k$ in $w$ when
$\stat: \symm_n \rightarrow \CC$ is $k$-local.
Since degree is submultiplicative, this observation alone has profound implications on all moments of $\stat$.
However, for the strongest possible results, we also need methods to group $R\, \one_{IJ}$'s in~\eqref{eq:intro-stat} that are the same and tally the $c_{IJ}$'s.
We introduce a new family of \emph{regular permutation statistics}, which are tailor-made to facilitate these computations and include nearly all weighted pattern counting statistics currently in the literature.
Using this outline, we prove that moments of regular statistics are rational functions in $n$ and the short cycle counts with predictable denominator.
As a consequence we show a large subfamily of regular statistics including classical and vincular pattern counts, when applied to a uniformly random permutations drawn from a sequence of cycle types, satisfy a weak law of large numbers depending only on the limiting proportion of fixed points.
Moreover, in the limit their variance depends only on the limiting proportion of fixed points and two cycles.
Subsequently, we make rigorous the heuristic that cycle type does not interact with the permuton structures of~\cite{HKMRS}.

We present an example of the above outline:
\begin{example}
\label{ex:intro-maj}
For $w \in \symm_n$, recall 
that an index $1 \leq i \leq n-1$ 
is a \emph{descent} in $w$ if $w(i) > w(i+1)$.
The \emph{major index}	$\maj:\symm_n \to \CC$ is the sum of descents.
In terms of the $\one_{IJ}$ we have the 
equivalent expression
\begin{equation}
\label{eq:maj-example-one}
\maj = \sum_{\substack{1 \, \leq \, i \, \leq \, n-1 \\ j \, < \,k}} i \cdot \one_{(i,i+1)(k,j)}
\end{equation}
so $\maj$ is $2$--local.

The graphs on $[n]$ 
associated to the partial permutations
appearing in Equation~\eqref{eq:maj-example-one}
have two edges: $i \rightarrow k$ and $i+1 \rightarrow j$.
There are several possibilities for the isomorphism
type of such a graph depending on the values
of $i, j,$ and $k$: in `most' situations we will 
have two paths of size 2 and $n-4$ paths of size 1,
but for special values of $j$ and $k$
we could also have one path of size 3 and $n-3$
paths of size 1 or one 2-cycle and $n-2$ paths of size 1.
We group terms appearing on the RHS of 
Equation~\eqref{eq:maj-example-one}
into separate sums such that the isomorphism
type of the directed graphs within each sum is the same:
\begin{align}
\nonumber
\maj =& \sum_{i} i \cdot \one_{(i,i+1)(i+1,i)} + \sum_{i < i+1 < j} i \cdot \one_{(i,i+1)(j,i)} + \sum_{i < i+1 < j} i \cdot \one_{(i,i+1)(j,i+1)} \\
&+ \sum_{i <j} j \cdot \one_{(j,j+1),(j,i)} + \sum_{i <j} j \cdot \one_{(j,j+1)(j+1,i)} +
\sum_{1 \leq i < i+1 < j < k} i \cdot \one_{(i,i+1)(k,j)} \label{eq:maj-example-two} \\ 
\nonumber
&+ \sum_{i < j < j+1 <k} j \cdot \one_{(j,j+1)(k,i)} + \sum_{i < j < k < k+1} k \cdot \one_{(k,k+1)(j,i)}.
\end{align}
Each of the eight
sums in \eqref{eq:maj-example-two}
is an example of a {\em constrained translate statistic}
on $\symm_n$; we introduce and study these statistics
in Chapter~\ref{Regular}.
Constrained translates are sums of indicator
statistics $\one_{IJ}$ weighted by polynomials 
in the indices $I$ and the values $J$;
they can feature requirements that positions
and/or values be consecutive
(here, we weight by the smaller of the two 
indices in $I$, require that the two indices in $I$
be consecutive, and require that the two values 
in $J$ be decreasing).

The next step in our method for analyzing the 
expectation of $\maj$ on conjugacy classes
is to turn \eqref{eq:maj-example-two}
into a class function by applying $R$ to both sides.
The class 
function $R \, \one_{IJ}$ is the same within
each summand on the RHS.
Therefore, by straightforward counting arguments we have
\begin{align}
\nonumber
R\,\maj =& \frac{(n)_2}{2} R\,\one_{(1,2)(2,1)} + \frac{(n)_3}{6} R\, \one_{(1,2),(3,1)} + \frac{(n)_3}{6}R\, \one_{(1,2)(3,2)} \\
\label{eq:maj-example-three}
&+ \frac{(n)_3}{3}  R\, \one_{(2,3)(2,1)} + \frac{(n)_3}{3} R\, \one_{(2,3)(3,1)} +
\frac{(n)_4}{24} R\,\one_{(1,2)(4,3)}\\
\nonumber
&+ \frac{(n)_4}{12} R\, \one_{(2,3)(4,1)} + \frac{(n)_4}{8} R\, \one_{(3,4)(2,1)},
\end{align}
where $(n)_k = n(n-1)\dots (n-k+1)$ is the falling factorial.
Observe that Equation~\eqref{eq:maj-example-three}
has only finitely many (eight) terms.
Furthermore, we have
\begin{equation}
    R \, \one_{(1,2),(3,1)} = 
    R \, \one_{(1,2),(3,2)} =
    R \, \one_{(2,3),(2,1)} = 
    R \, \one_{(2,3),(3,1)}
\end{equation} 
and
\begin{equation}
    R \, \one_{(1,2),(4,3)} = 
    R \, \one_{(2,3),(4,1)} = 
    R \, \one_{(3,4),(2,1)}
\end{equation}
since the directed graphs coresponding to these partial
permutations are isomorphic.
The expansion in Equation~\eqref{eq:maj-example-three}
therefore involves only three distinct class functions.

Applying the Frobenius isomorphism to each $R\, \one_{IJ}$ to obtain $A_{n,I,J}$, writing each $A_{n,I,J}$ as a product of path power sums
and classical power sums, and combining like terms, 
we have
\begin{multline}
\label{eq:maj-example-four}
\ch_n (R\, \maj) = \\ \frac{1}{n!}\left(\frac{(n)_2}{2} \cdot  \vec{p}_{1^{n-2}} p_2 
+ \frac{(n)_3}{2} \cdot  \vec{p}_{21^{n-3}} p_1
+ \frac{(n)_3}{2} \cdot  \vec{p}_{31^{n-3}} 
+ \frac{(n)_4}{4} \cdot \vec{p}_{2^21^{n-4}}
\right).
\end{multline}
Applying the Path Murnaghan--Nakayama formula as in Example~\ref{ex:intro-stability} and the usual Murnaghan--Nakayama as in~\eqref{eq:intro-MN}, we compute
the Schur expansion
of \eqref{eq:maj-example-four} to obtain
\begin{align}
\nonumber
\ch_n(R\,\maj) =&\  
\frac{1}{2}\left(-s_{n-2,1,1} + s_{n-2,2} + s_{n} \right) \\
\nonumber
& + \frac{1}{2} \left(
- s_{n-2,1,1} - s_{n-2,2} + (n-3) \cdot s_{n-1,1} + (n-2) \cdot s_n
\right) \\
\label{eq:maj-example-five}
& + \frac{1}{2} 
\left( s_{n-2,1,1} - s_{n-2,2} - s_{n-1,1} + (n-2) \cdot  s_n \right)\\
\nonumber
& + \frac{1}{2}\left(s_{n-2,2} - (n-3) \cdot s_{n-1,1} + \binom{n-2}{2} \cdot s_{n} \right)\\
\nonumber
=& \frac{n(n-1)}{4} \cdot s_n 
- \frac{1}{2} \cdot s_{n-1,1} 
- \frac{1}{2} \cdot s_{n-2,1,1}
\end{align}
Observe that for any term $s_\lambda$ appearing in 
\eqref{eq:maj-example-five}, the partition
$\lambda \vdash n$ has $\leq 2$ boxes outside the 
first row. 
This is not an accident, and reflects the fact that
$\maj$ is a 2-local statistic.
Also observe that the coefficient of each $s_\lambda$
is a rational function 
$f_\lambda(n)$ in $n$, and that the 
degree of $f_\lambda(n)$ is maximized when 
$\lambda = (n)$. Again, this is a general phenomenon.

Replacing each Schur function $s_\lambda$ in
Equation~\eqref{eq:maj-example-five} with the 
corresponding irreducible character
$\chi^\lambda$ expresses $R \, \maj$ as 
a linear combination of irreducible characters.
While such an expansion is interesting,
for combinatorial and probabilistic purposes
one often wants a polynomial expansion of
$R \, \maj$ in terms of cycle counts.
To achieve this, we apply character polynomials
as follows.

For $w \in \symm_n$, let $m_i(w)$ 
be the number of $i$--cycles in $w$.
We use the Schur expansion of $\ch_n(R \, \maj)$
computed above to express $R \, \maj$ as a polynomial
function of $n$ and the $m_i$.
To this end,
recall that 
$\chi^\lambda = \ch_n^{-1}(s_\lambda)$.
From the theory of character polynomials
(see e.g. \cite[p. 323]{Specht} or
\cite[Table 5.1]{GG}), 
we have 
\begin{equation}
\label{eq:maj-example-six}
\chi^{(n)} = 1,\ \chi^{(n-1,1)} = m_1 - 1,\ \mbox{and} \ \chi^{(n-2,1,1)} = \binom{m_1}{2} - 
\binom{m_2}{1} - \binom{m_1}{1} + 1
\end{equation}
as functions $\symm_n \rightarrow \CC$
where  
$\binom{x}{k} := \frac{x(x-1) \cdots (x-k+1)}{k!}$
for a variable $x$.
Therefore we conclude
\begin{equation}
R\, \maj =\  \frac{n(n-1)}{4} - \frac{1}{4} \cdot m_1^2 + \frac{1}{2} \cdot m_2 + \frac{1}{4}\cdot m_1.
\end{equation}
Note the leading term by degree\footnote{Here we 
use the na\"ive notion of degree which takes
$n$ and all of the $m_i$ to have degree 1.
This is different from the more typical
notion of degree in character polynomials which
assigns $m_i$ to degree $i$.} 
is $(n^2 - m_1^2)/4$, 
so for $\{\lambda^{(n)}\}$ a sequence of partitions with $\lim_{n \to \infty} m_1(\lambda^{(n)})/n = \alpha$ and $\Sigma_n$ a uniformly random element of $K_{\lambda^{(n)}}$ we have 
\[
\lim_{n \to \infty} \EE[ \maj(\Sigma_n) ] = (n^2 - \alpha^2)/4 + O(n).
\]

For probabilistic reasons, one often 
needs to understand higher moments of statistics
in addition to their expectations.
A similar but far more involved computation 
for the second moment $\maj^2$
shows that
\begin{align}
\nonumber
\ch_n \, R \, \maj^2 = & 
\frac{9 n^4 - 14 n^3 + 15 n^2 - 10 n}{144} \cdot s_n 
+ \frac{-3 n^2 + 3n + 8}{12} s_{n-1,1} \\
\label{eq:maj-example-seven}
&+ \frac{7}{6} s_{n-2,2} + 
\frac{-3 n^2 + 3n + 8}{12} s_{n-2,1,1}  
+ \frac{7}{6} s_{n-2,2,1} + \frac{1}{2} s_{n-3,1^3} \\
\nonumber
&+ \frac{1}{2} s_{n-4,2,2} + \frac{1}{2} s_{n-4,1^4}.
\end{align}
The Schur expansion \eqref{eq:maj-example-seven}
is quite complicated, 
but some structure remains.
Observe that $s_\lambda$ appears 
only if $\lambda$ has $\leq 4$ boxes outside the first
row. This reflects the fact that the statistic
$\maj^2$ is 4-local;
in general, the pointwise product of a $k$-local 
statistic and an $\ell$-local statistic is
$(k + \ell)$-local. Furthermore, observe that 
the coefficients are 
polynomials in $n$, with maximal degree achieved
when $\lambda = (n)$.

As in the case of $R \, \maj$, we can use 
Equation~\eqref{eq:maj-example-seven}
and character polynomials
to write 
$R \, \maj^2$ as a polynomial in $n$ and the $m_i$.
Consulting the character polynomial
data in \cite[p. 323]{Specht} yields
\begin{align}
\nonumber
    R \, \maj^2 = &
    \frac{1}{16} \cdot m_1^4 - 
    \frac{1}{8} \cdot m_1^2n^2 
    + \frac{1}{16} \cdot n^4 - \frac{11}{72} \cdot m_1^3 - 
    \frac{1}{4} \cdot m_1^2 m_2 + \frac{1}{8} \cdot m_1^2 n + \frac{1}{8} \cdot m_1 n^2 
    \\&+ 
    \frac{1}{4} \cdot m_2 n^2  - 
    \frac{7}{72} \cdot n^3 + 
    \frac{1}{48} \cdot m_1^2 + 
    \frac{1}{4} \cdot m_1 m_2 + 
    \frac{3}{4} \cdot m_2^2 - \frac{1}{8} \cdot m_1 n - 
    \frac{1}{4} \cdot m_2 n 
    \label{eq:maj-example-eight}
    \\&+ 
    \frac{5}{48} \cdot n^2 + 
    \frac{5}{72} \cdot m_1 - \frac{3}{4} \cdot m_2 - 
    \frac{2}{3} \cdot m_3 - \frac{1}{2} \cdot m_4 - 
    \frac{5}{72} \cdot n
    \nonumber
\end{align}
for the conditional expectation of $\maj^2$
on a given conjugacy class.
The variance of $\maj$ on a conjugacy class
is therefore the somewhat more compact expression
\begin{align}
    R \, \maj^2 - \left( R \, \maj \right)^2 =
    &-\frac{1}{36} \cdot m_1^3 +
    \frac{1}{36} \cdot n^3 - \frac{1}{24} \cdot m_1^2 + 
    \frac{1}{2} \cdot m_2^2 + \frac{1}{24} \cdot n^2 
    \nonumber
    \\ &+ 
    \frac{5}{72} \cdot m_1 - \frac{3}{4} \cdot m_2 - 
    \frac{2}{3} \cdot m_3 -
    \frac{1}{2} \cdot m_4 - \frac{5}{72} \cdot n.
    \label{eq:maj-example-nine}
\end{align}
The leading term by degree of the variance 
is therefore $(n^3 - m_1^3)/36$.
This variance computation is, to our knowledge, previously unknown.
\end{example}

\section{Partial Permutations, Local Functions and Class Functions}

The combinatorial Definition~\ref{k-local-definition} for low degree functions on permutations has a representation-theoretic reformulation in terms of the Fourier transform (or 
Artin-Wedderburn Theorem).
Recall that irreducible representations $V^{\lambda}$ 
of $\symm_n$ over $\CC$ are parametrized by partitions $\lambda \vdash n$.
Writing $\End_{\CC}(V^{\lambda})$ for the space of $\CC$-linear maps $V^{\lambda} \rightarrow V^{\lambda}$,
the natural map
$\Psi: \CC[\symm_n] \rightarrow \bigoplus_{\lambda \vdash n} \End_{\CC}(V^{\lambda})$ is an $\CC$-algebra isomorphism.
The space of functions $\Fun(\symm_n,\CC)$ of the form $f: \symm_n \rightarrow \CC$ can be identified with the group algebra $\CC[\symm_n]$ via $\alpha_f := \sum_{w \in \symm_n} f(w) \cdot w \in \CC[\symm_n]$.

\begin{theorem}
\label{aw-local-preview} 
{\em (\cite{Diaconis2,EFP}, $=$ Corollary~\ref{aw-local})}
Let $f: \symm_n \rightarrow \CC$ be a function with $\Psi(\alpha_f) = \left( A^{\lambda} \right)_{\lambda \vdash n}$.
Then $f$ is $k$-local if and only if $A^{\lambda} = 0$ whenever $\lambda_1 < n-k$.
\end{theorem}

The forward direction of Theorem~\ref{aw-local-preview}  has appeared in various guises throughout the literature. 
It is implicit in Diaconis \cite{Diaconis2} and appeared more explicitly in Huang et. al. \cite[Appendix C]{HGG}.
The converse appears as~\cite[Thm. 7]{EFP}, while Even-Zohar \cite[Lemma 3]{EZ} proved a result which is `dual' to Theorem~\ref{aw-local-preview}.
Since this result is less known in the enumerative and algebraic combinatorics communities, we give a self-contained treatment of this theory in Chapter~\ref{Local}.
We include an exposition of the second author's recent basis for $k$--local functions whose terms are certain $\one_{IJ}$'s defined in terms of Viennot's shadow construction from RSK.

Since the irreducible characters form a basis for the vector space $\Class(\symm_n,\CC)$ of class functions, for $f \in \Class(\symm_n,\CC)$ there are unique coefficients $c_{f,\lambda} \in \CC$ so that
\begin{equation}
f = \sum_{\lambda \vdash n} c_{f,\lambda} \cdot \chi^{\lambda} \quad \quad \text{as maps $\symm_n \rightarrow \CC$.}
\end{equation}
Characters have the benefit of being defined independent of choice of basis for $\symm_n$-modules.
Since each irreducible character $\chi^\lambda$ is supported on $V^\lambda$, Theorem~\ref{aw-local-preview} implies for $f$ a class function that
\begin{equation}
\text{degree of $f$} = \max  \{ n - \lambda_1 \,:\,  \lambda \vdash n \text{ and } c_{f,\lambda} \neq 0 \}
\end{equation}
(see also Corollary~\ref{character-local}).

We are able to prove stronger results for the $k$-local maps $f: \symm_n \rightarrow \CC$ that are also class functions thanks to the 
powerful machinery
of symmetric functions. 
The Reynolds operator $R$ is a projection operator
\begin{equation}
R: \Fun(\symm_n,\CC) \twoheadrightarrow \Class(\symm_n,\CC).
\end{equation}
The map $R\,f$ is the closest class function approximation to $f$.
Writing $K_{\lambda} \subseteq \symm_n$ for the conjugacy class of cycle type $\lambda$, the map $R\,f$ sends any $w \in K_{\lambda}$ to the average value of $f$ on $K_{\lambda}$.
As mentioned after~\eqref{eq:intro-reynolds}, this is the expected value of $f$ for the uniform distribution on $K_{\lambda}$.

We study the class function approximations $Rf$ of maps $f: \symm_n \rightarrow \CC$ using symmetric function theory; see 
Chapter~\ref{Background} for the relevant definitions and Chapter~\ref{Atomic} for a complete exposition.
Let $\Lambda$ be the ring of symmetric functions and $\Lambda_n$ its $n$th graded component.
Then \emph{Frobenius characteristic isomorphism} is the map
\begin{equation}
 \ch_n: \Class(\symm_n,\CC) \xrightarrow{ \, \sim \, } \Lambda_n \quad \mbox{with}\quad \ch_n(\chi^\lambda) = s_\lambda.
 \end{equation}
Here $s_\lambda$ is a {\em Schur function}.

From Definition~\ref{k-local-definition} we know the indicator functions $\{ \one_{IJ} \,:\, (I,J) \in \symm_{n,k} \}$ are the building blocks of $k$-local statistics. 
As observed in~\eqref{eq:intro-stat}, we can understand arbitrary local statistics via the $\one_{IJ}$.
To this end, we introduce the following class of symmetric functions.

\begin{defn}
\label{atomic-symmetric-function-definition}
Let $(I,J) \in \symm_{n,k}$ be a partial permutation. The {\em atomic  symmetric function} is
\begin{equation}
\label{eq:intro-A}
A_{n,I,J} := n! \cdot \ch_n(R \, \one_{IJ}) = 
\sum_{\lambda \vdash n} \left( \sum_{w \in K_{\lambda}} \one_{I,J}(w) \right) \cdot p_{\lambda}
\end{equation}
where $p_{\lambda} \in \Lambda_n$ is the power sum symmetric function.
\end{defn}

The normalization factor
of $n!$ in 
Definition~\ref{atomic-symmetric-function-definition}
simplifies our formulas.
The symmetric group $\symm_n$ acts on the family
of $k$-tuples
of distinct elements in $[n]$ by the rule
\begin{equation}
w(I) = (w(i_1), \dots, w(i_k)) \quad \quad 
\text{for $w \in \symm_n$ and $I = (i_1, \dots, i_k).$}
\end{equation}
Then for $(I,J) \in \symm_{n,k}$, by definition we have 
\begin{equation}
R\,\one_{IJ} = R\,\one_{w(I)w(J)} \quad \mbox{so} \quad A_{n,I,J} = A_{n,w(I)w(J)}.
\end{equation}
The underlying structure of $(I,J)$ that determines $A_{n,I,J}$ is its directed \emph{graph} $G_n(I,J)$ with vertex set $[n]$ and edges 
$i_1 \rightarrow j_1, \dots, i_k \rightarrow j_k.$
Up to isomorphism $G_n(I,J)$ is characterized by the \emph{cycle-path type} of $(I,J)$, which is the pair of partitions $(\nu,\mu)$ of cycle lengths and path lengths, respectively, with isolated vertices viewed as paths of length one.
The cycle-path type plays the role of cycle type for partial permutations; a closely related concept appears in the setting of the symmetric association scheme of injections~\cite{BBIT}.

Let $(I,J) \in \symm_n$ with cycle-path type $(\nu,\mu)$.
A careful examination of~\eqref{eq:intro-A} shows we can factor $p_\nu$ out of $A_{n,I,J}$.
The resulting symmetric function, which depends only on the partition $\mu$ of path lengths, is the {\em path power sum} $\vec{p}_\mu$.
Path power sums are a new basis for $\Lambda$ and have relatively straightforward descriptions in terms of power sum symmetric functions (Proposition~\ref{path-to-classical}) and monomial symmetric functions (Proposition~\ref{path-to-monomial}).


\section{The Path Murnaghan--Nakayama Formula}

The classical Murnaghan--Nakayama formula computes the product $s_\lambda \cdot p_k$.
A \emph{ribbon} $\xi$ is an edgewise-connected skew partition whose Young diagram does not contain any $2 \times 2$ square.
The \emph{height} $\height(\xi)$ of the ribbon $\xi$ is the number of rows occupied by $\xi$ minus one and the \emph{sign} of $\xi$ is $\sign(\xi) = (-1)^{\height(\xi)}$.
\begin{theorem}[Classical Murnaghan--Nakayama rule]
\label{t:intro-mn}
	For $\lambda$ an integer partition and $k$ a positive integer, 
	\[
	s_\lambda \cdot p_k = \sum_{\mu:\ |\mu| - |\lambda| = k,\ \mu/\lambda \mathrm{\ is\ a\  ribbon}} \sign(\mu/\lambda) s_\mu.
	\]
\end{theorem}
For example, the fact that $s_{3} \cdot p_2 = s_5 + s_{32} - s_{311}$ is witnessed by the tableaux
\begin{equation}
\label{eq:intro-MN}
\begin{tikzpicture}[scale=.25]
 \begin{scope}
    \clip (0,0) -| (5,1) -| (0,0);
    \draw [color=black!25] (0,0) grid (5,1);
  \end{scope}

  \draw [thick] (0,0) -| (5,1) -| (0,0);

  \draw [thick, rounded corners] (3.5,0.5) -- (4.5,0.5);
  \draw [color=black,fill=black,thick] (4.5,0.5) circle (.4ex);
  \node [draw, circle, fill = white, inner sep = 1pt] at (3.5,0.5) { };

   \begin{scope}
    \clip (11,-1) -| (13,0) -| (14,1) -| (11,-1);
    \draw [color=black!25] (11,-1) grid (14,1);
  \end{scope}

  \draw [thick] (11,-1) -| (13,0) -| (14,1) -| (11,-1);
  \draw [thick, rounded corners] (11.5,-0.5) -- (12.5,-0.5);
\draw [color=black,fill=black,thick] (12.5,-0.5) circle (.4ex);
  \node [draw, circle, fill = white, inner sep = 1pt] at (11.5,-0.5) { };
  
 \begin{scope}
    \clip (20,1) -| (23,0) -| (21,-2) -| (20,1);
    \draw [color=black!25] (20,-1) grid (26,1);
  \end{scope}

  \draw [thick] (20,1) -| (23,0) -| (21,-2) -| (20,1);

  \draw [thick, rounded corners] (20.5,-0.5) -- (20.5,-1.5);
  \draw [color=black,fill=black,thick] (20.5,-0.5) circle (.4ex);
  \node [draw, circle, fill = white, inner sep = 1pt] at (20.5,-1.5) { };
\end{tikzpicture}.
\end{equation}
Note this computation extends to 
$s_{n-2} \cdot p_2 = s_n + s_{n-1,2} - s_{n-2,1,1}$ 
as used in Example~\ref{ex:intro-maj}.

For $\lambda = (\lambda_1,\dots,\lambda_k)$ a partition, repeated application of the Murnaghan--Nakayma rule for the product $p_\lambda = p_{\lambda_1} \cdots p_{\lambda_k}$ gives the expansion of $p_\lambda$ into the Schur basis as a signed sum of ribbon tilings, which are sequences 
\[\varnothing = \mu^0 \subseteq \mu^1 \subseteq \dots \subseteq \mu^k
\]
so that $\mu^i/\mu^{i-1}$ is a ribbon of size $\lambda_i$.
See Theorem~\ref{mn-rule} for a precise statement.
At the level of representation theory, for $w \in K_\lambda$ the coefficient of $s_\mu$ in $p_\lambda$ computes $\chi^\mu(w)$, so the Murnaghan--Nakayama rule evaluates irreducible characters.

As demonstrated in Example~\ref{ex:intro-maj}, our outline requires computing the Schur expansion of $A_{n,I,J}$ for every $(I,J) \in \symm_{n,k}$, which we do in Chapter~\ref{Path}.
With $(\nu,\mu)$ the cycle-path type of $(I,J)$, we have $A_{n,I,J} = p_\nu \cdot \vec{p}_\mu$ so the classical Murnaghan--Nakayama rule reduces this problem to expanding $\vec{p}_\mu$ into the Schur basis.

The expansion of $\vec{p}_\mu$ in the Schur basis may
be computed na\"ively using the Murnaghan-Nakayama Rule
as follows.
If $\mu = (\mu_1, \dots \mu_r)$ has $r$ parts, 
we show (Proposition~\ref{path-to-classical}) that
the 
path power sum $\vec{p}_\mu$ expands into classical
power sums according to the rule
\begin{equation}
    \label{intro-path-to-classical}
    \vec{p}_\mu = \sum_{\sigma \in \Pi_r}
    \prod_{B \in \sigma} (|B| - 1)! \cdot p_{\mu_B},
\end{equation}
where
\begin{itemize}
    \item the sum is over the family $\Pi_r$ of set 
    partitions of $[r]$,
    \item the product is over blocks $B$ of the set 
    partition $\sigma$, and
    \item $\mu_B = \sum_{i \in B} \mu_i$ is the sum
    of the parts of $\mu$ indexed by elements of the block
    $B$.
\end{itemize}
The classical Murnaghan-Nakayama formula may 
be applied to find the $s$-expansion
of each term on the right-hand side of 
\eqref{intro-path-to-classical}.
However, since the family $\Pi_r$ of set partitions 
of $[r]$ grows rapidly with $r$, this involves 
a huge number of ribbon tableaux as $r$ grows.
We will be interested in the $s$-expansion of $A_{n,I,J}$
in the limit $n \rightarrow \infty$, which corresponds
to adding parts of size $1$ to the path partition $\mu$,
thus increasing $r$.
Fortunately, there is a massive amount of 
cancellation in the $s$-expansion
of the right-hand side of \eqref{intro-path-to-classical},
and this $s$-expansion may be expressed compactly 
using a monotonicity condition on ribbon tilings.

For $T = (\varnothing = \mu^0,\mu^1,\dots,\mu^k = \mu)$ a ribbon tiling with ribbons $\xi^i = \mu^i/\mu^{i-1}$, let 
$\sign(T) = \sign(\xi^1)\cdot \ldots \cdot \sign(\xi^k)$.
Let $\alpha(T) = (|\xi^1|,\dots,|\xi^k|)$ and $\shape(T) = \mu$.
We say $T$ is \emph{monotonic} if the southwest cells $(r_i,c_i)$ of the ribbons $\xi^i$ satisfy 
\begin{center}
$
c_1 < \dots < c_k$ and $r_1 \geq \dots \geq r_k$.
\end{center}
See Figure~\ref{fig:intro-monotonic-and-not} for two ribbon tilings, one of which is monotonic and the other not.
To expand $A_{n,I,J}$ into the Schur basis, or equivalently $R\,\one_{IJ}$ into the irreducible character basis, we use:
\begin{theorem}[Path Murnaghan--Nakayama formula]
\label{t:intro-path-mn}
For $\lambda = (\lambda_1,\dots,\lambda_k)$ a partition,
\[
\vec{p}_\lambda = \sum_{w \, \in \, \symm_k} \sum_{\substack{T \; \mathrm{monotonic} \\ \ w(\alpha(T)) \, = \, \lambda}} \sign(T) \cdot s_{\shape(T)}.
\]
\end{theorem}

The proof of Theorem~\ref{t:intro-path-mn} adapts one of the standard proofs of the Murnaghan--Nakayama formula, but with several additional complexities.
Restricting to a finite (but large) variable set,
the Schur function $s_\lambda$ may be expressed 
as a bialternant, or ratio of determinants.
The proof of Theorem~\ref{t:intro-mn} we adapt proceeds by   multiplying $p_k$ through the numerator of the bialternant and identifying the terms that do not vanish with ribbon 
additions.
We take this same tact, using 
\eqref{intro-path-to-classical} to 
multiply $\vec{p}_\lambda$ 
through the numerator of the bialternant 
(in our setting the numerator is 
simply the Vandermonde determinant).
We give a sign-reversing involution on the non-vanishing terms so obtained, and identify the surviving 
terms with monotonic ribbon tilings.
Since the simple multiplicative structure
$p_{\lambda} = p_{\lambda_1} p_{\lambda_2} \cdots$
is replaced by the more complicated formula
\eqref{intro-path-to-classical}, the proof 
of Theorem~\ref{t:intro-path-mn} is 
substantially more 
involved than that of the classical
Murnaghan-Nakayama formula.

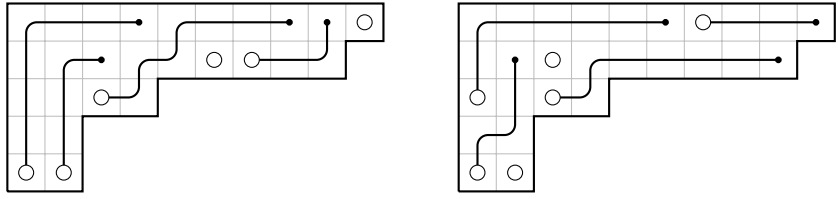
\begin{figure}
\begin{center}
\begin{tikzpicture}[scale = 0.5]

  \begin{scope}
    \clip (0,0) -| (2,2) -| (4,3) -| (9,4) -| (10,5) -| (0,0);
    \draw [color=black!25] (0,0) grid (10,5);
  \end{scope}

  \draw [thick] (0,0) -| (2,2) -| (4,3) -| (9,4) -| (10,5) -| (0,0);

  \draw [thick, rounded corners] (0.5,0.5) |- (3.5,4.5);
  \draw [color=black,fill=black,thick] (3.5,4.5) circle (.4ex);
  \node [draw, circle, fill = white, inner sep = 2pt] at (0.5,0.5) { };
  
  \draw [thick, rounded corners] (1.5,0.5) |- (2.5,3.5);
  \draw [color=black,fill=black,thick] (2.5,3.5) circle (.4ex);
  \node [draw, circle, fill = white, inner sep = 2pt] at (1.5,0.5) { };
  
  \draw [thick, rounded corners] (2.5,2.5) -| (3.5,3.5) -| (4.5,4.5) -- (7.5,4.5);
  \draw [color=black,fill=black,thick] (7.5,4.5) circle (.4ex);
  \node [draw, circle, fill = white, inner sep = 2pt] at (2.5,2.5) { };

    \node [draw, circle, fill = white, inner sep = 2pt] at (5.5,3.5) { };
    
  \draw [thick, rounded corners] (6.5,3.5) -| (8.5,4.5);
  \draw [color=black,fill=black,thick] (8.5,4.5) circle (.4ex);
  \node [draw, circle, fill = white, inner sep = 2pt] at (6.5,3.5) { };   
  
      \node [draw, circle, fill = white, inner sep = 2pt] at (9.5,4.5) { };

   \begin{scope}
    \clip (12,0) -| (14,2) -| (16,3) -| (21,4) -| (22,5) -| (12,0);
    \draw [color=black!25] (12,0) grid (22,5);
  \end{scope}

   \draw [thick] (12,0) -| (14,2) -| (16,3) -| (21,4) -| (22,5) -| (12,0);    
  
  \draw [thick, rounded corners] (12.5,2.5) |- (17.5,4.5);
  \draw [color=black,fill=black,thick] (17.5,4.5) circle (.4ex);
  \node [draw, circle, fill = white, inner sep = 2pt] at (12.5,2.5) { };   
  
  \draw [thick, rounded corners] (12.5,0.5) |- (13.5,1.5) -- (13.5,3.5);
  \draw [color=black,fill=black,thick] (13.5,3.5) circle (.4ex);
  \node [draw, circle, fill = white, inner sep = 2pt] at (12.5,0.5) { };   
  
  \draw [thick, rounded corners] (14.5,2.5) -| (15.5,3.5) -- (20.5,3.5);
  \draw [color=black,fill=black,thick] (20.5,3.5) circle (.4ex);
  \node [draw, circle, fill = white, inner sep = 2pt] at (14.5,2.5) { };   
  
  \draw [thick, rounded corners]  (18.5,4.5) -- (21.5,4.5);
  \draw [color=black,fill=black,thick] (21.5,4.5) circle (.4ex);
  \node [draw, circle, fill = white, inner sep = 2pt] at (18.5,4.5) { };

   \node [draw, circle, fill = white, inner sep = 2pt] at (13.5,0.5) { };

    \node [draw, circle, fill = white, inner sep = 2pt] at (14.5,3.5) { };

\end{tikzpicture}
\end{center}
\caption{A monotonic ribbon tiling (left) and a tiling by ribbons which is not monotonic (right).}
\label{fig:intro-monotonic-and-not}
\end{figure}

\begin{example}
\label{ex:intro-path-to-s}
    To get a flavor of how the monotonicity condition of
    Theorem~\ref{t:intro-path-mn} behaves, 
    we calculate the $s$-expansion
    of the path power sum $\vec{p}_{31^3}$.
    This expansion
    \begin{equation*}
        \vec{p}_{31^3} = 3! \cdot \left( 
        s_{41^2} + s_{51} + s_{42} + 4 s_6
        \right)
    \end{equation*}
    is witnessed by the monotonic tilings
    
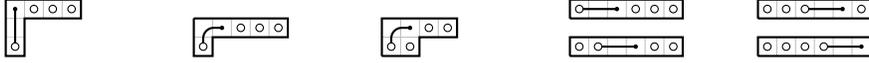
\[
\begin{tikzpicture}[scale=.25]
 \begin{scope}
    \clip (0,3) -| (4,2) -| (1,1) -| (0,3);
    \draw [color=black!25] (0,0) grid (6,3);
  \end{scope}

  \draw [thick] (0,3) -| (4,2) -| (1,0) -| (0,3);

  \draw [thick, rounded corners] (0.5,2.5) -- (0.5,0.5);
  \draw [color=black,fill=black,thick] (0.5,2.5) circle (.4ex);
  \node [draw, circle, fill = white, inner sep = 1pt] at (0.5,0.5) { };
  \node [draw, circle, fill = white, inner sep = 1pt] at (1.5,2.5) { };
  \node [draw, circle, fill = white, inner sep = 1pt] at (2.5,2.5) { };
  \node [draw, circle, fill = white, inner sep = 1pt] at (3.5,2.5) { };
  
\begin{scope}
    \clip (10,2) -| (15,1) -| (11,0) -| (10,0);
    \draw [color=black!25] (10,-1) grid (16,2);
\end{scope}

  \draw [thick] (10,2) -| (15,1) -| (11,0) -| (10,2);

  \draw [thick, rounded corners] (10.5,.5) -- (10.5,1.5) -- (11.5,1.5);
  \draw [color=black,fill=black,thick] (11.5,1.5) circle (.4ex);
  \node [draw, circle, fill = white, inner sep = 1pt] at (10.5,0.5) { };
  \node [draw, circle, fill = white, inner sep = 1pt] at (12.5,1.5) { };
  \node [draw, circle, fill = white, inner sep = 1pt] at (13.5,1.5) { };
  \node [draw, circle, fill = white, inner sep = 1pt] at (14.5,1.5) { };

\begin{scope}
    \clip (20,0) -| (22,1) -| (24,2) -| (20,0);
    \draw [color=black!25] (20,0) grid (24,2);
\end{scope}

\draw [thick] (20,0) -| (22,1) -| (24,2) -| (20,0);

\draw [thick, rounded corners] (20.5,.5) -- (20.5,1.5) -- (21.5,1.5);
  \draw [color=black,fill=black,thick] (21.5,1.5) circle (.4ex);
  \node [draw, circle, fill = white, inner sep = 1pt] at (20.5,0.5) { };
  \node [draw, circle, fill = white, inner sep = 1pt] at (21.5,0.5) { };
  \node [draw, circle, fill = white, inner sep = 1pt] at (22.5,1.5) { };
  \node [draw, circle, fill = white, inner sep = 1pt] at (23.5,1.5) { };

\begin{scope}
    \clip (30,2) -| (36,3) -| (30,2);
    \draw [color=black!25] (30,2) grid (36,3);
\end{scope}

\draw [thick] (30,2) -| (36,3) -| (30,2);  
\draw [thick, rounded corners] (30.5,2.5) -- (32.5,2.5);
\draw [color=black,fill=black,thick] (32.5,2.5) circle (.4ex);
\node [draw, circle, fill = white, inner sep = 1pt] at (30.5,2.5) { };
\node [draw, circle, fill = white, inner sep = 1pt] at (33.5,2.5) { };
\node [draw, circle, fill = white, inner sep = 1pt] at (34.5,2.5) { };
\node [draw, circle, fill = white, inner sep = 1pt] at (35.5,2.5) { };

\begin{scope}
    \clip (30,0) -| (36,1) -| (30,0);
    \draw [color=black!25] (30,0) grid (36,1);
\end{scope}

\draw [thick] (30,0) -| (36,1) -| (30,0);
\draw [thick, rounded corners] (31.5,0.5) -- (33.5,0.5);
\draw [color=black,fill=black,thick] (33.5,0.5) circle (.4ex);
\node [draw, circle, fill = white, inner sep = 1pt] at (31.5,0.5) { };
\node [draw, circle, fill = white, inner sep = 1pt] at (30.5,0.5) { };
\node [draw, circle, fill = white, inner sep = 1pt] at (34.5,0.5) { };
\node [draw, circle, fill = white, inner sep = 1pt] at (35.5,0.5) { };

\begin{scope}
    \clip (40,2) -| (46,3) -| (40,2);
    \draw [color=black!25] (40,2) grid (46,3);
\end{scope}

\draw [thick] (40,2) -| (46,3) -| (40,2);  
\draw [thick, rounded corners] (42.5,2.5) -- (44.5,2.5);
\draw [color=black,fill=black,thick] (44.5,2.5) circle (.4ex);
\node [draw, circle, fill = white, inner sep = 1pt] at (42.5,2.5) { };
\node [draw, circle, fill = white, inner sep = 1pt] at (40.5,2.5) { };
\node [draw, circle, fill = white, inner sep = 1pt] at (41.5,2.5) { };
\node [draw, circle, fill = white, inner sep = 1pt] at (45.5,2.5) { };

\begin{scope}
    \clip (40,0) -| (46,1) -| (40,0);
    \draw [color=black!25] (40,0) grid (46,1);
\end{scope}

\draw [thick] (40,0) -| (46,1) -| (40,0);
\draw [thick, rounded corners] (43.5,0.5) -- (45.5,0.5);
\draw [color=black,fill=black,thick] (45.5,0.5) circle (.4ex);
\node [draw, circle, fill = white, inner sep = 1pt] at (43.5,0.5) { };
\node [draw, circle, fill = white, inner sep = 1pt] at (40.5,0.5) { };
\node [draw, circle, fill = white, inner sep = 1pt] at (41.5,0.5) { };
\node [draw, circle, fill = white, inner sep = 1pt] at (42.5,0.5) { };
\end{tikzpicture}.
\]
The factor of $3!$ arises from the three ribbons 
of size 1.

It is instructive to compare with the na\"ive approach
of computing the $s$-expansion using the classical
Murnaghan-Nakayama formula.
Applying \eqref{intro-path-to-classical}, the $p$-expansion
of $\vec{p}_{31^3}$ is given by
\begin{equation*}
    \vec{p}_{31^3} = 
    p_{31^3} + 3 p_{321} + 2p_{33} + 3p_{411} + 3p_{42} + 6 p_{51} + 6p_6.
\end{equation*}
Term-by-term, the number of 
ribbon tableaux involved 
in the use of classical Murnaghan-Nakayama
to compute the $s$-expansion of the right-hand side is
\begin{equation*}
    52 + 22 + 12 + 28 + 12 + 13 + 6 = 145
\end{equation*}
assuming we add ribbons to the empty shape
from largest to smallest. This immediately reveals
the computational advantages of 
the Path Murnaghan-Nakayama formula.
Furthermore, the na\"ive approach yields 
ribbon tableaux corresponding to many terms
such as $s_{1^6}$ and $s_{2 1^4}$ which are precluded
by the monotonicity condition 
(and ultimately cancel out in subtle ways).
\end{example}

The monotonicity condition for ribbon tilings guarantees that every shape $\lambda$ occurring in the Schur expansion of $\vec{p}_\mu$ with non-zero coefficient is greater than or equal to the transpose $\mu'$ of $\mu$ in dominance order.
Beginning with this observation, 
Theorem~\ref{t:intro-path-mn} 
allows us to prove a stability phenomenon 
for path power sums with many parts of size 1.
Roughly speaking, after a certain threshold,
adding a part of size $1$ to $\mu$ will increase
the length of the first row of each term in the 
$s$-expansion of $\vec{p}_\mu$.
The coefficients in this $s$-expansion will change in 
predictable ways corresponding to how
the horizontal ribbons in this long first row 
intermingle.

More specifically, given $k \leq n$ and a partition 
$\mu  \vdash k$, construct a new partition
$\mu(n)$ by appending $1^{n-k}$ to $\mu$. Note that 
the conjugate
$\mu(n)'$ has first part at least $n-k$.
For $T$ a monotonic ribbon tiling whose ribbon sizes are a permutation of $\mu(n)$, the monotonicity condition guarantees at most $k$ of these ribbons are not contained in the first row.
Therefore, there are only finitely many configurations for the first $k$ columns.
The remaining entries can be placed in polynomially many ways, leading to the corollary:

\begin{corollary}
	\label{c:intro-stability}
	Let $\mu = (\mu_1,\dots,\mu_k)$ be a partition whose $k$ parts are all of size $> 1$ and $n \geq 2|\mu| - 2r$.
	We have the Schur expansion
	\[
	\frac{1}{(n - |\mu|)!} \cdot \vec{p}_{\mu(n)} = \sum_{|\lambda| \leq |\mu| - r} f_\lambda(n) \cdot s_{\lambda[n]}
	\]	
	where each $f_\lambda(n)$ is a polynomial function of $n$ of degree at most $r$ with equality when $\lambda = \varnothing$.
\end{corollary}

For $(I,J)$ a partial permutation, this result extends to the symmetric functions $A_{n,I,J}$ with sharper degree bounds.
See Corollary~\ref{atomic-asymptotics} for a precise statement.
A stronger polynomiality result for $R\, \one_{IJ}$ in terms of character polynomials appears as Proposition~\ref{indicator-polynomiality-proposition}.

\begin{example}
\label{ex:intro-stability}
For $n \geq 6$, the calcuation
in Example~\ref{ex:intro-path-to-s} extends to compute the
$s$-expansion of $\vec{p}_{31^{n-3}}$.
The leftmost three tilings extend uniquely to monotonic
tilings of the shape $(n-2,1,1), (n-1,1)$, and $(n-2,1)$.
The four tilings on the right are replaced by 
$n-2$ tilings of shape $(n)$ that are determined by the 
placement of the ribbon of size 3.
Therefore
\[
\vec{p}_{31^{n-3}} = (n-3)!
\left(s_{n-2,1,1} - s_{n-1,1} - s_{n-2,2} + (n-2) \cdot s_n\right).
\]
Similarly, we compute
\[
\vec{p}_{221^{n-4}} = 2! \cdot (n-4)! \left(s_{n-2,2} - (n-3) \cdot s_{n-1,1} + \binom{n-2}{2} \cdot s_n \right)
\]
with the $n=5$ case addressed via the monotonic ribbon tilings
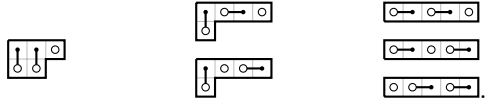
\[
\begin{tikzpicture}[scale=.25]
 \begin{scope}
    \clip (0,3) -| (3,2) -| (2,1) -| (0,3);
    \draw [color=black!25] (0,0) grid (6,4);
  \end{scope}

  \draw [thick] (0,3) -| (3,2) -| (2,1) -| (0,3);

  \draw [thick, rounded corners] (0.5,1.5) -- (0.5,2.5);
  \draw [color=black,fill=black,thick] (0.5,2.5) circle (.4ex);
  \node [draw, circle, fill = white, inner sep = 1pt] at (0.5,1.5) { };
  \draw [thick, rounded corners] (1.5,1.5) -- (1.5,2.5);
  \draw [color=black,fill=black,thick] (1.5,2.5) circle (.4ex);
  \node [draw, circle, fill = white, inner sep = 1pt] at (1.5,1.5) { };
  \node [draw, circle, fill = white, inner sep = 1pt] at (2.5,2.5) { };
  
   \begin{scope}
    \clip (10,2) -| (14,1) -| (11,0) -| (10,2);
    \draw [color=black!25] (10,-1) grid (16,3);
  \end{scope}

  \draw [thick] (10,2) -| (14,1) -| (11,0) -| (10,2);

  \draw [thick, rounded corners] (10.5,.5) -- (10.5,1.5);
  \draw [color=black,fill=black,thick] (10.5,1.5) circle (.4ex);
  \node [draw, circle, fill = white, inner sep = 1pt] at (10.5,0.5) { };
    \draw [thick, rounded corners] (12.5,1.5) -- (13.5,1.5);
  \draw [color=black,fill=black,thick] (13.5,1.5) circle (.4ex);
  \node [draw, circle, fill = white, inner sep = 1pt] at (12.5,1.5) { };
  \node [draw, circle, fill = white, inner sep = 1pt] at (11.5,1.5) { };
  
     \begin{scope}
    \clip (10,5) -| (14,4) -| (11,3) -| (10,5);
    \draw [color=black!25] (10,1) grid (16,5);
  \end{scope}

  \draw [thick] (10,5) -| (14,4) -| (11,3) -| (10,5);

  \draw [thick, rounded corners] (10.5,3.5) -- (10.5,4.5);
  \draw [color=black,fill=black,thick] (10.5,4.5) circle (.4ex);
  \node [draw, circle, fill = white, inner sep = 1pt] at (10.5,3.5) { };
    \draw [thick, rounded corners] (11.5,4.5) -- (12.5,4.5);
  \draw [color=black,fill=black,thick] (12.5,4.5) circle (.4ex);
  \node [draw, circle, fill = white, inner sep = 1pt] at (11.5,4.5) { };
  \node [draw, circle, fill = white, inner sep = 1pt] at (13.5,4.5) { };

  \begin{scope}
    \clip (20,0) -| (25,1) -| (20,0);
    \draw [color=black!25] (20,0) grid (25,2);
  \end{scope}

  \draw [thick] (20,0) -| (25,1) -| (25,1) -| (20,0);

  \draw [thick, rounded corners] (21.5,0.5) -- (22.5,0.5);
  \draw [color=black,fill=black,thick] (22.5,0.5) circle (.4ex);
  \node [draw, circle, fill = white, inner sep = 1pt] at (21.5,0.5) { };
    \draw [thick, rounded corners] (23.5,0.5) -- (24.5,0.5);
  \draw [color=black,fill=black,thick] (24.5,0.5) circle (.4ex);
  \node [draw, circle, fill = white, inner sep = 1pt] at (23.5,0.5) { };
  \node [draw, circle, fill = white, inner sep = 1pt] at (20.5,0.5) { };
  
   \begin{scope}
    \clip (20,2) -| (25,3) -| (20,2);
    \draw [color=black!25] (20,2) grid (26,4);
  \end{scope}

  \draw [thick] (20,2) -| (25,3) -| (25,3) -| (20,2);

  \draw [thick, rounded corners] (20.5,2.5) -- (21.5,2.5);
  \draw [color=black,fill=black,thick] (21.5,2.5) circle (.4ex);
  \node [draw, circle, fill = white, inner sep = 1pt] at (20.5,2.5) { };
    \draw [thick, rounded corners] (23.5,2.5) -- (24.5,2.5);
  \draw [color=black,fill=black,thick] (24.5,2.5) circle (.4ex);
  \node [draw, circle, fill = white, inner sep = 1pt] at (23.5,2.5) { };
  \node [draw, circle, fill = white, inner sep = 1pt] at (22.5,2.5) { };

   \begin{scope}
    \clip (20,4) -| (25,5) -| (20,4);
    \draw [color=black!25] (20,4) grid (26,6);
  \end{scope}

  \draw [thick] (20,4) -| (25,5) -| (25,5) -| (20,4);

  \draw [thick, rounded corners] (20.5,4.5) -- (21.5,4.5);
  \draw [color=black,fill=black,thick] (21.5,4.5) circle (.4ex);
  \node [draw, circle, fill = white, inner sep = 1pt] at (20.5,4.5) { };
    \draw [thick, rounded corners] (22.5,4.5) -- (23.5,4.5);
  \draw [color=black,fill=black,thick] (23.5,4.5) circle (.4ex);
  \node [draw, circle, fill = white, inner sep = 1pt] at (22.5,4.5) { };
  \node [draw, circle, fill = white, inner sep = 1pt] at (24.5,4.5) { };
  
\end{tikzpicture}.
\]
Here, the term $2! \cdot (n-4)!$ accounts for permutations of the 2 ribbons of size 2 and the $n-4$ ribbons of size 1.
The coefficients $(n-3)$ and $\binom{n-2}{2}$  account for possible locations of horizontal ribbons of length two.
\end{example}


\section{Regular Statistics and Their Asymptotics}

In algebraic and enumerative combinatorics, permutation statistics are usually defined on all symmetric groups simultaneously.
Fundamental questions in a wide variety of computational fields including computer science and statistics are best understood by characterizing the typical behavior of permutation statistics.
We introduce a novel family of pattern counting permutation statistics called \emph{regular statistics} and show how to apply our previous results to understand their asymptotic properties when applied to random permutations of a given cycle type.

For $(I,J)$ a partial permutation, note $\one_{IJ}$ is a statistic that detects whether a permutation contains $(I,J)$ as a pattern.
From this perspective, any permutation statistic that performs a weighted pattern count should have a simple interpretation in terms of $\one_{IJ}$'s, with degree bounded by the size of the patterns being counted.
As we have demonstrated in Example~\ref{ex:intro-maj}, it is natural to group partial permutations by their graphs.
Our regular statistics are designed to facilitate this grouping.

Notation is needed to define regular statistics.
We say $(U,V)$ is a \emph{packed partial permutation} if $U \cup V = [m]$ with $m$ a positive integer.
For $S = \{s_1 < \dots < s_m\} \subseteq \NN$, let $S(U) = (s_{i_1},\dots,s_{i_k})$ and $S(V) = (s_{j_1},\dots,s_{j_k})$
where $U = (i_1 < \cdots < i_k)$
and $V = (j_1 < \cdots < j_k)$.
Note that the graphs of $(U,V)$ and $(S(U),S(V))$ are isomorphic.
Lastly, for $C \subseteq [m-1]$ let 
$\binom{\NN}{m}_C$ be the set of $m$--element subsets $S = \{s_1 < \dots < s_m\}$ of $\NN$ so that for $i \in C$ we have $s_{i+1} = s_i + 1$.
\begin{definition}
	\label{d:intro-regular}

For $(U,V)$ a packed partial permutation with $U \cup V = [m]$, $C \subseteq [m-1]$ and $f \in \CC[x_1,\dots,x_m]$, the associated \emph{constrained translate} is
\[
T^f_{(U,V),C} = \sum_{L = \{\ell_1 < \dots < \ell_m\} \in \binom{\NN}{m}_C} f(\ell_1,\dots,\ell_m) \cdot \one_{L(U)\,L(V)}.
\]
A \emph{regular statistic} is a linear combination of constrained translates.
\end{definition}

Notice that the constrained translate $T^f_{(U,V),C}$
defines a function $\symm_n \rightarrow \CC$ for each 
$n$.
The first equation of Example~\ref{ex:intro-maj} shows that $\maj$ is a regular statistic:
\begin{align*}
\maj =& T^{x_1}_{((1,2),(2,1)),\{1\}} + T^{x_1}_{((1,2),(3,1)),\{1\}} + T^{x_1}_{((1,3),(2,1)),\{1\}}	\\
&+ T^{x_2}_{(2,3),(2,1),\{2\}}  + T^{x_2}_{(2,3),(3,1),\{2\}} + T^{x_1}_{(1,2),(4,3),\{1\}}\\
&+ T^{x_2}_{(2,3),(4,1),\{2\}} + T^{x_3}_{(3,4),(2,1),\{3\}}.
\end{align*}
As we have seen in Example~\ref{ex:intro-maj}, an explicit realization of a statistic $\stat$ as regular facilitates the computation of $R \, \stat$.
Additionally, regular statistics are closed under multiplication, hence form an algebra.
The following result is a portion of 
Proposition~\ref{p:simple-prod}.

\begin{proposition}
	\label{p:intro-regular-algebra}
	A pointwise 
 product of regular statistics is regular.
\end{proposition}

For $(U,V) \in \symm_k$ a packed partial permutation with $U \cup V = [m]$, $C \subseteq [m-1]$ and $f \in \CC[x_1,\dots,x_m]$, we say the constrained translate $T^f_{(U,V),C}$ has \emph{size} $k$, \emph{shift} $|C|$ and \emph{power} $k + \deg f - |C|$.
Similarly, when expressing a regular statistic as a linear combination of constrained translates, the size, shift and power of that expansion are the maxima amongst all translates in the expansion.
These properties determine the asymptotic behavior of a regular statistic.

\begin{theorem}
\label{t:intro-regular-character}
	Let $\stat$ be a regular statistic of size $k$, shift $q$ and power $p$. Then
	\begin{equation}
	\label{eq:intro-regular}
	\ch_n (R\,\stat) = \sum_{|\lambda| \leq k} c_\lambda(n) s_{\lambda[n]}	
	\end{equation}
where on the domain $n \geq 2k$ each $(n)_q \cdot c_\lambda(n)$ is a polynomial function of $n$ with degree at most $p+q - |\lambda|$.
\end{theorem}

We refine Proposition~\ref{p:intro-regular-algebra} as Proposition~\ref{p:simple-prod} by showing all three of the properties size, shift, and power are subadditive.
Therefore, Theorem~\ref{t:intro-regular-character} applies to products of regular statistic, and in particular higher moments.

Recall that an irreducible character $\chi^\lambda:\symm_n \to \CC$ is a class function, hence only depends on the cycle type of the input.
For $w \in \symm_n$, let $m_i(w)$ be the number of $i$-cycles in $w$.
The theory of character polynomials, introduced in~\cite{Specht}, gives a structural description of $\chi^\lambda$ as a function of $m_i$'s.

\begin{theorem}
	\label{t:intro-character-polynomial}
	For $\lambda$ a partition of size $n$, the chracter $\chi^\lambda$ is a polynomial in the variables $m_1,m_2,\dots,m_n$ whose degree is $n - |\lambda_1|$ where $\deg (m_i) = i$.
\end{theorem}

Combining Theorems~\ref{t:intro-regular-character} and~\ref{t:intro-character-polynomial} gives a powerful structural description for moments of regular statistics on arbitrary cycle types.

\begin{corollary}[= Corollary~\ref{cycle-type-asymptotics}]
	\label{c:intro-regular-moments}
	Let $\stat$ be a regular statistic with size $k$, shift $q$ and power $p$, and let $d \geq 1$.
	Then for $\Sigma$ a uniformly random permutation of cycle type $\lambda$,
	\[
	\EE(\stat^d(\Sigma)) = \frac{f(n,m_1(\lambda),m_2(\lambda),\dots,m_{dk}(\lambda))}{(n)_{dq}}
	\]
	where $f$ is a polynomial of degree $dp + dq$ where $\deg n =1$ and $\deg m_i = i$.
\end{corollary}

There are many previous results on moments of pattern counting statistics in the literature, beginning with Zeilberger's result that all moments for the random variable counting subwords of size $k$ in a uniformly random permutation with a given relative order is a polynomial in the size of the permutation~\cite{Zeilberger}.
There are many generalizations and extensions of Zeilberger's result, both to more general statistics~\cite{DK} and allowing for specific~\cite{KLY} or arbitrary cycle type~\cite{Hultman,GR,GP}.
Corollary~\ref{c:intro-regular-moments} generalizes all previous results on this topic we are aware of.
Chapter~\ref{Pattern} begins with a more detailed discussion of this history and explains why our results generalize this prior work.

We next apply Corollary~\ref{c:intro-regular-moments} to understand the limiting behavior of regular statistics.
Our first result in this vein demonstrates the value of character polynomials as a tool for asymptotic analysis.
Hofer has shown \emph{vincular pattern counts} for uniformly random permutations are asymptotically normal~\cite{Hofer}.
Via a generating function argument, Fulman showed the descent statistic, which is vincular, is asymptotically normal for random permutations with no short cycles.
As a corollary of locality and the theory of character polynomials, we extend Fulman's result to all vincular patterns (see Theorem~\ref{our-long-cycles}), a result incomparable to Kammoun's recent work on conjugacy invariant distributions~\cite{Kammoun}.

Recent work of Kim and Lee extends Fulman's result for descents to random permutations on any cycle type~\cite{KL1}.
While we do not demonstrate normality in these settings, we demonstrate both a law of large numbers and a structural characterization of the variance for regular statistics with the correct scaling behavior.

\begin{theorem} \label{t:intro-regular-asymptotics}
 {\em ($\subset$ Theorems~\ref{expectation-fixed-point} and~\ref{regular-variance})}
	Let $\stat$ be a regular statistic with power $p$ and $\{\lambda^{(n)}\}$ be a sequence of partitions of $n$
	so that for $\alpha,\beta \in [0,1]$ we have
	\begin{equation}
		\lim_{n \to \infty} \frac{m_1(\lambda^{(n)})}{n} = \alpha
		\quad \mbox{and} \quad
		\lim_{n \to \infty} \frac{m_2(\lambda^{(n)})}{n} = \beta.
	\end{equation}
	Viewing $\stat$ as a random variable with respect to the uniform distribution, there exist polynomials $f,g,h \in \RR[x]$ independent
	of $\{ \lambda^{(n)} \}$ so that
\begin{equation}\label{eq:intro-asymptotics}
	\lim_{n \to \infty} \EE\left( \frac{\stat}{n^p} \mid K_{\lambda^{(n)}} \right) = f(\alpha), \quad \  
	\lim_{n \to \infty} \VV\left(\frac{\stat}{n^{2p-1}} \mid K_{\lambda^{(n)}}\right)= g(\alpha) + \beta h(\alpha).
\end{equation}
\end{theorem}

We remark that the scaling factors in~\eqref{eq:intro-asymptotics} can dominate the numerator, for instance when $\stat$ is the number of two cycles.
However, for classical and vincular pattern counting statistics, they are correct.
While Theorem~\ref{t:intro-regular-asymptotics} does not characterize the limiting behavior of regular statistics, it has already been used to prove to prove normality at this generality for classical pattern counts and some vincular pattern counts~\cite{FK}.
In addition, we can use \eqref{eq:intro-asymptotics} to show that uniformly random permutations on a sequence of cycle types converge to a limiting object called a permuton (see Theorem~\ref{alpha-permuton}).

\section*{Outline}
The rest of the document is organized as follows.
In {\bf Chapter~\ref{Background}} we give background information on symmetric functions, the representation theory of $\symm_n$, and basic probability.
{\bf Chapter~\ref{Local}} characterizes the spaces of low frequency functions  and low frequency class functions.
{\bf Chapter~\ref{Atomic}} studies the atomic symmetric functions $A_{n,I,J}$. We prove that for $(I,J) \in \symm_{n,k}$ the Schur expansion
of $A_{n,I,J}$ is supported on partitions $\lambda$ where $\lambda_1 \geq n-k$ and give a factorization
$A_{n,I,J} = \vec{p}_{\mu} \cdot p_{\nu}$ of the atomic functions into a ``path power sum" $\vec{p}_{\mu}$ and a classical power sum $p_{\nu}$.
{\bf Chapter~\ref{Path}} is the most technical of the paper; the main result is the Schur expansion of the path power sum $\vec{p}_{\mu}$.
This gives our combinatorial interpretation of the character evaluations.
The factorization $A_{n,I,J} = \vec{p}_{\mu} \cdot p_{\nu}$ leads to asymptotic results on the atomic functions as $n \rightarrow \infty$.
{\bf Chapter~\ref{Regular}} defines regular permutation statistics and uses the Path Murnaghan-Nakayama Rule to obtain asymptotics for their Schur 
expansions.
{\bf Chapter~\ref{Pattern}} applies the theory of regular statistics to pattern enumeration.
In {\bf Chapter~\ref{Convergence}} we apply our results to prove convergence properties of local statistics and regular statistics.
We close in {\bf Chapter~\ref{Conclusion}} with directions for future research.

\newpage

\chapter{Background}
\label{Background}

We give necessary preliminaries on symmetric functions, partial permutations, representation theory of the symmetric group, and basic probability.
The material on symmetric function and representation theory is standard, with textbook treatments~\cite{Macdonald,Sagan}.
The results on character polynomials can be found in~\cite{Macdonald}.
The notation and terminology for partial permutations has not been standardized -- our presentation is closely related to that in~\cite{DIL}, which summarizes classical results on them as orbitals in the symmetric association scheme of injections treated in~\cite{BBIT}.
Our treatment of probability is standard, summarizing standard facts and some classic results on convergence.
Many textbook treatments exist, e.g. \cite{Billingsley}.

\section{Symmetric functions}
A {\em (strong) composition} of $n$ is a sequence $\alpha = (\alpha_1, \dots, \alpha_k)$ of positive integers such that $\alpha_1 + \cdots + \alpha_k = n$.
We write $\alpha \models n$ to indicate the $\alpha$ is a composition of $n$ and $|\alpha| = n$ for the sum of the parts of $\alpha$.
We let $m_i(\alpha)$ be the multiplicity of $i \geq 1$ as  a part of $\alpha$ and introduce the notation 
$m(\alpha)! := m_1(\alpha)! \cdots m_n(\alpha)!$.

A {\em partition} of $n$ is a composition $\lambda = (\lambda_1 \geq \cdots \geq \lambda_k)$ of $n$ 
that is weakly decreasing. We write $\lambda \vdash n$ to indicate that $\lambda$ is a partition of $n$ and $\ell(\lambda)$ for the number
of positive parts of $\lambda$.
We identify $\lambda$ with its {\em  (English) Young diagram} which consists of $\lambda_i$ left-justified boxes in row $i$.
We write $|\lambda| = n$ to indicate the number of boxes in the Young diagram of $\lambda$.
We let $\lambda' \vdash n$ be the conjugate partition to $\lambda$ whose Young diagram is obtained by reflecting across the main diagonal.
The Young diagrams of $\lambda = (4,3,1) \vdash 8$ and $\lambda' = (3,2,2,1)$ are shown below.

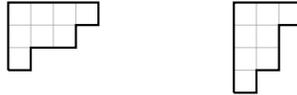
\begin{center}
\begin{tikzpicture}[scale = .3]
  \begin{scope}
    \clip (0,0) -| (1,1) -| (3,2) -| (4,3) -| (0,0);
    \draw [color=black!25] (0,0) grid (4,3);
  \end{scope}

  \draw [thick] (0,0) -| (1,1) -| (3,2) -| (4,3) -| (0,0);

   \begin{scope}
    \clip (10,-1) -| (11,0) -| (12,2) -| (13,3) -| (10,-1);
    \draw [color=black!25] (10,-1) grid (13,3);
  \end{scope}

  \draw [thick] (10,-1) -| (11,0) -| (12,2) -| (13,3) -| (10,-1);
\end{tikzpicture}
\end{center}

We make occasional use of the {\em dominance order}
on partitions of $n$.
This is the partial order $\leq_{\mathrm{dom}}$ defined by
$\mu \leq_{\mathrm{dom}} \lambda$ if and only if
$\mu_1 + \cdots + \mu_i \leq 
\lambda_1 + \cdots + \lambda_i$ for all $i \geq 0$.
(Here we pad the sequences $\mu$ and $\lambda$ with 
an infinite string of 0s so that these sums 
make sense.)

We will make more frequent use of {\em Young's 
Lattice}. This is the partial order $\subseteq$
defined on partitions by
$\mu \subseteq \lambda$ if $\mu_i \leq \lambda_i$ for all $i$.
Young's Lattice is a graded poset where 
the rank of a partition $\lambda$ is given by
$|\lambda|$.

Let $\lambda, \mu$ be partitions with 
$\mu \subseteq \lambda$.
The {\em skew shape} $\lambda/\mu$ is the set-theoretic difference $\lambda/\mu := \lambda - \mu$ of Young diagrams.
We write $|\lambda/\mu| := |\lambda|  - |\mu|$ for the number of boxes in $\lambda/\mu$.
The Young diagram of the skew shape $(4,4,1)/(2,1)$ is shown below.

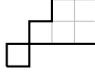
\begin{center}
\begin{tikzpicture}[scale = .3]
  \begin{scope}
    \clip (0,0) -| (1,1) -| (4,3) -| (2,2) -| (1,1) -| (0,0);
    \draw [color=black!25] (0,0) grid (4,3);
  \end{scope}
 
 \draw [thick] (0,0) -| (1,1) -| (4,3) -| (2,2) -| (1,1) -| (0,0);
\end{tikzpicture}
\end{center}

Let $\lambda \vdash n$ be a partition of $n$ with part multiplicities $m_1(\lambda), m_2(\lambda), \dots, m_n(\lambda)$.  We let $z_{\lambda}$ be 
\begin{equation}
z_{\lambda} := 1^{m_1(\lambda)} 2^{m_2(\lambda)} \cdots n^{m_n(\lambda)} \cdot m_1(\lambda)! m_2(\lambda)! \cdots m_n(\lambda)!.
\end{equation}
In algebraic terms, the number $z_{\lambda}$ enumerates the centralizer of any permutation $w \in \symm_n$ of cycle type $\lambda$.

The ring $\Lambda$ of symmetric functions in an infinite variable set $\{ x_1, x_2, \dots \}$ over the ground field $\CC$
will play a central role in our work.
We work over the field of complex numbers for representation-theoretic convenience, 
but all results relating to symmetric functions and representation theory of the symmetric
group are valid over any field $\mathbb{F}$ of characteristic zero.
We will sometimes work over the real numbers $\RR$ when dealing with probabilistic ideas.

Let 
$\CC[[x_1, x_2, \dots ]]_n$ be the $\CC$-vector space of power series in the infinite variable set $\{x_1, x_2, \dots \}$ of homogeneous degree $n$.
For $n > 0$, the {\em power sum}, {\em homogeneous}, and {\em elementary} symmetric functions are the elements of 
$\CC[[x_1, x_2, \dots ]]_n$ 
 defined by
\begin{equation}
p_n := \sum_{i \geq 1} x_i^n \quad \quad 
h_n := \sum_{i_1 \leq \cdots \leq i_n} x_{i_1} \cdots x_{i_n} \quad \quad
e_n := \sum_{i_1 < \cdots < i_n} x_{i_1} \cdots x_{i_n},
\end{equation}
respectively. 

The direct sum $\bigoplus_{n \geq 0} \CC[[x_1, x_2, \dots ]]_n$ has the structure of a graded ring under power series multiplication.
We define $\Lambda$ to be the unital $\CC$-subalgebra of $\bigoplus_{n \geq 0} \CC[[x_1, x_2, \dots ]]_n$
generated by $\{ p_n \,:\, n \geq 1\}$.  The
three sets $\{ p_n \,:\, n \geq 1\}$,
$\{ h_n \,:\, n \geq 1 \}$, and $\{ e_n \,:\, n \geq 1 \}$
are algebraically independent generating sets of $\Lambda$.
 In symbols, we have
\begin{equation}
\Lambda = \CC[p_1, p_2, \dots ] = \CC[h_1, h_2, \dots ] = \CC[e_1, e_2, \dots ].
\end{equation}
We write $\Lambda_n := \Lambda \cap \CC[[x_1, x_2, \dots ]]_n$ for the vector space of symmetric functions of homogeneous degree $n$.

The ring $\Lambda = \bigoplus_{n \geq 0} \Lambda_n$ is graded by power series degree.
 Bases of $\Lambda_n$ are indexed by partitions $\lambda \vdash n$. The {\em power sum}, {\em homogeneous}, and {\em elementary} bases 
 are given by the multiplicative rules
 \begin{equation}
 p_{\lambda} := p_{\lambda_1} p_{\lambda_2} \cdots \quad \quad
 h_{\lambda} := h_{\lambda_1} h_{\lambda_2} \cdots \quad \quad
 e_{\lambda} := e_{\lambda_1} e_{\lambda_2} \cdots
 \end{equation}
 for all partitions $\lambda = (\lambda_1, \lambda_2, \dots) \vdash n$.  
 For a partition $\lambda = (\lambda_1, \dots, \lambda_k)$, the {\em monomial
 symmetric function} is 
 \begin{equation}
 m_{\lambda} := \sum_{(a_1, \dots, a_k)} \sum_{i_1 < \cdots < i_k} x_{i_1}^{a_1} \cdots x_{i_k}^{a_k}
 \end{equation}
 where the outer sum is over all distinct rearrangements $(a_1, \dots, a_k)$ of the sequence $(\lambda_1, \dots, \lambda_k)$.

 The {\em Schur basis} $s_{\lambda}$ of $\Lambda$ was originally
 defined as follows. Let us temporarily restrict to a finite number $N$ of variables $x_1, \dots, x_N$.
 Recall that $\CC[\symm_N]$ denotes the group algebra of the symmetric group $\symm_N$. The algebra $\CC[\symm_N]$ acts naturally on
$\CC[x_1, \dots, x_N]$ by subscript permutation.
 Let $\varepsilon \in \CC[\symm_N]$ be the antisymmetrizing element 
 \begin{equation}
 \varepsilon := \sum_{w \in \symm_N} \sign(w) \cdot w.
 \end{equation}
 The {\em Vandermonde determinant} is the polynomial in $\CC[x_1, \dots, x_N]$ given by
 \begin{equation}
 \varepsilon \cdot x^{\delta} = \varepsilon \cdot (x_1^{\delta_1} \cdots x_N^{\delta_N})
 \end{equation}
where the exponent sequence 
 $\delta := (N-1, N-2, \dots, 1, 0)$ is the staircase of length $N$. 
 Let $\lambda = (\lambda_1, \dots, \lambda_N)$ be a partition with $\leq N$ parts (where we pad with zeros if needed to achieve a sequence of length $N$).
The {\em Schur polynomial} $s_{\lambda}(x_1, \dots, x_N)$ is the quotient
 \begin{equation}
 \label{bialternant-schur-formula}
 s_{\lambda}(x_1, \dots, x_N) := \frac{\varepsilon \cdot x^{\lambda + \delta}}{\varepsilon \cdot x^{\delta}} =
 \frac{ \varepsilon \cdot (x_1^{\lambda_1 + \delta_1} \cdots x_N^{\lambda_N + \delta_N})}{ \varepsilon \cdot (x_1^{\delta_1} \cdots x_N^{\delta_N})}
 \end{equation}
 where the addition $\lambda + \delta$ of exponents is interpreted componentwise.
 It is not hard to see that the rational expression $s_{\lambda}(x_1, \dots ,x_N)$ is a polynomial that is symmetric in $x_1, \dots, x_N$.

A polynomial $f \in \CC[x_1, \dots, x_N]$ is called {\em alternating} if 
\begin{equation}
w \cdot f = \sign(w) \cdot w
\end{equation}
 for all $w \in \symm_N$.
As the numerator and denominator in \eqref{bialternant-schur-formula} are alternating, \eqref{bialternant-schur-formula} 
is called the {\em bialternant formula} for the Schur polynomial $s_\lambda(x_1, \dots, x_N)$.
Equation~\eqref{bialternant-schur-formula} is equivalent to the {\em Weyl Character Formula} which computes the trace of a diagonal matrix acting
on an irreducible polynomial $GL_N(\CC)$-module.
 It is not hard to check that 
 \begin{equation}
 s_{\lambda}(x_1, \dots, x_N,0) = s_{\lambda}(x_1, \dots, x_N)
 \end{equation} 
 for any partition $\lambda$ with $\leq N$ parts, so that 
 the limit
 \begin{equation}
 s_{\lambda} := \lim_{N \rightarrow \infty} s_{\lambda}(x_1, \dots, x_N)
 \end{equation}
 is a well-defined formal power series belonging to $\Lambda$.
 Strictly speaking, to make  sense of this limit we use an
  equivalent definition of $\Lambda = \bigoplus_{n \geq 0} \Lambda_n$ obtained by identifying the graded piece $\Lambda_n$ with the inverse limit
 \begin{equation}
 \Lambda_n := \lim_{\substack{\longleftarrow \\ N}} \left( \CC[x_1, \dots, x_N, x_{N+1}]^{\symm_{N+1}}_n \xrightarrow{ \, \, \pi_N \, \, } \CC[x_1, \dots, x_N]^{\symm_N}_n  \right).
 \end{equation}
 Here $\CC[x_1, \dots, x_N]^{\symm_N}_n$ denotes the vector
 space of homogeneous $\symm_N$-invariant polynomials in $x_1, \dots, x_N$ of degree $n$ and the map $\pi_N$ 
is given by 
\begin{equation}
\pi_N: f(x_1, \dots, x_N, x_{N+1}) \mapsto f(x_1, \dots, x_N, 0).
\end{equation}
 The limit $s_\lambda \in \Lambda$ is called a {\em Schur function}, and the collection
 $\{ s_{\lambda} \}$ forms the {\em Schur basis}   of $\Lambda$.

The Schur function $s_\lambda$ admits a combinatorial definition in terms of tableaux. 
If $\lambda$ is a partition, a {\em semistandard tableau} of shape $\lambda$ is a filling $T: \lambda \rightarrow \ZZ_{> 0}$
of the Young diagram of $\lambda$ with positive integers
that is weakly increasing across rows and strictly increasing down columns.
We have
\begin{equation}
\label{combinatorial-schur-function}
s_{\lambda} := \sum_{T} x^T
\end{equation}
where the sum is over all semistandard tableaux of shape $\lambda$ and $x^T = x^{c_1} x^{c_2} \cdots $ where $c_i$ is the number 
of $i$'s in $T$. For example, the semistandard tableau
\begin{center}
\begin{tikzpicture}[scale = .3]
  \begin{scope}
    \clip (0,0) -| (1,1) -| (3,2) -| (4,3) -| (0,0);
    \draw [color=black!25] (0,0) grid (4,3);
  \end{scope}

  \draw [thick] (0,0) -| (1,1) -| (3,2) -| (4,3) -| (0,0);
  
  \node at (0.5,2.5) {$1$};
  \node at (1.5,2.5) {$1$};
  \node at (2.5,2.5) {$3$};
  \node at (3.5,2.5) {$3$};
  \node at (0.5,1.5) {$2$};
  \node at (1.5,1.5) {$3$};
  \node at (2.5,1.5) {$4$};
  \node at (0.5,0.5) {$4$};
\end{tikzpicture}
\end{center}
contributes $x_1^2 x_2 x_3^3 x_4^2$ to $s_{431}$.
We will not make use of the combinatorial definition \eqref{combinatorial-schur-function} of $s_\lambda$
 in this manuscript; the bialternant definition 
\eqref{bialternant-schur-formula} will be more useful for our purposes.

A semistandard tableau $T$ with $n$ boxes is called {\em standard} if it is a bijection onto $\{1, \dots, n \}$.
A standard tableau of shape $(4,3,1)$ is shown below.
\begin{center}
\begin{tikzpicture}[scale = .3]
  \begin{scope}
    \clip (0,0) -| (1,1) -| (3,2) -| (4,3) -| (0,0);
    \draw [color=black!25] (0,0) grid (4,3);
  \end{scope}

  \draw [thick] (0,0) -| (1,1) -| (3,2) -| (4,3) -| (0,0);
  
  \node at (0.5,2.5) {$1$};
  \node at (1.5,2.5) {$2$};
  \node at (2.5,2.5) {$4$};
  \node at (3.5,2.5) {$6$};
  
  \node at (0.5,1.5) {$3$};
  \node at (1.5,1.5) {$7$};
  \node at (2.5,1.5) {$8$};
  
  \node at (0.5,0.5) {$5$};
\end{tikzpicture}
\end{center}
Tableaux of various kinds play a central role in combinatorial representation theory in general, and symmetric function theory in particular.

 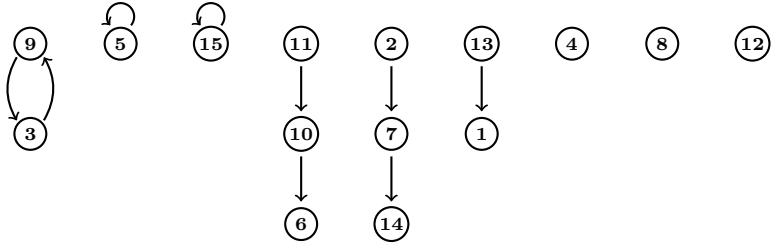
\begin{figure}
 \begin{center}
 \begin{tikzpicture}[scale = 0.6]
 
 \coordinate (v11) at (6,0);
 
 \coordinate (v10) at (6,-2);
 
 \coordinate (v6) at (6,-4);

 \coordinate (v2) at (8,0);
 
 \coordinate (v7) at (8,-2);
 
\coordinate (v14) at (8,-4);
  
 \coordinate (v13) at (10,0);
   
 \coordinate (v1) at (10,-2);
    
 \coordinate (v8) at (14,0);
 
 \coordinate (v4) at (12,0);
 
 \coordinate (v12) at (16,0);
 
 \coordinate (v9) at (0,0);
 
 \coordinate (v3) at (0,-2);
 
 \coordinate (v5) at (2,0);
 
 \coordinate (v15) at (4,0);

   \node [draw, circle, fill = white, inner sep = 2pt, thick] at (v1)
  {\scriptsize {\bf 1} };
     \node [draw, circle, fill = white, inner sep = 2pt, thick] at (v2)
  {\scriptsize {\bf 2}};
     \node [draw, circle, fill = white, inner sep = 2pt, thick] at (v3)
  {\scriptsize {\bf 3}};
     \node [draw, circle, fill = white, inner sep = 2pt, thick] at (v4)
  {\scriptsize {\bf 4}};
     \node [draw, circle, fill = white, inner sep = 2pt, thick] at (v5)
  {\scriptsize {\bf 5}};
     \node [draw, circle, fill = white, inner sep = 2pt, thick] at (v6)
  {\scriptsize {\bf 6}};
     \node [draw, circle, fill = white, inner sep = 2pt, thick] at (v7)
  {\scriptsize {\bf 7}};
     \node [draw, circle, fill = white, inner sep = 2pt, thick] at (v8)
  {\scriptsize {\bf 8}};
     \node [draw, circle, fill = white, inner sep = 2pt, thick] at (v9)
  {\scriptsize {\bf 9}};
     \node [draw, circle, fill = white, inner sep = 1pt, thick] at (v10)
  {\scriptsize {\bf 10}};
     \node [draw, circle, fill = white, inner sep = 1pt, thick] at (v11)
  {\scriptsize {\bf 11}};
     \node [draw, circle, fill = white, inner sep = 1pt, thick] at (v12)
  {\scriptsize {\bf 12}};
     \node [draw, circle, fill = white, inner sep = 1pt, thick] at (v13)
  {\scriptsize {\bf 13}};
     \node [draw, circle, fill = white, inner sep = 1pt, thick] at (v14)
  {\scriptsize {\bf 14}};
     \node [draw, circle, fill = white, inner sep = 1pt, thick] at (v15)
  {\scriptsize {\bf 15}};
  
  \draw [->, thick] (6,-0.5) -- (6,-1.5);
  \draw [->, thick] (6,-2.5) -- (6,-3.5);
 \draw [->, thick] (8,-0.5) -- (8,-1.5);
 \draw [->, thick] (8,-2.5) -- (8,-3.5);
 \draw [->, thick] (10,-0.5) -- (10,-1.5);
 \draw [->, thick]  (-0.3,-0.3) to[bend right]  (-0.3,-1.7);
 \draw [->, thick]  (0.3,-1.7) to[bend right]  (0.3,-0.3);
 
 \draw[->, thick]  ($(2,0.6) + (-40:3mm)$) arc (-40:220:3mm);

 \draw[->, thick]  ($(4,0.6) + (-40:3mm)$) arc (-40:220:3mm);
 
 \end{tikzpicture} 
 \end{center}
 \caption{The graph $G_n(I,J)$  with $I = [11,10,2,7,13,9,3,5,15]$ and $J = [10,6,7,14,1,3,9,5,15]$. Here $(I,J) \in \symm_{15,9}$.
 The cycle partition of is $(2,1,1)$ and the path partition is $(3,3,2,1,1,1)$.  The reduced path partition is $(3,3,2)$.}
 \label{fig:graph}
 \end{figure}

\section{Permutations and Partial permutations}

We repeat several key definitions from the introduction.
For $S$ a finite set, let $\symm_S$ be the permutations of $S$.
In particular, if $[n] := \{1, \dots, n \}$, then $\symm_{[n]}$ is the usual symmetric group $\symm_n$.
A \emph{partial permutation} of size $k$ in $\symm_n$ is a bijection $S \to T$ where $S, T \subseteq [n]$ and $|S| = |T| = k$.
We represent a partial permutation as a pair $(I,J)$ of tuples $I = (i_1,\dots,i_k) \in \symm_S$ and $J = (j_1,\dots,j_k) \in \symm_T$ where $i_1 \mapsto j_1,\dots, i_k \mapsto j_k$.
We write $\symm_{n,k}$ for the set of all partial permutations of size $k$ in $\symm_n$.
There are eighteen partial permutations in $\symm_{3,2}$; the six with $I = (12)$ are:
\begin{align*}
(12,12), (12, 13), (12,21), (12,23), (12,31), (12, 32).
\end{align*}
The symmetric group $\symm_k$ acts on pairs of lists $(I,J)$ of length $k$ by permuting their entries simultaneously.
Partial permutations $(I,J) \in \symm_{n,k}$ are unchanged under this $\symm_k$ action, so we can always write $I$ in increasing order.

Let $(I,J) \in \symm_{n,k}$ with $I = (i_1, \dots, i_k)$ and $J = (j_1, \dots, j_k)$.
The {\em graph} $G_n(I,J)$ of $(I,J)$ is the directed graph on the vertex set $[n]$ with edges $i_1 \rightarrow j_1, \dots, i_k \rightarrow j_k$.
The graph $G_n(I,J)$ of a partial permutation $(I,J) \in \symm_{n,k}$ is an extension of the disjoint cycle notation for a permutation in $\symm_n$.
Every connected component of $G_n(I,J)$ is a directed path or a directed cycle.
In particular, the 1-vertex paths in $G_n(I,J)$ correspond precisely to the elements in $[n] - (I \cup J)$ which are not involved in the partial permutation $(I,J)$.
The {\em cycle partition} $\nu = (\nu_1, \nu_2, \dots )$ has parts given by the cycle lengths in $G_n(I,J)$.
Similarly, the {\em path partition} $\mu = (\mu_1, \mu_2, \dots )$ of $(I,J)$ is obtained by listing the path sizes in $G_n(I,J)$ in weakly decreasing order.
The \emph{cycle-path type} of $(I,J) \in \symm_{n,k}$ is the pair $(\nu,\mu)$ given the cycle and path decomposition of $G_n(I,J)$.
Oberve that $|\nu| + |\mu| = n$.
The cycle-path type of a partial permutation plays the role of the cycle type for a permutation.
See Figure~\ref{fig:graph} for an example of these concepts.

The symmetric group $\symm_n$ acts naturally on length $k$ lists of elements of $[n]$ by the rule 
\begin{equation}
w(I) = w((i_1, \dots, i_k)) := (w(i_1), \dots, w(i_k))
\end{equation}
where $I = (i_1, \dots, i_k)$.
This induces an action of $\symm_n$ on $\symm_{n,k}$ by the diagonal rule 
\begin{equation}
w(I,J) := (w(I), w(J)).
\end{equation}
The orbits of this action are parametrized by cycle-path type.

\begin{proposition}
\label{p:cycle-path}
	The partial permutations $(I,J), (I',J')$ of $[n]$
	have the same cycle-path type if and only if there is a genuine permutation $w \in \symm_n$ which satisfies $I' = w(I)$ and $J' = w(J)$.
\end{proposition}

\begin{proof}
	Observe $(I,J)$ and $(I',J')$ have the same cycle-path type if and only if $G_n(I,J)$ and $G_n(I',J')$ are isomorphic as graphs. 
	The graphs $G_n(I,J)$ and $G_n(I',J')$ are isomorphic if and only
	if there is a relabeling of $I \cup J$ to $I'\cup J'$ sending edges in $G_n(I,J)$ to edges in $G_n(I',J')$ and visa versa.
	This relabeling extends to a permutation $w$ with $I' = w(I)$ and $J' = w(J)$, and such permutations restrict to labelings with this property.
\end{proof}

In the statement of Proposition~\ref{p:cycle-path}, note that the partial permutations $(I,J)$ and $(I',J')$ of $[n]$ do not a priori have the same size.
In particular, we see that
cycle-path type determines size.
In fact, for $(I,J)$ a partial permutation of size $k$ with cycle-path type $(\nu,\mu)$ we have
\begin{equation}
\label{eq:cycle-path-size}
k = |\mu| + |\nu| - \ell(\nu).	
\end{equation}

Given two sequences $I, J$ of distinct positive integers with the same length $k$, the pair $(I,J)$ may be regarded as a partial permutation in 
$\symm_{n,k}$ whenever $n \geq \max(I \cup J)$.
The transition $n \leadsto n+1$ adds a 1-path labelled $n+1$ to the graph of $(I,J)$.
For probabilistic purposes, we will often want to consider $(I,J)$ as a partial permutation of $[n]$ in the limit $n \rightarrow \infty$.
As such, we define the {\em reduced path partition} $\overline{\mu}$ of $(I,J)$ to be the partition obtained by listing the path lengths $> 1$ 
in $G_n(I,J)$ in weakly decreasing order.


\section{$\symm_n$-representation theory} 
Let $G$ be a finite group. A complex-valued function $\varphi: G \rightarrow \CC$ is a {\em class function} if 
$\varphi$ is constant on conjugacy classes, i.e. 
\begin{equation}
\varphi(h g h^{-1}) = \varphi(g) \text{ for all } g, h \in G.
\end{equation}
The set $\Class(G,\CC)$ of all class functions $G \rightarrow \CC$ forms a vector space under the operations of pointwise
addition and scalar multiplication.
If $V$ is a finite-dimensional $\CC[G]$-module, 
the {\em character} $\chi_V: G \rightarrow \CC$
of $V$
is defined by
\begin{equation}
\chi_V(g) := \mathrm{trace}_V(v \mapsto g \cdot v).
\end{equation}
Since trace is invariant under matrix conjugation,
the map $\chi_V$ is a class function.

If $V_1$ and $V_2$ are $\CC[G]$-modules,
the direct sum $V_1 \oplus V_2$ is a 
$\CC[G]$-module via 
$g \cdot (v_1, v_2) := (g \cdot v_1, g \cdot v_2)$.
We have 
$\chi_{V_1 \oplus V_2} = \chi_{V_1} + \chi_{V_2}$
as functions on $G$, where the function
$(\chi_{V_1} + \chi_{V_2}): G \rightarrow \CC$ 
is interpreted pointwise.
A nonzero $\CC[G]$-module $V$ is {\em irreducible}
if its only submodules are $0$ and $V$.
Since the group algebra $\CC[G]$ is semisimple,
a nonzero $\CC[G]$-module $V$ is irreducible
if $V = V_1 \oplus V_2$ implies $V = V_1$ or $V = V_2$.
For any $\CC[G]$-modules $V_1$ and $V_2$, we have
$V_1 \cong_{\CC[G]} V_2$ if and only if 
$\chi_{V_1} = \chi_{V_2}$.
A character $\chi_V$ of $G$ is called irreducible 
if its corresponding module $V$ is irreducible.

The vector space $\Class(G,\CC)$ is endowed with the {\em class function inner product} defined by
\begin{equation}
\label{G-class-function-inner}
\langle \varphi, \psi \rangle := \frac{1}{|G|} \sum_{g \in G} \varphi(g) \cdot \overline{\psi(g)}.
\end{equation}
The irreducible characters of $G$ form an orthonormal basis of $\Class(G,\CC)$ with respect to this inner product.
We will focus on class functions on the symmetric group.
In this setting, Equation~\eqref{G-class-function-inner} reads
\begin{equation}
\langle \varphi, \psi \rangle = \frac{1}{n!} \sum_{w \in \symm_n} \varphi(w) \cdot \overline{\psi(w)}
\end{equation}
for class functions $\varphi, \psi: \symm_n \rightarrow \CC$.

The {\em characteristic map} $\ch_n: \Class(\symm_n,\CC) \rightarrow \Lambda_n$ carries class functions to symmetric functions. It is defined by
\begin{equation}
\ch_n(\varphi) := \frac{1}{n!} \sum_{w \in \symm_n} \varphi(w) \cdot p_{\lambda(w)}
\end{equation}
where $\lambda(w) \vdash n$ is the cycle type of the permutation $w \in \symm_n$. The map $\ch_n$ is an isomorphism of $\CC$-vector spaces.
Furthermore, if we endow $\Lambda_n$ with the {\em Hall inner product} $\langle -, - \rangle$ defined by either of the equivalent conditions
\begin{equation}
\langle p_{\lambda}, p_{\mu} \rangle = z_{\lambda} \cdot \delta_{\lambda,\mu} \quad \quad \text{or} \quad \quad
\langle s_{\lambda}, s_{\mu} \rangle = \delta_{\lambda,\mu}
\end{equation}
where $\delta_{\lambda,\mu}$ is the Kronecker delta, the map $\ch_n: \Class(\symm_n,\CC) \rightarrow \Lambda_n$ is an isometry.

Irreducible representations of $\symm_n$ over $\CC$ are in one-to-one correspondence 
with partitions $\lambda \vdash n$. 
Given a partition $\lambda \vdash n$, the corresponding irreducible representation $V^\lambda$ may be constructed as follows.
For a set $X \subseteq \symm_n$ of permutations, let $[X]_+, [X]_- \in \CC[\symm_n]$ be the group algebra elements
\begin{equation}
[X]_+ := \sum_{w \in X} w \quad \quad [X]_- := \sum_{w \in X} \sign(w) \cdot w
\end{equation}
given by the sum and signed sum over $X$. Fix a bijective filling $T: \lambda \rightarrow [n]$ of the Young diagram of $\lambda$ with $\{1, \dots, n \}$
and let $R(T), C(T) \subseteq \symm_n$ be the subgroups of $\symm_n$ which stabilize the rows and columns of the filling $T$.
The {\em Young idempotent} is the element $\varepsilon_\lambda \in \CC[\symm_n]$ given by
\begin{equation}
\varepsilon_\lambda := [R(T)]_+ \cdot [C(T)]_- \in \CC[\symm_n].
\end{equation}
The group algebra
 element $\varepsilon_\lambda$ depends on the filling $T$, so this is (standard) notational abuse. However, changing $T$ merely conjugates
  $\varepsilon_\lambda$ by an element of $\symm_n$.
 The representation $V^\lambda$ is the left ideal in $\CC[\symm_n]$ generated by $\varepsilon_\lambda$.
 In symbols, we have
 \begin{equation}
 V^\lambda := \CC[\symm_n] \cdot \varepsilon_\lambda
 \end{equation}
 with the natural left action of $\symm_n$. 
 It can be shown that $\{ V^\lambda \,:\, \lambda \vdash n \}$ gives a complete set of irreducible $\symm_n$-modules. For example, the trivial representation
 is $V^{(n)}$ while the sign representation is $V^{(1^n)}$.  We will make no further use of this explicit construction of $V^\lambda$.

 If $\lambda \vdash n$ is a partition, we write
 $\chi^{\lambda}: \symm_n \rightarrow \CC$ for the character of the irreducible module
 $V^{\lambda}$ and $f^{\lambda} := \chi^{\lambda}(e)$ for the dimension of 
$V^{\lambda}$.
If $\mu \vdash n$ is another partition, we write $\chi^{\lambda}_{\mu}$ for the value $\chi^{\lambda}(w)$ of $\chi^{\lambda}$ on any permutation 
$w \in \symm_n$ of cycle type $\mu$.  We have 
\begin{equation}
\ch_n(\chi^{\lambda}) = s_{\lambda},
\end{equation}
so that the isomorphism $\ch_n$ carries the irreducible character basis of $\Class(\symm_n,\CC)$ to the Schur basis of $\Lambda_n$.

The matrix irreducible character values $(\chi^{\lambda}_{\mu})_{\lambda, \mu \vdash n}$ is the {\em character table} of $\symm_n$.
This array gives the transition matrices between the Schur and power sum bases of $\Lambda_n$. More precisely, we have
\begin{equation}
p_{\mu} = \sum_{\lambda \vdash n} \chi^{\lambda}_{\mu} \cdot s_{\lambda} \quad \quad \text{and} \quad \quad
s_{\lambda} = \sum_{\mu \vdash n} \frac{\chi^{\lambda}_{\mu}}{z_{\mu}} \cdot p_{\mu}.
\end{equation}
The {\em Murnaghan-Nakayama Rule} describes the numbers $\chi^{\lambda}_{\mu}$ using the combinatorics of
ribbon tableaux as follows.

 A {\em ribbon} $\xi$ is an edgewise connected skew partition whose Young diagram contains no $2 \times 2$ square. The {\em height} $\height(\xi)$
 of a ribbon $\xi$ is the number of rows occupied by $\xi$, less one. The {\em sign} of $\xi$ is $\sign(\xi) := (-1)^{\height(\xi)}$.
For example, the figure
 
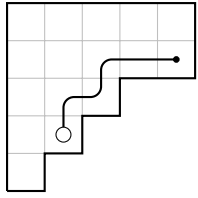
\begin{center}
 \begin{tikzpicture}[scale = .5]
  \begin{scope}
    \clip (0,0) -| (1,1) -| (2,2) -| (3,3) -| (5,5) -| (0,0);
    \draw [color=black!25] (0,0) grid (5,5);
  \end{scope}

  \draw [thick] (0,0) -| (1,1) -| (2,2) -| (3,3) -| (5,5) -| (0,0);

  \draw [thick, rounded corners] (1.5,1.5) |- (2.5,2.5) |- (4.5,3.5) ;
  \draw [color=black,fill=black,thick] (4.5,3.5) circle (.4ex);
  \node [draw, circle, fill = white, inner sep = 2pt] at (1.5,1.5) { };
  \end{tikzpicture}
 \end{center}
 shows a ribbon $\xi$ of size 6, satisfying $\height(\xi) = 2$ and $\sign(\xi) = (-1)^2 = +1$ embedded inside the 
 Young diagram of $(5,5,3,2,1)$.
 Ribbons whose removal from a Young diagram yields a Young diagram (such as above) are called {\em rim hooks}.
 The southwesternmost cell in a ribbon $\xi$ is called its {\em tail}; the tail is decorated with a white circle $\circ$ in the above picture.

 If $\lambda$ is a partition, a {\em standard ribbon tableau} (or {\em standard rim hook tableau}) of shape $\lambda$ is a sequence 
 \begin{equation}
T= ( \varnothing = \lambda^{(0)} \subseteq \lambda^{(1)} \subseteq \cdots \subseteq \lambda^{(r)} = \lambda)
 \end{equation}
 of nested partitions starting at the empty partition $\varnothing$
 and ending at $\lambda$ such that each difference $\xi^{(i)} := \lambda^{(i)}/\lambda^{(i-1)}$ is a ribbon.
 The {\em sign}  of  $T$ is the product 
 \begin{equation}
 \sign(T) := \sign(\xi^{(1)}) \cdots \sign(\xi^{(r)})
 \end{equation}
 of the signs of the ribbons in $T$.
 The {\em type} of $T$ is the composition $\mu = (\mu_1, \dots, \mu_r)$ where $\mu_i = |\xi^{(i)}|$ is the size of the ribbon $\xi^{(i)}$.

\begin{theorem}
\label{mn-rule}
{\em (Classical Murnaghan-Nakayama Rule)}
Let $\mu = (\mu_1, \dots, \mu_r)$ be a composition of $n$ and let $\lambda$ be a partition of $n$.
We have 
\begin{equation}
\label{classical-mn-sum}
p_{\mu} = \sum_T \sign(T) \cdot s_{\shape(T)}
\end{equation}
where the sum is over all standard ribbon tableaux $T$ of type $\mu$.
\end{theorem}

For example, let $\mu = (3,2,2,1)$. The coefficient $\chi^{431}_{3221}$ of 
$s_{431}$ in $p_{3221}$ is the signed sum of standard ribbon tableaux of 
shape $\lambda = (4,3,1)$ and type $\mu = (3,2,2,1)$.  The five ribbon tableaux in question are shown in Figure~\ref{fig:classical-mn}; we conclude that
$\chi^{431}_{3221} = 1 - 1 + 1 - 1 - 1 = -1.$

\begin{figure}
\begin{center}
\begin{tikzpicture}[scale = .4]
  \begin{scope}
    \clip (0,0) -| (1,1) -| (3,2) -| (4,3) -| (0,0);
    \draw [color=black!25] (0,0) grid (4,3);
  \end{scope}

  \draw [thick] (0,0) -| (1,1) -| (3,2) -| (4,3) -| (0,0);

  \draw [thick, rounded corners] (0.5,0.5) -- (0.5,2.5);
  \draw [color=black,fill=black,thick] (0.5,2.5) circle (.4ex);
  \node [draw, circle, fill = white, inner sep = 1pt] at (0.5,0.5)
  {\scriptsize 1};

  \draw [thick, rounded corners] (1.5,1.5) -- (1.5,2.5);
  \draw [color=black,fill=black,thick] (1.5,2.5) circle (.4ex);
  \node [draw, circle, fill = white, inner sep = 1pt] at (1.5,1.5)
  {\scriptsize 2};

  \draw [thick, rounded corners] (2.5,1.5) -- (2.5,2.5);
  \draw [color=black,fill=black,thick] (2.5,2.5) circle (.4ex);
  \node [draw, circle, fill = white, inner sep = 1pt] at (2.5,1.5)
  {\scriptsize 3};

  \node [draw, circle, fill = white, inner sep = 1pt] at (3.5,2.5)
  {\scriptsize 4};

 \node at (2.5,0) {$+1$};

   \begin{scope}
    \clip (6,0) -| (7,1) -| (9,2) -| (10,3) -| (6,0);
    \draw [color=black!25] (6,0) grid (10,3);
  \end{scope}

  \draw [thick] (6,0) -| (7,1) -| (9,2) -| (10,3) -| (6,0);

  \draw [thick, rounded corners] (6.5,0.5) -- (6.5,2.5);
  \draw [color=black,fill=black,thick] (6.5,2.5) circle (.4ex);
  \node [draw, circle, fill = white, inner sep = 1pt] at (6.5,0.5)
  {\scriptsize 1};

  \draw [thick, rounded corners] (7.5,1.5) -- (7.5,2.5);
  \draw [color=black,fill=black,thick] (7.5,2.5) circle (.4ex);
  \node [draw, circle, fill = white, inner sep = 1pt] at (7.5,1.5)
  {\scriptsize 2};

  \draw [thick, rounded corners] (8.5,2.5) -- (9.5,2.5);
  \draw [color=black,fill=black,thick] (9.5,2.5) circle (.4ex);
  \node [draw, circle, fill = white, inner sep = 1pt] at (8.5,2.5)
  {\scriptsize 3};

  \node [draw, circle, fill = white, inner sep = 1pt] at (8.5,1.5)
  {\scriptsize 4};
  
  \node at (8.5,0) {$-1$};

 \begin{scope}
    \clip (12,0) -| (13,1) -| (15,2) -| (16,3) -| (12,0);
    \draw [color=black!25] (12,0) grid (16,3);
  \end{scope}

  \draw [thick] (12,0) -| (13,1) -| (15,2) -| (16,3) -| (12,0);

  \draw [thick, rounded corners] (12.5,0.5) -- (12.5,2.5);
  \draw [color=black,fill=black,thick] (12.5,2.5) circle (.4ex);
  \node [draw, circle, fill = white, inner sep = 1pt] at (12.5,0.5)
  {\scriptsize 1};

  \draw [thick, rounded corners] (13.5,2.5) -- (14.5,2.5);
  \draw [color=black,fill=black,thick] (14.5,2.5) circle (.4ex);
  \node [draw, circle, fill = white, inner sep = 1pt] at (13.5,2.5)
  {\scriptsize 2};

  \draw [thick, rounded corners] (13.5,1.5) -- (14.5,1.5);
  \draw [color=black,fill=black,thick] (14.5,1.5) circle (.4ex);
  \node [draw, circle, fill = white, inner sep = 1pt] at (13.5,1.5)
  {\scriptsize 3};

  \node [draw, circle, fill = white, inner sep = 1pt] at (15.5,2.5)
  {\scriptsize 4};
  
  \node at (14.5,0) {$+1$};

   \begin{scope}
    \clip (18,0) -| (19,1) -| (21,2) -| (22,3) -| (18,0);
    \draw [color=black!25] (18,0) grid (22,3);
  \end{scope}

  \draw [thick]  (18,0) -| (19,1) -| (21,2) -| (22,3) -| (18,0);

  \draw [thick, rounded corners] (18.5,2.5) -- (20.5,2.5);
  \draw [color=black,fill=black,thick] (20.5,2.5) circle (.4ex);
  \node [draw, circle, fill = white, inner sep = 1pt] at (18.5,2.5)
  {\scriptsize 1};

  \draw [thick, rounded corners] (18.5,0.5) -- (18.5,1.5);
  \draw [color=black,fill=black,thick] (18.5,1.5) circle (.4ex);
  \node [draw, circle, fill = white, inner sep = 1pt] at (18.5,0.5)
  {\scriptsize 2};

  \draw [thick, rounded corners] (19.5,1.5) -- (20.5,1.5);
  \draw [color=black,fill=black,thick] (20.5,1.5) circle (.4ex);
  \node [draw, circle, fill = white, inner sep = 1pt] at (19.5,1.5)
  {\scriptsize 3};

  \node [draw, circle, fill = white, inner sep = 1pt] at (21.5,2.5)
  {\scriptsize 4};
  
  \node at (20.5,0) {$-1$};

     \begin{scope}
    \clip (24,0) -| (25,1) -| (27,2) -| (28,3) -| (24,0);
    \draw [color=black!25] (24,0) grid (28,3);
  \end{scope}

  \draw [thick]  (24,0) -| (25,1) -| (27,2) -| (28,3) -| (24,0);

  \draw [thick, rounded corners] (24.5,1.5) |- (25.5,2.5);
  \draw [color=black,fill=black,thick] (25.5,2.5) circle (.4ex);
  \node [draw, circle, fill = white, inner sep = 1pt] at (24.5,1.5)
  {\scriptsize 1};

  \draw [thick, rounded corners] (26.5,2.5) -- (27.5,2.5);
  \draw [color=black,fill=black,thick] (27.5,2.5) circle (.4ex);
  \node [draw, circle, fill = white, inner sep = 1pt] at (26.5,2.5)
  {\scriptsize 2};

  \draw [thick, rounded corners] (25.5,1.5) -- (26.5,1.5);
  \draw [color=black,fill=black,thick] (26.5,1.5) circle (.4ex);
  \node [draw, circle, fill = white, inner sep = 1pt] at (25.5,1.5)
  {\scriptsize 3};

  \node [draw, circle, fill = white, inner sep = 1pt] at (24.5,0.5)
  {\scriptsize 4};

\node at (26.5,0) {$-1$};
  
\end{tikzpicture}
\end{center}
\caption{The standard ribbon tableaux computing $\chi^{431}_{3221}$ together
with their signs. The numbers in the tails indicate the order in which ribbons are added.}
\label{fig:classical-mn}
\end{figure}
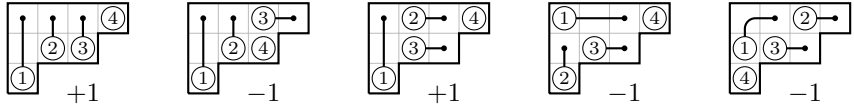

The status of $\chi^{\lambda}_{\mu}$ as a character evaluation guarantees that $\chi^{\lambda}_{\mu}$ is invariant
under permutations of $\mu = (\mu_1, \dots, \mu_r)$, although this is not evident from
Theorem~\ref{mn-rule}.
Changing the order of the sequence $\mu = (\mu_1, \dots, \mu_r)$ will usually result in different sets of standard ribbon tableaux
(and these sets can have different sizes), but the signed ribbon tableau count is independent of the order of $(\mu_1, \dots, \mu_r)$.
A combinatorial proof of this fact was discovered by 
Mendes  \cite{Mendes}.
Figure~\ref{fig:classical-mn} shows that the Murnaghan-Nakayama Rule is not cancellation free.
Finding a cancellation free expression for the irreducible character value $\chi^{\lambda}_{\mu}$ is an important open problem.

For a composition
$\mu \models n$, let $\symm_{\mu} = \symm_{\mu_1} \times \symm_{\mu_2} \times \cdots $ be the associated parabolic subgroup of 
$\symm_n$. The permutation modules
\begin{equation}
M^{\mu} := \CC[\symm_n/\symm_{\mu}] = \Ind_{\symm_{\mu}}^{\symm_n}( \triv_{\symm_{\mu}})
\end{equation}
are an important class of $\symm_n$-representations. Writing $\eta^{\mu}: \symm_n \rightarrow \CC$ for the character of $M^{\mu}$, 
we have
$\ch_n(\eta^{\mu}) = h_{\mu}$.

The coefficients in $K_{\lambda,\mu}$ in the equations
\begin{equation}
h_{\mu} = \sum_{\lambda \vdash n} K_{\lambda,\mu} \cdot s_{\lambda} \quad \quad \text{and} \quad \quad
\eta^{\mu} = \sum_{\lambda \vdash n} K_{\lambda, \mu} \cdot \chi^{\lambda}
\end{equation}
are called {\em Kostka numbers}. Combinatorially, they are given as follows.  
The {\em content} of a tableau $T$ is the sequence $\mu = (\mu_1, \mu_2, \dots )$ where $\mu_i$ is the number of copies of $i$ in $T$.
{\em Young's Rule} states that $K_{\lambda,\mu}$ is the number of semistandard tableaux of shape $\lambda$ and content $\mu$.



\section{Character polynomials}
We will be interested in class functions on $\symm_n$ as $n$ grows.
Given a partition $\lambda = (\lambda_1, \lambda_2, \dots )$, the {\em padded partition} is the sequence 
$\lambda[n] = (n - |\lambda|, \lambda_1, \lambda_2, \dots )$.
Note $\lambda[n]$ is a partition of $n$ if and only if $n \geq |\lambda| + \lambda_1.$

Let $\mu \vdash n$ be a partition with irreducible character $\chi^{\mu}: \symm_n \rightarrow \RR$ and let $w \in \symm_n$.
If $m_i(w)$ is the number of $i$-cycles in $w$, then $\chi^{\mu}$ is a function of $m_1(w), m_2(w), \dots, m_n(w)$.
The theory of {\em character polynomials} gives a more precise description of this dependence.

\begin{theorem}
\label{character-polynomial-theorem}
Let $\lambda$ be a partition of $k$.  For any $n \geq |\lambda| - \lambda_1$ and any $w \in \symm_n$,
the character evaluation $\chi^{\lambda[n]}(w)$ is a polynomial in $m_1(w), m_2(w), \dots, m_k(w)$ with rational coefficients.
Furthermore, with respect to the grading which places $m_i(w)$ in degree $i$, this polynomial has degree $k$.
\end{theorem}

Theorem~\ref{character-polynomial-theorem} implies that 
$\chi^{\lambda[n]}(w)$ does not depend on the number of $i$-cycles in $w$ for $i > k$.
The polynomial $q_{\lambda}(m_1,\dots,m_k) \in \QQ[m_1, \dots, m_k]$ which satisfies 
\begin{equation}
\label{character-polynomial-equation}
\chi^{\lambda[n]} = q_{\lambda}(m_1, \dots, m_k) \quad \quad \text{for all $n \geq |\lambda| + \lambda_1$}
\end{equation}
is the {\em character polynomial} of $\lambda$.  A famous example is given by
$q_{(1)}(m_1) = m_1 - 1$, corresponding to the fact that $\chi^{(n-1,1)}(w)$ is the number of fixed points of $w \in \symm_n$, less one.
Garsia and Goupil \cite{GG} described an algorithm for computing the polynomials $q_{\lambda}(m_1, \dots, m_k)$ in general.
See Figure~\ref{fig:char-poly} or the tables
in \cite{Specht}
for more examples.

\begin{figure}
	
	\begin{tabular}{c|c}
		Partition & Polynomial\\
		\hline
		$(n)$ & $1$\\
		\hline
		$(n-1,1)$ & $m_1 - 1$\\
		\hline
		$(n-2,2)$ & $m_2 - m_1 + (m_1)_2/2$\\
		\hline
		$(n-2,1,1)$ & $-m_2 - m_1 + (m_1)_2/2  + 1$\\
		\hline
		$(n-3,3)$ & $m_3 + m_1m_2 - m_2 - (m_1)_2/2 + (m_1)_3/6$ \\
		\hline
		$(n-3,2,1)$ & $-m_3 - (m_1)_2 + (m_1)_3/3 + m_1$\\
		\hline
		$(n-3,1,1,1)$ & $m_3 - m_1m_2 + m_2 - (m_1)_2/2 + (m_1)_3/6 + m_1 - 1$
	\end{tabular}
	
	\caption{\label{fig:char-poly} The character polynomials for $\chi^{\lambda[n]}$ with $\lambda \vdash k \leq 3$.
	Here $(x)_k = x(x-1) \dots (x-k+1)$.}
\end{figure}

If $n < |\lambda| + \lambda_1$, then $\lambda[n]$ is not a partition. Equation~\eqref{character-polynomial-equation} extends to this setting,
but we must interpret the class function $\chi^{\lambda[n]}: \symm_n \rightarrow \CC$ by {\em straightening} $\lambda[n]$ into a partition;
see \cite{Macdonald}.
This results in plus or minus an irreducible character $\chi^{\nu}: \symm_n \rightarrow \CC$, or zero.

For a proof of Theorem~\ref{character-polynomial-theorem} we refer the reader
to \cite[Chapter 1, Section 7, Example 14]{Macdonald}.
Theorem~\ref{character-polynomial-theorem} will allow us to transfer from Schur expansions of $\ch_n(R \, f)$ to 
asymptotics for $R \, f$ itself, where $f: \symm_n \rightarrow \CC$ is a $k$-local statistic.

\section{Probability on $\symm_n$}

We go over basic material on probability theory.
For $S$ a countable set, a \emph{discrete probability measure} is a function $\mu: S \to \RR_{\geq 0}$ such that
\begin{equation}
\sum_{w \in S} \mu(w) = 1.
\end{equation}
The pair $(S,\mu)$ is a \emph{(discrete) probability space}.
For $(S,\mu)$ a probability space, the \emph{probability} of $A \subseteq S$ is
\begin{equation}
\PP(A) = \sum_{a \in A} \mu(a).	
\end{equation}
Unless otherwise stated, we assume our measure on $\symm_n$ is the uniform measure on the symmetric group, given by $\mu(w) = \frac{1}{n!}$ for all $w \in \symm_n$.

A {\em (real) random variable} on a discrete probability space $(S,\mu)$  is a function $X: S \rightarrow \RR$.
We often view permutation statistics as random variables.
The {\em expectation} of $X$ is the quantity
\begin{equation}
\EE(X) := \sum_{w \in S} X(w) \cdot \mu(w).
\end{equation}
For any subset $A \subseteq S$ with $\PP(A) \neq 0$, the {\em conditional expectation} is given by
\begin{equation}
\EE(X \mid A) := \frac{1}{\PP(A)} \sum_{w \in A} X(w) \cdot \mu(w). 
\end{equation}
Viewing $f:\symm_n \to \RR$ as a random variable with respect to the uniform measure, for $w \in K_\lambda$ note that
\begin{equation}
Rf(w) =	\EE(f \mid K_\lambda).
\end{equation}

The {\em cumulative distribution function} of a real random variable $X$
 is the map $F_X: \RR \rightarrow \RR$ given by
\begin{equation}
F_X(x) := \PP(X \leq x).
\end{equation}
The {\em probability density function} of $X$ is the distribution
\begin{equation}
f_X = \sum_{x \in \RR}  \PP(X = x) \cdot \delta_x
\end{equation}
in the discrete case (where $\delta_x$ is the Dirac distribution concentrated at $x \in \RR$)  and the function
\begin{equation}
f_X = \frac{d F_X}{dx}
\end{equation}
when $F_X: \RR \rightarrow \RR$ is differentiable.
A {\em probability distribution} is a function (or more generally a distribution) $\rho: \RR \rightarrow \RR_{\geq 0}$ satisfying 
$\int_{\RR} \rho = 1$.
If $\rho$ is a probability distribution and $X$ is a random variable, we write $X \sim \rho$ to mean that $f_X = \rho$.
For this paper, the most important distribution is the {\em normal distribution} $\NNN(\mu,\sigma^2)$ of mean $\mu$ and variance $\sigma^2$
given by the density 
\begin{equation}
f(x) = \frac{1}{\sigma \sqrt{2 \pi}} e^{- \frac{1}{2} \left( \frac{x - \mu}{\sigma} \right)^2}.	
\end{equation}

If $X_1, X_2, \dots, Y$ are random variables, we say that the sequence $X_n$ {\em converges to $Y$ in distribution} and write
$X_n \xrightarrow{d} Y$ if 
\begin{equation}
\lim_{n \rightarrow \infty} F_{X_n}(x) = F_Y(x)
\end{equation}
for all $x \in \RR$.
If $\rho$ is a probability distribution, we write $X_n \xrightarrow{d} \rho$ to mean $X_n \xrightarrow{d} Y$ for $Y \sim \rho$.
A stronger notion of convergence is \emph{convergence in probability}, denoted $X_n \xrightarrow{p} Y$, which says for every $\epsilon >0$ that
\begin{equation}
\lim_{n \to \infty}	\PP(|X_n - Y|> \epsilon) = 0.
\end{equation}
Another notion of convergence is \emph{almost sure convergence}, which says
\begin{equation}
	\PP\left( \lim_{n\to \infty} X_n = Y \right) = 1.
\end{equation}
Almost sure convergence implies convergence in probability, which in turn implies convergence in distribution.

Let $X$ be a random variable on a probability space.
For $d \geq 0$, the {\em $d^{th}$ moment} of $X$ is $\EE(X^d)$, provided this expectation is finite.
 In particular, the first moment is the usual expectation while the second moment is closely related to the {\em variance} $\VV(X) := \EE(X^2) - \EE(X)^2$.
If we have a convergence $X_n \xrightarrow{d} Y$ of random variables in distribution, for any $d \geq 0$ the corresponding moments satisfy
\begin{equation}
\lim_{n \rightarrow \infty} \EE(X_n^d) = \EE(Y^d),
\end{equation}
provided these expectations are finite. The {\em Method of Moments} gives a kind of converse to this statement as follows; see e.g.
\cite[Thm. 30.2]{Billingsley} for a proof.

\begin{theorem}
\label{method-of-moments}
{\em (Method of Moments)}
Let $X_1, X_2, \dots $ and $Y$ be real random variables with finite $d^{th}$ moments for all $d$.  
Assume that the random variable $Y$ is determined by its moments.
If
$\lim_{n \rightarrow \infty} \EE(X_n^d) = \EE(Y^d)$ for all $d$,
then $X \xrightarrow{d} Y$.
\end{theorem}

The condition in  Theorem~\ref{method-of-moments} that $Y$ is determined by its moments is satisfied when the moments 
$\EE(Y^d)$ do not grow too quickly in $d$.  
In our applications, the random variable $Y$ will always be normally distributed, and this condition will be satisfied.

\chapter{Local Statistics}
\label{Local}

We develop fundamental algebraic properties of the space $\Loc_k(\symm_n,\CC)$ of $k$-local functions on $\symm_n$.
The dimension of $\Loc_k(\symm_n,\CC)$ is the number permutations $w \in \symm_n$ that have an increasing subsequence of length $n-k$.
We present a vector space basis of $\Loc_k(\symm_n,\CC)$ related to the shadow line construction of Viennot \cite{Viennot} due to the second author~\cite{Rhoades}.
Our primary tool is an interpretation of $\Loc_k(\symm_n,\CC)$ in terms of the Artin-Weddernburn isomorphism applied to the symmetric group algebra 
$\CC[\symm_n]$.
For any function $f: \symm_n \rightarrow \CC$, the Artin-Wedderburn perspective gives rise to a best $k$-local approximation $L_k f \in \Loc_k(\symm_n,\CC)$
to the function $f$.
As another consequence, we describe the image of the space of $k$-local class functions under the Frobenius characteristic map $\ch_n$.
We also deduce an old result of Murnaghan \cite{Murnaghan} on bounds for Kronecker coefficients.

Most of the results in this chapter (aside from those in Section~\ref{local-basis})
have appeared in various guises in the probability, machine learning, and extremal combinatorics literature.
The main purpose of this chapter is to give a comprehensive exposition of these ideas through the lens of Artin-Wedderburn theory.
We hope that this perspective will be enriching for 
those
the enumerative and algebraic combinatorics communities
who have are outsiders to the probabilistic perspective.
We give a leisurely exposition of the representation-theoretic ideas involved for the benefit of readers from outside the realms of algebra.

\section{The local filtration}
Recall from the introduction
 that a statistic $f: \symm_n \rightarrow \CC$ is $k$-local if 
there exist constants $c_{I,J} \in \CC$ such that $f = \sum_{(I,J) \in \symm_{n,k}} c_{I,J} \cdot \one_{I,J}$ where 
$\one_{I,J}(w) = 1$ if $w$ sends the list $I$ to the list $J$ and $0$ otherwise.
We write
\begin{equation}
\Loc_k(\symm_n,\CC) := \{ f: \symm_n \rightarrow \CC \,:\, \text{$f$ is $k$-local} \}
\end{equation}
for the $\CC$-vector space of all $k$-local statistics on $\symm_n$. 
It may be tempting to define a statistic $f$ to be `$\leq k$-local' if it is a linear combination of $\one_{I,J}$ for partial permutations $(I,J)$
of size $\leq k$, but the following observation
makes this unnecessary.
If $I = (i_1, \dots, i_k)$ is a list of integers and $a$ is another integer, we use the shorthand $Ia := (i_1, \dots, i_k,a)$.

\begin{observation}
\label{lower-closure}
For any $(I,J) \in \symm_{n,k-1}$, the statistic $\one_{I,J}: \symm_n \rightarrow \CC$ is $k$-local.
Indeed, for any fixed $a \notin I$ we have 
\begin{equation}
\one_{I,J}(w) = \sum_{b \notin J} \one_{Ia,Jb}(w)
\end{equation}
for any $w \in \symm_n$. In particular, any $(k-1)$-local statistic is also $k$-local.
\end{observation}

Since the values of any permutation $w \in \symm_n$ at any $n-1$ positions determine the value of $w$ at the remaining position,
any permutation statistic $f: \symm_n \rightarrow \CC$ is $(n-1)$-local.
Observation~\ref{lower-closure} implies that we have a filtration
of the space $\Fun(\symm_n,\CC)$ of complex-valued functions on $\symm_n$ given by the subspaces
\begin{equation}
\label{local-filtration}
\Loc_0(\symm_n,\CC) \subseteq \Loc_1(\symm_n,\CC) \subseteq \cdots \subseteq \Loc_{n-1}(\symm_n,\CC) = \Fun(\symm_n,\CC).
\end{equation}

To understand how the filtration \eqref{local-filtration} behaves under pointwise multiplication, we use the following lemma.
Two partial permutations $(I_1,J_1), (I_2,J_2)$ of $[n]$ are {\em compatible} if there exists $w \in \symm_n$ such that
$w(I_1) = J_1$ and $w(I_2) = J_2$.  Equivalently, the directed graph on $[n]$ whose edges are the union
$G_n(I_1,J_1) \cup G_n(I_2,J_2)$ is a disjoint union of directed paths and directed cycles (i.e. the graph of a partial permutation).

\begin{lemma}
\label{compatible-lemma}
Let $(I_1, J_1)$ and $(I_2, J_2)$ be compatible partial permutations of $[n]$. Let $(I,J)$ be the partial permutation of $[n]$
whose graph has edges 
\begin{equation*}
G_n(I,J) = 
G_n(I_1,J_1) \cup G_n(I_2,J_2). 
\end{equation*}
Then 
$\one_{I_1,J_1} \cdot \one_{I_2,J_2} = \one_{I,J}$ as functions $\symm_n \rightarrow \CC$.
\end{lemma}

\begin{proof}
A permutation $w \in \symm_n$ satisfies $w(I) = J$ if and only if $w(I_1) = J_1$ and $w(I_2) = J_2$.
\end{proof}

Recall that the $d^{th}$ moment of a statistic $f: \symm_n \rightarrow \CC$ is the map
$f^d: \symm_n \rightarrow \CC$ given by $f^d(w) := f(w)^d$.
Pointwise products (and in particular moments) of statistics behave well with respect locality.

\begin{proposition}
\label{product-local}
If $f: \symm_n \rightarrow \CC$ is $k$-local and $g: \symm_n \rightarrow \CC$ is $\ell$-local, their pointwise product
$f \cdot g$ is $(k + \ell)$-local. In particular, the $d^{th}$ moment $f^d$ of $f$ is $(d \cdot \ell)$-local.
\end{proposition}

\begin{proof}
Apply Lemma~\ref{compatible-lemma}.
\end{proof}

Proposition~\ref{product-local} states that \eqref{local-filtration} is a filtration of the algebra $\Fun(\symm_n,\CC)$ of 
functions $f: \symm_n \rightarrow \CC$
under the operations of pointwise linear combinations and product.
The product group $\symm_n \times \symm_n$ acts naturally on the space $\Fun(\symm_n,\CC)$ by the rule
\begin{equation}
((u,v) \cdot f)(w) := f(u^{-1} w v) \quad \quad u, v, w \in \symm_n, \, \, f: \symm_n \rightarrow \CC.
\end{equation}
The following observation implies that the filtration \eqref{local-filtration} is preserved by this action.

\begin{lemma}
\label{product-action-on-indicators}
Let $(I,J) \in \symm_{n,k}$ be a partial permutation and let $u, v \in \symm_n$.  We have 
\begin{equation}
(u,v) \cdot \one_{I,J} = \one_{vI,uJ}.
\end{equation}
\end{lemma}

\begin{proof}
Let $w \in \symm_n$. We calculate
\begin{align}
((u,v) \cdot \one_{I,J})(w) = \one_{I,J}(u^{-1} w v) &= \begin{cases}
1 & u^{-1} w v I = J \\
0 & \text{else}
\end{cases}  \\
&= \begin{cases}
1 & w v I = u J \\
0 & \text{else}
\end{cases} = \one_{vI,uJ}(w).\nonumber
\end{align}
\end{proof}

Although the family $\{ \one_{I,J} \,:\, (I,J) \in \symm_{n,k} \}$ of indicator functions spans $\symm_{n,k}$ by definition, 
this spanning set has many linear dependencies.  For example, we have 
$\one_{1,1} + \one_{1,2} + \cdots + \one_{1,n} = \one_{1,1} + \one_{2,1} + \cdots + \one_{n,1}$ as functions on $\symm_n$.
The reader may find it pleasant to verify that 
\begin{equation}
\{ \one_{i,j} \,:\, 2 \leq i, j \leq n \} \cup \{ \one_{\varnothing,\varnothing} \}
\end{equation}
is a basis for the space $\Loc_1(\symm_n,\CC)$ of 1-local functions on $\symm_n$, so that 
\begin{equation}
\dim \Loc_1(\symm_n,\CC) = n^2 - 2n + 2.
\end{equation}
The dimension of $\Loc_k(\symm_n,\CC)$ for general $k$ is best understood via representation theory.

\section{Local statistics and the Artin-Wedderburn Theorem}

Let $G$ be a finite group and let $\Irr(G)$ be the set of (isomorphism classes of) complex irreducible representations of $G$.
Given $V \in \Irr(G)$, let $\End_\CC(V)$ be the $\CC$-algebra of endomorphisms of $V$. If $\dim V = d$, then $\End_\CC(V)$ is isomorphic to the algebra 
$\Mat_d(\CC)$ of $d \times d$ complex matrices.
The Artin-Wedderburn Theorem is the following paramount result in representation theory; we will apply it in the following form.

\begin{theorem}
\label{artin-wedderburn}
{\em (Artin-Wedderburn)}
Let $G$ be a finite group and let $\CC[G]$ be its group algebra with complex coefficients. We have an isomorphism of algebras
\begin{equation}
\Psi: \CC[G] \xrightarrow{ \, \, \sim \, \, } \bigoplus_{V \in \Irr(G)} \End_\CC(V)
\end{equation}
induced by $\Psi(\alpha): v \mapsto \alpha \cdot v$ for any irreducible representation $V \in \Irr(G)$ and any vector $v \in V$.
\end{theorem}

Although standard,
Theorem~\ref{artin-wedderburn} is so important that we give a proof.

\begin{proof}
It is easy check that $\Psi$ is linear and that $\Psi(\alpha \cdot \beta) = \Psi(\alpha) \circ \Psi(\beta)$ for all $\alpha, \beta \in \CC[G]$.
It is a consequence of character orthogonality (see e.g. \cite{Sagan}) that the left regular representation $\CC[G]$ of $G$ 
decomposes into irreducibles according to
\begin{equation}
\label{CG-decomposition}
\CC[G] \cong_G \bigoplus_{V \in \Irr(G)} (\dim V) \cdot V.
\end{equation} 
Taking dimensions gives $|G| = \sum_{V \in \Irr(G)} (\dim V)^2$, so 
 $\Psi$  is  a $\CC$-algebra homomorphism between algebras of the same 
dimension. It therefore suffices to show that $\Psi$ is injective.  

Let $\alpha \in \CC[G]$ and suppose $\Psi(\alpha) = 0$.
Then $\alpha$ acts by 0 on any irreducible representation of $G$.
In particular, the decomposition \eqref{CG-decomposition} implies that 
$\alpha$ acts by 0 on the group algebra $\CC[G]$. It follows that $\alpha = \alpha \cdot e = 0$ where $e \in G$ is the identity element.
\end{proof}

The isomorphism $\Psi$ called  the {\em Fourier transform} in probability \cite{Diaconis1, Diaconis2, Terras}
and machine learning  \cite{HGG, KHJ}.
 Indeed, when $G = Z_n$ is a cyclic group, the representations in $\Irr(Z_n)$ are labeled by $n^{th}$ roots-of-unity
 and the components of $\Psi$
record the usual
discrete Fourier transform of a function $Z_n \rightarrow \CC$.
We will apply Theorem~\ref{artin-wedderburn} in the case $G = \symm_n$.
In this setting, the isomorphism of Theorem~\ref{artin-wedderburn} is a map
\begin{equation}
\label{aw-sn}
\Psi: \CC[\symm_n] \xrightarrow{ \, \, \sim \, \, } \bigoplus_{\lambda \, \vdash \, n} \End_{\CC}(V^{\lambda}).
\end{equation}

We have a natural bijective correspondence between the space $\Fun(\symm_n,\CC)$ and the group algebra $\CC[\symm_n]$ given by
\begin{equation}
f: \symm_n \rightarrow \CC \quad \quad \leftrightarrow \quad \quad \alpha_f = \sum_{w \in \symm_n} f(w) \cdot w.
\end{equation}
Under this correspondence, we may regard $\Loc_k(\symm_n,\CC)$ as a subspace of $\CC[\symm_n]$ for all $k \geq 0$.
The image of $\Loc_k(\symm_n,\CC)$ under the Artin-Weddernburn isomorphism $\Psi$ has a nice description.
In order to prove this result, we recall two standard facts from ring theory.

Let $R$ be a ring, not necessarily commutative. Recall that $R$ is {\em simple} if the only two-sided ideals $I \subseteq R$ 
are $I = 0, R$.  
Ideals in direct sums of simple rings 
are described as follows.

\begin{lemma}
\label{semisimple-ideals}
Let $R_1, \dots, R_n$ be simple rings and let $R = R_1 \oplus \cdots \oplus R_n$ be their direct sum.
Ideals in $R$ are indexed by subsets of $[n]$; 
if $J \subseteq [n]$, the corresponding ideal is given by $I_J = (I_J)_1 \oplus \cdots \oplus (I_J)_n$ with components
\begin{equation}
(I_J)_i = \begin{cases}
R_i & i \in J, \\
0 & i \notin J.
\end{cases}
\end{equation}
\end{lemma}

We give the standard proof of Lemma~\ref{semisimple-ideals}.

\begin{proof}
For $1 \leq i \leq n$, let $e_i \in R$ denote the $i^{th}$ `standard basis vector', with a 1 in position $i$ and 0s elsewhere.
Then $e_i R e_i$ is a ring with multiplicative identity $e_i$ and we have a natural identification $e_i R e_i = R_i$.

Let $I \subseteq R$ be a two-sided ideal. We have $e_i I e_i = I \cap e_i R e_i = I \cap R_i$ for each $i$, where the containment $e_i I e_i \supseteq I \cap e_i R e_i$ uses 
$e_i^2 = e_i$. Since $R_i$ is simple, we have $e_i I e_i = 0$ or $R_i$ for each $i$. Furthermore, since $e_1 + \cdots + e_n = 1$ and $e_i R e_j = 0$ for any $i \neq j$, we have
\begin{multline}
I = (e_1 + \cdots + e_n) \cdot I \cdot (e_1 + \cdots + e_n) = \sum_{i, j \, = \, 1}^n e_i I e_j
\\ = \sum_{i \, = \, 1}^n e_i I e_i
= \bigoplus_{i \, = \, 1}^n e_i I e_i = \bigoplus_{j \, \in \, J} R_j = I_J
\end{multline}
where $J = \{1 \leq j \leq n \,:\, e_j I e_j \neq  0 \}$.
\end{proof}

Ring $R$  which are direct sums of simple rings
as in Lemma~\ref{semisimple-ideals} are called {\em semisimple}.
The following lemma implies that the rings arising in the Artin-Wedderburn Theorem are semisimple.

\begin{lemma}
\label{matrix-rings-are-simple}
Let $\FF$ be a field. The ring $\Mat_d(\FF)$ of $d \times d$ matrices over $\FF$ is simple.
\end{lemma}

Lemma~\ref{matrix-rings-are-simple} is proven in many places; we give a proof here for the convenience of the reader.

\begin{proof}
For $1 \leq i,j \leq d$, let $E_{i,j} \in \Mat_d(\mathbb{F})$ be the matrix with a $1$ in position $(i,j)$ and $0$s elsewhere.
Let $\III \subseteq \Mat_d(\FF)$ be a nonzero ideal and choose a nonzero matrix $A = (a_{i,j}) \in \III$.  
There exist $1 \leq i_0, j_0 \leq d$ such that $a_{i_0,j_0} \neq 0$.  It follows that 
\begin{equation}
E_{i_0, i_0} \cdot A \cdot E_{j_0,j_0} = a_{i_0,j_0}  E_{i_0,j_0} \in \III.
\end{equation}
Furthermore, if $I_d \in \Mat_d(\FF)$ is the identity matrix, we have 
\begin{equation}
\left( a_{i_0,j_0}^{-1} I_d \right) \cdot \left( a_{i_0,j_0} E_{i_0,j_0} \right) = E_{i_0, j_0} \in \III
\end{equation}
where we used the fact that $a_{i_0,j_0}$ is a nonzero element of $\FF$.
For any $1 \leq i, j \leq d$, we deduce that
\begin{equation}
E_{i,i_0} \cdot E_{i_0,j_0} \cdot E_{j_0,j} = E_{i,j} \in \III
\end{equation}
so that $\{ E_{i,j} \,:\, 1 \leq i,j \leq d \} \subseteq \III$.  Finally, since $(a I_d) \cdot B = a B$ for all $a \in \FF$ and $B \in \Mat_d(\FF)$,
we see that $\III$ is closed under taking $\FF$-linear combinations, which forces $\III = \Mat_d(\FF)$.
\end{proof}

With Lemma~\ref{matrix-rings-are-simple} in hand, we 
can describe the image of the $k$-local subspace of $\CC[\symm_n]$ under the Artin-Wedderburn isomorphism.
The following theorem is true
(with the same proof) 
when $\CC$ is replaced by any field
$\FF$ for which $\FF[\symm_n]$ is semisimple.
This happens when $\FF$ has characteristic 
zero or positive characteristic $p > n$.

\begin{theorem}
\label{aw-local}
Let $k \geq 0$ and regard the space $\Loc_k(\symm_n,\CC)$ of $k$-local functions $\symm_n \rightarrow \CC$ as a subspace of the 
symmetric group algebra $\CC[\symm_n]$.  We have
 the equality of $\CC$-algebras
\begin{equation}
 \Psi \left(   \Loc_k(\symm_n,\CC) \right) = \bigoplus_{\lambda_1 \, \geq \, n-k} \End_\CC(V^{\lambda})
\end{equation}
where $\Psi$ is the Artin-Weddernburn isomorphism  \eqref{aw-sn}.
\end{theorem}

\begin{proof}
For $r \leq n$, embed $\symm_r \subseteq \symm_n$ by acting on the first $r$ letters.
We have 
the group algebra element $[\symm_r]_+ \in \CC[\symm_n]$ given by
$[\symm_r]_+ = \sum_{w \in \symm_r} w$. If $V$ is any $\symm_n$-module, it is not hard to see that the image of $V$ under 
$[\symm_r]_+$ coincides with the $\symm_r$-fixed subspace of $V$, that is
\begin{equation}
[\symm_r]_+ \cdot V = V^{\symm_r} := \{ v \in V \,:\, w \cdot v = v \text{ for all $w \in \symm_r$} \}.
\end{equation}

Observe that $[\symm_{n-k}]_+$ is the group algebra element corresponding to the indicator function $\one_{I,J}: \symm_n \rightarrow \CC$ 
where $I = J = (n-k+1, \dots, n-1, n)$.
Since $\Loc_k(\symm_n,\CC)$ is spanned by $\{ \one_{I,J} \,:\, (I,J) \in \symm_{n,k} \}$, Lemma~\ref{product-action-on-indicators} identifies 
$\Loc_k(\symm_n,\CC)$ with the two-sided ideal $\III_k \subseteq \CC[\symm_n]$ generated by $[\symm_{n-k}]_+$.
We show that 
\begin{equation}
 \Psi \left(   \III_k \right) = \bigoplus_{\lambda_1 \geq n-k} \End_\CC(V^\lambda)
\end{equation}
as follows.
 
 Lemma~\ref{matrix-rings-are-simple} implies that 
 the matrix ring $\End_\CC(V^\lambda)$ is simple for all $\lambda \vdash n$.
Since $\Psi(\III_k)$ is a two-sided ideal in the direct sum $\bigoplus_{\lambda \vdash n} \End_\CC(V^\lambda)$,
by Lemma~\ref{semisimple-ideals}
 there exists a family $P_k$ of partitions of $n$ 
such that 
\begin{equation}
 \Psi \left(   \III_k \right) = \bigoplus_{\lambda \in P_k} \End_\CC(V^\lambda).
\end{equation}
By Theorem~\ref{artin-wedderburn}, the family $P_k$ is characterized by
\begin{equation}
P_k = \{ \lambda \vdash n \,:\, \III_k \cdot V^{\lambda} \neq 0 \} = \{ \lambda \vdash n \,:\, [\symm_{n-k}]_+ \cdot V^{\lambda} \neq 0 \}
\end{equation}
where the second equality holds because the ideal $\III_k$ is generated by $[\symm_{n-k}]_+$.

Let $\lambda \vdash n$.  We are reduced to showing that 
$[\symm_{n-k}]_+ \cdot V^{\lambda} \neq 0$ if and only if $\lambda_1 \geq n-k$.
But we have 
\begin{equation}
[\symm_{n-k}]_+ \cdot V^{\lambda} = \left( V^{\lambda} \right)^{\symm_{n-k}} 
\end{equation}
so that
\begin{equation}
\dim \left( [\symm_{n-k}]_+ \cdot V^{\lambda}  \right) = \mult_\triv \left( \Res^{\symm_n}_{\symm_{n-k}} V^{\lambda} \right)
\end{equation}
is the multiplicity of the trivial character in the restriction of $V^\lambda$ from $\symm_n$ to $\symm_{n-k}$.
By the Branching Rule for irreducible $\symm_n$-representations (see e.g. \cite{Sagan}), this multiplicity is nonzero if and only if 
$\lambda_1 \geq n-k$.
\end{proof}

\begin{remark}
Theorem~\ref{aw-local} (or portions thereof) has appeared throughout the literature in various guises.
In machine learning, 
Huang, Guerstin, and Guibas \cite[Appendix C]{HGG} studied the isomorphism $\Psi$ with respect to {\em Young's orthogonal basis}
(or the {\em Gelfand-Tsetlin basis}) of $\symm_n$-irreducibles.
They proved a result equivalent to the assertion $\Psi(\Loc_k(\symm_n,\CC)) \subseteq \bigoplus_{\lambda_1 \, \geq \, n-k} \End_\CC(V^\lambda)$.
Meanwhile, Ellis, Friedgut and Pilpel proved the reverse inclusion as~\cite[Thm. 7]{EFP} using the Branching Rule for symmetric group representations.
In the context of permutation patterns, Even-Zohar \cite[Lem. 3]{EZ} proved a `dual' result which classifies functions 
$f: \symm_n \rightarrow \CC$ which vanish when summed over any two-sided coset $u \cdot \symm_k \cdot v$ of $\symm_k$ inside $\symm_n$.
The Peter-Weyl Theorem implies that Even-Zohar's \cite[Lem. 3]{EZ} is equivalent to Theorem~\ref{aw-local}.
\end{remark}

Taking dimensions in Theorem~\ref{aw-local} gives the following combinatorial interpretation of $\dim \Loc_k(\symm_n,\CC)$.
Recall that an {\em increasing subsequence} of length $r$ in a permutation $w \in \symm_n$ is a sequence $1 \leq i_1 < \cdots < i_r \leq n$
of positions such that $w(i_1) < \cdots < w(i_r)$.

\begin{corollary}
\label{increasing-subsequence-corollary}
The dimension of $\Loc_k(\symm_n,\CC)$ is the number of permutations in $\symm_n$ which have an increasing subsequence of length $n-k$.
\end{corollary}

\begin{proof}
The {\em RSK correspondence} 
is a  bijection $w \mapsto (P(w), Q(w))$ between permutations $w \in \symm_n$ and pairs
$(P,Q)$ of standard tableaux with $n$ boxes having the same shape (see e.g.~\cite{Sagan}). It is well-known that the length of the longest increasing 
subsequence of $w$ is the length of the first row of $P(w)$ (or $Q(w)$). The number of $w \in \symm_n$ which have 
an increasing subsequence of length $n-k$ is therefore $\sum_{\substack{\lambda \vdash n \\ \lambda_1 \geq n-k}} (f^{\lambda})^2$ .
By Theorem~\ref{aw-local} this is also the dimension of $\Loc_k(\symm_n,\CC)$.
\end{proof}

Let $k \geq 0$.
Given any statistic $f: \symm_n \rightarrow \CC$, Theorem~\ref{aw-local} gives a natural best $k$-local approximation $L_k f$ to the statistic $f$.
Indeed, if $\Psi(f) = (\varphi_\lambda)_{\lambda \vdash n}$ where $\varphi_\lambda: V^{\lambda} \rightarrow V^{\lambda}$ is a linear map, the statistic $L_k f$
is characterized by $\Psi(L_k f) (\psi_\lambda)_{\lambda \vdash n}$ where $\psi_\lambda: V^{\lambda} \rightarrow V^{\lambda}$ is given by
\begin{equation}
\psi_\lambda = \begin{cases}
\varphi_\lambda & \lambda_1 \geq n-k, \\
0 & \text{else.}
\end{cases}
\end{equation}
We record some simple properties of $L_k  f$.

\begin{proposition}
\label{local-approximation-properties}
Let $f: \symm_n \rightarrow \CC$ be a function and let $k \geq 0$.
\begin{enumerate}
\item  We have $L_n f = f$.
\item  The function $L_0 f: \symm_n \rightarrow \CC$ is the constant function $\frac{1}{n!} \sum_{w \in \symm_n} f(w)$.
\item  The function $f$ is $k$-local if and only if $L_k f = f$.
\end{enumerate}
\end{proposition}

\begin{proof}
Item (3) is a consequence of Theorem~\ref{aw-local}, and implies Item (1).  Item (2) follows because $V^{(n)}$ is the trivial representation.
\end{proof}

In practice, computing $L_k f$ for a given statistic $f$
requires choosing a basis $\BBB_\lambda$ for every irreducible $V^\lambda$ and coordinatizing the isomorphism 
$\Psi$ of \eqref{aw-sn} to get an isomorphism
\begin{equation}
\label{aw-coordinate}
\CC[\symm_n] \cong \bigoplus_{\lambda \vdash n} \Mat_{f^{\lambda}}(\CC)
\end{equation}
between $\CC[\symm_n]$ and a product of matrix algebras.
Although the matrix tuples representing $f$ and $L_k f$ depend on the choice of bases $\BBB_\lambda$, the statistic $L_k f: \symm_n \rightarrow \CC$
is uniquely determined by $f$.
It could be interesting to study efficient methods to compute $L_k f$, with related work appearing in~\cite{DR} and~\cite{HKMOW}.

The replacement of $f: \symm_n \rightarrow \CC$ by its $k$-local approximation $L_k f$ introduces some amount of error.
It could also be interesting to quantify this error, a la Taylor's Theorem for smooth functions.
Character theory gives a hint that this line of study could be fruitful.
 
For each partition $\lambda \vdash n$, choose a basis $\BBB_\lambda$ of the corresponding $\symm_n$-irreducible $V^\lambda$.
Given $w \in \symm_n$, let $A^\lambda(w)$ be the representing matrix for the action of $w$ on $V^\lambda$ with respect to the basis $\BBB_\lambda$. 
The space $\Fun(\symm_n,\CC)$ admits an inner product
\begin{equation}
\label{arbitrary-function-inner-product}
\langle f, g \rangle := \frac{1}{n!} \sum_{w \in \symm_n} f(w) \cdot \overline{g(w)}.
\end{equation}
The $n! = \sum_{\lambda \vdash n} (f^{\lambda})^2$ matrix elements of the matrix tuple $(A^{\lambda}(w))_{\lambda \vdash n}$ where
 $w \in \symm_n$
determine $n!$ functions $\symm_n \rightarrow \CC$. 
If $U_{\lambda} \subseteq \Fun(\symm_n,\CC)$ is the subspace spanned by the $(f^{\lambda})^2$
matrix elements in the $\lambda^{th}$ component, the {\em Peter-Weyl Theorem} implies that 
we have a direct sum decomposition
\begin{equation}
\Fun(\symm_n,\CC) = \bigoplus_{\lambda \vdash n} U_{\lambda}
\end{equation}
which is orthogonal with respect to the inner product \eqref{arbitrary-function-inner-product}.
Furthermore, the subspace $U_\lambda \subseteq \Fun(\symm_n,\CC)$ is independent of the choice of bases $\BBB_\lambda$ and 
we have the identification $\Loc_k(\symm_n,\CC) = \bigoplus_{\lambda_1 \geq n-k} U_{\lambda}$.
The map $L_k: \Fun(\symm_n,\CC) \twoheadrightarrow \Loc_k(\symm_n,\CC)$ admits a coordinate-free description
as orthogonal projection onto the subspace $\bigoplus_{\lambda_1 \geq n-k} U_{\lambda}$.
This gives another sense in which $L_k f$ is the best $k$-local approximation of $f$.

\section{A basis of $k$-local statistics}
\label{local-basis}

In principle, Theorem~\ref{aw-local} can be used to obtain a basis for the space $\Loc_k(\symm_n)$ of $k$-local statistics on $\symm_n$.
One could choose a basis $\BBB_{\lambda}$ for each $\symm_n$-irreducible $V^{\lambda}$, use these bases to coordinatize the Artin-Wedderburn isomorphism 
as in \eqref{aw-coordinate}, and compute the inverse image of the matrix coordinates in the algebras $\Mat_{f^\lambda}(\CC)$ for which $\lambda_1 \geq n-k$.
In practice, bases for $\Loc_k(\symm_n)$ obtained in this fashion are likely to be very complicated and hard to compute.
The reader who suspects that simple combinatorial bases of $\Loc_k(\symm_n)$ may be trivially obtained is encouraged to try out their intuition on 
$\Loc_2(\symm_n)$; Corollary~\ref{increasing-subsequence-corollary} 
states that such a basis would biject with permutations in $\symm_n$ with an increasing subsequence of length $n-2$.

In this subsection, we state without proof a result of the second author which gives a simple, explicit basis for $\Loc_k(\symm_n)$. In fact, we give a nested family of bases
for the filtration $\Loc_0(\symm_n,\CC) \subset \Loc_1(\symm_n,\CC) \subset \cdots \subset \Loc_{n-1}(\symm_n,\CC)$ of the space $\Fun(\symm_n,\CC)$ of functions
on $\symm_n$.
In order to describe these bases, we need the shadow lines construction of Viennot \cite{Viennot}.

We represent a permutation $w \in \symm_n$ with the graph of its function $w: [n] \rightarrow [n]$,
 i.e. the collection of points $\{ (i,w(i)) \,:\, 1 \leq i \leq n \}$ on the grid $[n] \times [n]$.
For example, the permutation $w = [4,1,8,5,3,6,2,7] \in \symm_8$ is given below in bullets. 

\begin{center}
\begin{tikzpicture}[scale = 0.5]

\draw (1,1) grid (8,8);

\node at (1,4) {$\bullet$};
\node at (2,1) {$\bullet$};
\node at (3,8) {$\bullet$};
\node at (4,5) {$\bullet$};
\node at (5,3) {$\bullet$};
\node at (6,6) {$\bullet$};
\node at (7,2) {$\bullet$};
\node at (8,7) {$\bullet$};

\end{tikzpicture}
\end{center}

Shine a flashlight northeast from the origin (0,0). Each bullet in the permutation casts a shadow to its northeast. The boundary of the shaded region is the 
{\em first shadow line}; in our example it is as follows.
\begin{center}
\begin{tikzpicture}[scale = 0.5]

\draw (1,1) grid (8,8);

\node at (1,4) {$\bullet$};
\node at (2,1) {$\bullet$};
\node at (3,8) {$\bullet$};
\node at (4,5) {$\bullet$};
\node at (5,3) {$\bullet$};
\node at (6,6) {$\bullet$};
\node at (7,2) {$\bullet$};
\node at (8,7) {$\bullet$};

\draw[very thick] (1,9) -- (1,4) -- (2,4) -- (2,1) -- (9,1);

\end{tikzpicture}
\end{center}
Removing the points on the first shadow line and repeating this procedure, we obtain the {\em second shadow line}. 
Iterating, we obtain the {\em third shadow line}, the {\em fourth shadow line}, and so on.
In our example, the shadow lines are shown below.
\begin{center}
\begin{tikzpicture}[scale = 0.5]

\draw (1,1) grid (8,8);

\node at (1,4) {$\bullet$};
\node at (2,1) {$\bullet$};
\node at (3,8) {$\bullet$};
\node at (4,5) {$\bullet$};
\node at (5,3) {$\bullet$};
\node at (6,6) {$\bullet$};
\node at (7,2) {$\bullet$};
\node at (8,7) {$\bullet$};

\draw[very thick] (1,9) -- (1,4) -- (2,4) -- (2,1) -- (9,1);

\draw[very thick] (3,9) -- (3,8) -- (4,8) -- (4,5) -- (5,5) -- (5,3) -- (7,3) -- (7,2) -- (9,2);

\draw[very thick] (6,9) -- (6,6) -- (9,6);

\draw[very thick] (8,9) -- (8,7) -- (9,7);

\node at (2,4) {${\color{red} \bullet}$};
\node at (4,8) {${\color{red} \bullet}$};
\node at (5,5) {${\color{red} \bullet}$};
\node at (7,3) {${\color{red} \bullet}$};

\end{tikzpicture}
\end{center}

The {\em Viennot shadow line construction} associates a permutation $w \in \symm_n$ to its sequence of shadow lines.
The points on the graph of $w$ (shown in black above) form the southwest corners of the shadow lines.
We will be especially interested in the northeast corners of lines in the shadow diagram; these are shown in red.
The northeast corners of the shadow diagram of $w$ may be naturally regarded as a partial permutation; call this partial 
permutation $(I(w), J(w))$.  In the example above we have $I(w) = 2457$ and $J(w) = 4853$.

Viennot used \cite{Viennot} the shadow line construction to give a reformulation of the RSK bijection $w \mapsto (P(w), Q(w))$.
He proved that the entries of the first row of $P(w)$ are the $y$-coordinates of the horizontal rays in the shadow line construction and that
 the entries of the first row of $Q(w)$ are the $x$-coordinates of the vertical rays in the shadow line construction.
The northeast corners of the shadow diagram of $w$ form the partial permutation $(I(w), J(w))$ to which the shadow line construction may be applied;
this yields the second rows of $P(w)$ and $Q(w)$. The remaining rows of $P(w)$ and $Q(w)$ are obtained recursively.
The Viennot construction makes various properties of the Schensted correspondence transparent, notably its behavior
$P(w^{-1}) = Q(w), Q(w^{-1}) = P(w)$ under group-theoretic inversion $w \mapsto w^{-1}$.

We will use the shadow line construction to get a basis for the space of $k$-local statistics.
For $w \in \symm_n$,
let $s(w)$ be the number of northeast corners in the shadow diagram (or, equivalently, the length of the partial permutation $(I(w), J(w))$);
we have $s(w) = 4$ in our example.

\begin{theorem} 
\label{shadow-basis}
\cite[Thm. 3.12]{Rhoades} 
For any $n, k \geq 0$, the indicator functions 
\begin{equation}
\left \{ \one_{I(w),J(w)} \,:\, w \in \symm_n, \, \, s(w) \leq k  \right \}
\end{equation}
form a basis of $\Loc_k(\symm_n,\CC)$.
\end{theorem}

 Theorem~\ref{shadow-basis} also holds when one replaces $\CC$ with any commutative ring $R$
 (in which case the given indicator functions form an $R$-module basis of $\Loc_k(\symm_n,R)$).
Observation~\ref{lower-closure} guarantees that the given indicator functions lie in $\Loc_k(\symm_n,\CC)$.
The nested bases of $\Loc_0(\symm_3,\CC) \subset \Loc_1(\symm_3,\CC) \subset \Loc_2(\symm_3,\CC)$ given by
Theorem~\ref{shadow-basis} are as follows.
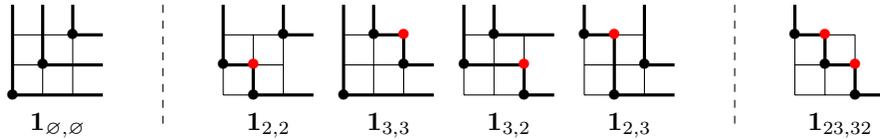
\begin{center}
\begin{tikzpicture}[scale = 0.4]

\draw (1,1) grid (3,3);

\draw[very thick] (1,4) -- (1,1) -- (4,1);
\draw[very thick] (2,4) -- (2,2) -- (4,2);
\draw[very thick] (3,4) -- (3,3) -- (4,3);

\node at (1,1) {$\bullet$};
\node at (2,2) {$\bullet$};
\node at (3,3) {$\bullet$};

\node at (2.5,0) {$\one_{\varnothing,\varnothing}$};

\draw[dashed] (6,4) -- (6,0);

\draw (8,1) grid (10,3);

\draw[very thick] (8,4) -- (8,2) -- (9,2) -- (9,1) -- (11,1);
\draw[very thick] (10,4) -- (10,3) -- (11,3);

\node at (8,2) {$\bullet$};
\node at (9,1) {$\bullet$};
\node at (10,3) {$\bullet$};

\node at (9,2) {${\color{red} \bullet}$};

\node at (9.5,0) {$\one_{2,2}$};

\draw (12,1) grid (14,3);

\draw[very thick] (12,4) -- (12,1) -- (15,1);
\draw[very thick] (13,4) -- (13,3) -- (14,3) -- (14,2) -- (15,2);

\node at (12,1) {$\bullet$};
\node at (13,3) {$\bullet$};
\node at (14,2) {$\bullet$};

\node at (14,3) {${\color{red} \bullet}$};

\node at (13.5,0) {$\one_{3,3}$};

\draw (16,1) grid (18,3);

\draw[very thick] (16,4) -- (16,2) -- (18,2) -- (18,1) -- (19,1);
\draw[very thick] (17,4) -- (17,3) -- (19,3);

\node at (16,2) {$\bullet$};
\node at (17,3) {$\bullet$};
\node at (18,1) {$\bullet$};

\node at (18,2) {${\color{red} \bullet}$};

\node at (17.5,0) {$\one_{3,2}$};

\draw (20,1) grid (22,3);

\draw[very thick] (20,4) -- (20,3) -- (21,3) -- (21,1) -- (23,1);
\draw[very thick] (22,4) -- (22,2) -- (23,2);

\node at (20,3) {$\bullet$};
\node at (21,1) {$\bullet$};
\node at (22,2) {$\bullet$};

\node at (21,3) {${\color{red} \bullet}$};

\node at (21.5,0) {$\one_{2,3}$};

\draw[dashed] (25,4) -- (25,0);

\draw (27,1) grid (29,3);

\draw[very thick] (27,4) -- (27,3) -- (28,3) -- (28,2) -- (29,2) -- (29,1) -- (30,1);

\node at (27,3) {$\bullet$};
\node at (28,2) {$\bullet$};
\node at (29,1) {$\bullet$};

\node at (28,3) {${\color{red} \bullet}$};
\node at (29,2) {${\color{red} \bullet}$};

\node at (28.5,0) {$\one_{23,32}$};

\end{tikzpicture}
\end{center}

Theorem~\ref{shadow-basis} has a geometric interpretation. Regard the set $\Mat_n(\CC)$ of $n \times n$ complex matrices as 
an affine space with coordinate ring $\CC[\xx]$ where $\xx = (x_{i,j})$ is an $n \times n$ matrix of variables.
We have a permutation matrix embedding $\symm_n \subseteq \Mat_n(\CC)$ which realizes $\symm_n$ as a locus of $n!$ permutation matrices.
By Lagrange Interpolation, any function $\symm_n \rightarrow \CC$ is the restriction of some polynomial $f \in \CC[\xx]$ to the locus $\symm_n$.
In symbols, we have
\begin{equation}
\CC[\symm_n] = \CC[\xx]/\II(\symm_n)
\end{equation}
where we identify $\CC[\symm_n]$ with the family of (polynomial) functions $\symm_n \rightarrow \CC$ and 
\begin{equation}
\II(\symm_n) = \{ f \in \CC[\xx] \,:\, f(w) = 0 \text{ for all $w \in \symm_n$} \} \subseteq \CC[\xx]
\end{equation}
is the ideal of polynomials which vanish on $\symm_n$.

Let $(I,J) \in \symm_{n,k}$ be a partial permutation with $I = (i_1, \dots, i_k)$ and $J = (j_1, \dots, j_k)$.  
Clearly the indicator function $\one_{I,J}: \symm_n \rightarrow \CC$ is the restriction of the polynomial function
$x_{I,J} := x_{i_1,j_1} \cdots x_{i_k,j_k}$ from $\Mat_n(\CC)$ to $\symm_n$. Thanks to Observation~\ref{lower-closure}, we have an identification
\begin{equation}
\label{geometric-local-identification}
\Loc_k(\symm_n,\CC) \cong \mathrm{Image} \left(  \CC[\xx]_{\leq k} \hookrightarrow \CC[\xx] \twoheadrightarrow \CC[\xx]/\II(\symm_n)  \right)
\end{equation}
where $\CC[\xx]_{\leq k} \subseteq \CC[\xx]$ is the subspace of polynomials of degree $\leq k$ and the maps are inclusion followed by projection.
Theorem~\ref{shadow-basis} implies that the cosets 
\begin{equation}
\{x_{I(w),J(w)} + \II(\symm_n) \,:\, s(w) \leq k \}
\end{equation}
form a basis of $\Loc_k(\symm_n,\CC)$ under the identification \eqref{geometric-local-identification}.  
In particular, if $\gr \, \II(\symm_n) \subseteq \CC[\xx]$ denotes the {\em associated graded ideal} of $\II(\symm_n)$
generated by the highest degree homogeneous components of polynomials $f \in \II(\symm_n)$, we have the Hilbert series
\begin{equation}
\Hilb \left(   \CC[\xx]/\gr \, \II(\symm_n) ; q \right) = \sum_{d = 0}^{n-1} a_{n,n-d} \cdot q^d
\end{equation}
where $a_{n,n-d}$ counts permutations $w \in \symm_n$ whose longest increasing subsequence has length $n-d$.
The ideal $\gr \, \II(\symm_n)$ is generated \cite[Thm. 3.12]{Rhoades}
by the family of all row sums $x_{i,1} + \cdots + x_{i,n}$ of variables, column sums $x_{1,j} + \cdots + x_{n,j}$
of variables, and products $x_{i,j} \cdot x_{i,j'}$ and $x_{i,j} \cdot x_{i',j}$ with $1 \leq i, i', j, j' \leq n$.

In addition to respecting the filtration of $\Fun(\symm_n,\CC)$ by local statistics,
 the bases of Theorem~\ref{shadow-basis} are also stable under the operation $n \rightarrow n+1$.
 Indeed, if $w \in \symm_n$ is a permutation, we can build a larger permutation $\hat{w} \in \symm_{n+1}$ by letting 
 $\hat{w}(i) = w(i)$ for $1 \leq i \leq n$ and setting $\hat{w}(n+1) = n+1$.  It is not hard to see from the shadow line construction that 
 $I(w) = I(\hat{w})$ and $J(w) = J(\hat{w})$, so that $w$ and $\hat{w}$ give rise to the same indicator function in Theorem~\ref{shadow-basis}.
 On the other hand,
the authors do not know how to extract a basis of $\Loc_k(\symm_n,\CC)$ from its definitional spanning 
set $\{ \one_{I,J} \,:\, (I,J) \in \symm_{n,k} \}$.

\section{Local class functions}

Under the identification of the symmetric group algebra $\CC[\symm_n]$ with the space $\Fun(\symm_n,\CC)$ of functions on $\symm_n$, the center $Z(\CC[\symm_n])$
of $\CC[\symm_n]$ corresponds to the subspace $\Class(\symm_n,\CC)$ of functions $\symm_n \rightarrow \CC$ which are constant on conjugacy classes.
Since the center of a complex matrix ring $\Mat_d(\CC)$ is its 1-dimensional subalgebra of scalar matrices, 
the Artin-Wedderburn Theorem gives rise to isomorphisms
\begin{equation}
\label{aw-central}
Z(\CC[\symm_n]) \cong \bigoplus_{\lambda \, \vdash \, n} \CC \cdot \mathrm{id}_{V^{\lambda}} \cong \bigoplus_{\lambda \, \vdash \, n} \CC.
\end{equation}
Restricting to $k$-local class functions yields the following.

\begin{corollary}
\label{k-local-central-aw}
The space $\Loc_k(\symm_n,\CC) \cap \Class(\symm_n,\CC)$ of $k$-local class functions on $\symm_n$ corresponds to 
\begin{equation}
\bigoplus_{\lambda_1 \, \geq \, n-k} \CC \cdot \mathrm{id}_{V^{\lambda}} 
\end{equation}
under the Artin-Wedderburn isomorphism $\Psi$.
\end{corollary}

\begin{proof}
Combine the isomorphisms \eqref{aw-central} with Theorem~\ref{aw-local}.
\end{proof}

Given a partition $\lambda \vdash n$, the irreducible character $\chi^{\lambda}: \symm_n \rightarrow \CC$ is a class function on $\symm_n$.
It is a well-known consequence of character orthogonality and 
Schur's Lemma that the group algebra element $\sum_{w \in \symm_n} \chi^{\lambda}(w) \cdot w \in \CC[\symm_n]$
acts by a nonzero scalar on $V^{\lambda}$ and annihilates any irreducible representations $V^{\mu}$ of $\symm_n$ with $\mu \neq \lambda$.
The following result is now an immediate consequence of Theorem~\ref{aw-local}.

\begin{corollary}
\label{character-local}
Let $\lambda \vdash n$. The irreducible character $\chi^{\lambda}: \symm_n \rightarrow \CC$ of $\symm_n$ is $k$-local if and only if 
$\lambda_1 \geq n-k$.
\end{corollary}

For example, it is impossible to write the sign function $\sign: \symm_n \rightarrow \{ \pm 1 \}$ is a linear combination of indicator functions $\one_{I,J}$
for partial permutations $(I,J) \in \symm_{n,k}$ with $k < n-1$.

\begin{remark}
\label{kronecker-remark}
Corollary~\ref{character-local} states that for a partition $\lambda$, the irreducible character $\chi^{\lambda[n]}: \symm_n \rightarrow \CC$ 
is a $|\lambda|$-local statistic. 
In turn, Proposition~\ref{product-local} implies that if $\mu$ is another partition, the pointwise product (or {\em Kronecker product})
$\chi^{\lambda[n]} \cdot \chi^{\mu[n]}: \symm_n \rightarrow \CC$ is $(|\lambda| + |\mu|)$-local.
Finding the expansion of $\chi^{\lambda[n]} \cdot \chi^{\mu[n]}$ into irreducible characters is a famous open problem; the above argument 
shows that the only characters $\chi^{\nu}$ which appear in this expansion correspond to partitions $\nu \vdash n$ satisfying $n - \nu_1 \leq |\lambda| + |\mu|$.
This support result was first discovered by Murnaghan \cite{Murnaghan}.
\end{remark}

Recall that $\ch_n: \Class(\symm_n,\CC) \rightarrow \Lambda_n$ is the Frobenius isomorphism.
The space of $k$-local class functions maps to an easily identifiable space of symmetric functions.

\begin{theorem}
\label{local-support-theorem}
For any $n, k \geq 0$ we have the equality of subspaces 
\begin{equation}
\ch_n \left(  \Loc_k(\symm_n,\CC) \cap \Class(\symm_n,\CC) \right) = 
\mathrm{span}  \left\{   s_{\lambda} \,:\, \lambda \vdash n, \, \, \lambda_1 \geq n-k \right\}
\end{equation}
of the space $\Lambda_n$ of degree $n$ symmetric functions.
\end{theorem}

\begin{proof}
Recall that $\ch_n(\chi^\lambda) = s_\lambda$ and apply Corollary~\ref{character-local}.
\end{proof}

\chapter{Atomic Symmetric Functions and Path Power Sums}
\label{Atomic}

The atomic  functions $A_{n,I,J}$ will be our road to understanding the asymptotic behavior of permutation statistics.
In this chapter we develop basic properties of these symmetric functions.
Prominent among these (Proposition~\ref{atomic-factorization})
is a factorization $A_{n,I,J} = p_\nu \cdot \vec{p}_\mu$ of the atomic function into 
a classical power sum $p_\nu$ and new symmetric function $\vec{p}_\mu$ which we term a 
`path power sum'.
Since the power sums $p_\nu$ are classical objects, understanding the atomic symmetric functions will come down to understanding the path power
sums $\vec{p}_\mu$.

We give combinatorial formulas for the $p$-expansion and $m$-expansion of the path power sums (and hence of the atomics),
as well as graph-theoretic recursions for the $A_{n,I,J}$ under the action of skewing operators $h_j^\perp$ and $p_j^\perp$.
We also give a representation-theoretic interpretation (Proposition~\ref{atomic-character-interpretation}) for the $s$-expansion of $A_{n,I,J}$ in terms
of irreducible character evaluations.

\section{Atomics to power sums}
Let $(I,J) \in \symm_{n,k}$ be a partial permutation.
Recall that the atomic symmetric function $A_{n,I,J}$ is given by
\begin{equation}
    A_{n,I,J} = n! \cdot \ch_n(R \, \one_{I,J})
\end{equation}
where $\one_{I,J}: \symm_n \rightarrow \CC$ is the indicator function for $w(I) = J$ and 
\begin{equation}
    R: \Fun(\symm_n,\CC) \twoheadrightarrow \Class(\symm_n,\CC)
\end{equation} 
is the Reynolds projection.
In this section we study the combinatorics of the $A_{n,I,J}$.

Let $\GGG_n$ be the family of directed graphs on the vertex set $[n]$ whose connected components are all directed cycles and directed paths.
We may identify $\GGG_n$ with the collection of all partial permutations of $[n]$.
The first result of this chapter states that the atomic symmetric function $A_{n,I,J}$ of a partial permutation $(I,J)$
only depends on the isomorphism type of its directed graph $G_n(I,J) \in \GGG_n$.

\begin{proposition}
\label{atomic-g-invariance}
Let $(I,J) \in \symm_{n,k}$ be a partial permutation.  We have 
$A_{n,I,J} = A_{n, w(I), w(J)}$ for all $w \in \symm_n$.
\end{proposition}

\begin{proof}
The symmetric group $\symm_n$ acts on the space $\Fun(\symm_n,\CC)$ by the rule $(w \cdot \psi)(v) := \psi(w^{-1}vw)$.
By Lemma~\ref{product-action-on-indicators} we have 
$w \cdot \one_{I,J} = \one_{w(I),w(J)}$.  Since $R(w \cdot \psi) = R(\psi)$ for any function $\psi: \symm_n \rightarrow \CC$,  we have
\begin{multline}
A_{n,w(I),w(J)} = n! \cdot \ch_n( R \, \one_{w(I),w(J)} ) \\ = n! \cdot \ch_n( R (w \cdot \one_{I,J})) =
n! \cdot \ch_n( R \, \one_{I,J}) = A_{n,I,J},
\end{multline}
as desired.
\end{proof}

As mentioned in Chapter 1, Proposition~\ref{atomic-g-invariance} implies that the true indexing set of the 
atomic symmetric functions is given by ordered pairs $(\nu,\mu)$ of partitions where $\nu$ records the 
cycle sizes in $G_n(I,J)$
and $\mu$ records the path sizes in $G_n(I,J)$. The atomic symmetric functions will factor
(Proposition~\ref{atomic-factorization})
into a `cycle part' indexed by $\nu$
and a `path part' in indexed by $\mu$. 
The cycle part is amenable to classical analysis 
whereas the path part requires 
significantly more work.

The power sum expansion of $A_{n,I,J}$ has a nice combinatorial interpretation.
If $(I,J) \in \symm_{n,I,J}$ is a partial permutation whose graph $G_n(I,J)$ has path partition $\mu$ and cycle partition $\nu$,
let $\domcyc(I,J) \vdash n$ be the partition of $n$ whose parts consist of the sum $|\mu|$ of the path lengths together 
with the cycle lengths in $\nu$.
The notation $\domcyc(I,J)$ is justified by the following proposition.

\begin{proposition}
\label{p-combinatorial-interpretation}
Let $(I,J) \in \symm_{n,k}$ be a partial permutation and consider the $p$-expansion
\begin{equation}
A_{n,I,J} = \sum_{\lambda \, \vdash \, n} a_{\lambda} \cdot p_{\lambda}
\end{equation}
of the atomic symmetric function $A_{n,I,J}$.  The coefficient $a_{\lambda}$ counts $w \in \symm_n$ of cycle type $\lambda$
such that $w(I) = J$.  We have $a_{\lambda} = 0$ unless $\lambda \leq \domcyc(I,J)$ in dominance order.
If $\lambda = \domcyc(I,J)$ then $a_{\lambda} \neq 0$.
\end{proposition}

\begin{proof}
The combinatorial interpretation of $a_{\lambda}$ is immediate from Definition~\ref{atomic-symmetric-function-definition}; the 
triangularity assertion follows from the definition of dominance order.
\end{proof}

In combinatorial terms, Proposition~\ref{p-combinatorial-interpretation} asserts that the $p$-expansion of $A_{n,I,J}$
may be obtained from the graph $G_n(I,J)$ by joining the ends of paths in $G_n(I,J)$ to form cycles in all possible ways (while leaving the cycles
in $G_n(I,J)$ alone) and recording the partition $\lambda$ of cycle lengths so obtained in the power sum $p_\lambda$.

For example, consider the partial permutation $\symm_{6,3}$ given by $(I,J) = (456,346)$.
The path partition of $(I,J)$ is $\mu = (2,1,1)$ and the cycle partition of $(I,J)$ is $\nu = (2)$.
The graph $G_6(I,J)$ is as follows.

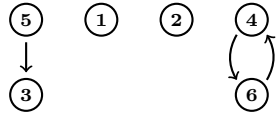
\begin{center}
\begin{tikzpicture}[scale = 0.5]

   \node [draw, circle, fill = white, inner sep = 2pt, thick] at (2,0)
  {\scriptsize {\bf 1} };
     \node [draw, circle, fill = white, inner sep = 2pt, thick] at (6,-2)
  {\scriptsize {\bf 6}};
     \node [draw, circle, fill = white, inner sep = 2pt, thick] at (0,-2)
  {\scriptsize {\bf 3}};
     \node [draw, circle, fill = white, inner sep = 2pt, thick] at (4,0)
  {\scriptsize {\bf 2}};
     \node [draw, circle, fill = white, inner sep = 2pt, thick] at (0,0)
  {\scriptsize {\bf 5}};
     \node [draw, circle, fill = white, inner sep = 2pt, thick] at (6,0)
  {\scriptsize {\bf 4}};
  \draw [->, thick] (0,-0.6) -- (0,-1.4);
 \draw [->, thick]  (5.6,-0.4) to[bend right]  (5.6,-1.6);
 \draw [->, thick]  (6.4,-1.6) to[bend right]  (6.4,-0.4);
\end{tikzpicture}
\end{center}
To compute the $p$-expansion of $A_{6,I,J}$, we glue the paths in this diagram to form cycles in all possible ways, yielding the 
following pictures.
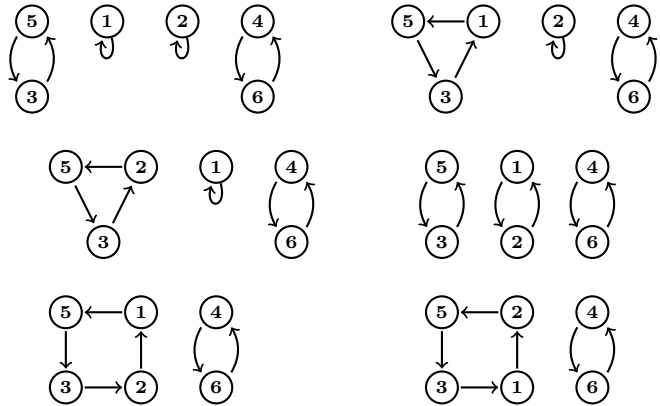
\begin{center}
\begin{tikzpicture}[scale = 0.5]

   \node [draw, circle, fill = white, inner sep = 2pt, thick] (1) at (2,0)
  {\scriptsize {\bf 1} };
     \node [draw, circle, fill = white, inner sep = 2pt, thick]  (6) at (6,-2)
  {\scriptsize {\bf 6}};
     \node [draw, circle, fill = white, inner sep = 2pt, thick] (3) at (0,-2)
  {\scriptsize {\bf 3}};
     \node [draw, circle, fill = white, inner sep = 2pt, thick] (2) at (4,0)
  {\scriptsize {\bf 2}};
     \node [draw, circle, fill = white, inner sep = 2pt, thick]  (5) at (0,0)
  {\scriptsize {\bf 5}};
     \node [draw, circle, fill = white, inner sep = 2pt, thick] (4) at (6,0)
  {\scriptsize {\bf 4}};
 \draw [->, thick]  (-0.4,-0.4) to[bend right]  (-0.4,-1.6);
 \draw [->, thick]  (0.4,-1.6) to[bend right]  (0.4,-0.4);
 \draw [->, thick]  (5.6,-0.4) to[bend right]  (5.6,-1.6);
 \draw [->, thick]  (6.4,-1.6) to[bend right]  (6.4,-0.4);
 
  \path  (1) edge [loop below, thick]  (1);
  \path (2) edge [loop below, thick] (2);
  
     \node [draw, circle, fill = white, inner sep = 2pt, thick] (1a) at (12,0)
  {\scriptsize {\bf 1} };
     \node [draw, circle, fill = white, inner sep = 2pt, thick]  (6a) at (16,-2)
  {\scriptsize {\bf 6}};
     \node [draw, circle, fill = white, inner sep = 2pt, thick] (3a) at (11,-2)
  {\scriptsize {\bf 3}};
     \node [draw, circle, fill = white, inner sep = 2pt, thick] (2a) at (14,0)
  {\scriptsize {\bf 2}};
     \node [draw, circle, fill = white, inner sep = 2pt, thick]  (5a) at (10,0)
  {\scriptsize {\bf 5}};
     \node [draw, circle, fill = white, inner sep = 2pt, thick] (4a) at (16,0)
  {\scriptsize {\bf 4}};
   \draw [->, thick]  (15.6,-0.4) to[bend right]  (15.6,-1.6);
   \draw [->, thick]  (16.4,-1.6) to[bend right]  (16.4,-0.4);
   \draw [->, thick] (10.25,-0.5) to (10.75,-1.5);
   \draw [->, thick] (11.25,-1.5) to (11.75,-0.5);
    \draw [->, thick] (11.5,0) to (10.5,0);
      \path (2a) edge [loop below, thick] (2a);
\end{tikzpicture}
\end{center}
\vspace{0.1in}
\begin{center}
\begin{tikzpicture}[scale = 0.5]
       \node [draw, circle, fill = white, inner sep = 2pt, thick] (2a) at (2,0)
  {\scriptsize {\bf 2} };
     \node [draw, circle, fill = white, inner sep = 2pt, thick]  (6a) at (6,-2)
  {\scriptsize {\bf 6}};
     \node [draw, circle, fill = white, inner sep = 2pt, thick] (3a) at (1,-2)
  {\scriptsize {\bf 3}};
     \node [draw, circle, fill = white, inner sep = 2pt, thick] (1a) at (4,0)
  {\scriptsize {\bf 1}};
     \node [draw, circle, fill = white, inner sep = 2pt, thick]  (5a) at (0,0)
  {\scriptsize {\bf 5}};
     \node [draw, circle, fill = white, inner sep = 2pt, thick] (4a) at (6,0)
  {\scriptsize {\bf 4}};
   \draw [->, thick]  (5.6,-0.4) to[bend right]  (5.6,-1.6);
   \draw [->, thick]  (6.4,-1.6) to[bend right]  (6.4,-0.4);
   \draw [->, thick] (0.25,-0.5) to (0.75,-1.5);
   \draw [->, thick] (1.25,-1.5) to (1.75,-0.5);
    \draw [->, thick] (1.5,0) to (0.5,0);
      \path (1a) edge [loop below, thick] (1a);

   \node [draw, circle, fill = white, inner sep = 2pt, thick] (1) at (12,0)
  {\scriptsize {\bf 1} };
     \node [draw, circle, fill = white, inner sep = 2pt, thick]  (6) at (14,-2)
  {\scriptsize {\bf 6}};
     \node [draw, circle, fill = white, inner sep = 2pt, thick] (3) at (10,-2)
  {\scriptsize {\bf 3}};
     \node [draw, circle, fill = white, inner sep = 2pt, thick] (2) at (12,-2)
  {\scriptsize {\bf 2}};
     \node [draw, circle, fill = white, inner sep = 2pt, thick]  (5) at (10,0)
  {\scriptsize {\bf 5}};
     \node [draw, circle, fill = white, inner sep = 2pt, thick] (4) at (14,0)
  {\scriptsize {\bf 4}};
     \draw [->, thick] (13.6,-0.4) to[bend right]  (13.6,-1.6);
   \draw [->, thick]  (14.4,-1.6) to[bend right]  (14.4,-0.4);
   
\draw [->, thick] (11.6,-0.4) to[bend right]  (11.6,-1.6);
   \draw [->, thick]  (12.4,-1.6) to[bend right]  (12.4,-0.4);

 \draw [->, thick] (9.6,-0.4) to[bend right]  (9.6,-1.6);
   \draw [->, thick]  (10.4,-1.6) to[bend right]  (10.4,-0.4);
\end{tikzpicture}
\end{center}
\vspace{0.1in}
\begin{center}
\begin{tikzpicture}[scale = 0.5]
   \node [draw, circle, fill = white, inner sep = 2pt, thick] (1a) at (2,0)
  {\scriptsize {\bf 1} };
     \node [draw, circle, fill = white, inner sep = 2pt, thick]  (6a) at (4,-2)
  {\scriptsize {\bf 6}};
     \node [draw, circle, fill = white, inner sep = 2pt, thick] (3a) at (0,-2)
  {\scriptsize {\bf 3}};
     \node [draw, circle, fill = white, inner sep = 2pt, thick] (2a) at (2,-2)
  {\scriptsize {\bf 2}};
     \node [draw, circle, fill = white, inner sep = 2pt, thick]  (5a) at (0,0)
  {\scriptsize {\bf 5}};
     \node [draw, circle, fill = white, inner sep = 2pt, thick] (4a) at (4,0)
  {\scriptsize {\bf 4}};
   \draw [->, thick] (3.6,-0.4) to[bend right]  (3.6,-1.6);
   \draw [->, thick]  (4.4,-1.6) to[bend right]  (4.4,-0.4);
   \draw [->, thick] (0,-0.5) to (0,-1.5);
   \draw [->, thick] (0.5,-2) to (1.5,-2);
   \draw [->, thick] (1.5,0) to (0.5,0);
    \draw [->, thick] (2,-1.5) to (2,-0.5);

   \node [draw, circle, fill = white, inner sep = 2pt, thick] (2) at (12,0)
  {\scriptsize {\bf 2} };
     \node [draw, circle, fill = white, inner sep = 2pt, thick]  (6) at (14,-2)
  {\scriptsize {\bf 6}};
     \node [draw, circle, fill = white, inner sep = 2pt, thick] (3) at (10,-2)
  {\scriptsize {\bf 3}};
     \node [draw, circle, fill = white, inner sep = 2pt, thick] (1) at (12,-2)
  {\scriptsize {\bf 1}};
     \node [draw, circle, fill = white, inner sep = 2pt, thick]  (5) at (10,0)
  {\scriptsize {\bf 5}};
     \node [draw, circle, fill = white, inner sep = 2pt, thick] (4) at (14,0)
  {\scriptsize {\bf 4}};
     \draw [->, thick] (13.6,-0.4) to[bend right]  (13.6,-1.6);
   \draw [->, thick]  (14.4,-1.6) to[bend right]  (14.4,-0.4);
      \draw [->, thick] (10,-0.5) to (10,-1.5);
   \draw [->, thick] (10.5,-2) to (11.5,-2);
   \draw [->, thick] (11.5,0) to (10.5,0);
    \draw [->, thick] (12,-1.5) to (12,-0.5);

\end{tikzpicture}
\end{center}
These pictures correspond to ways to complete the partial permutation $(I,J) \in \symm_{6,3}$ to a genuine permutation in $\symm_6$.
By Proposition~\ref{p-combinatorial-interpretation} we have the power sum expansion
\begin{equation*}
A_{6,I,J} = 2 p_{4,2} + 2 p_{3,2,1} + p_{2,2,2} + p_{2,2,1,1}
\end{equation*}
where the term corresponding to $\domcyc(I,J) = (4,2) \vdash 6$ is shown first.

We will be interested in the Schur expansion of $A_{n,I,J}$.
As a first step in this direction, Proposition~\ref{p-combinatorial-interpretation} yields a crucial factorization property of the atomic functions.

\begin{proposition}
\label{atomic-factorization}
Let $(I,J) \in \symm_{n,k}$ be a partial permutation, let $\mu = (\mu_1 \geq \mu_2 \geq \cdots ) \vdash r$ be the sizes of the paths in $G_n(I,J)$,
and let $\nu = (\nu_1 \geq \nu_2 \geq \cdots ) \vdash s$ be the sizes of the cycles in $G_n(I,J)$.  Then
\begin{equation}
\label{a-factor}
A_{n,I,J} = A_{r,I_1,J_1} \cdot A_{s,I_2,J_2}
\end{equation}
where $G_r(I_1,J_1)$ is a digraph on $[r]$ consisting of paths of sizes $\mu_1, \mu_2, \dots $ and $G_s(I_2,J_2)$ is a digraph 
on $[s]$ consisting of cycles of sizes $\nu_1, \nu_2, \dots $. Furthermore, the second factor on the RHS of \eqref{a-factor} is the power sum
\begin{equation}
A_{s,I_2,J_2} = p_{\nu}
\end{equation}
indexed by $\nu \vdash s$.
\end{proposition}

\begin{proof}
Recall that the $p$-basis of $\Lambda$ is multiplicative: we have $p_{\lambda} = p_{\lambda_1} p_{\lambda_2} \cdots $ for any partition
$\lambda$.  The result follows from the combinatorial description of $a_{\lambda}$ in Proposition~\ref{p-combinatorial-interpretation}.
\end{proof}

It is {\bf not} true that the first factor $A_{r,I_1,J_1}$ in \eqref{a-factor} corresponding to paths in $G_n(I,J)$ is the 
power sum $p_{\mu}$. Paths and cycles behave very differently with regards to atomic functions. We notate
the first factor of \eqref{a-factor} as follows.

\begin{defn}
\label{path-power-sum-definition}
Let $\mu = (\mu_1, \mu_2, \dots )$ be a partition of $n$. The {\em path power sum} $\vec{p}_{\mu} \in \Lambda_n$ is the atomic
symmetric function 
\begin{equation}
\vec{p}_{\mu} = A_{n,I,J}
\end{equation}
where $(I,J)$ is a partial permutation of $[n]$ whose graph $G_n(I,J)$ consists of disjoint paths of sizes 
given by the parts $\mu_1, \mu_2, \dots $ of $\mu$.
\end{defn}

The vector notation $\vec{p}_{\mu}$ is meant to recall the shape of the graph $G_n(I,J)$ in Definition~\ref{path-power-sum-definition}.
With this notation, Proposition~\ref{atomic-factorization} reads
\begin{equation}
A_{n,I,J} = \vec{p}_{\mu} \cdot p_{\nu}
\end{equation}
whenever $G_n(I,J)$ has path partition $\mu$ and cycle partition $\nu$.
The analysis of $\vec{p}_{\mu}$ is the most technical part of this manuscript and will take place in Chapter~\ref{Path}.

\section{Atomics to Schurs}

Proposition~\ref{p-combinatorial-interpretation} gives a combinatorial interpretation of the $p$-expansion of an atomic 
function $A_{n,I,J}$.
The $s$-expansion of $A_{n,I,J}$ has a representation-theoretic interpretation. If
$(I,J) \in \symm_{n,k}$ is a partial permutation, we write
\begin{equation}
[I,J] := \sum_{\substack{w  \, \in \, \symm_n \\ w(I) \, = \, J}} w \in \CC[\symm_n].
\end{equation}
for the group algebra sum of all permutations $w \in \symm_n$ sending $I$ to $J$.

\begin{proposition}
\label{atomic-character-interpretation}
Let $(I,J) \in \symm_{n,k}$ be a partial permutation.
The expansion of $A_{n,I,J}$ in the Schur basis is given by
\begin{equation}
A_{n,I,J} = \sum_{\lambda \, \vdash \, n} \chi^{\lambda}( [I,J] ) \cdot s_{\lambda}
\end{equation}
where $\chi^{\lambda}: \CC[\symm_n] \rightarrow \CC$ is the linearly extended irreducible character of $\symm_n$ indexed by 
$\lambda$.
\end{proposition}

\begin{proof}
If $w \in \symm_n$ has cycle type $\mu$ then $p_{\mu} = \sum_{\lambda \vdash n} \chi^{\lambda}(w) \cdot s_{\lambda}$.
Proposition~\ref{p-combinatorial-interpretation} and linearity give the result.
\end{proof}

The coefficients appearing in Proposition~\ref{atomic-character-interpretation} admit another characterization in terms of double cosets.
If $G$ is a group, $H \subseteq G$ is a subgroup, and $g \in G$, recall that the {\em double coset} is 
$H g H := \{ hgh' \,:\, h, h' \in H \}$.
Write $H \backslash G / H := \{ H g H \,:\, g \in G \}$ for the family of all such double cosets.

Given a partial permutation $(I,J) \in \symm_{n,k}$, there exist (Proposition~\ref{product-action-on-indicators}) elements $u,v \in \symm_n$ such that 
$[I,J] = u \cdot [ \symm_{n-k} ]_+ \cdot v$. Applying the cyclic invariance of trace, we see that
\begin{equation}
\chi^{\lambda}( [I,J] ) = \chi^{\lambda} \left(u \cdot [ \symm_{n-k} ]_+ \cdot v  \right)
 = \chi^{\lambda} \left(  w \cdot [ \symm_{n-k} ]_+   \right)
\end{equation}
where $w := vu$.  For $1 \leq j \leq n$, let $\eta_j \in \CC[\symm_n]$ be the group algebra element
\begin{equation}
\eta_j := \frac{1}{j!} \sum_{x \, \in \, \symm_j} x = \frac{1}{j!} \cdot [\symm_j]_+.
\end{equation}
It is easy to see that $\eta_j^2 = \eta_j$ is an idempotent element of $\CC[\symm_n]$ so that 
\begin{multline}
\chi^{\lambda} \left(  w \cdot [ \symm_{n-k} ]_+   \right) = 
(n-k)! \cdot \chi^{\lambda} \left(  w \cdot \eta_{n-k}   \right) = \\
(n-k)! \cdot \chi^{\lambda} \left(  w \cdot \eta_{n-k}^2   \right) =
(n-k)! \cdot \chi^{\lambda} \left(  \eta_{n-k} \cdot w \cdot \eta_{n-k}   \right).
\end{multline}
In other words, the quantity $\chi^\lambda([I,J])$ is, up to rescaling, the sum of evaluations of the irreducible character $\chi^{\lambda}$ 
over the double coset $\symm_{n-k} w \symm_{n-k}$.
C. Ryba proved a generating function which gives irreducible character evaluations related to $\chi^{\lambda}([I,J])$.  We describe how his result relates to our work.

Let $\mu = (\mu_1, \dots, \mu_r) \vdash n$ be a partition of length $r$ 
and consider the parabolic subgroup $\symm_\mu = \symm_{\mu_1} \times \cdots \times \symm_{\mu_r}$ 
of $\symm_n$. 
Double cosets in $\symm_\mu \backslash \symm_n / \symm_\mu$ are indexed by $\ZZ_{\geq 0}$-matrices $Q = (q_{i,j})_{1 \leq i,j \leq r}$ with row and column sums equal
to $\mu = (\mu_1, \mu_2, \dots )$.  Such matrices are also known as {\em contingency tables}. 

In a MathOverflow answer,
Ryba \cite{RybaOverflow} (see also \cite[Lem. 2.12]{Ryba}) described a generating function for the irreducible character $\chi^{\mu}$ over double cosets corresponding
to contingency tables.
In particular, Ryba proved that
\begin{equation}
\label{ryba-generating-function}
\sum_Q c_Q \prod_{i,j \, = \, 1}^r x_{i,j}^{q_{i,j}} = \det(x_{1,1})^{\mu_1 - \mu_2} \cdot \det \begin{pmatrix} x_{1,1} & x_{1,2} \\ x_{2,1} & x_{2,2} \end{pmatrix}^{\mu_2 - \mu_3} \cdots 
  \, \, \det(X)^{\mu_r}
\end{equation}
where
\begin{itemize}
\item the sum is over all contingency tables $Q$ with row and column sums $\mu$,
\item if a contingency table $Q$ corresponds to a double coset $\symm_\mu w \symm_\mu$, then $c_Q$ is the trace of the operator
\begin{equation}
\frac{|\symm_\mu w \symm_\mu|}{ |\symm_\mu| } \sum_{u,v \, \in \, \symm_\mu} uwv \in \CC[\symm_n]
\end{equation}
on the $\symm_n$-irreducible $V^{\mu}$, and
\item $X = (x_{i,j})_{1 \leq i, j \leq r}$ is an $r \times r$ matrix of variables.
\end{itemize}

Ryba's Equation~\eqref{ryba-generating-function} may be used to compute $\chi^{\mu} \left(  w \cdot [\symm_{\mu}]_+ \right)$ where $\mu$ is an arbitrary partition
of $n$ by expanding determinants and renormalizing.
On the other hand, we are interested in character evaluations of the form
$\chi^{\lambda} \left(  w \cdot [\symm_{\mu}]_+ \right)$ where $\mu = (n-k,1^k) \vdash n$ is a hook and $\lambda \vdash n$ is an arbitrary partition.
It may be interesting to combine these settings and study
$\chi^{\lambda} \left(  w \cdot [\symm_{\mu}]_+ \right)$ where $\lambda, \mu \vdash n$ are  arbitrary and 
potentially distinct.

Since $\chi^\lambda([I,J]) = \sum_{w(I) = J} \chi^\lambda(w)$, na\"ively calculating the $s$-expansion of $A_{n,I,J}$ 
using Proposition~\ref{atomic-character-interpretation} would involve $(n-k)!$ applications of the Murnaghan-Nakayama Rule.
In our probabilistic applications we will be interested in the behavior of the $s$-expansion as $n \rightarrow \infty$, so this is not
a practical approach.
As a first source of computational savings,
the character evaluations $\chi^{\lambda}([I,J])$ appearing in Proposition~\ref{atomic-character-interpretation}
exhibit 
nontrivial vanishing properties.

\begin{theorem}
\label{atomic-support}
Let $(I,J) \in \symm_{n,k}$ be a partial permutation of $[n]$ of size $k$.  The coefficients $b_{\lambda}$ in the Schur expansion
\begin{equation}
A_{n,I,J} = \sum_{\lambda \, \vdash \, n} b_{\lambda} \cdot s_{\lambda}
\end{equation}
of the atomic local symmetric function $A_{n,I,J}$ satisfy $b_{\lambda} = 0$ unless $\lambda_1 \geq n-k$.
\end{theorem}

\begin{proof}
Lemma~\ref{product-action-on-indicators} implies that $R \, \one_{I,J}: \symm_n \rightarrow \CC$ is a $k$-local class function. 
Since $A_{n,I,J} = n! \cdot \ch_n(R \, \one_{I,J})$, the result follows from Theorem~\ref{local-support-theorem}.
\end{proof}

As a sanity check, we consider Theorem~\ref{atomic-support} in the case $I = J = \varnothing$.  We have 
\begin{equation*}
[I,J] = [\varnothing, \varnothing] = 
\sum_{w \, \in \, \symm_n} w \in \CC[\symm_n],
\end{equation*}
an operator which annihilates every $\symm_n$-irreducible $V^\lambda$ for which $\lambda \neq (n)$.
Accordingly, the only surviving term in the $s$-expansion of Proposition~\ref{atomic-character-interpretation} is $s_n$, which is consistent 
with Theorem~\ref{atomic-support}.

The factorization $A_{n,I,J} = \vec{p}_{\mu} \cdot p_{\nu}$ of Proposition~\ref{atomic-factorization} and the
Path Murnaghan-Nakayama formula (Theorem~\ref{path-murnaghan-nakayama}) will sharpen
Theorem~\ref{atomic-support}.
In the $I = J = \varnothing$ case, we will
recover the fact that 
$[\varnothing, \varnothing] = \sum_{w \in \symm_n} w$
acts on 
the trivial representation $V^{(n)}$ with 
trace $n!$.

\section{Path and classical power sums}

By Proposition~\ref{atomic-factorization}, the $p$-expansion of $A_{n,I,J}$ is determined by the expansion of path power sums into classical power sums.
This expansion is best understood via the theory of M\"obius inversion on posets.

The poset of interest for us is the {\em partition lattice} $\Pi_r$.
 This is the family of all set partitions 
$\sigma$ of $[r]$ partially ordered by $\sigma \leq \sigma'$ if and only if $\sigma$ refines $\sigma'$.
The poset $\Pi_3$ is shown below together with the M\"obius function values $\Mob_{\Pi_3}(\hat{0},\sigma)$ corresponding 
to lower order intervals where $\hat{0} = 1/2/3$ is the minimum element.
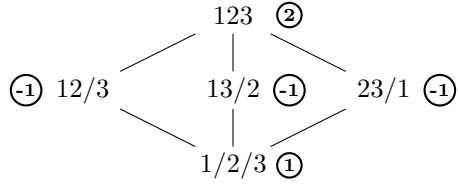
\begin{center}
\begin{tikzpicture}[scale = 0.5]
\node at (0,0) {$1/2/3$};
\draw [-] (0,0.5) -- (0,1.5);
\draw [-] (0,2.5) -- (0,3.5);
\draw [-] (1,0.5) -- (3,1.5);
\draw [-] (-1,0.5) -- (-3,1.5);
\draw [-] (-1,3.5) -- (-3,2.5);
\draw [-] (1,3.5) -- (3,2.5);
\node at (0,2) {$13/2$};
\node at (-4,2) {$12/3$};
\node at (4,2) {$23/1$};
\node at (0,4) {$123$};

  \node [draw, circle, fill = white, inner sep = 1pt, thick] at (1.5,0)
  {\scriptsize {\bf 1}};
 \node [draw, circle, fill = white, inner sep = 1pt, thick] at (5.5,2)
  {\scriptsize {\bf -1}};
  \node [draw, circle, fill = white, inner sep = 1pt, thick] at (1.5,2)
  {\scriptsize {\bf -1}};  
   \node [draw, circle, fill = white, inner sep = 1pt, thick] at (-5.5,2)
  {\scriptsize {\bf -1}};   
   \node [draw, circle, fill = white, inner sep = 1pt, thick] at (1.5,4)
  {\scriptsize {\bf 2}};     
\end{tikzpicture}
\end{center}

For any $r \geq 1$, 
the M\"obius function of $\Pi_r$ satisfies
\begin{equation}
\label{mobius-product}
\Mob_{\Pi_r}(\hat{0}, \hat{1}) = (-1)^{r-1} \cdot (r-1)!
\end{equation}
where $\hat{0}$ is the set partition of $[r]$ in which all blocks are singletons and $\hat{1}$ is the set partition of $[r]$ with a single block.
Since the M\"obius function on posets $P, Q$
satisfies the multiplicative property
\begin{equation}
    \Mob_{P \times Q}( (p_1,q_1), (p_2, q_2) ) = 
    \Mob_P(p_1, p_2) \cdot \Mob_Q(q_1, q_2)
\end{equation}
where $P \times Q$ is given the componentwise
partial order,
Equation~\eqref{mobius-product} implies that the lower interval M\"obius function evaluations on $\Pi_r$ are given by
\begin{equation}
\label{mobius-partition}
\Mob_{\Pi_r}(\hat{0}, \sigma) = (-1)^{r - |\sigma|} \cdot \prod_{B \, \in \, \sigma} ( |B| - 1)!.
\end{equation}
If $\mu = (\mu_1, \dots, \mu_r)$ is a partition with $r$ parts,
the expansion of the path power sum $\vec{p}_{\mu}$ into classical power sums has a similar formula, but without the signs.

\begin{proposition}
\label{path-to-classical}
Let $\mu = (\mu_1, \dots, \mu_r)$ be a partition with $r$ parts.  We have 
\begin{equation}
\vec{p}_{\mu} = \sum_{w \, \in \, \symm_r} \prod_{C \, \in \, w} p_{\mu_C} = \sum_{\sigma \, \in \, \Pi_r} \prod_{B \, \in \, \sigma} (|B|-1)! \cdot p_{\mu_B}
\end{equation}
where the product in the middle expression is over all cycles $C$ belonging to the permutation $w \in \symm_r$.
\end{proposition}

For example, if $\mu = (a,b,c)$ then $r = 3$ and
\begin{equation*}
\vec{p}_{a,b,c} = p_{a,b,c} + p_{a+b,c} + p_{a+c,b} + p_{b+c,a} + 2 \cdot p_{a+b+c}
\end{equation*}
where the coefficient $2$ arises since the cycles $(a,b,c)$ and $(c,b,a)$ are different permutations.

\begin{proof}
By Definition~\ref{path-power-sum-definition} and
 Proposition~\ref{p-combinatorial-interpretation}, the $p$-expansion of $\vec{p}_{\mu}$ is obtained by starting with disjoint paths of 
 sizes $\mu_1, \dots, \mu_r$ glueing these paths to form disjoint cycles in all possible ways. The result follows.
\end{proof}

Proposition~\ref{path-to-classical} implies that
the terms $p_{\nu}$ in the $p$-expansion of $\vec{p}_{\mu}$ are indexed by partitions $\nu$ obtained by combining parts of $\mu$.
Since the $p$-functions are a basis of the space of symmetric functions,  by triangularity  we have the following.

\begin{corollary}
\label{path-basis}
The path power sums $\{ \vec{p}_{\mu} \,:\, \mu \vdash n \}$ form a basis of the vector space $\Lambda_n$.
\end{corollary}

By Proposition~\ref{path-to-classical} and the fact that the $p$-basis is multiplicative,
 the path power sums satisfy the multiplication formula
\begin{equation}
\vec{p}_a \cdot \vec{p}_{(\mu_1, \dots, \, \mu_r)} = \vec{p}_{(a, \, \mu_1, \dots, \, \mu_r)} - \sum_{i \, = \, 1}^r \vec{p}_{(\mu_1,  \dots, \, \mu_i + a, \dots, \, \mu_r)}
\end{equation}
where $\mu = (\mu_1, \dots, \mu_r)$ is any composition. Thanks to Proposition~\ref{atomic-factorization} and the fact that
$\vec{p}_a = p_a$, this gives the expansion of any atomic symmetric function $A_{n,I,J}$
in the basis $\{ \vec{p}_{\mu} \}$ of Corollary~\ref{path-basis}.

In the special case where all parts of $\mu = (k^r)$ are equal, the path power sum $\vec{p}_{\mu}$ has a plethystic 
interpretation. Let $F[G]$ be the {\em plethysm} operation defined on symmetric functions $F, G$ characterized by
\begin{equation}
(F_1 \cdot F_2)[G] = F_1[G] \cdot F_2[G] \quad 
(c_1 F_1 + c_2 F_2)[G] = c_1 F_1[G] + c_2 F_2[G] \quad
c[G] = c
\end{equation}
for symmetric functions $F_1, F_2, G$ and constants $c_1, c_2, c \in \CC$ together with the condition
\begin{equation}
p_k[G] = p_k[G(x_1, x_2, \dots)] = G(x_1^k, x_2^k, \dots ).
\end{equation}
We claim that
\begin{equation}
\label{path-power-plethysm}
\vec{p}_{(k^r)} = r! \cdot h_r[p_k].
\end{equation}
Indeed, the $p$-expansion of the complete homogeneous function is $h_r = \sum_{\lambda \vdash r} \frac{1}{z_{\lambda}} p_{\lambda}$ so that
\begin{equation}
r! \cdot h_r[p_k] = 
\sum_{\lambda \, \vdash \, r} \frac{r!}{z_{\lambda}} p_{\lambda}[p_k] = 
\sum_{\lambda \, \vdash \, r} |K_{\lambda}| 
\cdot p_{k \cdot \lambda}
\end{equation}
where $k \cdot \lambda = (k \lambda_1, k \lambda_2, \dots )$ is the partition obtained by multiplying each part of $\lambda$ by $k$ and $K_\lambda \subseteq \symm_n$
is the conjugacy class of permutations of cycle type $\lambda$.
Proposition~\ref{path-to-classical} implies that
$\vec{p}_{(k^r)} = \sum_{\lambda \, \vdash \, r} 
|K_{\lambda}| \cdot p_{k \cdot \lambda}$, which implies 
\eqref{path-power-plethysm}.
Since the path power sum $\vec{p}_{\mu}$ is not multiplicative in $\mu$, Equation~\eqref{path-power-plethysm}
does not na\"ively extend to nonconstant partitions $\mu$.

The expansion of the classical power sums in terms of path power sums is simpler than 
Proposition~\ref{path-to-classical} and involves signs.

\begin{proposition}
\label{classical-to-path}
Let $\mu = (\mu_1, \dots, \mu_r)$ be a partition with $r$ parts. We have
\begin{equation}
p_{\mu} = \sum_{\sigma \, \in \, \Pi_r}
(-1)^{r - |\sigma|}
\prod_{B \, \in \, \sigma} \vec{p}_{\mu_B}.
\end{equation}
\end{proposition}

If $\mu = (a,b,c)$ as before, Proposition~\ref{classical-to-path} yields
\begin{equation*}
p_{a,b,c} = \vec{p}_{a,b,c} - \vec{p}_{a+b,c} - \vec{p}_{a+c,b} - \vec{p}_{b+c,a} + \vec{p}_{a+b+c}.
\end{equation*}

\begin{proof}
This follows from Proposition~\ref{path-to-classical} and
 the M\"obius Inversion Formula as applied to the partition lattice $\Pi_r$.
\end{proof}

\section{Path power sums to monomial symmetric functions}

Let $\mu \vdash n$ be a partition of $n$ with $r$ parts. The classical power sum $p_\mu$ expands into the basis of monomial
symmetric functions as 
\begin{equation}
\label{classical-p-to-m}
p_\mu = \sum_{\lambda \, \vdash \, n} a_{\mu,\lambda} \cdot m_\lambda.
\end{equation}
If $\lambda$ has $s$ parts, then $a_{\mu,\lambda}$ counts ordered partitions $(B_1, \dots , B_s)$ of the set $[r]$ 
such that $\sum_{i \in B_j} \mu_i = \lambda_j$ for all $1 \leq j \leq s$.

The path power sum $\vec{p}_{\mu}$ also admits a simple expansion in the monomial basis.
A ribbon $\xi$ is {\em horizontal} if its boxes occupy a single row.
Recall that $m_i(\mu)$ denotes the multiplicity of $i$ as a part of a composition $\mu$.
The factor $m(\mu)! := m_1(\mu)! m_2(\mu)! \cdots $ appearing in the following result will also feature in the 
Schur expansion of $\vec{p}_{\mu}$.

\begin{proposition}
\label{path-to-monomial}
For any partition $\mu = (\mu_1, \dots, \mu_r) \vdash n$ we have
\begin{equation}
\vec{p}_{\mu} = m(\mu)! \cdot \sum_{\lambda \, \vdash \, n} d_{\lambda} \cdot m_{\lambda}
\end{equation} 
where $d_{\lambda}$ counts the number of ways to tile the Young diagram of $\lambda$ with 
horizontal ribbons of sizes $\mu_1, \dots, \mu_r$.
\end{proposition}

\begin{proof}
Combine Equation~\eqref{classical-p-to-m} with Proposition~\ref{path-to-classical}.
\end{proof}

For example, if $\mu = (3,2,2,1)$ then $m(\mu)! = 2$ and Proposition~\ref{path-to-monomial} gives the $m$-basis expansion
\begin{equation*}
\vec{p}_{3221} = \\ 2 \times \left[
\begin{array}{c}
m_{3221} + 4 m_{332} + 2 m_{422} + m_{431} + 4 m_{44}  \\ + \, 2 m_{521} + 7 m_{53} + 6 m_{62} + 3 m_{71} + 12 m_8
\end{array}
\right].
\end{equation*}
The coefficient $7$ of $m_{53}$ is witnessed by the collection of tilings in Figure~\ref{fig:horizontal-monomial}.
Equation~\eqref{classical-p-to-m} and Proposition~\ref{path-to-monomial} combine with Proposition~\ref{atomic-factorization}
 to give an expansion of any atomic function $A_{n,I,J}$ in the $m$-basis; we leave details to the interested reader.

\begin{figure}
\begin{center}
\begin{tikzpicture}[scale = 0.3]
  \begin{scope}
    \clip (0,0) -| (3,1) -| (5,2) -| (0,0);
    \draw [color=black!25] (0,0) grid (5,2);
  \end{scope}

  \draw [thick] (0,0) -| (3,1) -| (5,2) -| (0,0);

  \draw [thick, rounded corners] (0.5,1.5) -- (2.5,1.5);
  \draw [color=black,fill=black,thick] (2.5,1.5) circle (.4ex);
  \node [draw, circle, fill = white, inner sep = 1.2pt] at (0.5,1.5) { };
  
   \draw [thick, rounded corners] (3.5,1.5) -- (4.5,1.5);
  \draw [color=black,fill=black,thick] (4.5,1.5) circle (.4ex);
  \node [draw, circle, fill = white, inner sep = 1.2pt] at (3.5,1.5) { };
  
   \draw [thick, rounded corners] (0.5,0.5) -- (1.5,0.5);
  \draw [color=black,fill=black,thick] (1.5,0.5) circle (.4ex);
  \node [draw, circle, fill = white, inner sep = 1.2pt] at (0.5,0.5) { };
  
    \node [draw, circle, fill = white, inner sep = 1.2pt] at (2.5,0.5) { };

   \begin{scope}
    \clip (8,0) -| (11,1) -| (13,2) -| (8,0);
    \draw [color=black!25] (8,0) grid (13,2);
  \end{scope}

  \draw [thick] (8,0) -| (11,1) -| (13,2) -| (8,0);

   \draw [thick, rounded corners] (8.5,1.5) -- (10.5,1.5);
  \draw [color=black,fill=black,thick] (10.5,1.5) circle (.4ex);
  \node [draw, circle, fill = white, inner sep = 1.2pt] at (8.5,1.5) { };
  
   \draw [thick, rounded corners] (11.5,1.5) -- (12.5,1.5);
  \draw [color=black,fill=black,thick] (12.5,1.5) circle (.4ex);
  \node [draw, circle, fill = white, inner sep = 1.2pt] at (11.5,1.5) { };
  
   \draw [thick, rounded corners] (9.5,0.5) -- (10.5,0.5);
  \draw [color=black,fill=black,thick] (10.5,0.5) circle (.4ex);
  \node [draw, circle, fill = white, inner sep = 1.2pt] at (9.5,0.5) { };
  
   \node [draw, circle, fill = white, inner sep = 1.2pt] at (8.5,0.5) { };

   \begin{scope}
    \clip (16,0) -| (19,1) -| (21,2) -| (16,0);
    \draw [color=black!25] (16,0) grid (21,2);
  \end{scope}

  \draw [thick] (16,0) -| (19,1) -| (21,2) -| (16,0);
  
   \draw [thick, rounded corners] (18.5,1.5) -- (20.5,1.5);
  \draw [color=black,fill=black,thick] (20.5,1.5) circle (.4ex);
  \node [draw, circle, fill = white, inner sep = 1.2pt] at (18.5,1.5) { };
  
   \draw [thick, rounded corners] (16.5,1.5) -- (17.5,1.5);
  \draw [color=black,fill=black,thick] (17.5,1.5) circle (.4ex);
  \node [draw, circle, fill = white, inner sep = 1.2pt] at (16.5,1.5) { };
  
   \draw [thick, rounded corners] (16.5,0.5) -- (17.5,0.5);
  \draw [color=black,fill=black,thick] (17.5,0.5) circle (.4ex);
  \node [draw, circle, fill = white, inner sep = 1.2pt] at (16.5,0.5) { };
  
   \node [draw, circle, fill = white, inner sep = 1.2pt] at (18.5,0.5) { };

   \begin{scope}
    \clip (24,0) -| (27,1) -| (29,2) -| (24,0);
    \draw [color=black!25] (24,0) grid (29,2);
  \end{scope}

  \draw [thick] (24,0) -| (27,1) -| (29,2) -| (24,0);
  
   \draw [thick, rounded corners] (26.5,1.5) -- (28.5,1.5);
  \draw [color=black,fill=black,thick] (28.5,1.5) circle (.4ex);
  \node [draw, circle, fill = white, inner sep = 1.2pt] at (26.5,1.5) { };  
  
   \draw [thick, rounded corners] (24.5,1.5) -- (25.5,1.5);
  \draw [color=black,fill=black,thick] (25.5,1.5) circle (.4ex);
  \node [draw, circle, fill = white, inner sep = 1.2pt] at (24.5,1.5) { };
  
   \draw [thick, rounded corners] (25.5,0.5) -- (26.5,0.5);
  \draw [color=black,fill=black,thick] (26.5,0.5) circle (.4ex);
  \node [draw, circle, fill = white, inner sep = 1.2pt] at (25.5,0.5) { };
  
   \node [draw, circle, fill = white, inner sep = 1.2pt] at (24.5,0.5) { };

  \begin{scope}
    \clip (4,-3) -| (7,-2) -| (9,-1) -| (4,-3);
    \draw [color=black!25] (4,-3) grid (9,-1);
  \end{scope}

  \draw [thick]  (4,-3) -| (7,-2) -| (9,-1) -| (4,-3);

  \draw [thick, rounded corners] (4.5,-2.5) -- (6.5,-2.5);
  \draw [color=black,fill=black,thick] (6.5,-2.5) circle (.4ex);
  \node [draw, circle, fill = white, inner sep = 1.2pt] at (4.5,-2.5) { };
  
   \draw [thick, rounded corners] (4.5,-1.5) -- (5.5,-1.5);
  \draw [color=black,fill=black,thick] (5.5,-1.5) circle (.4ex);
  \node [draw, circle, fill = white, inner sep = 1.2pt] at (4.5,-1.5) { };
  
   \draw [thick, rounded corners] (6.5,-1.5) -- (7.5,-1.5);
  \draw [color=black,fill=black,thick] (7.5,-1.5) circle (.4ex);
  \node [draw, circle, fill = white, inner sep = 1.2pt] at (6.5,-1.5) { };
  
    \node [draw, circle, fill = white, inner sep = 1.2pt] at (8.5,-1.5) { };

  \begin{scope}
    \clip (12,-3) -| (15,-2) -| (17,-1) -| (12,-3);
    \draw [color=black!25] (12,-3) grid (17,-1);
  \end{scope}

  \draw [thick]  (12,-3) -| (15,-2) -| (17,-1) -| (12,-3); 
  
  \draw [thick, rounded corners] (12.5,-2.5) -- (14.5,-2.5);
  \draw [color=black,fill=black,thick] (14.5,-2.5) circle (.4ex);
  \node [draw, circle, fill = white, inner sep = 1.2pt] at (12.5,-2.5) { };
  
   \draw [thick, rounded corners] (12.5,-1.5) -- (13.5,-1.5);
  \draw [color=black,fill=black,thick] (13.5,-1.5) circle (.4ex);
  \node [draw, circle, fill = white, inner sep = 1.2pt] at (12.5,-1.5) { };
  
   \draw [thick, rounded corners] (15.5,-1.5) -- (16.5,-1.5);
  \draw [color=black,fill=black,thick] (16.5,-1.5) circle (.4ex);
  \node [draw, circle, fill = white, inner sep = 1.2pt] at (15.5,-1.5) { };
  
   \node [draw, circle, fill = white, inner sep = 1.2pt] at (14.5,-1.5) { };

  \begin{scope}
    \clip (20,-3) -| (23,-2) -| (25,-1) -| (20,-3);
    \draw [color=black!25] (20,-3) grid (25,-1);
  \end{scope}

  \draw [thick, rounded corners] (20.5,-2.5) -- (22.5,-2.5);
  \draw [color=black,fill=black,thick] (22.5,-2.5) circle (.4ex);
  \node [draw, circle, fill = white, inner sep = 1.2pt] at (20.5,-2.5) { };
  
   \draw [thick, rounded corners] (21.5,-1.5) -- (22.5,-1.5);
  \draw [color=black,fill=black,thick] (22.5,-1.5) circle (.4ex);
  \node [draw, circle, fill = white, inner sep = 1.2pt] at (21.5,-1.5) { };
  
   \draw [thick, rounded corners] (23.5,-1.5) -- (24.5,-1.5);
  \draw [color=black,fill=black,thick] (24.5,-1.5) circle (.4ex);
  \node [draw, circle, fill = white, inner sep = 1.2pt] at (23.5,-1.5) { };
  
   \node [draw, circle, fill = white, inner sep = 1.2pt] at (20.5,-1.5) { };   

  \draw [thick]  (20,-3) -| (23,-2) -| (25,-1) -| (20,-3);

\end{tikzpicture}
\end{center}
\caption{The coefficient of $m_{53}$ in $\vec{p}_{\mu}$ for $\mu = (3,2,2,1)$ is $m(\mu)! \cdot 7$.}
\label{fig:horizontal-monomial}
\end{figure}
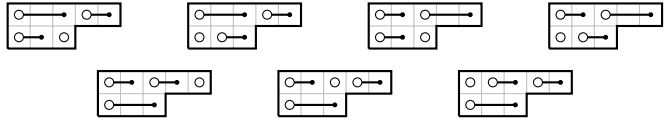

\section{Skewing operators}

For any symmetric function $F \in \Lambda$, let $F^{\perp}: \Lambda \rightarrow \Lambda$ be the adjoint to multiplication by $F$
under the Hall inner product. This operator is characterized by
\begin{equation}
\langle F^{\perp} G, H \rangle = \langle G, FH \rangle \quad \quad \text{for all $G,H \in \Lambda$.}
\end{equation}
In particular, the operator $h_j^{\perp}$ lowers the degree of a symmetric function by $j$. More precisely, the Pieri Rule
implies
\begin{equation}
h_j^{\perp} s_{\lambda} = \sum_{\mu} s_{\mu}
\end{equation}
where the sum is over all partitions $\mu \subseteq \lambda$ such that the skew shape $\lambda/\mu$ consists of $j$ boxes,
each in different columns. 
The operators $h_j^{\perp}$ can be used to recursively characterize symmetric functions; for any $F, G \in \Lambda_n$  of positive
homogeneous degree $n \geq 1$
we have
\begin{equation}
\label{skewing-assertion}
F = G \text{ if and only if } h_j^{\perp} F = h_j^{\perp} G \text{ for all $j \geq 1$.}
\end{equation}
The assertion \eqref{skewing-assertion} is true because the elements $h_1, h_2, \dots $ constitute an algebraically independent
generating set of $\Lambda$.

The atomic functions behave nicely under the action of $h_j^{\perp}$. This recursion is best stated using the graphical representation
of partial permutations.  If $(I,J)$ is a partial permutation of $[n]$ with graph $G = G_n(I,J) \in \GGG_n$, we write
$A_G := A_{n,I,J}$ for the atomic symmetric function corresponding to $(I,J)$.

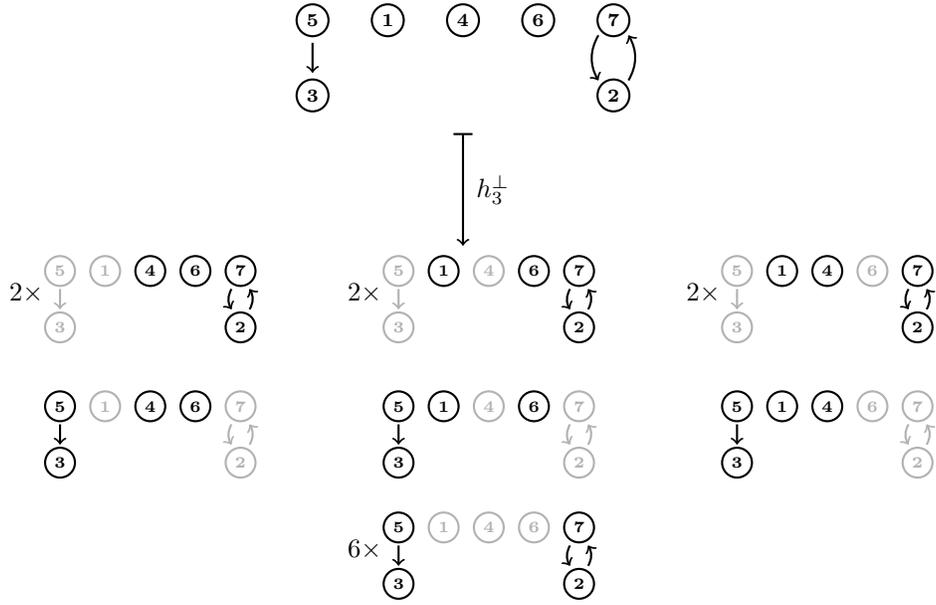
\begin{figure}
\begin{center}
\begin{tikzpicture}[scale = 0.5]

   \node [draw, circle, fill = white, inner sep = 2pt, thick] at (2,0)
  {\scriptsize {\bf 1} };
     \node [draw, circle, fill = white, inner sep = 2pt, thick] at (8,-2)
  {\scriptsize {\bf 2}};
     \node [draw, circle, fill = white, inner sep = 2pt, thick] at (0,-2)
  {\scriptsize {\bf 3}};
     \node [draw, circle, fill = white, inner sep = 2pt, thick] at (4,0)
  {\scriptsize {\bf 4}};
     \node [draw, circle, fill = white, inner sep = 2pt, thick] at (0,0)
  {\scriptsize {\bf 5}};
     \node [draw, circle, fill = white, inner sep = 2pt, thick] at (6,0)
  {\scriptsize {\bf 6}};
     \node [draw, circle, fill = white, inner sep = 2pt, thick] at (8,0)
  {\scriptsize {\bf 7}};
  \draw [->, thick] (0,-0.6) -- (0,-1.4);
 \draw [->, thick]  (7.6,-0.4) to[bend right]  (7.6,-1.6);
 \draw [->, thick]  (8.4,-1.6) to[bend right]  (8.4,-0.4);
 
 \draw [|->, thick] (4,-3) -- (4,-6);
 
\node at (4.8,-4.5) {$h_3^{\perp}$};

\end{tikzpicture}
\end{center}

\begin{center}
\begin{tikzpicture}[scale = 0.3]

\node at (-1.5,-1) {$2 \times$};
   \node [draw, circle, black!30, fill = white, inner sep = 2pt, thick] at (2,0)
  {\tiny {\bf 1} };
     \node [draw, circle, fill = white, inner sep = 2pt, thick] at (8,-2.5)
  {\tiny {\bf 2}};
     \node [draw, circle, black!30, fill = white, inner sep = 2pt, thick] at (0,-2.5)
  {\tiny {\bf 3}};
     \node [draw, circle, fill = white, inner sep = 2pt, thick] at (4,0)
  {\tiny {\bf 4}};
     \node [draw, circle, black!30, fill = white, inner sep = 2pt, thick] at (0,0)
  {\tiny {\bf 5}};
     \node [draw, circle, fill = white, inner sep = 2pt, thick] at (6,0)
  {\tiny {\bf 6}};
     \node [draw, circle, fill = white, inner sep = 2pt, thick] at (8,0)
  {\tiny {\bf 7}};
  \draw [->, thick, black!30] (0,-0.8) -- (0,-1.7);
 \draw [->, thick]  (7.6,-0.8) to[bend right]  (7.6,-1.7);
 \draw [->, thick]  (8.4,-1.7) to[bend right]  (8.4,-0.8);
 
 \node at (13.5,-1) {$2 \times$};
    \node [draw, circle, fill = white, inner sep = 2pt, thick] at (17,0)
  {\tiny {\bf 1} };
     \node [draw, circle, fill = white, inner sep = 2pt, thick] at (23,-2.5)
  {\tiny {\bf 2}};
     \node [draw, circle, black!30, fill = white, inner sep = 2pt, thick] at (15,-2.5)
  {\tiny {\bf 3}};
     \node [draw, circle, black!30, fill = white, inner sep = 2pt, thick] at (19,0)
  {\tiny {\bf 4}};
     \node [draw, circle, black!30, fill = white, inner sep = 2pt, thick] at (15,0)
  {\tiny {\bf 5}};
     \node [draw, circle, fill = white, inner sep = 2pt, thick] at (21,0)
  {\tiny {\bf 6}};
     \node [draw, circle, fill = white, inner sep = 2pt, thick] at (23,0)
  {\tiny {\bf 7}};
  \draw [->, thick, black!30] (15,-0.8) -- (15,-1.7);
 \draw [->, thick]  (22.6,-0.8) to[bend right]  (22.6,-1.7);
 \draw [->, thick]  (23.4,-1.7) to[bend right]  (23.4,-0.8);

 \node at (28.5,-1) {$2 \times$};
    \node [draw, circle, fill = white, inner sep = 2pt, thick] at (32,0)
  {\tiny {\bf 1} };
     \node [draw, circle, fill = white, inner sep = 2pt, thick] at (38,-2.5)
  {\tiny {\bf 2}};
     \node [draw, circle, black!30, fill = white, inner sep = 2pt, thick] at (30,-2.5)
  {\tiny {\bf 3}};
     \node [draw, circle, fill = white, inner sep = 2pt, thick] at (34,0)
  {\tiny {\bf 4}};
     \node [draw, circle, black!30, fill = white, inner sep = 2pt, thick] at (30,0)
  {\tiny {\bf 5}};
     \node [draw, circle, black!30, fill = white, inner sep = 2pt, thick] at (36,0)
  {\tiny {\bf 6}};
     \node [draw, circle, fill = white, inner sep = 2pt, thick] at (38,0)
  {\tiny {\bf 7}};
  \draw [->, thick, black!30] (30,-0.8) -- (30,-1.7);
 \draw [->, thick]  (37.6,-0.8) to[bend right]  (37.6,-1.7);
 \draw [->, thick]  (38.4,-1.7) to[bend right]  (38.4,-0.8);

    \node [draw, circle, black!30, fill = white, inner sep = 2pt, thick] at (2,-6)
  {\tiny {\bf 1} };
     \node [draw, circle, black!30, fill = white, inner sep = 2pt, thick] at (8,-8.5)
  {\tiny {\bf 2}};
     \node [draw, circle, fill = white, inner sep = 2pt, thick] at (0,-8.5)
  {\tiny {\bf 3}};
     \node [draw, circle, fill = white, inner sep = 2pt, thick] at (4,-6)
  {\tiny {\bf 4}};
     \node [draw, circle, fill = white, inner sep = 2pt, thick] at (0,-6)
  {\tiny {\bf 5}};
     \node [draw, circle, fill = white, inner sep = 2pt, thick] at (6,-6)
  {\tiny {\bf 6}};
     \node [draw, circle, black!30, fill = white, inner sep = 2pt, thick] at (8,-6)
  {\tiny {\bf 7}};
  \draw [->, thick] (0,-6.8) -- (0,-7.7);
 \draw [->, thick, black!30]  (7.6,-6.8) to[bend right]  (7.6,-7.7);
 \draw [->, thick, black!30]  (8.4,-7.7) to[bend right]  (8.4,-6.8);

    \node [draw, circle, fill = white, inner sep = 2pt, thick] at (17,-6)
  {\tiny {\bf 1} };
     \node [draw, circle, black!30, fill = white, inner sep = 2pt, thick] at (23,-8.5)
  {\tiny {\bf 2}};
     \node [draw, circle, fill = white, inner sep = 2pt, thick] at (15,-8.5)
  {\tiny {\bf 3}};
     \node [draw, circle, black!30, fill = white, inner sep = 2pt, thick] at (19,-6)
  {\tiny {\bf 4}};
     \node [draw, circle, fill = white, inner sep = 2pt, thick] at (15,-6)
  {\tiny {\bf 5}};
     \node [draw, circle, fill = white, inner sep = 2pt, thick] at (21,-6)
  {\tiny {\bf 6}};
     \node [draw, circle, black!30, fill = white, inner sep = 2pt, thick] at (23,-6)
  {\tiny {\bf 7}};
  \draw [->, thick] (15,-6.8) -- (15,-7.7);
 \draw [->, thick, black!30]  (22.6,-6.8) to[bend right]  (22.6,-7.7);
 \draw [->, thick, black!30]  (23.4,-7.7) to[bend right]  (23.4,-6.8);

    \node [draw, circle, fill = white, inner sep = 2pt, thick] at (32,-6)
  {\tiny {\bf 1} };
     \node [draw, circle, black!30, fill = white, inner sep = 2pt, thick] at (38,-8.5)
  {\tiny {\bf 2}};
     \node [draw, circle, fill = white, inner sep = 2pt, thick] at (30,-8.5)
  {\tiny {\bf 3}};
     \node [draw, circle, fill = white, inner sep = 2pt, thick] at (34,-6)
  {\tiny {\bf 4}};
     \node [draw, circle, fill = white, inner sep = 2pt, thick] at (30,-6)
  {\tiny {\bf 5}};
     \node [draw, circle, black!30, fill = white, inner sep = 2pt, thick] at (36,-6)
  {\tiny {\bf 6}};
     \node [draw, circle, black!30, fill = white, inner sep = 2pt, thick] at (38,-6)
  {\tiny {\bf 7}};
  \draw [->, thick] (30,-6.8) -- (30,-7.7);
 \draw [->, thick, black!30]  (37.6,-6.8) to[bend right]  (37.6,-7.7);
 \draw [->, thick, black!30]  (38.4,-7.7) to[bend right]  (38.4,-6.8);

\end{tikzpicture}
\end{center}
\vspace{0.1in}
\begin{center}
\begin{tikzpicture}[scale = 0.3]
  \node at (43.5,-1) {$6 \times$};
     \node [draw, circle, black!30, fill = white, inner sep = 2pt, thick] at (47,0)
  {\tiny {\bf 1} };
     \node [draw, circle, fill = white, inner sep = 2pt, thick] at (53,-2.5)
  {\tiny {\bf 2}};
     \node [draw, circle, fill = white, inner sep = 2pt, thick] at (45,-2.5)
  {\tiny {\bf 3}};
     \node [draw, circle, black!30, fill = white, inner sep = 2pt, thick] at (49,0)
  {\tiny {\bf 4}};
     \node [draw, circle, fill = white, inner sep = 2pt, thick] at (45,0)
  {\tiny {\bf 5}};
     \node [draw, circle, black!30, fill = white, inner sep = 2pt, thick] at (51,0)
  {\tiny {\bf 6}};
     \node [draw, circle, fill = white, inner sep = 2pt, thick] at (53,0)
  {\tiny {\bf 7}};
  \draw [->, thick] (45,-0.8) -- (45,-1.7);
 \draw [->, thick]  (52.6,-0.8) to[bend right]  (52.6,-1.7);
 \draw [->, thick]  (53.4,-1.7) to[bend right]  (53.4,-0.8);

\end{tikzpicture}
\end{center}

\caption{An application of $h_3^{\perp}$ to an atomic symmetric function.}
\label{fig:h-perp}
\end{figure}

\begin{proposition}
\label{skewing-recursion}
Let $G \in \GGG_n$ be a graph corresponding to a partial permutation and let $j \geq 1$.
We have
\begin{equation}
h_j^{\perp} A_G = \sum_{H} \path(G - H)! \cdot A_H
\end{equation}
where the sum is over all subgraphs $H \subseteq G$ obtained by removing components of $G$ such that $G - H$ has $j$ vertices.
Here $\path(G-H)$ is the number of components in the removed graph $G-H$ which are paths (instead of cycles).
\end{proposition}

\begin{proof}
The application of $h_j^{\perp}$ has the following class function interpretation. Recall that $\eta_j \in \CC[\symm_n]$ is the group algebra element
\begin{equation}
\eta_j = \frac{1}{j!} \cdot \sum_{w \, \in \, \symm_j} w
\end{equation}
which symmetrizes over the first $j$ letters in $[n]$. 
For any (linearly extended) class function 
$\psi: \CC[\symm_n] \rightarrow \CC$ we have a class function
$\eta_j \psi: \CC[\symm_{n-j}] \rightarrow \CC$ on the smaller symmetric group $\symm_{n-j}$ given by
\begin{equation}
\eta_j \psi( v) := \psi( \eta_j v)
\end{equation}
where we interpret the argument of $\psi(\eta_j v)$ via the subgroup $\symm_j \times \symm_{n-j}$ of $\symm_n$.  Frobenius Reciprocity
yields
the equality
\begin{equation}
h_j^{\perp} \ch_n(\psi) = \ch_{n-j}( \eta_j \psi)
\end{equation}
of symmetric functions.

The class function $\psi^{\mu}: \symm_n \rightarrow \CC$ satisfying $\ch_n(\psi^{\mu}) = p_{\mu}$ 
is the scaled indicator function for the conjugacy class $K_{\mu}$. We have
\begin{equation}
\psi^{\mu}(w) = \begin{cases}
z_{\mu} & \text{$w$ has cycle type $\mu$} \\
0 & \text{otherwise.}
\end{cases}
\end{equation}
If we write $\mu = (\mu_1, \mu_2, \dots, \mu_{\ell})$ for the parts of $\mu$, it is not hard to see that
\begin{equation}
\eta_j \psi^{\mu} = 
\sum_{\substack{1 \, \leq \, i_1 \, < \, \cdots \, < \, i_r \, \leq \, \ell \\ \mu_{i_1} + \cdots + \mu_{i_r} \, = \, n-j}}
\psi^{(\mu_{i_1}, \dots, \mu_{i_r})}
\end{equation}
as class functions on $\symm_{n-j}$.
Applying the characteristic map gives
\begin{equation}
h_j^{\perp} p_{\mu} 
\sum_{\substack{1 \, \leq \, i_1 \, < \, \cdots \, < \, i_r \, \leq \, \ell \\ \mu_{i_1} + \cdots + \mu_{i_r} \, = \, n-j}}
p_{(\mu_{i_1}, \dots, \mu_{i_r})}.
\end{equation}
That is, to apply $h_j^{\perp}$ to $p_{\mu}$, we sum over all $p_{\nu}$'s for which 
$\nu$ is obtained by removing parts of $\mu$ which sum to $j$.
The result follows from Proposition~\ref{p-combinatorial-interpretation} and linearity.
\end{proof}

An example of Proposition~\ref{skewing-recursion} is shown in Figure~\ref{fig:h-perp}.
Let $(I,J) \in \symm_{7,3}$ be the partial permutation $(I,J) = (257,732)$.  The graph $G_n(I,J)$ has cycle partition
$(2)$ and path partition $(2,1,1,1)$. To compute $h_3^{\perp} A_{7,I,J} = h_3^{\perp} (p_2 \cdot \vec{p}_{2,1,1,1})$, we remove components from $G_n(I,J)$ in all possible ways
such that a total of 3 vertices are removed.
Removing $r$ paths in this fashion contributes a factor of $r!$.
This results in the formula
\begin{equation*}
h_3^{\perp} A_{7,I,J} = h_3^{\perp} (p_2 \cdot \vec{p}_{2,1,1,1}) = 6 p_2 \cdot \vec{p}_{1,1} + 3 \vec{p}_{2,1,1} + 6 p_2 \cdot \vec{p}_2.
\end{equation*}

Next we study the action of the skewing operator $p_j^\perp$.
The formula has the same flavor as Proposition~\ref{skewing-recursion}.

\begin{proposition}
\label{p-skewing-recursion}
Let $G \in \GGG_n$ be a graph corresponding to a partial permutation and let $j \geq 1$.
We have
\begin{equation}
p_j^{\perp} A_G = j \times \sum_{H} (\path(G - H)-1)! \cdot A_H
\end{equation}
where the sum is over all subgraphs $H \subseteq G$ obtained by removing components of $G$ such that $G - H$ is a disjoint union of paths with $j$ total
vertices or a single cycle with $j$ vertices.
\end{proposition}

\begin{proof}
Let 
$F \in \Lambda$ be an arbitrary symmetric function and
express $F$ uniquely as a polynomial in the power sums $\{ p_1, p_2, \dots \}$ via
\begin{equation}
F = \sum_{\mathbf{a}} c_{\mathbf{a}} \cdot p_1^{a_1} p_2^{a_2} \cdots 
\end{equation} 
where the sum ranges over all sequences $\mathbf{a} = (a_1, a_2, \dots )$ in $\ZZ_{\geq 0}$ with finite support and all but finitely many 
of the coefficients $c_{\mathbf{a}}$ are zero.
The application of $p_j^\perp$ to $F$ satisfies
\begin{equation}
\label{p-skew-derivative}
p_j^\perp F = j (\partial / \partial p_j) F =  \sum_{\mathbf{a}} j a_j c_{\mathbf{a}} \cdot p_1^{a_1} p_2^{a_2} \cdots p_j^{a_j - 1} \cdots 
\end{equation}
where $\partial/\partial p_j$ is differentiation with respect to $p_j$.

Thanks to Proposition~\ref{atomic-factorization}, the atomic symmetric function $A_G$ factors as
 $A_{G} = p_\nu \cdot \vec{p}_\mu$ where $\nu$ and $\mu$ record the cycle 
 and path sizes in $G$. By Equation~\eqref{p-skew-derivative} and the product rule we have
\begin{equation}
p_j^\perp A_{G} = ( p_j^\perp p_\nu ) \cdot p_\mu + p_\nu \cdot ( p_j^\perp \vec{p}_\mu ).
\end{equation}
Equation~\eqref{p-skew-derivative} gives the $p$-expansion of  $p_j^\perp p_\nu$ immediately.
The $\vec{p}$-expansion of $p_j^\perp \vec{p}_\mu$ follows from
Proposition~\ref{classical-to-path}. Combining these expansions gives the desired result.
\end{proof}

We close this chapter with basic questions and problems concerning the symmetric functions $A_{n,I,J}$ and $\vec{p}_{\mu}$.
Let $\omega: \Lambda \rightarrow \Lambda$ be the involution which interchanges $h_n$ and $e_n$.

\begin{question}
\label{omega-question}
Does the image of $A_{n,I,J}$ or $\vec{p}_\mu$ under $\omega$ have a simple form?
\end{question}

An answer to Question~\ref{omega-question} would relate the two parts of the following problem to one another.

\begin{problem}
\label{h-and-e-expansion}
Describe the expansion of $A_{n,I,J}$ and $\vec{p}_\mu$ in the $h$-basis or $e$-basis of $\Lambda$.
\end{problem}

Although Propositions~\ref{skewing-recursion} and \ref{p-skewing-recursion} give recursive formulas for the image of an atomic function
under the skewing operators $h_j^\perp$ and $p_j^\perp$, their image under other skewing operators is unknown.

\begin{problem}
\label{skewing-problem}
Give a recursion for the image of an atomic symmetric function $A_{n,I,J}$ under the skewing operators $e_j^\perp$ and $s_\rho^\perp$, where $\rho \vdash j$.
\end{problem}

\chapter{The Path Murnaghan-Nakayama Rule}
\label{Path}

In order to perform asymptotic analysis of class functions on $\symm_n$, we study the Schur expansion of the 
atomic functions $A_{n,I,J}$ for $(I,J) \in \symm_{n,k}$.
By 
Theorem~\ref{atomic-support}, the $s$-expansion of $A_{n,I,J}$ is supported on partitions $\lambda \vdash n$
for which $\lambda_1 \geq n-k$. As we explained, pieces of this support result have appeared in the literature in various guises.
This section develops new symmetric function theory to understand what these coefficients actually are, starting with the factorization
$A_{n,I,J} = \vec{p}_{\mu} \cdot p_{\nu}$ of Proposition~\ref{atomic-factorization}.  
As the classical power sum factor $p_{\mu}$ of $A_{n,I,J}$ is well studied, this section 
analyzes the path power sum $\vec{p}_{\mu}$ factor in detail.

The main result of this chapter (Theorem~\ref{path-murnaghan-nakayama})
is a combinatorial expansion of
 $\vec{p}_{\mu}$ in the Schur basis using `monotonic ribbon tilings'.
 Theorem~\ref{path-murnaghan-nakayama} and its proof form the technical heart of this manuscript.

We use the following notation throughout this section.
Let $\mu = (\mu_1, \dots, \mu_r)$ be a composition with $r$ parts and let $S$ 
be one of the following objects:
\begin{itemize}
\item a word $S = i_1 \dots i_m$ in the alphabet $\{1, \dots, r\}$,
\item a subset $S = \{i_1, \dots, i_m \}$ of $\{1, \dots, r \}$, or
\item a cycle $S = (i_1, \dots, i_m)$ in a permutation $w \in \symm_r$.
\end{itemize}
In each of these situations, we write
\begin{equation}
\mu_S := \sum_{j \, = \, 1}^m \mu_{i_j}
\end{equation}
for the sum of the parts of $\mu$ indexed by entries in $S$.  For example, we have
\begin{equation*}
\mu_{325} = \mu_{\{3,2,5\}} = \mu_{(3,2,5)} = \mu_3 + \mu_2 + \mu_5
\end{equation*}
whenever $\mu$ has $\geq 5$ parts.

\section{Alternants and classical Murnaghan-Nakayama}
The classical Murnaghan-Nakayama Rule describes the coefficients $\chi^{\lambda}_{\mu}$ in the expansion
$p_{\mu} = \sum_{\lambda \vdash n} \chi^{\lambda}_{\mu} \cdot s_{\lambda}$ in terms of standard ribbon tableaux.
There are several ways to prove this result. The proof that extends best to the path power sum setting uses alternants;
we recall this argument here before  treating the more elaborate case of
 $\vec{p}_{\mu}$ in the following sections.

We restrict to a finite variable set $x_1, \dots, x_N$ and use the bialternant expression for the Schur polynomial 
\begin{equation}
s_{\lambda}(x_1, \dots, x_N) = \frac{ \varepsilon \cdot x^{\lambda + \delta} }{ \varepsilon \cdot x^{\delta}}
\end{equation}
where we assume that $\lambda = (\lambda_1, \dots, \lambda_N)$ has $\leq N$ parts and $\varepsilon \in \CC[\symm_N]$
is given by $\varepsilon = \sum_{w \in \symm_N} \sign(w) \cdot w$.
 The polynomial $\varepsilon \cdot x^{\lambda + \delta}$
appearing in the numerator is the {\em alternant}
\begin{equation}
a_{\lambda}(x_1, \dots, x_N) := \varepsilon \cdot x^{\lambda + \delta}.
\end{equation}
Recall from the introduction that a polynomial $f \in \CC[x_1, \dots, x_N]$ is {\em alternating}
if $w \cdot f = \sign(w) \cdot f$ for all $w \in \symm_N$.
The set $\{ a_{\lambda} \}$ 
of all alternants corresponding to partitions $\lambda$ with $\leq N$ parts is a basis of the vector space of alternating 
polynomials in $x_1, \dots, x_N$.

Multiplying through by the Vandermonde determinant $\varepsilon \cdot x^{\delta}$, the Murnaghan-Nakayama Rule
is equivalent to finding the expansion of 
\begin{equation}
\label{first-classical}
p_{\mu}(x_1, \dots, x_N) \times \varepsilon \cdot (x^{\delta}) =
\varepsilon \cdot \left( p_{\mu}(x_1, \dots, x_N) \times x^{\delta} \right)
\end{equation}
in the alternant basis $\{ a_{\lambda} \,:\, \ell(\lambda) \leq N \}$.
Equation~\eqref{first-classical} uses the fact that the action of $\CC[\symm_N]$ on $\CC[x_1, \dots, x_N]$ commutes with multiplication
by symmetric polynomials.  

Since the $p$-basis is multiplicative, the alternant expansion of \eqref{first-classical} can be understood inductively. In particular,
for $k \geq 1$ and a partition $\lambda = (\lambda_1, \dots, \lambda_N)$ we have
\begin{multline}
\label{inductive-step}
p_k(x_1, \dots, x_N) \cdot a_{\lambda}(x_1, \dots, x_N) = 
\varepsilon \cdot \left( p_k(x_1, \dots, x_N) \cdot x^{\lambda + \delta}  \right) \\ =
\sum_{i \, = \, 1}^N \varepsilon \cdot \left( x_1^{\lambda_1 + N-1} \cdots x_i^{k + \lambda_i + N-i} \cdots x_N^{\lambda_N} \right).
\end{multline}
Each term 
$\varepsilon \cdot \left( x_1^{\lambda_1 + N-1} \cdots x_i^{k + \lambda_i + N-i} \cdots x_N^{\lambda_N} \right)$ in this sum 
is an alternant, the negative of an alternant, or zero. 
Indeed, we have 
\begin{equation*}
\varepsilon \cdot \left( x_1^{\lambda_1 + N-1} \cdots x_i^{k + \lambda_i + N-i} \cdots x_N^{\lambda_N} \right) = 0
\end{equation*}
precisely when $k + \lambda_i + N-i$ coincides with one of the other exponents 
\begin{equation*}
(\lambda_1 + N - 1, \dots, \lambda_{i+1} + i, \lambda_{i-1} + i - 2, \dots, \lambda_N).
\end{equation*}
The following crucial observation describes 
$\varepsilon \cdot \left( x_1^{\lambda_1 + N-1} \cdots x_i^{k + \lambda_i + N-i} \cdots x_N^{\lambda_N} \right)$
combinatorially.

\begin{observation}
\label{ribbon-addition-lemma}
For $1 \leq i \leq N$, we have 
\begin{equation*}
\varepsilon \cdot \left( x_1^{\lambda_1 + N-1} \cdots x_i^{k + \lambda_i + N-i} \cdots x_N^{\lambda_N} \right) \neq 0
\end{equation*}
if and only if it is possible to add a size $k$ ribbon $\xi$ to the Young diagram of $\lambda$ such that the tail of $\xi$ is in row $i$
and $\lambda \cup \xi$ is the Young diagram of a partition. In this case, we have
\begin{equation}
\varepsilon \cdot \left( x_1^{\lambda_1 + N-1} \cdots x_i^{k + \lambda_i + N-i} \cdots x_N^{\lambda_N} \right) =
\sign(\xi) \cdot a_{\lambda \cup \xi}.
\end{equation}
\end{observation}

To see Observation~\ref{ribbon-addition-lemma} in action, take $k = 3, \lambda = (3,3,1),$ and $N = 6$. Starting with the length
$N$ sequence $(\lambda_1, \dots, \lambda_6) + (\delta_1, \dots, \delta_6) = (8,7,4,2,1,0)$, adding $k = 3$ in all possible positions yields
\begin{multline*}
(11,7,4,2,1,0), \quad (8,10,4,2,1,0), \quad (8,7,7,2,1,0), \\ \quad (8,7,4,5,1,0),  \quad (8,7,4,2,4,0), \quad \text{and} \quad
(8,7,4,2,1,3).
\end{multline*}
The third and fifth sequences above have repeated terms; the corresponding monomials in $\CC[x_1, \dots, x_6]$
are annihilated by $\varepsilon$. Sorting the remaining sequences into decreasing order and applying $\varepsilon$
introduces signs of $+1, -1, -1,$ and $+1$ into the first, second, fourth, and sixth sequences (respectively).
At the level of partitions, these are the four signs associated to the ways to add size $3$ ribbons to $\lambda = (3,3,1)$.
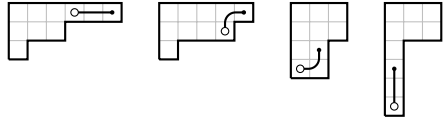
\begin{center}
\begin{tikzpicture}[scale = 0.25]

  \begin{scope}
    \clip (37,0) -| (38,4) -| (40,6) -| (37,0);
    \draw [color=black!25] (37,0) grid (40,6);
  \end{scope}

  \draw [thick] (37,0) -| (38,4) -| (40,6) -| (37,0);

  \draw [thick, rounded corners] (37.5,0.5) -- (37.5,2.5);
  \draw [color=black,fill=black,thick] (37.5,2.5) circle (.4ex);
  \node [draw, circle, fill = white, inner sep = 1pt] at (37.5,0.5) { };

  \begin{scope}
    \clip (32,2) -| (34,4) -| (35,6) -| (32,2);
    \draw [color=black!25] (32,2) grid (35,6);
  \end{scope}

  \draw [thick] (32,2) -| (34,4) -| (35,6) -| (32,2);

  \draw [thick, rounded corners] (32.5,2.5) -| (33.5,3.5);
  \draw [color=black,fill=black,thick] (33.5,3.5) circle (.4ex);
  \node [draw, circle, fill = white, inner sep = 1pt] at (32.5,2.5) { };

 \begin{scope}
    \clip (25,3) -| (26,4) -| (29,5) -| (30,6) -| (25,3);
    \draw [color=black!25] (25,3) grid (30,6);
  \end{scope}

  \draw [thick] (25,3) -| (26,4) -| (29,5) -| (30,6) -| (25,3);

   \draw [thick, rounded corners] (28.5,4.5) |- (29.5,5.5);
  \draw [color=black,fill=black,thick] (29.5,5.5) circle (.4ex);
  \node [draw, circle, fill = white, inner sep = 1pt] at (28.5,4.5) { };

 \begin{scope}
    \clip (17,3) -| (18,4) -| (20,5) -| (23,6) -| (17,3);
    \draw [color=black!25] (17,3) grid (23,6);
  \end{scope}

  \draw [thick]  (17,3) -| (18,4) -| (20,5) -| (23,6) -| (17,3);
  
   \draw [thick, rounded corners] (20.5,5.5) -- (22.5,5.5);
  \draw [color=black,fill=black,thick] (22.5,5.5) circle (.4ex);
  \node [draw, circle, fill = white, inner sep = 1pt] at (20.5,5.5) { };

\end{tikzpicture}
\end{center}
The tails of these added ribbons are in rows 1, 2, 4, and 6.  These agree with the relative values 
of the terms in $\lambda + \delta = (8,7,4,2,1,0)$ to which $k  =3$ was added. The illegal ribbon additions
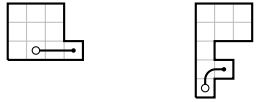
\begin{center}
\begin{tikzpicture}[scale = 0.25]

\begin{scope}
   \clip (0,0) -| (4,1) -| (3,3) -| (0,0);
   \draw [color=black!25] (0,0) grid (4,3);
\end{scope}

\draw [thick] (0,0) -| (4,1) -| (3,3) -| (0,0);

  \draw [thick, rounded corners] (1.5,0.5) -- (3.5,0.5);
  \draw [color=black,fill=black,thick] (3.5,0.5) circle (.4ex);
  \node [draw, circle, fill = white, inner sep = 1pt] at (1.5,0.5) { };

\begin{scope}
   \clip (10,-2) -|  (11,-1) -| (12,0) -| (11,1) -| (13,3) -| (10,-2);
   \draw [color=black!25] (13,3) grid (10,-2);
\end{scope}  

\draw [thick] (10,-2) -|  (11,-1) -| (12,0) -| (11,1) -| (13,3) -| (10,-2);

  \draw [thick, rounded corners] (10.5,-1.5) |- (11.5,-0.5);
  \draw [color=black,fill=black,thick] (11.5,-0.5) circle (.4ex);
  \node [draw, circle, fill = white, inner sep = 1pt] at (10.5,-1.5) { };

\end{tikzpicture}
\end{center}
obtained by adding a 3-ribbon with tail in rows 3 and 5 correspond precisely to the sequences 
$\lambda + \delta + (0,0,3,0,0,0)$ and $\lambda + \delta + (0,0,0,0,3,0)$ having repeated terms.

Observation~\ref{ribbon-addition-lemma} and induction on the number of parts of $\mu$ show that
multiplying $p_{\mu} = p_{\mu_1} p_{\mu_2} \cdots $ by $\varepsilon \cdot x^{\delta}$ corresponds to building up a standard ribbon
tableau of type $\mu$ ribbon by ribbon, starting with the empty tableau $\varnothing$. That is, we have
\begin{equation}
p_{\mu}(x_1, \dots, x_N) \times a_{\delta}(x_1, \dots, x_N) =
\sum_T \sign(T) \times  a_{\shape(T) + \delta}(x_1, \dots, x_N)
\end{equation}
where the sum is over all standard ribbon tableaux $T$ of type $\mu$. Dividing by the Vandermonde $a_{\delta}$ and taking the limit
as $N \rightarrow \infty$ yields the Murnaghan-Nakayama Rule as presented in Theorem~\ref{mn-rule}.

\section{Word arrays} 
We embark on finding the $s$-expansion of $\vec{p}_{\mu}$.  The main difficulty in adapting the proof
in the previous section to the path power sum setting is that the polynomials $\vec{p}_{\mu}$ are not multiplicative in $\mu$
so the induction argument breaks down. In spite of this, the $s$-expansion of the path power sum $\vec{p}_{\mu}$ will admit 
a reasonably compact expansion in terms of certain ribbon tilings.
In this section we set the stage for this result by giving a combinatorial description of the alternating
polynomial $\vec{p}_{\mu}(x_1, \dots, x_N) \times a_{\delta}(x_1, \dots, x_N)$.

Fix an integer $N \gg 0$ and a partition $\mu = (\mu_1, \dots, \mu_r)$ with $r$ parts.  A {\em $\mu$-word array} of length $N$
is  a sequence $\omega = ( w_1 \mid \cdots \mid w_N)$ of finite (possibly empty) words $w_1, \dots, w_N$
over the alphabet $\{1, \dots, r \}$.
The word array $\omega$ is {\em standard} if each letter $1, \dots, r$ appears exactly once among the words 
$w_1,  \dots, w_N$.
We will mainly be interested in standard word arrays, but will pass through more general word arrays 
in the course of our inductive arguments.

If $\omega = (w_1 \mid w_2 \mid \cdots \mid w_{N-1} \mid w_N)$ is a $\mu$-word array of length $N$, the {\em weight}
$\wt(\omega)$ is the sequence of integers
\begin{multline}
\wt(\omega) := ( \mu_{w_1} + N - 1, \mu_{w_2} + N - 2, \dots, \mu_{w_{N-1}} + 1, \mu_{w_N} ) \\ =
( \mu_{w_1}, \mu_{w_2}, \dots, \mu_{w_{N-1}}, \mu_{w_N}) + (N-1,N-2, \dots, 1,0)
\end{multline}
where the addition in the second line is componentwise.  Recall that $\mu_{w_i}$ is shorthand for the sum of the parts $\mu_j$
corresponding to the letters $j$ of the word $w_i$.
Observe that $\wt(\omega)$ depends on $\mu$; we leave this dependence implicit to reduce clutter.
Word arrays have the following connection to path power sums.

\begin{lemma}
\label{path-to-array}
Let $\mu = (\mu_1, \dots, \mu_r)$ be a partition. 
We have
\begin{equation}
\vec{p}_{\mu}(x_1, \dots, x_N) \times a_{\delta}(x_1, \dots, x_N) = \sum_{\omega}
 \varepsilon \cdot x^{\wt(\omega)}
\end{equation}
where the sum is over all standard $\mu$-word arrays $\omega = (w_1 \mid \cdots \mid w_N)$ of length $N$.
\end{lemma}

\begin{proof}
Proposition~\ref{path-to-classical} implies that 
\begin{equation}
\label{path-to-array-one}
\vec{p}_{\mu}(x_1, \dots, x_N) = \sum_{w \, \in \, \symm_r} \prod_{C \, \in \, w} p_{\mu_C}(x_1, \dots, x_N)
\end{equation}
where the notation $C \in w$ indicates that $C$ is a cycle of the permutation $w \in \symm_r$.
Multiplying both sides of Equation~\eqref{path-to-array-one} by $a_{\delta}(x_1, \dots, x_N)$ yields
\begin{multline}
\label{path-to-array-two}
\vec{p}_{\mu}(x_1, \dots, x_N) \times
a_{\delta}(x_1, \dots, x_N) = 
\vec{p}_{\mu}(x_1, \dots, x_N) \times
\varepsilon \cdot x^\delta \\ = 
\varepsilon \cdot \left(
\vec{p}_{\mu}(x_1, \dots, x_N) \times
x^\delta  \right) = 
\varepsilon \cdot \left( \sum_{w \in \symm_r} \prod_{C \in w} p_{\mu_C}(x_1, \dots, x_N) \times x^{\delta} \right)
\end{multline}
where the second equality uses the fact that 
$\vec{p}_\mu(x_1, \dots, x_N)$ is a symmetric polynomial,
and so commutes with $\varepsilon$.

Consider the monomial expansion of 
\begin{equation}
\label{path-to-array-three}
 \sum_{w \, \in \, \symm_r} 
 \prod_{C \, \in \, w} p_{\mu_C}(x_1, \dots, x_N) \times x^{\delta} = 
  \sum_{w \, \in \, \symm_r} \prod_{C \, \in \, w} (x_1^{\mu_C} + \cdots + x_N^{\mu_C}) \times x^{\delta}.
 \end{equation}
 A typical term in this expansion is obtained by first selecting a permutation $w \in \symm_r$, then assigning the cycles $C$ of $w$
 to the $N$ exponents of $x_1, \dots, x_N$ (where a given exponent can get no, one, or multiple cycles), and finally 
 multiplying by $x^{\delta}$. This given, applying the  antisymmetrizer $\varepsilon$ to both sides of 
 Equation~\eqref{path-to-array-three} completes the proof.
\end{proof}

Each term $\varepsilon \cdot x^{\wt(\omega)}$ on the RHS of Lemma~\ref{path-to-array}
 is 0 when $\wt(\omega)$ has repeated entries, 
 or $\pm$ some alternant 
 $a_{\lambda(\omega)}(x_1, \dots, x_N)$
 where $\lambda(\omega)$ is a partition.
Dividing through by the Vandermonde $a_{\delta}(x_1, \dots, x_N)$ and taking the limit as $N \rightarrow \infty$
gives an $s$-expansion of $\vec{p}_\mu$, but this expansion is far too large and
involves too much cancellation to be of much use.
Roughly speaking, the $s$-expansion
coming from Lemma~\ref{path-to-array} arises
by computing the $p$-expansion of $\vec{p}_\mu$,
and then applying classical Murnaghan-Nakayama
to each term.
In the following section we use a sign-reversing involution to cut this expansion down to size.

\section{A sign-reversing involution}  
The word array expansion of Lemma~\ref{path-to-array} involves a massive amount 
of cancellation. In this section we describe a sign-reversing involution $\iota$ on word arrays
which removes a large part (but not all) of this cancellation
and reduces Lemma~\ref{path-to-array} to a combinatorially useful sum.

The involution $\iota$ will act by swapping `unstable pairs' in word arrays.  Let $\omega = (w_1 \mid \cdots \mid w_N)$ 
be a $\mu$-word array. An {\em unstable pair} in $\omega$ consists of two prefix-suffix factorizations
\begin{equation}
w_i = u_i v_i \quad \text{and} \quad w_j = u_j v_j 
\end{equation}
of the words $w_i$ and $w_j$
at distinct positions $1 \leq i < j \leq N$ such that we have the equality
\begin{equation}
\label{unstable-equality}
\mu_{v_i} - i = \mu_{v_j} - j
\end{equation}
involving the suffixes of these factorizations.
A word array $\omega$ can admit unstable pairs at multiple pairs of positions $i < j$, and even at the same pair of positions
there could be more than one factorization
$w_i = u_i v_i$ and $w_j = u_j v_j$ of the words $w_i$ and $w_j$  witnessing an unstable pair. 
It is possible for either or both of the prefixes $u_i$ and $u_j$ in an unstable pair to be empty, but at least one of the suffixes
$v_i$ or $v_j$ must be nonempty.
We define the {\em score}
of an unstable pair to be the common value $\mu_{v_i} - i = \mu_{v_j} - j$ of \eqref{unstable-equality}.
An example should help clarify these definitions.

\begin{example}
\label{unstable-score-example}
Let $\mu = (\mu_1, \dots, \mu_9) = (3,3,3,3,2,2,2,1,1)$ and $N = 6$.  We have the standard $\mu$-word array
\begin{equation*}
\omega = (w_1 \mid w_2 \mid w_3 \mid w_4 \mid w_5 \mid w_6) = ( \varnothing \mid \varnothing \mid 3 \, 7 \mid 8 \, 1 \, 4 \mid \varnothing \mid 9 \, 5 \, 2 \, 6).
\end{equation*}
The ten unstable pairs in $\omega$ correspond to the prefix-suffix factorizations
\begin{equation*}
(w_1, w_3) = ( \varnothing , 3 \, \cdot \, 7),  \quad  
(w_1, w_4) = ( \varnothing, 8 \, 1 \, \cdot \, 4), \quad 
(w_1, w_6) = (\varnothing ,  9 \, \cdot \,  5 \, 2 \, 6),  
\end{equation*}
\begin{equation*}
(w_3, w_4) = ( 3  \, \cdot \,  7 , 8 \, 1 \, \cdot \, 4),  \quad 
(w_3, w_4) = ( \cdot \, 3 \,  7 , 8 \, \cdot \, 1 \, 4),  \quad 
(w_3, w_6) = (3 \cdot 7, 9 \, 5 \cdot  2 \, 6),
\end{equation*}
\begin{equation*}
(w_3, w_6) = ( \cdot \,  3 \, 7 , \cdot \, 9 \, 5 \, 2 \, 6),   \quad 
(w_4, w_6) = ( 8 \, 1 \, 4 \, \cdot \, , 9 \, 5 \, 2 \, \cdot \, 6), \quad 
(w_4, w_6) = (8 \, 1 \, \cdot \, 4, 9 \, 5 \, \cdot \, 2 \, 6),
\end{equation*}
\begin{equation*}
\text{ and } \quad 
(w_4, w_6) = (8 \, \cdot \, 1 \, 4, \, \cdot \, 9 \, 5 \, 2 \, 6).
\end{equation*}
In reading order from the top left, the scores of these unstable pairs are 
\begin{equation*}
-1, -1, -1, -1, 2, -1, 2, -4, -1, \text{ and } 2.
\end{equation*}
\end{example}

A $\mu$-word array $\omega = (w_1 \mid \cdots \mid w_N)$ is {\em stable} if $\omega$ has no unstable pairs.
Stability is sufficient to detect when  the monomial $x^{\wt(\omega)}$ is annihilated by the antisymmetrizer $\varepsilon$.

\begin{lemma}
\label{collision-instability}
Let $\mu$ be a partition and let
 $\omega = (w_1 \mid \cdots \mid w_N)$ be a $\mu$-word array with weight sequence
 $\wt(\omega) = (\wt(\omega)_1, \dots, \wt(\omega)_N)$.  If there exist $i < j$ such that $\wt(\omega)_i = \wt(\omega)_j$
 then $\omega$ is unstable.
\end{lemma}

\begin{proof}
This instability is witnessed by the prefix-suffix
 factorizations $w_i = u_i v_i$ and $w_j = u_j v_j$ where $u_i = u_j = \varnothing$ are empty prefixes.
\end{proof}

Lemma~\ref{collision-instability} says that any individual term in the expansion
\begin{equation*}
 \vec{p}_{\mu}(x_1, \dots, x_N) \times a_{\delta}(x_1, \dots, x_N) = \sum_{\text{$\omega$ standard}} \varepsilon \cdot x^{\wt(\omega)}
 \end{equation*}
of Lemma~\ref{path-to-array} which vanishes corresponds to an unstable
word array.  We will use stability to cancel more terms of this expansion pairwise.
Before doing so, we record
the following hereditary property of stable word arrays which will be crucial for our inductive arguments.

\begin{lemma}
\label{hereditary-stability}
Let $\mu$ be a partition and let $\omega = (w_1 \mid \cdots \mid w_N)$ be a stable $\mu$-word array.  
For each position $i$, let $w_i = u_i v_i$ be a prefix-suffix factorization of the word $w_i$.  
The  $\mu$-word array
$(v_1 \mid \cdots \mid v_N)$ formed by the suffices $v_i$ of the words $w_i$ is also stable.
\end{lemma}

\begin{proof}
The suffix  array $(v_1 \mid \cdots \mid v_N)$ is stable because
 the condition \eqref{unstable-equality} defining unstable pairs
depends only on the suffixes of the words $w_1, \dots, w_N$.
\end{proof}

We define a swapping operation on unstable pairs as follows.
Consider an unstable $\mu$-word array $\omega = ( w_1 \mid \cdots \mid w_N)$ at positions $1 \leq i < j \leq N$ 
with prefix-suffix factorizations 
$w_i = u_i v_i$ and $w_j = u_j v_j$. The {\em swap} $\sigma(\omega)$ of $\omega$ at this unstable pair is the 
$\mu$-word array
\begin{multline}
\omega = (w_1 \mid \cdots \mid w_i \mid \cdots \mid w_j \mid \cdots \mid w_N) \quad \leadsto \\ 
\sigma(\omega) := (w_1 \mid \cdots \mid w'_i \mid \cdots \mid w'_j \mid \cdots \mid w_N)
\end{multline}
where 
\begin{equation}
w'_i = u_j v_i \quad \text{and} \quad w'_j = u_i v_j. 
\end{equation}
That is, the swap $\sigma(\omega)$ is obtained from $\omega$ by interchanging the 
prefixes $u_i$ and $u_j$ while leaving the suffixes
$v_i$ and $v_j$ in the same positions.
It is possible for $\sigma$ to leave the 
word array $\omega$ unchanged:
if both prefixes $u_i$ and $u_j$ are empty, we have $\sigma(\omega) = \omega$.
The instability condition \eqref{unstable-equality} implies
\begin{equation}
\mu_{w_i} + N - i = \mu_{u_i} + \mu_{v_i} + N - i = \mu_{u_i} + \mu_{v_j} + N - j = \mu_{w'_j} + N - j
\end{equation}
and
\begin{equation}
\mu_{w_j} + N - j = \mu_{u_j} + \mu_{v_j} + N - j = \mu_{u_j} + \mu_{v_i} + N - i = \mu_{w'_i} + N - i
\end{equation}
We record this crucial property of swapping in the following observation.

\begin{observation}
\label{swap-reverse}
Let $\omega = (w_1 \mid \cdots \mid w_N)$ be a 
$\mu$-word array which is unstable
at positions $1 \leq i < j \leq n$ with respect to some prefix-suffix factorizations
$w_i = u_i v_i$ and $w_j = u_j v_j$
 of $w_i$ and $w_j$.  Let $\sigma(\omega)$ be the swapped $\mu$-word array.

The weight sequence $\wt(\sigma(\omega)) = (\wt(\sigma(\omega))_1, \dots, \wt(\sigma(\omega))_N)$ of the swapped word array $\sigma(\omega)$
 is obtained from the weight sequence
$\wt(\omega) = (\wt(\omega)_1, \dots, \wt(\omega)_N)$ of the original word array $\omega$ 
by interchanging the entries in positions $i$ and $j$.
\end{observation}

Observation~\ref{swap-reverse}
 implies that
\begin{equation}
\label{basic-swapping-cancellation}
\varepsilon \cdot x^{\wt(\omega)} + \varepsilon \cdot x^{\wt(\sigma(\omega))} = 0
\end{equation}
so that (for standard arrays) the terms corresponding to $\omega$ and $\sigma(\omega)$ appearing in 
Lemma~\ref{path-to-array} cancel.
If both prefixes in the swap $\sigma$ are empty so that $\sigma(\omega) = \omega$, the corresponding entries of $\wt(\omega) = \wt(\sigma(\omega))$ coincide
and
Equation~\eqref{basic-swapping-cancellation} reads $0 + 0 = 0$.

As Example~\ref{unstable-score-example} illustrates, an unstable word array $\omega$ can have multiple unstable pairs.
Our sign-reversing involution $\iota$ will act on unstable arrays by making a  swap
$\iota: \omega \mapsto \sigma(\omega)$ at a strategically chosen unstable pair of $\omega$.
Since the positions and prefix-suffix factorizations of unstable pairs in unstable arrays can be affected by swapping,
care must be taken to ensure the map $\iota$ is actually an involution. 
The score of an unstable pair was introduced to solve this problem.

In order to state the next result, we need one more definition. If $\mu = (\mu_1, \dots, \mu_r)$
is a partition with $r$ parts, the {\em content} of a $\mu$-word
array $\omega = (w_1 \mid \cdots \mid w_N)$ is the sequence $\content(\omega) = (\content(\omega)_1, \dots, \content(\omega)_r)$
where $\content(\omega)_i$ counts the total number of copies of the letter $i$ among the words $w_1, \dots, w_N$.
In particular, a standard word array is simply a word array of content $(1^r)$.

\begin{lemma}
\label{sign-reversing}
Let $\mu = (\mu_1, \dots, \mu_r)$ be a partition with $r$ parts and let $\gamma = (\gamma_1, \dots, \gamma_r)$ be a length
$r$ sequence of nonnegative integers.  We have
\begin{equation}
\sum_{\omega}  \varepsilon \cdot x^{\wt(\omega)} = \sum_{\text{$\omega$ {\em stable}}} \varepsilon \cdot x^{\wt(\omega)}
\end{equation}
where the sum on the LHS is over all $\mu$-word arrays of length $N$ and content $\gamma$ and the sum
on the RHS is over all {\em stable} $\mu$-word arrays of length $N$ and content $\gamma$.
\end{lemma}

In particular, Lemma~\ref{path-to-array} and \ref{sign-reversing} imply that $\vec{p}_{\mu}$ may be expressed
\begin{equation}
\label{sign-standard-case}
\vec{p}_{\mu}(x_1, \dots, x_N) \times 
a_{\delta}(x_1, \dots, x_N) = \sum_{\substack{ \text{$\omega$ standard} \\ \text{$\omega$ stable}}} 
\varepsilon \cdot x^{\wt(\omega)}.
\end{equation}
in terms of $\mu$-word arrays $\omega = (w_1 \mid \cdots \mid w_N)$ which are both standard and stable. Upon division by $a_{\delta}$,
Equation~\eqref{sign-standard-case} expresses the $s$-expansion of $\vec{p}_{\mu}$ in the most efficient way known to the authors.
In the next section we give a combinatorially amenable avatar of standard stable arrays.

\begin{proof}
Let $\WWW$ be the family of all $\mu$-word arrays $\omega = ( w_1 \mid \cdots \mid w_N)$ of length $N$ and content $\gamma$. Let 
 $\UUU \subseteq \WWW$ be the subfamily unstable arrays.
 We define a function $\iota: \WWW \rightarrow \WWW$ as follows.
 If $\omega \in \WWW - \UUU$ is stable, we set $\iota(\omega) := \omega$.
 The definition of $\iota(\omega)$ for $\omega \in \UUU$ unstable requires a fact about unstable pairs of minimal score.

Let $\omega = (w_1 \mid \cdots \mid w_N) \in \UUU$. There may be multiple factorization pairs $(w_i = u_i v_i, w_j = u_j v_j)$ 
and multiple pairs of positions which are unstable in $\omega$.  Choose one such unstable pair such that the score
$m := \mu_{v_i} - i = \mu_{v_j} - j$ is as small as possible. Let 
$\sigma(\omega) = (w_1 \mid \cdots \mid w'_i \mid \cdots \mid w'_j \mid \cdots \mid w_N)$ be the result of swapping this pair,
so that $w'_i = u_j v_i$ and $w'_j = u_i v_j$.

{\bf Claim:} {\em The swapped array $\sigma(\omega) = (w_1 \mid \cdots \mid w'_i \mid \cdots \mid w'_j \mid \cdots \mid w_N)$ 
has an unstable pair at positions $i < j$ of score $m$ given by the factorizations $w'_i = u_j v_i$ and $w'_j = u_i v_j$.
Furthermore, the array $\sigma(\omega)$ has no unstable pairs 
of score $< m$. Finally, the unstable pairs in the array $\sigma(\omega)$ of score $m$ occur at the same positions as the unstable pairs of score $m$ in 
$\sigma$.}

The first part of the claim is immediate from the defining condition \eqref{unstable-equality} of unstable pairs. For the second part,
suppose $\sigma(\omega)$ admitted an unstable pair of score $< m$. By the minimality of $m$, this unstable pair must involve 
one of the positions $i, j$ and some other position $k \neq i,j$.  Without loss of generality, assume that this unstable pair
involves the positions $i$ and $k$ and  corresponds to the
prefix-suffix factorizations $w'_i = u'_i v'_i$ and $w_k = u_k v_k$ of the words in at positions $i$ and $k$ of $\sigma(\omega)$.
By assumption the score of this pair is
\begin{equation}
\mu_{v'_i} - i < m = \mu_{v_i} - i 
\end{equation}
which implies that $v'_i$ is a (proper) suffix of $v_i$. If we write $v_i = t_i v'_i$ then $w_i = u_i t_i v'_i$ and
 $v'_i$ is also a suffix of $w_i$. The factorizations $w_i = (u_i t_i) v'_i$ and $w_k = u_k v_k$ witness an unstable pair in $\omega$
 of score $< m$, which contradicts the choice of $m$.  
 For the last part of the claim, if there is an unstable pair in $w$ of score $m$ at positions $i$ and $k$ with corresponding factorization
 $w_k = u_k v_k$, we have 
 \begin{equation}
 \mu_{v_k} - k = m = \mu_{v_i} - i = \mu_{v_j} - j,
 \end{equation}
 which implies that $w$ also has an unstable pair of score $m$ at positions $j$ and $k$.  Since the suffices $v_i, v_j,$ and $v_k$ remain unchanged
 by the operation $\omega \mapsto \sigma(\omega)$, we see that $\sigma(\omega)$ also has unstable pairs of score $m$ at positions $(i,k)$ and $(j,k)$.
 This completes the proof of the claim.
 
 With the claim in hand, the rest of the proof is straightforward. Given $\omega \in \UUU$, define 
 $\iota(\omega) := \sigma(\omega)$ where $\sigma$ swaps the unstable pair of minimal score $m$
 (if there are multiple unstable pairs of score $m$, let $\sigma$ swap the unstable pair of score $m$ at positions 
 $1 \leq i < j \leq N$ which are lexicographically maximal).   Since $\iota$ fixes any stable array, we have a function $\iota: \WWW \rightarrow \WWW$.
 The claim guarantees that $\iota(\iota(\omega)) = \omega$
 for all $\omega \in \UUU$ so that $\iota$ is an involution.
 The discussion following Lemma~\ref{hereditary-stability} shows that
 \begin{equation}
 \sum_{\omega \, \in \, \UUU} \varepsilon \cdot x^{\wt(\omega)} = 0
 \end{equation}
 which completes the proof.
\end{proof}

To show how the involution $\iota$ in the proof of Lemma~\ref{sign-reversing} works, we consider the unstable word
array of Example~\ref{unstable-score-example}.

\begin{example}
\label{involution-example}
Consider applying $\iota$ to the $\mu$-word array $\omega = (w_1 \mid w_2 \mid w_3 \mid w_4 \mid w_5 \mid w_6)$ 
of Example~\ref{unstable-score-example}.
The unstable pair of lowest score $m = -4$  is $(w_4, w_6) = ( 8 \, 1 \, 4 \, \cdot \, , 9 \, 5 \, 2 \, \cdot \, 6)$. Performing a swap at this unstable pair 
interchanges these two prefixes. The resulting array is 
\begin{equation*}
\iota(\omega) = ( \varnothing \mid \varnothing \mid 3 \, 7 \mid 9 \, 5 \, 2 \mid \varnothing \mid 8 \, 1 \, 4 \, 6).
\end{equation*}
The factorizations $(9 \, 5 \, 2 \, \cdot \, , 8 \, 1 \, 4 \, \cdot \, 6)$ yield an unstable pair in $\iota(\omega)$,
also of score $m = -4$. 
\end{example}

\section{Monotonic ribbon tilings}
Lemma~\ref{sign-reversing} is equivalent to our expansion of $\vec{p}_{\mu}$ in the $s$-basis, but we will need 
this expansion in a more combinatorially transparent form.
To achieve this, we consider ribbon tilings of Young diagrams that satisfy a monotone condition on their tails.

\begin{defn}
\label{monotonic-definition}
A {\em monotonic ribbon tiling} $T$ of a shape $\lambda$ is a disjoint union decomposition $\lambda = \xi^{(1)} \sqcup \cdots \sqcup \xi^{(r)}$
of the Young diagram of $\lambda$ into ribbons $\xi^{(1)}, \dots, \xi^{(r)}$ such that 
\begin{itemize}
\item
the tails of $\xi^{(1)}, \dots, \xi^{(r)}$ occupy distinct columns
$c_1 < \cdots < c_r$ of $\lambda$, and 
\item
each initial union $\xi^{(1)} \sqcup \cdots \sqcup \xi^{(i)}$ of ribbons is the Young diagram of a partition for $0 \leq i \leq r$.
\end{itemize}
\end{defn}

The term {\em monotonic} in Definition~\ref{monotonic-definition} refers to the fact that if $b_i$ is the row containing the tail of $\xi^{(i)}$, we have
$b_1 \geq \cdots \geq b_r$.
If $T$ is a monotonic ribbon tiling of shape $\lambda$ with ribbons $\xi^{(1)}, \dots, \xi^{(r)}$, there is precisely one standard ribbon tableau
of shape $\lambda$ with these ribbons.
The tiling shown on the 
left of Figure~\ref{fig:monotonic-and-not} 
is monotonic with tail depth sequence $(b_1 \geq \cdots \geq b_6) = (5,5,3,2,2,1)$; observe that there is a unique
standard ribbon tableau with this underlying tiling.
The tiling shown on the right of Figure~\ref{fig:monotonic-and-not}
 is not monotonic; there are multiple standard ribbon tableaux with this underlying tiling.

\begin{figure}
\begin{center}
\begin{tikzpicture}[scale = 0.5]

  \begin{scope}
    \clip (0,0) -| (2,2) -| (4,3) -| (9,4) -| (10,5) -| (0,0);
    \draw [color=black!25] (0,0) grid (10,5);
  \end{scope}

  \draw [thick] (0,0) -| (2,2) -| (4,3) -| (9,4) -| (10,5) -| (0,0);

  \draw [thick, rounded corners] (0.5,0.5) |- (3.5,4.5);
  \draw [color=black,fill=black,thick] (3.5,4.5) circle (.4ex);
  \node [draw, circle, fill = white, inner sep = 2pt] at (0.5,0.5) { };
  
  \draw [thick, rounded corners] (1.5,0.5) |- (2.5,3.5);
  \draw [color=black,fill=black,thick] (2.5,3.5) circle (.4ex);
  \node [draw, circle, fill = white, inner sep = 2pt] at (1.5,0.5) { };
  
  \draw [thick, rounded corners] (2.5,2.5) -| (3.5,3.5) -| (4.5,4.5) -- (7.5,4.5);
  \draw [color=black,fill=black,thick] (7.5,4.5) circle (.4ex);
  \node [draw, circle, fill = white, inner sep = 2pt] at (2.5,2.5) { };

    \node [draw, circle, fill = white, inner sep = 2pt] at (5.5,3.5) { };
    
  \draw [thick, rounded corners] (6.5,3.5) -| (8.5,4.5);
  \draw [color=black,fill=black,thick] (8.5,4.5) circle (.4ex);
  \node [draw, circle, fill = white, inner sep = 2pt] at (6.5,3.5) { };   
  
      \node [draw, circle, fill = white, inner sep = 2pt] at (9.5,4.5) { };

   \begin{scope}
    \clip (12,0) -| (14,2) -| (16,3) -| (21,4) -| (22,5) -| (12,0);
    \draw [color=black!25] (12,0) grid (22,5);
  \end{scope}

   \draw [thick] (12,0) -| (14,2) -| (16,3) -| (21,4) -| (22,5) -| (12,0);    
  
  \draw [thick, rounded corners] (12.5,2.5) |- (17.5,4.5);
  \draw [color=black,fill=black,thick] (17.5,4.5) circle (.4ex);
  \node [draw, circle, fill = white, inner sep = 2pt] at (12.5,2.5) { };   
  
  \draw [thick, rounded corners] (12.5,0.5) |- (13.5,1.5) -- (13.5,3.5);
  \draw [color=black,fill=black,thick] (13.5,3.5) circle (.4ex);
  \node [draw, circle, fill = white, inner sep = 2pt] at (12.5,0.5) { };   
  
  \draw [thick, rounded corners] (14.5,2.5) -| (15.5,3.5) -- (20.5,3.5);
  \draw [color=black,fill=black,thick] (20.5,3.5) circle (.4ex);
  \node [draw, circle, fill = white, inner sep = 2pt] at (14.5,2.5) { };   
  
  \draw [thick, rounded corners]  (18.5,4.5) -- (21.5,4.5);
  \draw [color=black,fill=black,thick] (21.5,4.5) circle (.4ex);
  \node [draw, circle, fill = white, inner sep = 2pt] at (18.5,4.5) { };

   \node [draw, circle, fill = white, inner sep = 2pt] at (13.5,0.5) { };

    \node [draw, circle, fill = white, inner sep = 2pt] at (14.5,3.5) { };

\end{tikzpicture}
\end{center}
\caption{A monotonic ribbon tiling (left) and a tiling by ribbons which is not monotonic (right).}
\label{fig:monotonic-and-not}
\end{figure}
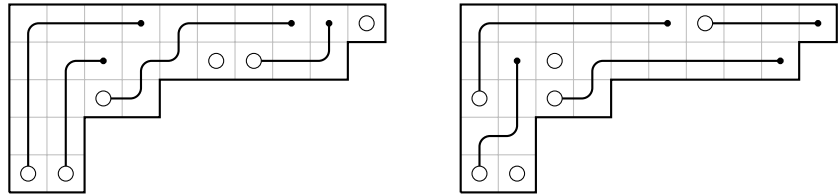

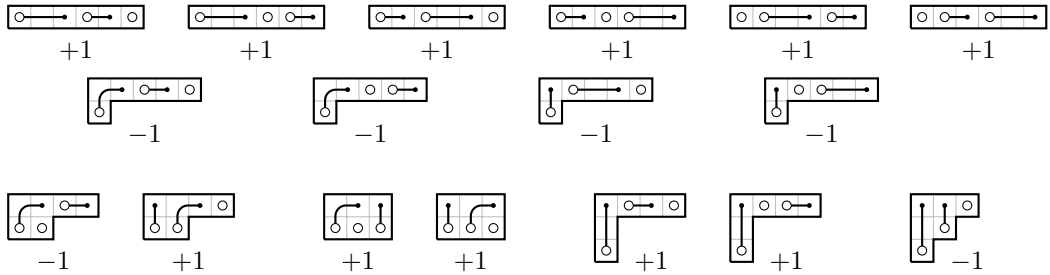
\begin{figure}
\begin{center}
\begin{tikzpicture}[scale = 0.3]

\begin{scope}
   \clip (0,0) -| (6,1) -| (0,0);
    \draw [color=black!25] (0,0) grid (6,1);
\end{scope}

\draw [thick] (0,0) -| (6,1) -| (0,0);

  \draw [thick, rounded corners]  (0.5,0.5) -- (2.5,0.5);
  \draw [color=black,fill=black,thick] (2.5,0.5) circle (.4ex);
  \node [draw, circle, fill = white, inner sep = 1.2pt] at (0.5,0.5) { };    
  
  \draw [thick, rounded corners]  (3.5,0.5) -- (4.5,0.5);
  \draw [color=black,fill=black,thick] (4.5,0.5) circle (.4ex);
  \node [draw, circle, fill = white, inner sep = 1.2pt] at (3.5,0.5) { };    
  
    \node [draw, circle, fill = white, inner sep = 1.2pt] at (5.5,0.5) { };    
    
    \node at (3,-1) {$+1$};

  \begin{scope}
   \clip (8,0) -| (14,1) -| (8,0);
    \draw [color=black!25] (8,0) grid (14,1);
\end{scope}

\draw [thick] (8,0) -| (14,1) -| (8,0);

  \draw [thick, rounded corners]  (8.5,0.5) -- (10.5,0.5);
  \draw [color=black,fill=black,thick] (10.5,0.5) circle (.4ex);
  \node [draw, circle, fill = white, inner sep = 1.2pt] at (8.5,0.5) { };    
  
  \draw [thick, rounded corners]  (12.5,0.5) -- (13.5,0.5);
  \draw [color=black,fill=black,thick] (13.5,0.5) circle (.4ex);
  \node [draw, circle, fill = white, inner sep = 1.2pt] at (12.5,0.5) { };    
  
    \node [draw, circle, fill = white, inner sep = 1.2pt] at (11.5,0.5) { };    
    
    \node at (11,-1) {$+1$};

  \begin{scope}
   \clip (16,0) -| (22,1) -| (16,0);
    \draw [color=black!25] (16,0) grid (22,1);
\end{scope}

\draw [thick] (16,0) -| (22,1) -| (16,0);

  \draw [thick, rounded corners]  (16.5,0.5) -- (17.5,0.5);
  \draw [color=black,fill=black,thick] (17.5,0.5) circle (.4ex);
  \node [draw, circle, fill = white, inner sep = 1.2pt] at (16.5,0.5) { };    
  
  \draw [thick, rounded corners]  (18.5,0.5) -- (20.5,0.5);
  \draw [color=black,fill=black,thick] (20.5,0.5) circle (.4ex);
  \node [draw, circle, fill = white, inner sep = 1.2pt] at (18.5,0.5) { };    
  
    \node [draw, circle, fill = white, inner sep = 1.2pt] at (21.5,0.5) { };    
    
    \node at (19,-1) {$+1$};

 \begin{scope}
   \clip (24,0) -| (30,1) -| (24,0);
    \draw [color=black!25] (24,0) grid (30,1);
\end{scope}

\draw [thick] (24,0) -| (30,1) -| (24,0);

  \draw [thick, rounded corners]  (24.5,0.5) -- (25.5,0.5);
  \draw [color=black,fill=black,thick] (25.5,0.5) circle (.4ex);
  \node [draw, circle, fill = white, inner sep = 1.2pt] at (24.5,0.5) { };    
  
  \draw [thick, rounded corners]  (27.5,0.5) -- (29.5,0.5);
  \draw [color=black,fill=black,thick] (29.5,0.5) circle (.4ex);
  \node [draw, circle, fill = white, inner sep = 1.2pt] at (27.5,0.5) { };    
  
    \node [draw, circle, fill = white, inner sep = 1.2pt] at (26.5,0.5) { };    
    
    \node at (27,-1) {$+1$};     
    
 \begin{scope}
   \clip (32,0) -| (38,1) -| (32,0);
    \draw [color=black!25] (32,0) grid (38,1);
\end{scope}

\draw [thick] (32,0) -| (38,1) -| (32,0);

  \draw [thick, rounded corners]  (33.5,0.5) -- (35.5,0.5);
  \draw [color=black,fill=black,thick] (35.5,0.5) circle (.4ex);
  \node [draw, circle, fill = white, inner sep = 1.2pt] at (33.5,0.5) { };    
  
  \draw [thick, rounded corners]  (36.5,0.5) -- (37.5,0.5);
  \draw [color=black,fill=black,thick] (37.5,0.5) circle (.4ex);
  \node [draw, circle, fill = white, inner sep = 1.2pt] at (36.5,0.5) { };    
  
    \node [draw, circle, fill = white, inner sep = 1.2pt] at (32.5,0.5) { };    
    
    \node at (35,-1) {$+1$};

 \begin{scope}
   \clip (40,0) -| (46,1) -| (40,0);
    \draw [color=black!25] (40,0) grid (46,1);
\end{scope}

\draw [thick] (40,0) -| (46,1) -| (40,0);

  \draw [thick, rounded corners]  (41.5,0.5) -- (42.5,0.5);
  \draw [color=black,fill=black,thick] (42.5,0.5) circle (.4ex);
  \node [draw, circle, fill = white, inner sep = 1.2pt] at (41.5,0.5) { };    
  
  \draw [thick, rounded corners]  (43.5,0.5) -- (45.5,0.5);
  \draw [color=black,fill=black,thick] (45.5,0.5) circle (.4ex);
  \node [draw, circle, fill = white, inner sep = 1.2pt] at (43.5,0.5) { };    
  
    \node [draw, circle, fill = white, inner sep = 1.2pt] at (40.5,0.5) { };    
    
    \node at (43,-1) {$+1$};

\end{tikzpicture}
\end{center}

\begin{center}
\begin{tikzpicture}[scale = 0.3]

\begin{scope}
   \clip (0,0) -| (1,1) -| (5,2) -| (0,0);
    \draw [color=black!25] (0,0) grid (5,2);
\end{scope}

\draw [thick] (0,0) -| (1,1) -| (5,2) -| (0,0);

\node at (2.5,-0.5) {$-1$};

  \draw [thick, rounded corners]  (0.5,0.5) |- (1.5,1.5);
  \draw [color=black,fill=black,thick] (1.5,1.5) circle (.4ex);
  \node [draw, circle, fill = white, inner sep = 1.2pt] at (0.5,0.5) { };

    \draw [thick, rounded corners]  (2.5,1.5) -- (3.5,1.5);
  \draw [color=black,fill=black,thick] (3.5,1.5) circle (.4ex);
  \node [draw, circle, fill = white, inner sep = 1.2pt] at (2.5,1.5) { };  
  
    \node [draw, circle, fill = white, inner sep = 1.2pt] at (4.5,1.5) { };

\begin{scope}
   \clip (10,0) -| (11,1) -| (15,2) -| (10,0);
    \draw [color=black!25] (10,0) grid (15,2);
\end{scope}

\draw [thick] (10,0) -| (11,1) -| (15,2) -| (10,0);

\node at (12.5,-0.5) {$-1$};

  \draw [thick, rounded corners]  (10.5,0.5) |- (11.5,1.5);
  \draw [color=black,fill=black,thick] (11.5,1.5) circle (.4ex);
  \node [draw, circle, fill = white, inner sep = 1.2pt] at (10.5,0.5) { };  
  
      \node [draw, circle, fill = white, inner sep = 1.2pt] at (12.5,1.5) { };

      \draw [thick, rounded corners]  (13.5,1.5) -- (14.5,1.5);
  \draw [color=black,fill=black,thick] (14.5,1.5) circle (.4ex);
  \node [draw, circle, fill = white, inner sep = 1.2pt] at (13.5,1.5) { };

\begin{scope}
   \clip (20,0) -| (21,1) -| (25,2) -| (20,0);
    \draw [color=black!25] (20,0) grid (25,2);
\end{scope}

\draw [thick] (20,0) -| (21,1) -| (25,2) -| (20,0);

\node at (22.5,-0.5) {$-1$};

  \draw [thick, rounded corners]  (20.5,0.5) -- (20.5,1.5);
  \draw [color=black,fill=black,thick] (20.5,1.5) circle (.4ex);
  \node [draw, circle, fill = white, inner sep = 1.2pt] at (20.5,0.5) { };

   \draw [thick, rounded corners]  (21.5,1.5) -- (23.5,1.5);
  \draw [color=black,fill=black,thick] (23.5,1.5) circle (.4ex);
  \node [draw, circle, fill = white, inner sep = 1.2pt] at (21.5,1.5) { };  
  
  \node [draw, circle, fill = white, inner sep = 1.2pt] at (24.5,1.5) { };

\begin{scope}
   \clip (30,0) -| (31,1) -| (35,2) -| (30,0);
    \draw [color=black!25] (30,0) grid (35,2);
\end{scope}

\draw [thick] (30,0) -| (31,1) -| (35,2) -| (30,0);

\node at (32.5,-0.5) {$-1$};

  \draw [thick, rounded corners]  (30.5,0.5) -- (30.5,1.5);
  \draw [color=black,fill=black,thick] (30.5,1.5) circle (.4ex);
  \node [draw, circle, fill = white, inner sep = 1.2pt] at (30.5,0.5) { };

   \draw [thick, rounded corners]  (32.5,1.5) -- (34.5,1.5);
  \draw [color=black,fill=black,thick] (34.5,1.5) circle (.4ex);
  \node [draw, circle, fill = white, inner sep = 1.2pt] at (32.5,1.5) { };  
  
    \node [draw, circle, fill = white, inner sep = 1.2pt] at (31.5,1.5) { };  

\end{tikzpicture}
\end{center}

\begin{center}
\begin{tikzpicture}[scale = 0.3]

\begin{scope}
\clip (0,0) -| (2,1) -| (4,2) -| (0,0);
\draw [color=black!25] (0,0) grid (4,2);
\end{scope}

\draw [thick] (0,0) -| (2,1) -| (4,2) -| (0,0);

\node at (2,-1) {$-1$};

  \draw [thick, rounded corners]  (0.5,0.5) |- (1.5,1.5);
  \draw [color=black,fill=black,thick] (1.5,1.5) circle (.4ex);
  \node [draw, circle, fill = white, inner sep = 1.2pt] at (0.5,0.5) { };

    \draw [thick, rounded corners]  (2.5,1.5) -- (3.5,1.5);
  \draw [color=black,fill=black,thick] (3.5,1.5) circle (.4ex);
  \node [draw, circle, fill = white, inner sep = 1.2pt] at (2.5,1.5) { };  
  
    \node [draw, circle, fill = white, inner sep = 1.2pt] at (1.5,0.5) { };

\begin{scope}
\clip (6,0) -| (8,1) -| (10,2) -| (6,0);
\draw [color=black!25] (6,0) grid (10,2);
\end{scope}

\draw [thick] (6,0) -| (8,1) -| (10,2) -| (6,0);

\node at (8,-1) {$+1$};

  \draw [thick, rounded corners]  (6.5,0.5) -- (6.5,1.5);
  \draw [color=black,fill=black,thick] (6.5,1.5) circle (.4ex);
  \node [draw, circle, fill = white, inner sep = 1.2pt] at (6.5,0.5) { };

  \draw [thick, rounded corners]  (7.5,0.5) |- (8.5,1.5);
  \draw [color=black,fill=black,thick] (8.5,1.5) circle (.4ex);
  \node [draw, circle, fill = white, inner sep = 1.2pt] at (7.5,0.5) { };

 \node [draw, circle, fill = white, inner sep = 1.2pt] at (9.5,1.5) { };

\begin{scope}
\clip (14,0) -| (17,2) -| (14,0);
\draw [color=black!25] (14,0) grid (17,2);
\end{scope}

\draw [thick] (14,0) -| (17,2) -| (14,0);

\node at (15.5,-1) {$+1$};

  \draw [thick, rounded corners]  (14.5,0.5) |- (15.5,1.5);
  \draw [color=black,fill=black,thick] (15.5,1.5) circle (.4ex);
  \node [draw, circle, fill = white, inner sep = 1.2pt] at (14.5,0.5) { };  
  
  \draw [thick, rounded corners]  (16.5,0.5) -- (16.5,1.5);
  \draw [color=black,fill=black,thick] (16.5,1.5) circle (.4ex);
  \node [draw, circle, fill = white, inner sep = 1.2pt] at (16.5,0.5) { };    
  
   \node [draw, circle, fill = white, inner sep = 1.2pt] at (15.5,0.5) { };

\begin{scope}
\clip (19,0) -| (22,2) -| (19,0);
\draw [color=black!25] (19,0) grid (22,2);
\end{scope}

\draw [thick] (19,0) -| (22,2) -| (19,0);

\node at (20.5,-1) {$+1$};

  \draw [thick, rounded corners]  (19.5,0.5) -- (19.5,1.5);
  \draw [color=black,fill=black,thick] (19.5,1.5) circle (.4ex);
  \node [draw, circle, fill = white, inner sep = 1.2pt] at (19.5,0.5) { };  
  
  \draw [thick, rounded corners]  (20.5,0.5) |- (21.5,1.5);
  \draw [color=black,fill=black,thick] (21.5,1.5) circle (.4ex);
  \node [draw, circle, fill = white, inner sep = 1.2pt] at (20.5,0.5) { };   
  
   \node [draw, circle, fill = white, inner sep = 1.2pt] at (21.5,0.5) { };

\begin{scope}
\clip (26,-1) -| (27,1) -| (30,2) -| (26,-1);
\draw [color=black!25] (26,-1) grid (30,2);
\end{scope}

\draw [thick] (26,-1) -| (27,1) -| (30,2) -| (26,-1);

\node at (28.5,-1) {$+1$};

  \draw [thick, rounded corners]  (26.5,-0.5) -- (26.5,1.5);
  \draw [color=black,fill=black,thick] (26.5,1.5) circle (.4ex);
  \node [draw, circle, fill = white, inner sep = 1.2pt] at (26.5,-0.5) { };

    \draw [thick, rounded corners]  (27.5,1.5) -- (28.5,1.5);
  \draw [color=black,fill=black,thick] (28.5,1.5) circle (.4ex);
  \node [draw, circle, fill = white, inner sep = 1.2pt] at (27.5,1.5) { };  
  
   \node [draw, circle, fill = white, inner sep = 1.2pt] at (29.5,1.5) { };

\begin{scope}
\clip (32,-1) -| (33,1) -| (36,2) -| (32,-1);
\draw [color=black!25] (32,-1) grid (36,2);
\end{scope}

\draw [thick]  (32,-1) -| (33,1) -| (36,2) -| (32,-1);

\node at (34.5,-1) {$+1$};

  \draw [thick, rounded corners]  (32.5,-0.5) -- (32.5,1.5);
  \draw [color=black,fill=black,thick] (32.5,1.5) circle (.4ex);
  \node [draw, circle, fill = white, inner sep = 1.2pt] at (32.5,-0.5) { };  
  
    \node [draw, circle, fill = white, inner sep = 1.2pt] at (33.5,1.5) { };  
    
  \draw [thick, rounded corners]  (34.5,1.5) -- (35.5,1.5);
  \draw [color=black,fill=black,thick] (35.5,1.5) circle (.4ex);
  \node [draw, circle, fill = white, inner sep = 1.2pt] at (34.5,1.5) { };

\begin{scope}
\clip (40,-1) -| (41,0) -| (42,1) -| (43,2) -| (40,-1);
\draw [color=black!25] (40,-1) grid (43,2);
\end{scope}

\draw [thick] (40,-1) -| (41,0) -| (42,1) -| (43,2) -| (40,-1);

\node at (42.5,-1) {$-1$};

  \draw [thick, rounded corners]  (40.5,-0.5) -- (40.5,1.5);
  \draw [color=black,fill=black,thick] (40.5,1.5) circle (.4ex);
  \node [draw, circle, fill = white, inner sep = 1.2pt] at (40.5,-0.5) { };  
  
  \draw [thick, rounded corners]  (41.5,0.5) -- (41.5,1.5);
  \draw [color=black,fill=black,thick] (41.5,1.5) circle (.4ex);
  \node [draw, circle, fill = white, inner sep = 1.2pt] at (41.5,0.5) { };

  \node [draw, circle, fill = white, inner sep = 1.2pt] at (42.5,1.5) { };

\end{tikzpicture}
\end{center}

\caption{The monotonic ribbon tilings used to calculate the Schur expansion of $\vec{p}_{321}$ using 
Theorem~\ref{path-murnaghan-nakayama}.}
\label{fig:tilings}
\end{figure}

The sequence $b_1 \geq \cdots \geq b_r$ of tail row depths, together with the sizes of the ribbons at each depth, characterizes a
monotonic ribbon tiling completely.  This is an efficient way to encode monotonic tilings which will be useful in our proofs.

\begin{observation}
\label{monotonic-characterization}
Let $(k_1, \dots, k_r)$ and $(b_1 \geq \cdots \geq b_r)$ be sequences of positive integers of the same length where the sequence of 
$b$'s is weakly decreasing. There is at most one monotonic ribbon tiling 
with ribbons whose tail rows are $b_1 \geq \cdots \geq b_r$ whose ribbon
sizes are $k_1, \dots, k_r$ reading left to right.
\end{observation}

For example, in the monotonic tiling 
on the left of Figure~\ref{fig:monotonic-and-not} we have $(k_1, \dots, k_r) = (8,5,8,1,4,1)$ and (as mentioned before) 
$(b_1, \dots, b_r) = (5,5,3,2,2,1)$. Not every pair of sequences as in Observation~\ref{monotonic-characterization} corresponds to a monotonic
ribbon tiling.
For example, the pair $(k_1, k_2, k_3) = (5,3,3)$ and $(b_1, b_2, b_3) = (3,2,1)$ does not yield a monotonic ribbon tiling. Indeed, the tiling
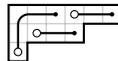
\begin{center}
\begin{tikzpicture}[scale = 0.25]

  \begin{scope}
    \clip (0,0) -| (1,1) -| (4,2) -| (6,3)  -| (0,0);
    \draw [color=black!25] (0,0) grid (6,3);
  \end{scope}

  \draw [thick] (0,0) -| (1,1) -| (4,2) -| (6,3)  -| (0,0);
  
  \draw [thick, rounded corners]  (0.5,0.5) |- (2.5,2.5);
  \draw [color=black,fill=black,thick] (2.5,2.5) circle (.4ex);
  \node [draw, circle, fill = white, inner sep = 1pt] at (0.5,0.5) { };

   \draw [thick, rounded corners]  (3.5,2.5) -- (5.5,2.5);
  \draw [color=black,fill=black,thick] (5.5,2.5) circle (.4ex);
  \node [draw, circle, fill = white, inner sep = 1pt] at (3.5,2.5) { };

     \draw [thick, rounded corners]  (1.5,1.5) -- (3.5,1.5);
  \draw [color=black,fill=black,thick] (3.5,1.5) circle (.4ex);
  \node [draw, circle, fill = white, inner sep = 1pt] at (1.5,1.5) { };    

\end{tikzpicture}
\end{center}
satisfies the conclusion of Observation~\ref{monotonic-characterization} but
is not monotonic because the union of the ribbons with the two leftmost tails
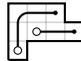
\begin{center}
\begin{tikzpicture}[scale = 0.25]

  \begin{scope}
    \clip (0,0) -| (1,1) -| (4,2) -| (3,3)  |- (2,3) -| (0,0);
    \draw [color=black!25] (0,0) grid (6,3);
  \end{scope}

  \draw [thick] (0,0) -| (1,1) -| (4,2) -| (3,3)  |- (2,3) -| (0,0);
  
  \draw [thick, rounded corners]  (0.5,0.5) |- (2.5,2.5);
  \draw [color=black,fill=black,thick] (2.5,2.5) circle (.4ex);
  \node [draw, circle, fill = white, inner sep = 1pt] at (0.5,0.5) { };

   \draw [thick, rounded corners]  (1.5,1.5) -- (3.5,1.5);
  \draw [color=black,fill=black,thick] (3.5,1.5) circle (.4ex);
  \node [draw, circle, fill = white, inner sep = 1pt] at (1.5,1.5) { };    

\end{tikzpicture}
\end{center}
is not the Young diagram of a partition.

\begin{remark}
\label{rmk:young-lattice}
    Every monotonic ribbon tiling of shape $\lambda$
    gives rise to 
    a chain in Young's Lattice $\YY$
    from $\varnothing$ to $\lambda$
    Indeed, if 
    $\lambda = \xi^{(1)} \sqcup \cdots \sqcup \xi^{(r)}$
    is a monotonic ribbon tiling
    we have the nested partitions 
    $\lambda^{(0)} \subseteq \lambda^{(1)} \subseteq \cdots \subseteq \lambda^{(r)} = \lambda$
    given by 
    $\lambda^{(i)} = \xi^{(1)} \sqcup \cdots \sqcup \xi^{(i)}$.
\end{remark}

The sign  of a monotonic ribbon tiling $T$ is (as with classical ribbon tableaux) the product of the signs of the ribbons in $T$.
The main result of this section is as follows.

\begin{theorem}
\label{path-murnaghan-nakayama}
{\em (Path Murnaghan-Nakayama Rule)}  Let $\mu = (\mu_1, \dots, \mu_r)$ be a partition of $n$.  We have 
\begin{equation}
\vec{p}_{\mu} = m(\mu)! \cdot \sum_T \sign(T) \cdot s_{\shape(T)}
\end{equation}
where the sum is over all monotonic ribbon tilings $T$ whose ribbon lengths $(k_1, \dots, k_r)$ form a rearrangement
of $(\mu_1, \dots, \mu_r)$. 
\end{theorem}

To prove Theorem~\ref{path-murnaghan-nakayama} we will find a bijection between standard stable word arrays and 
monotonic ribbon tilings. We defer this proof to the next section.
For now, we focus on how to apply Theorem~\ref{path-murnaghan-nakayama}.

Let $\mu = (3,2,1)$ so that $m(\mu)! = 1$.  We use Theorem~\ref{path-murnaghan-nakayama} to calculate the Schur expansion of 
$\vec{p}_{\mu}$.
The relevant monotonic tilings, together with their signs, are as shown in Figure~\ref{fig:tilings}.
Using Figure~\ref{fig:tilings} we calculate the Schur expansion
\begin{equation*}
\vec{p}_{321} = 6 s_6 - 4 s_{51} + 2 s_{411} + 2 s_{33} - s_{321}.
\end{equation*}

Figure~\ref{fig:tilings} illustrates the two  differences between applying the classical Murnaghan-Nakayama Rule to find the 
$s$-expansion of $p_{\mu}$ and using the Path Murnaghan-Nakayama Rule to calculate the $s$-expansion of 
$\vec{p}_{\mu}$.
\begin{enumerate}
\item  When calculating the $s$-expansion of $p_{\mu}$, one first fixes a definite order of the composition $\mu = (\mu_1, \dots, \mu_r)$
and always adds ribbons to the empty shape of sizes $\mu_1, \dots, \mu_r$ in that order. 
When calculating the $s$-expansion of $\vec{p}_{\mu}$, one can add ribbons of the sizes $\mu_1, \dots, \mu_r$ in any order.
The factor of $m(\mu)!$ in Theorem~\ref{path-murnaghan-nakayama} distinguishes between repeated parts of $\mu$.
\item  When calculating the $s$-expansion of $p_{\mu}$, one need only add the ribbons in a standard fashion.
When calculating the $s$-expansion of $\vec{p}_{\mu}$, Theorem~\ref{path-murnaghan-nakayama} imposes the stronger
monotonic condition.
\end{enumerate}
Condition (1) above expands the tilings under consideration relative to the classical case while Condition (2)
restricts them.

Figure~\ref{fig:tilings} also shows that the Path Murnaghan-Nakayama Rule (like the classical Murnaghan-Nakayama Rule)
is not cancellation-free. Indeed, the two monotonic tilings contributing to the coefficient of $s_{42}$ have opposite signs.

\begin{problem}
\label{cancellation-free-problem}
Find a cancellation-free formula for the Schur expansion of $\vec{p}_{\mu}$.
\end{problem}

The classical analogue of Problem~\ref{cancellation-free-problem} is equivalent to finding a cancellation-free formula
for the irreducible character evaluations $\chi^{\lambda}_{\mu}$ on the symmetric group. 
This is a famous open problem.
In light of this, Problem~\ref{cancellation-free-problem} could be quite difficult and it may be more realistic to ask for a more compact 
Schur expansion of $\vec{p}_{\mu}$.

On the bright side, Theorem~\ref{path-murnaghan-nakayama} represents a significant improvement over 
expanding $\vec{p}_{\mu}$ into the $p$-basis using Proposition~\ref{path-to-classical} and
then applying the Murnaghan-Nakayama
rule to expand in the Schur basis.
In our example $\mu = (3,2,1)$ Proposition~\ref{path-to-classical} yields
\begin{equation*}
\vec{p}_{321} = p_{321} + p_{51} + p_{42} + p_{33} + 2 p_6.
\end{equation*}
The number of standard ribbon tableaux in the Murnaghan-Nakayama expansion of $p_{\mu}$ depends on the order 
of $\mu = (\mu_1, \dots, \mu_r)$ in general.  In our case, we have at least 8 tableaux when computing $p_{321}$, at least
 5 tableaux when computing $p_{51}$, at least 8 tableaux when computing $p_{42}$, precisely 12 tableaux
when computing $p_{33}$, and precisely 6 tableaux when computing $p_6$. The number of tableaux considered using 
classical Murnaghan-Nakayama is therefore at least 
$8 + 5  + 8 + 12 + 6 = 39$ (with larger numbers possible if suboptimal choices for composition orders are made)
compared with the 17 monotonic tilings in Figure~\ref{fig:tilings}.
This discrepancy grows in larger examples.

In addition to computational advantages, Theorem~\ref{path-murnaghan-nakayama} also adds conceptual understanding 
to the $s$-expansion of $\vec{p}_{\mu}$ over the classical Murnaghan-Nakayama Rule.
In the case of $\vec{p}_{321}$, the $s$-expansion of $p_{321}$ contains terms such as $s_{3111}$ which are cancelled 
by the $s$-expansion of the remaining portion $ p_{51} + p_{42} + p_{33} + 2 p_6 $ of the $p$-expansion of $\vec{p}_{\mu}$.
More generally, we have the following support result.

\begin{corollary}
\label{path-length-bound}
Let $\mu \vdash n$ be a partition and consider the $s$-expansion 
$\vec{p}_{\mu} = \sum_{\lambda \vdash n} c_{\lambda} \cdot s_{\lambda}$.
If $\ell(\lambda) > \mu_1$ then $c_{\lambda} = 0$.
\end{corollary}

\begin{proof}
No monotonic ribbon tiling using ribbons of sizes $\mu = (\mu_1 \geq \mu_2 \geq \cdots )$ can occupy rows lower than $\mu_1$.
\end{proof}

\section{Proof of the Path Murnaghan-Nakayama Rule}
The goal of this section is to prove Theorem~\ref{path-murnaghan-nakayama}.
As mentioned earlier, our proof  
will be bijective.
For the rest of this section, fix a partition $\mu = (\mu_1, \dots, \mu_r)$ with $r$ positive parts
and an integer $N \gg 0$.

The factor $m(\mu)!$ in Theorem~\ref{path-murnaghan-nakayama} is a nuisance when defining a bijection,
so we decorate our ribbon tilings to get around it.  An {\em enhanced monotonic ribbon tiling} with ribbon sizes
$\mu$ is a pair $(T,w)$ where
\begin{itemize}
\item $T$ is a monotonic ribbon tiling whose ribbons (from west to east) have sizes given by some rearrangement
$(k_1, \dots, k_r)$ of $(\mu_1, \dots, \mu_r)$, and
\item   $w \in \symm_r$ is a permutation such that $\mu_{w(i)} = k_i$ for $1 \leq i \leq r$.
\end{itemize}
Theorem~\ref{path-murnaghan-nakayama} is equivalent to the assertion that
\begin{equation}
\label{our-formulation}
\vec{p}_{\mu} = \sum_{T \, \in \, \TTT} \sign(T) \cdot s_{\shape}(T)
\end{equation}
where 
\begin{equation}
\TTT := \{ \text{all enhanced monotonic ribbon tilings $(T,w)$ with ribbon sizes $\mu$} \}.
\end{equation}
In light of Lemmas~\ref{path-to-array} and \ref{sign-reversing},
Equation~\eqref{our-formulation} in the finite variable set $(x_1, \dots, x_N)$ is equivalent to the equality
\begin{equation}
\label{goal-formulation}
\sum_{\omega \, \in \, \OOO} \varepsilon \cdot x^{\wt(\omega)} 
= \sum_{T \, \in \, \TTT} \sign(T) \cdot a_{\shape(T)}(x_1, \dots, x_N).
\end{equation}
of alternating polynomials in $\CC[x_1, \dots, x_N]$ where
\begin{equation}
\OOO := \{ \text{all standard stable $\mu$-word arrays $\omega = ( u_1 \mid \cdots \mid u_N )$} \}.
\end{equation}
We prove Equation~\eqref{goal-formulation} by exhibiting a bijection between $\TTT$ and $\OOO$ that preserves 
weights and signs.

\begin{proof} {\em (of Theorem~\ref{path-murnaghan-nakayama})}
We define a function $\varphi: \TTT \rightarrow \OOO$ recursively as follows. Start with an enhanced
monotonic ribbon tiling $(T, w) \in \TTT$. 
As in Observation~\ref{monotonic-characterization}, the tiling $T$ corresponds to a pair 
$(k_1, \dots, k_r), (b_1 \geq \cdots \geq b_r)$ of integer sequences where the ribbons of $T$ are, read from west to east,
of sizes $k_1, \dots, k_r$ and have tails occupying rows $b_1, \dots, b_r$.  We initialize
\begin{equation*}
\varphi(T,w) := (\varnothing \mid \cdots \mid \varnothing) 
\end{equation*} 
to be $N$ copies of the empty word with weight sequence
$\wt(\varnothing \mid \cdots \mid \varnothing) = (N-1, N-2, \dots, 1, 0)$.
For $i = 1, 2, \dots, r$ we do the following.
\begin{enumerate}
\item  Suppose the word array $\varphi(T,w) = (u_1 \mid \cdots \mid u_N)$ contains precisely
one copy of each of the letters $w(1), w(2), \dots, w(i-1)$, and no other letters.  
Let $\wt( u_1 \mid \cdots \mid u_N) = (a_1, \dots, a_N)$ be the weight sequence of this array; the entries in this weight
sequence are distinct.
\item  Let $1 \leq j \leq n$ be such that $a_j$ is the $b_i^{th}$ largest element among $a_1, \dots, a_N$.
Prepend the letter $w(i)$ to the beginning of the word $u_j$ and change the $j^{th}$ entry of the weight sequence
from $a_j$ to $a_j + k_i$.
\end{enumerate}

We verify that $\varphi$ is  a well-defined function from $\TTT$ to $\OOO$. 
Ill-definedness of $\varphi$ could happen in one of two ways.
\begin{enumerate}
\item  The replacement $a_j \leadsto a_j + k_i$ in (2) could cause two entries in the weight sequence
\begin{equation*}
(a_1, \dots, a_{j-1}, a_j + k_i, a_{j+1}, \dots, a_N)
\end{equation*}
 to coincide. Depending on the value of $b_{i+1}$ in the next iteration, the choice of which word
 among $u_1, \dots, u_{j-1}, w(i) \cdot u_j, u_{j+1}, \dots, u_N$ to prepend with $w(i+1)$ could be ill defined.
\item  The replacement $u_j \leadsto w(i) \cdot u_j$ could destabilize the word array $\varphi(T,w)$.
\end{enumerate}
We need to show that neither of these things actually happen.
The monotonicity of $T$ guarantees that adding a ribbon of size $k_i$ whose tail occupies row $b_i$
results in the Young diagram of a partition, so
Observation~\ref{ribbon-addition-lemma} applies to show that (1) cannot happen. As for (2), since we prepend $w(i)$ 
the only new suffix that appears in the transformation of word arrays
\begin{equation*}
(u_1 \mid \cdots \mid u_{j-1} \mid  u_j \mid u_{j+1} \mid \cdots \mid u_N) \quad \leadsto \quad
(u_1 \mid \cdots \mid u_{j-1} \mid w(i) \cdot u_j \mid u_{j+1} \mid \cdots \mid u_N)
\end{equation*}
is the entire word $w(i) \cdot u_j$ in position $j$. Applying the defining condition \eqref{unstable-equality},
if this transformation destabilized $\varphi(T,w)$ there would be some position $s \neq j$ and a prefix-suffix factorization
$u_s = u'_s u''_s$ of $u_s$ such that 
\begin{equation}
\label{trouble-equation}
\mu_{u''_s} - s = \mu_{w(i) \cdot u_j} - j = \mu_{w(i)} + \mu_{u_j} - j = k_i + \mu_{u_j} - j.
\end{equation}
The prefix $u'_s$ is either empty or nonempty. 
\begin{itemize}
\item
If $u'_s$ is empty, then $u''_s = u_s$ and
Equation~\eqref{trouble-equation} reads 
\begin{equation*}
\mu_{u_s} - s = k_i + \mu_{u_j} - j. 
\end{equation*}
By Observation~\ref{ribbon-addition-lemma}
this means that adding a ribbon of size $k_i$ whose tail occupies row $b_i$ cannot result in a Young diagram,
contradicting the monotonicity of $T$.
\item  If $u'_s$ is not empty, the last letter of $u'_s$ is $w(p)$ for some $p < i$ and corresponds to a ribbon of $T$ located
strictly to the west of the $i^{th}$ west-to-east ribbon. Since $k_i > 0$, Equation~\eqref{trouble-equation}
implies $\mu_{u''_s} - s > \mu_{u_j} - j$ and the algorithm for computing $\varphi(T,w)$ implies that $b_p < b_i$.
Since $p < i$, by the 
remarks following Definition~\ref{monotonic-definition} this also contradicts the monotonicity of $T$.
\end{itemize}
In summary, the function $\varphi: \TTT \rightarrow \OOO$ is well-defined.

The map $\varphi: \TTT \rightarrow \OOO$ is best understood with an example.
Let $r = 6, \mu = (8,8,5,4,1,1),$ and let $T$ be the monotonic ribbon tiling on the left of Figure~\ref{fig:monotonic-and-not}
with ribbon tail depth sequence $(b_1, \dots, b_6) = (5,5,3,2,2,1)$.
There are $m(\mu)! = 2! \cdot 2! = 4$ permutations $w \in \symm_6$ such that $(T,w) \in \TTT$. Among these, we take 
$w = [2,3,1,5,4,6]$.  The stable word array $\varphi(T,w)$ is computed using the following table.

\begin{center}
\begin{tabular}{c | c | c | c | c}
$i$ & $w(i)$ & $b_i$ &  $\varphi(T,w)$ & $\wt(\varphi(T,w))$ \\ \hline
 &  &  & $(\varnothing \mid \varnothing \mid \varnothing \mid \varnothing \mid \varnothing \mid \varnothing)$ & $(5,4,3,2,1,0)$ \\
1 & 2 & 5 & $(\varnothing \mid \varnothing \mid \varnothing \mid \varnothing \mid 2 \mid \varnothing)$ & $(5,4,3,2,9,0)$ \\
2 & 3 & 5 & $(\varnothing \mid \varnothing \mid \varnothing \mid 3 \mid 2 \mid \varnothing)$ & $(5,4,3,7,9,0)$ \\
3 & 1 & 3 & $(1 \mid \varnothing \mid \varnothing \mid 3 \mid 2 \mid \varnothing)$ & $(13,4,3,7,9,0)$ \\
4 & 5 & 2 & $(1 \mid \varnothing \mid \varnothing \mid 3 \mid 5 \, \, 2 \mid \varnothing)$ & $(13,4,3,7,10,0)$ \\
5 & 4 & 2 & $(1 \mid \varnothing \mid \varnothing \mid 3 \mid 4 \, \, 5 \, \, 2 \mid \varnothing)$ & $(13,4,3,7,14,0)$ \\
6 & 6 & 1 & $(1 \mid \varnothing \mid \varnothing \mid 3 \mid 6 \, \,  4 \, \, 5 \, \, 2 \mid \varnothing)$ & $(13,4,3,7,15,0)$ \\
\end{tabular}
\end{center}
We conclude that $\varphi(T,w) = (1 \mid \varnothing \mid \varnothing \mid 3 \mid 6 \, \,  4 \, \, 5 \, \, 2 \mid \varnothing) \in \OOO$.

Our next task is to build the inverse map. To this end, we define a function $\psi: \OOO \rightarrow \TTT$ recursively.
Let $\omega  \in \OOO$. We initialize $\psi(\omega) := (T,w)$ by letting $T$ be the 
empty filling and letting $w$ be the empty word. 
We also initialize two `current word arrays' as $(u_1 \mid \cdots \mid u_N) := \omega$ and
$(v_1 \mid \cdots \mid v_N) := (\varnothing \mid \cdots \mid \varnothing)$.
We build up $T$ and $w$ step-by-step according to the following 
algorithm. Over the course of this algorithm, letters will move from the $u$-array to the $v$-array.
For $i =1, \dots, r$ we do the following.
\begin{enumerate}
\item  Suppose the first $i-1$ values $w(1), w(2), \dots, w(i-1)$ have been determined, these letters are precisely
those in the current $v$-array $(v_1 \mid \cdots \mid v_N)$ (which is stable) while the current $u$-array
$(u_1 \mid \cdots \mid u_N)$ consists of the remaining entries in $[r] - \{ w(1), \dots, w(i-1) \}$,
each occurring exactly once.  Suppose $T$ is a monotonic ribbon tiling consisting of ribbons of sizes 
$\mu_{w(1)}, \dots, \mu_{w(i-1)}$ when read from west to east.  Let $\wt(v_1 \mid \cdots \mid v_N) = (a_1, \dots, a_N)$
be the weight sequence of the $v$-array.
\item At least one of the words $u_1, \dots, u_N$ is not empty. Among the positions $1 \leq j \leq N$ for which $u_j$ is not 
empty, choose $j$ so that $a_j$ is minimal.  Suppose $a_j$ is the $b_i^{th}$ largest number 
in the list  $(a_1, \dots, a_N)$.
\item
If $\ell$ is the last letter of the word $u_j$, set $w(i) := \ell$, erase $\ell$ from the end of $u_j$ and prepend $\ell$ to the start
of $v_j$. Add a ribbon of size $\mu_{\ell}$ to $T$ whose tail occupies row $b_i$.
\end{enumerate}

As before, we verify that $\psi$ is a well-defined function $\OOO \rightarrow \TTT$.
This follows from two observations.
\begin{itemize}
\item
Lemma~\ref{hereditary-stability} says that the stability of an array is inherited by taking any suffices of the words in that 
array.  Since the words $v_1, \dots, v_N$ in the $v$-array $(v_1 \mid \cdots \mid v_N)$ 
are always suffices of those in the stable array $\omega$, the $v$-arrays are always stable and  by
Lemma~\ref{collision-instability} their weight sequences 
$\wt(v_1 \mid \cdots \mid v_N) = (a_1, \dots, a_N)$  have distinct entries. 
Consequently, the minimal value of $a_j$ in Step 2 is always uniquely achieved so that the position $j$ is uniquely determined 
going into
Step 3.
\item  Once we are in Step 3, we need to show that we can add a ribbon $\xi$ to $T$ of size $|\xi| = \mu_{\ell}$ whose tail
occupies row $b_i$ such that $T \sqcup \xi$ remains a monotonic ribbon tiling. First of all, 
Lemma~\ref{hereditary-stability}, Lemma~\ref{collision-instability}, and Observation~\ref{ribbon-addition-lemma}
guarantee that a ribbon $\xi$ of size $\mu_{\ell}$ whose tail occupies row $b_i$ can be added to $T$ 
such that $T \sqcup \xi$ is the Young diagram of a partition. For the remaining part of the monotonicity condition of 
Definition~\ref{monotonic-definition}, observe that the choice of $a_j$ in 
Step 2 forces $b_1 \geq b_2 \geq \cdots \geq b_i$ to be a weakly decreasing sequence.
\end{itemize}
In conclusion, the function $\psi: \OOO \rightarrow \TTT$ is well-defined.

As with $\varphi$, an example should help clarify the definition of $\psi: \OOO \rightarrow \TTT$. 
Suppose we are given $\omega = (1 \mid \varnothing \mid \varnothing \mid 3 \mid 6 \, \,  4 \, \, 5 \, \, 2 \mid \varnothing) \in \OOO$
where $\mu = (8,8,5,4,1,1)$ so that $r = 6$.
\begin{center}
\begin{tabular}{ c | c | c | c | c | c | c}
$i$ & $w(i)$ & $b_i$ & $\mu_{w(i)}$ & $(u_1 \mid \cdots \mid u_N)$  & $(v_1 \mid \cdots \mid v_N)$ & 
$\wt(v_1 \mid \cdots \mid v_N)$ \\ \hline
  &  &  &  & $ (1 \mid \varnothing \mid \varnothing \mid 3 \mid 6 \, \,  4 \, \, 5 \, \, 2 \mid \varnothing)$ &
  $(\varnothing \mid \varnothing \mid \varnothing \mid \varnothing \mid \varnothing \mid \varnothing)$  & 
  $(5,4,3,2,1,0)$ \\
 1 & 2 & 5 & 8 & $ (1 \mid \varnothing \mid \varnothing \mid 3 \mid 6 \, \,  4 \, \, 5  \mid \varnothing)$ &
  $(\varnothing \mid \varnothing \mid \varnothing \mid \varnothing \mid 2 \mid \varnothing)$  & 
  $(5,4,3,2,9,0)$ \\
   2 & 3 & 5 & 5 & $ (1 \mid \varnothing \mid \varnothing \mid \varnothing \mid 6 \, \,  4 \, \, 5  \mid \varnothing)$ &
  $(\varnothing \mid \varnothing \mid \varnothing \mid 3 \mid 2 \mid \varnothing)$  & 
  $(5,4,3,7,9,0)$ \\
   3 & 1 & 3 & 8 & $ (\varnothing \mid \varnothing \mid \varnothing \mid \varnothing \mid 6 \, \,  4 \, \, 5  \mid \varnothing)$ &
  $(1 \mid \varnothing \mid \varnothing \mid 3 \mid 2 \mid \varnothing)$  & 
  $(13,4,3,7,9,0)$ \\
   4 & 5 & 2 & 1 & $ (\varnothing \mid \varnothing \mid \varnothing \mid \varnothing \mid 6 \, \,  4  \mid \varnothing)$ &
  $(1 \mid \varnothing \mid \varnothing \mid 3 \mid 5 \, \,  2 \mid \varnothing)$  & 
  $(13,4,3,7,10,0)$ \\
   5 & 4 & 2 & 4 & $ (\varnothing \mid \varnothing \mid \varnothing \mid \varnothing \mid 6   \mid \varnothing)$ &
  $(1 \mid \varnothing \mid \varnothing \mid 3 \mid 4 \, \,  5 \, \,  2 \mid \varnothing)$  & 
  $(13,4,3,7,14,0)$ \\
    6  & 6 & 1 & 1 & $ (\varnothing \mid \varnothing \mid \varnothing \mid \varnothing \mid \varnothing   \mid \varnothing)$ &
  $(1 \mid \varnothing \mid \varnothing \mid 3 \mid 6 \, \,  4 \, \,  5 \, \,  2 \mid \varnothing)$  & 
  $(13,4,3,7,15,0)$ \\
\end{tabular}
\end{center}
We conclude that $\psi(\omega) = (T,w)$ where $T$ is the unique monotonic $\mu$-ribbon tiling with ribbons of sizes 
$8,5,8,1,4,1$ being added from west to east with tails occupying rows $5,5,3,2,2,1$ (respectively) and 
$w = [2,3,1,5,4,6] \in \symm_6$.
This is again the ribbon tableau from Figure~\ref{fig:monotonic-and-not}.

Our next task to verify that $\varphi: \TTT \rightarrow \OOO$ and $\psi: \OOO \rightarrow \TTT$ are mutually inverse bijections.
Since $\varphi$ and $\psi$ are defined recursively, an inductive proof is tempting. 
However, since both $\varphi$ and $\psi$ work `from the inside out' (constructing the words in $\varphi(T,w) = \omega$ from back to front
and building the enhanced tiling $\psi(\omega) = (T,w)$ from west to east) this approach is cumbersome.
A `from the outside in' version of these maps is difficult to formulate because stable sequences are closed under taking
suffixes (Lemma~\ref{hereditary-stability}) but not prefixes.
To get around this problem, we characterize the map $\psi: \OOO \rightarrow \TTT$ as follows.

{\bf Claim:}  {\em Let $\omega = (u_1 \mid \cdots \mid u_N) \in \OOO$.
Call a permutation $w = [w(1), \dots, w(r)] \in \symm_r$ a {\em $\omega$-suffix permutation} if 
\begin{quote}
for all $0 \leq  i \leq r$ the word array
$\omega^{(w,i)} = (u_1^{(i)} \mid \cdots \mid u_N^{(i)} )$ obtained by restricting $\omega$ to the letters $w(1), w(2), \dots, w(i)$
has the property that each $u_j^{(i)}$ is a suffix of $u_j$. 
\end{quote}
For any $\omega$-suffix permutation $w$, 
sorting each term in the list of the weight sequences
\begin{equation}
\wt(\omega^{(w,0)}), \wt(\omega^{(w,1)}), \dots, \wt(\omega^{(w,r)})
\end{equation}
yields a standard ribbon tableau
\begin{equation}
T^{(w)} = (\varnothing = \lambda^{(0)} \subset \lambda^{(1)} \subset \cdots \subset \lambda^{(r)})
\end{equation}
where $\lambda^{(i)} + (N-1, N-2, \dots, 1, 0)$ is the nonincreasing rearrangement of $\wt(\omega^{(w,i)})$.
There is precisely one $\omega$-suffix permutation $w$ for which the tiling underlying the standard ribbon tableau $T^{(w)}$ is monotonic;
we have $\psi(\omega) = (T^{(w)},w)$.}

To see why the claim is true,
Lemma~\ref{hereditary-stability} guarantees that for any $\omega$-suffix permutation $w$,
each of the word arrays $\omega^{(w,i)}$ is stable for $0 \leq i \leq r$. 
Lemma~\ref{collision-instability} implies that each of the weight sequences 
\begin{equation*}
\wt(\omega^{(w,0)}), \wt(\omega^{(w,1)}), \dots, \wt(\omega^{(w,r)})
\end{equation*} 
will have distinct terms, so Observation~\ref{ribbon-addition-lemma}
implies that $T^{(w)}$ is a valid standard ribbon tableau.
For uniqueness, observe that the monotonicity of the tiling underlying $T^{(w)}$ is equivalent to the sequence $b_1, \dots, b_r$
of ranks in the weight sequences $\wt(\omega^{(w,0)}), \wt(\omega^{(w,1)}), \dots, \wt(\omega^{(w,r)})$ of positions at which entries are added to 
\begin{equation*}
(\varnothing \mid \cdots \mid \varnothing) = \omega^{(w,0)}, \omega^{(w,1)}, \dots, \omega^{(w,r)} = \omega
\end{equation*} satisfying 
$b_1 \geq \cdots \geq b_r$. There is a unique $\omega$-suffix permutation $w$ which achieves this by at every stage adding an entry at greatest possible 
rank. 
If $\psi(\omega) = (T',w')$, by the definition of $\psi$ the permutation $w' \in \symm_r$ is an $\omega$-suffix permutation and 
$T' = T^{(w')}$ is monotonic; uniqueness gives the final assertion, proving the claim.

We use the claim to show that $\varphi$ and $\psi$ are mutually inverse.
If $(T,w) \in \TTT$ is an enhanced tiling with $\varphi(T,w) = \omega \in \OOO$, by the algorithm defining $\varphi(T,w)$ we know that 
$w$ is an $\omega$-suffix permutation such that $T$ is the tiling underlying $(T,w)$. Since $T$ is monotonic the claim forces 
\begin{equation}
\psi(\omega) = (\psi \circ \varphi)(T,w) = (T,w)
\end{equation}
so that $\psi \circ \varphi = \id_{\TTT}$.
Since $\varphi$ and $\psi$ are maps of finite sets, it is enough to show that $\varphi$ is surjective.
Indeed, given $\omega \in \OOO$, 
we apply the claim and let $w \in \symm_r$ be the unique $\omega$-suffix permutation such that $T^{(w)}$ is monotonic.
Then $(T^{(w)},w) \in \TTT$ is an enhanced tiling and the algorithm defining $\varphi$ forces $\varphi(T^{(w)},w) = \omega$.

The inverse bijections $\varphi: \TTT \rightarrow \OOO$ and $\psi: \OOO \rightarrow \TTT$ prove
Equation~\eqref{goal-formulation}, which we restate here:
\begin{equation*}
\sum_{\omega \, \in \, \OOO} \varepsilon \cdot x^{\wt(\omega)} 
= \sum_{T \, \in \, \TTT} \sign(T) \cdot a_{\shape(T)}(x_1, \dots, x_N).
\end{equation*}
Indeed, these maps provide a bijection between the terms $\varepsilon \cdot x^{\wt(\omega)}$ and
$\sign(T) \cdot a_{\shape(T)}(x_1, \dots, x_N)$ in the above sums.  The ribbon addition result in
 Observation~\ref{ribbon-addition-lemma} shows that if $\varphi(T,w) = \omega$ we have the equality of terms
 \begin{equation}
 \label{term-equality}
 \varepsilon \cdot x^{\wt(\omega)} = \sign(T) \cdot a_{\shape(T)}(x_1, \dots, x_N)  
 \end{equation}
 where the RHS of Equation~\eqref{term-equality} comes from $(T,w) \in \TTT$; i.e., the application of 
 $\varepsilon$ on the LHS of Equation~\eqref{term-equality} induces multiplication by $\sign(T) = \pm 1$ on the LHS.
 This proves Equation~\eqref{goal-formulation}.
 Dividing both sides of Equation~\eqref{goal-formulation} by the Vandermonde determinant $\varepsilon \cdot (x_1^{N-1} \cdots x_{N-1}^1 x_N^0)$
 and taking the limit as $N \rightarrow \infty$ completes the proof of 
 Theorem~\ref{path-murnaghan-nakayama}.
\end{proof}

\begin{remark}
    \label{rmk:multiple-chains-remark}
    The Claim
    in the above proof shows that 
    if $\omega = ( \omega_1 \mid \cdots \mid \omega_N)$
    is a stable word array, we get a chain
    \begin{equation*}
    \lambda^{(0,w)} \subset \lambda^{(1,w)} \subset
    \cdots \subset
    \lambda^{(r,w)}
    \end{equation*}
    in Young's Lattice $\YY$ 
    for every $\omega$-suffix permutation
    $w \in \symm_r$, where each successive difference
    $\lambda^{(i,w)}/\lambda^{(i-1,w)}$
    is a ribbon with $\mu_{w(i)}$ boxes.
    Stable word arrays can therefore yield
    multiple chains in Young's Lattice.
    Monotonic ribbon tilings are obtained by selecting
    the unique $\omega$-suffix permutation $w$ 
    which at every stage
    adds a ribbon whose tail is in the lowest possible
    row.
    From this viewpoint, the Path Murnaghan-Nakayama
    formula is obtained by selecting a special 
    $\omega$-suffix permutation $w$. 
    We could obtain another combinatorial
    rule for the $s$-expansion
    of the path power sum $\vec{p}_\mu$ by choosing
    the $\omega$-suffix permutation $w$
    which adds a ribbon to the {\em highest}
    possible row at each stage.

    For the example in the above proof where
    $\omega = (1 \mid \varnothing \mid \varnothing \mid 3 \mid 6 \, 4 \, 5 \, 2 \mid \varnothing)$
    and $\mu = (8,8,5,4,1,1)$ (so that $r = 6$),
    a permutation $w \in \symm_6$ is an $\omega$-suffix
    permutation if and only if 
    \begin{equation*}
        w^{-1}(2) < w^{-1}(5) < w^{-1}(4) < w^{-1}(6),
    \end{equation*}
    i.e. in the one-line notation 
    $w = [w(1), \dots, w(6)]$ the numbers $2,5,4,$ and 
    $6$ appear in that order from left to right.
    There are 30 such permutations in $\symm_6$,
    each of which gives a chain in $\YY$
    where successive differences are ribbon shapes.
    The corresponding 30 ribbon tableaux are shown
    below. Each has the same sign,
    $+1$. The unique monotonic ribbon tiling
    is shown in red.
\end{remark}

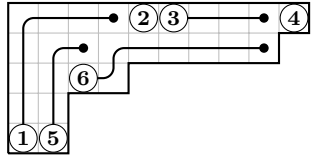
\begin{center}
\begin{tikzpicture}[scale = 0.4]

  \begin{scope}
    \clip (0,0) -| (2,2) -| (4,3) -| (9,4) -| (10,5) -| (0,0);
    \draw [color=black!25] (0,0) grid (10,5);
  \end{scope}

  \draw [thick] (0,0) -| (2,2) -| (4,3) -| (9,4) -| (10,5) -| (0,0);

  \draw [thick, rounded corners]  (0.5,0.5) |- (3.5,4.5);
  \draw [color=black,fill=black,thick] (3.5,4.5) circle (0.8ex);
  \node [draw, circle, fill = white, inner sep = 1.2pt] at (0.5,0.5) {{{\footnotesize \bf 1 }}};

  \draw [thick, rounded corners]  (1.5,0.5) |- (2.5,3.5);
  \draw [color=black,fill=black,thick] (2.5,3.5) circle (0.8ex);
  \node [draw, circle, fill = white, inner sep = 1.2pt] at (1.5,0.5) {{{\footnotesize \bf 5 }}};

  \draw [thick, rounded corners]  (2.5,2.5) -| (3.5,3.5) -- (8.5,3.5);
  \draw [color=black,fill=black,thick] (8.5,3.5) circle (0.8ex);
  \node [draw, circle, fill = white, inner sep = 1.2pt] at (2.5,2.5) {{{\footnotesize \bf 6 }}};

  \node [draw, circle, fill = white, inner sep = 1.2pt] at (4.5,4.5) {{{\footnotesize \bf 2 }}};

  \draw [thick, rounded corners]  (5.5,4.5) -- (8.5,4.5);
  \draw [color=black,fill=black,thick] (8.5,4.5) circle (0.8ex);
  \node [draw, circle, fill = white, inner sep = 1.2pt] at (5.5,4.5) {{{\footnotesize \bf 3 }}};

  \node [draw, circle, fill = white, inner sep = 1.2pt] at (9.5,4.5) {{{\footnotesize \bf 4 }}};

  \node at (4.5,-2) {$w = [2,5,4,6,3,1]$};

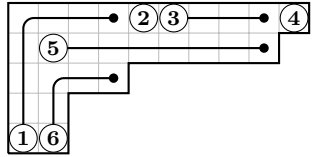
\begin{scope}
    \clip (12,0) -| (14,2) -| (16,3) -| (21,4) -| (22,5) -| (12,0);
    \draw [color=black!25] (12,0) grid (22,5);
\end{scope}

  \draw [thick] (12,0) -| (14,2) -| (16,3) -| (21,4) -| (22,5) -| (12,0);

  \draw [thick, rounded corners]  (12.5,0.5) |- (15.5,4.5);
  \draw [color=black,fill=black,thick] (15.5,4.5) circle (0.8ex);
  \node [draw, circle, fill = white, inner sep = 1.2pt] at (12.5,0.5) {{{\footnotesize \bf 1 }}};  

  \node [draw, circle, fill = white, inner sep = 1.2pt] at (16.5,4.5) {{{\footnotesize \bf 2 }}};

  \draw [thick, rounded corners]  (17.5,4.5) -- (20.5,4.5);
  \draw [color=black,fill=black,thick] (20.5,4.5) circle (0.8ex);
  \node [draw, circle, fill = white, inner sep = 1.2pt] at (17.5,4.5) {{{\footnotesize \bf 3 }}};  

  \node [draw, circle, fill = white, inner sep = 1.2pt] at (21.5,4.5) {{{\footnotesize \bf 4 }}}; 

  \draw [thick, rounded corners]  (13.5,3.5) -- (20.5,3.5);
  \draw [color=black,fill=black,thick] (20.5,3.5) circle (0.8ex);
  \node [draw, circle, fill = white, inner sep = 1.2pt] at (13.5,3.5) {{{\footnotesize \bf 5 }}};  

  \draw [thick, rounded corners]  (13.5,0.5) |- (15.5,2.5);
  \draw [color=black,fill=black,thick] (15.5,2.5) circle (0.8ex);
  \node [draw, circle, fill = white, inner sep = 1.2pt] at (13.5,0.5) {{{\footnotesize \bf 6 }}};    

  \node at (16.5,-2) {$w = [2,5,4,6,1,3]$};

 \begin{scope}
    \clip (24,0) -| (26,2) -| (28,3) -| (33,4) -| (34,5) -| (24,0);
    \draw [color=black!25] (24,0) grid (34,5);
\end{scope}

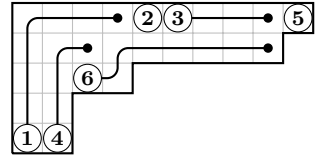
  \draw [thick] (24,0) -| (26,2) -| (28,3) -| (33,4) -| (34,5) -| (24,0);

  \draw [thick, rounded corners]  (24.5,0.5) |- (27.5,4.5);
  \draw [color=black,fill=black,thick] (27.5,4.5) circle (0.8ex);
  \node [draw, circle, fill = white, inner sep = 1.2pt] at (24.5,0.5) {{{\footnotesize \bf 1 }}};  

  \node [draw, circle, fill = white, inner sep = 1.2pt] at (28.5,4.5) {{{\footnotesize \bf 2 }}};

  \draw [thick, rounded corners]  (29.5,4.5) -- (32.5,4.5);
  \draw [color=black,fill=black,thick] (32.5,4.5) circle (0.8ex);
  \node [draw, circle, fill = white, inner sep = 1.2pt] at (29.5,4.5) {{{\footnotesize \bf 3 }}}; 

  \draw [thick, rounded corners]  (25.5,0.5) |- (26.5,3.5);
  \draw [color=black,fill=black,thick] (26.5,3.5) circle (0.8ex);
  \node [draw, circle, fill = white, inner sep = 1.2pt] at (25.5,0.5) {{{\footnotesize \bf 4 }}};

  \node [draw, circle, fill = white, inner sep = 1.2pt] at (33.5,4.5) {{{\footnotesize \bf 5 }}};

  \draw [thick, rounded corners]  (26.5,2.5) -| (27.5,3.5) -- (32.5,3.5);
  \draw [color=black,fill=black,thick] (32.5,3.5) circle (0.8ex);
  \node [draw, circle, fill = white, inner sep = 1.2pt] at (26.5,2.5) {{{\footnotesize \bf 6 }}};

  \node at (28.5,-2) {$w = [2,5,4,3,6,1]$};
\end{tikzpicture}
\end{center}
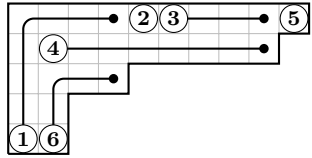
\begin{center}
\begin{tikzpicture}[scale = 0.4]

  \begin{scope}
    \clip (0,0) -| (2,2) -| (4,3) -| (9,4) -| (10,5) -| (0,0);
    \draw [color=black!25] (0,0) grid (10,5);
  \end{scope}

  \draw [thick] (0,0) -| (2,2) -| (4,3) -| (9,4) -| (10,5) -| (0,0);

  \draw [thick, rounded corners]  (0.5,0.5) |- (3.5,4.5);
  \draw [color=black,fill=black,thick] (3.5,4.5) circle (0.8ex);
  \node [draw, circle, fill = white, inner sep = 1.2pt] at (0.5,0.5) {{{\footnotesize \bf 1 }}}; 

  \node [draw, circle, fill = white, inner sep = 1.2pt] at (4.5,4.5) {{{\footnotesize \bf 2 }}}; 

  \draw [thick, rounded corners]  (5.5,4.5) -- (8.5,4.5);
  \draw [color=black,fill=black,thick] (8.5,4.5) circle (0.8ex);
  \node [draw, circle, fill = white, inner sep = 1.2pt] at (5.5,4.5) {{{\footnotesize \bf 3 }}};

  \draw [thick, rounded corners]  (1.5,3.5) -- (8.5,3.5);
  \draw [color=black,fill=black,thick] (8.5,3.5) circle (0.8ex);
  \node [draw, circle, fill = white, inner sep = 1.2pt] at (1.5,3.5) {{{\footnotesize \bf 4 }}};

  \node [draw, circle, fill = white, inner sep = 1.2pt] at (9.5,4.5) {{{\footnotesize \bf 5 }}};  

  \draw [thick, rounded corners]  (1.5,0.5) |- (3.5,2.5);
  \draw [color=black,fill=black,thick] (3.5,2.5) circle (0.8ex);
  \node [draw, circle, fill = white, inner sep = 1.2pt] at (1.5,0.5) {{{\footnotesize \bf 6 }}};  

  \node at (4.5,-2) {$w = [2,5,4,1,6,3]$};

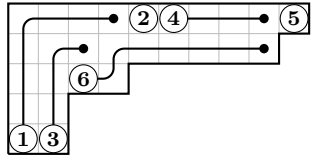
\begin{scope}
    \clip (12,0) -| (14,2) -| (16,3) -| (21,4) -| (22,5) -| (12,0);
    \draw [color=black!25] (12,0) grid (22,5);
\end{scope}

  \draw [thick] (12,0) -| (14,2) -| (16,3) -| (21,4) -| (22,5) -| (12,0);

  \draw [thick, rounded corners]  (12.5,0.5) |- (15.5,4.5);
  \draw [color=black,fill=black,thick] (15.5,4.5) circle (0.8ex);
  \node [draw, circle, fill = white, inner sep = 1.2pt] at (12.5,0.5) {{{\footnotesize \bf 1 }}}; 

  \node [draw, circle, fill = white, inner sep = 1.2pt] at (16.5,4.5) {{{\footnotesize \bf 2 }}};

 \draw [thick, rounded corners]  (13.5,0.5) |- (14.5,3.5);
  \draw [color=black,fill=black,thick] (14.5,3.5) circle (0.8ex);
  \node [draw, circle, fill = white, inner sep = 1.2pt] at (13.5,0.5) {{{\footnotesize \bf 3 }}};   

 \draw [thick, rounded corners]  (17.5,4.5) -- (20.5,4.5);
  \draw [color=black,fill=black,thick] (20.5,4.5) circle (0.8ex);
  \node [draw, circle, fill = white, inner sep = 1.2pt] at (17.5,4.5) {{{\footnotesize \bf 4 }}}; 

  \node [draw, circle, fill = white, inner sep = 1.2pt] at (21.5,4.5) {{{\footnotesize \bf 5 }}}; 

 \draw [thick, rounded corners]  (14.5,2.5) -| (15.5,3.5) -- (20.5,3.5);
  \draw [color=black,fill=black,thick] (20.5,3.5) circle (0.8ex);
  \node [draw, circle, fill = white, inner sep = 1.2pt] at (14.5,2.5) {{{\footnotesize \bf 6 }}};   

  \node at (16.5,-2) {$w = [2,5,3,4,6,1]$};

 \begin{scope}
    \clip (24,0) -| (26,2) -| (28,3) -| (33,4) -| (34,5) -| (24,0);
    \draw [color=black!25] (24,0) grid (34,5);
\end{scope}

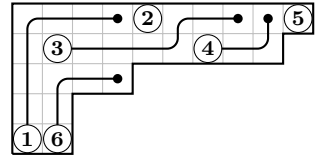
  \draw [thick] (24,0) -| (26,2) -| (28,3) -| (33,4) -| (34,5) -| (24,0);

  \draw [thick, rounded corners]  (24.5,0.5) |- (27.5,4.5);
  \draw [color=black,fill=black,thick] (27.5,4.5) circle (0.8ex);
  \node [draw, circle, fill = white, inner sep = 1.2pt] at (24.5,0.5) {{{\footnotesize \bf 1 }}}; 

  \node [draw, circle, fill = white, inner sep = 1.2pt] at (28.5,4.5) {{{\footnotesize \bf 2 }}};

  \draw [thick, rounded corners]  (25.5,3.5) -|(29.5,4.5) -- (31.5,4.5);
  \draw [color=black,fill=black,thick] (31.5,4.5) circle (0.8ex);
  \node [draw, circle, fill = white, inner sep = 1.2pt] at (25.5,3.5) {{{\footnotesize \bf 3 }}};

  \draw [thick, rounded corners]  (30.5,3.5) -|(32.5,4.5);
  \draw [color=black,fill=black,thick] (32.5,4.5) circle (0.8ex);
  \node [draw, circle, fill = white, inner sep = 1.2pt] at (30.5,3.5) {{{\footnotesize \bf 4 }}};

  \node [draw, circle, fill = white, inner sep = 1.2pt] at (33.5,4.5) {{{\footnotesize \bf 5 }}};

  \draw [thick, rounded corners]  (25.5,0.5) |- (27.5,2.5);
  \draw [color=black,fill=black,thick] (27.5,2.5) circle (0.8ex);
  \node [draw, circle, fill = white, inner sep = 1.2pt] at (25.5,0.5) {{{\footnotesize \bf 6 }}};  

  \node at (28.5,-2) {$w = [2,5,1,4,6,3]$};
\end{tikzpicture}
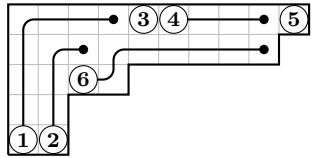
\begin{center}
\begin{tikzpicture}[scale = 0.4]

  \begin{scope}
    \clip (0,0) -| (2,2) -| (4,3) -| (9,4) -| (10,5) -| (0,0);
    \draw [color=black!25] (0,0) grid (10,5);
  \end{scope}

  \draw [thick, rounded corners]  (0.5,0.5) |- (3.5,4.5);
  \draw [color=black,fill=black,thick] (3.5,4.5) circle (0.8ex);
  \node [draw, circle, fill = white, inner sep = 1.2pt] at (0.5,0.5) {{{\footnotesize \bf 1 }}};

  \draw [thick, rounded corners]  (1.5,0.5) |- (2.5,3.5);
  \draw [color=black,fill=black,thick] (2.5,3.5) circle (0.8ex);
  \node [draw, circle, fill = white, inner sep = 1.2pt] at (1.5,0.5) {{{\footnotesize \bf 2 }}};

  \node [draw, circle, fill = white, inner sep = 1.2pt] at (4.5,4.5) {{{\footnotesize \bf 3 }}}; 

  \draw [thick, rounded corners]  (5.5,4.5) -- (8.5,4.5);
  \draw [color=black,fill=black,thick] (8.5,4.5) circle (0.8ex);
  \node [draw, circle, fill = white, inner sep = 1.2pt] at (5.5,4.5) {{{\footnotesize \bf 4 }}}; 

  \node [draw, circle, fill = white, inner sep = 1.2pt] at (9.5,4.5) {{{\footnotesize \bf 5 }}}; 

  \draw [thick, rounded corners]  (2.5,2.5) -| (3.5,3.5) -- (8.5,3.5);
  \draw [color=black,fill=black,thick] (8.5,3.5) circle (0.8ex);
  \node [draw, circle, fill = white, inner sep = 1.2pt] at (2.5,2.5) {{{\footnotesize \bf 6 }}};  

  \draw [thick] (0,0) -| (2,2) -| (4,3) -| (9,4) -| (10,5) -| (0,0);

  \node at (4.5,-2) {$w = [2,3,5,4,6,1]$};

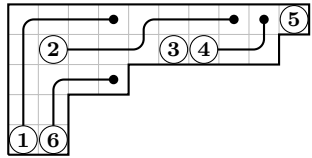
\begin{scope}
    \clip (12,0) -| (14,2) -| (16,3) -| (21,4) -| (22,5) -| (12,0);
    \draw [color=black!25] (12,0) grid (22,5);
\end{scope}

  \draw [thick] (12,0) -| (14,2) -| (16,3) -| (21,4) -| (22,5) -| (12,0);

  \draw [thick, rounded corners]  (12.5,0.5) |- (15.5,4.5);
  \draw [color=black,fill=black,thick] (15.5,4.5) circle (0.8ex);
  \node [draw, circle, fill = white, inner sep = 1.2pt] at (12.5,0.5) {{{\footnotesize \bf 1 }}}; 

  \draw [thick, rounded corners]  (13.5,3.5) -| (16.5,4.5) -- (19.5,4.5);
  \draw [color=black,fill=black,thick] (19.5,4.5) circle (0.8ex);
  \node [draw, circle, fill = white, inner sep = 1.2pt] at (13.5,3.5) {{{\footnotesize \bf 2 }}};

  \node [draw, circle, fill = white, inner sep = 1.2pt] at (17.5,3.5) {{{\footnotesize \bf 3 }}};

  \draw [thick, rounded corners]  (18.5,3.5) -| (20.5,4.5);
  \draw [color=black,fill=black,thick] (20.5,4.5) circle (0.8ex);
  \node [draw, circle, fill = white, inner sep = 1.2pt] at (18.5,3.5) {{{\footnotesize \bf 4 }}}; 

  \node [draw, circle, fill = white, inner sep = 1.2pt] at (21.5,4.5) {{{\footnotesize \bf 5 }}};  

  \draw [thick, rounded corners]  (13.5,0.5) |- (15.5,2.5);
  \draw [color=black,fill=black,thick] (15.5,2.5) circle (0.8ex);
  \node [draw, circle, fill = white, inner sep = 1.2pt] at (13.5,0.5) {{{\footnotesize \bf 6 }}};  

  \node at (16.5,-2) {$w = [2,1,5,4,6,3]$};

 \begin{scope}
    \clip (24,0) -| (26,2) -| (28,3) -| (33,4) -| (34,5) -| (24,0);
    \draw [color=black!25] (24,0) grid (34,5);
\end{scope}

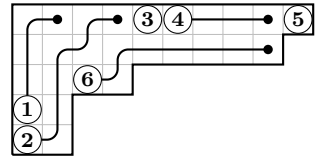
  \draw [thick] (24,0) -| (26,2) -| (28,3) -| (33,4) -| (34,5) -| (24,0);

  \draw [thick, rounded corners]  (24.5,1.5) |- (25.5,4.5);
  \draw [color=black,fill=black,thick] (25.5,4.5) circle (0.8ex);
  \node [draw, circle, fill = white, inner sep = 1.2pt] at (24.5,1.5) {{{\footnotesize \bf 1 }}};

  \draw [thick, rounded corners]  (24.5,0.5) -|(25.5,3.5) -| (26.5,4.5) -- (27.5,4.5);
  \draw [color=black,fill=black,thick] (27.5,4.5) circle (0.8ex);
  \node [draw, circle, fill = white, inner sep = 1.2pt] at (24.5,0.5) {{{\footnotesize \bf 2 }}}; 

  \node [draw, circle, fill = white, inner sep = 1.2pt] at (28.5,4.5) {{{\footnotesize \bf 3 }}};  

  \draw [thick, rounded corners] (29.5,4.5) -- (32.5,4.5);
  \draw [color=black,fill=black,thick] (32.5,4.5) circle (0.8ex);
  \node [draw, circle, fill = white, inner sep = 1.2pt] at (29.5,4.5) {{{\footnotesize \bf 4 }}};  

  \node [draw, circle, fill = white, inner sep = 1.2pt] at (33.5,4.5) {{{\footnotesize \bf 5 }}}; 

  \draw [thick, rounded corners] (26.5,2.5) -| (27.5,3.5) -- (32.5,3.5);
  \draw [color=black,fill=black,thick] (32.5,3.5) circle (0.8ex);
  \node [draw, circle, fill = white, inner sep = 1.2pt] at (26.5,2.5) {{{\footnotesize \bf 6 }}};    

  \node at (28.5,-2) {$w = [3,2,5,4,6,1]$};
\end{tikzpicture}
\end{center}
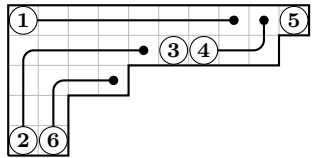
\begin{center}
\begin{tikzpicture}[scale = 0.4]

  \begin{scope}
    \clip (0,0) -| (2,2) -| (4,3) -| (9,4) -| (10,5) -| (0,0);
    \draw [color=black!25] (0,0) grid (10,5);
  \end{scope}

  \draw [thick] (0,0) -| (2,2) -| (4,3) -| (9,4) -| (10,5) -| (0,0);

  \draw [thick, rounded corners] (0.5,4.5) -- (7.5,4.5);
  \draw [color=black,fill=black,thick] (7.5,4.5) circle (0.8ex);
  \node [draw, circle, fill = white, inner sep = 1.2pt] at (0.5,4.5) {{{\footnotesize \bf 1 }}};

  \draw [thick, rounded corners] (0.5,0.5) |- (4.5,3.5);
  \draw [color=black,fill=black,thick] (4.5,3.5) circle (0.8ex);
  \node [draw, circle, fill = white, inner sep = 1.2pt] at (0.5,0.5) {{{\footnotesize \bf 2 }}}; 

  \node [draw, circle, fill = white, inner sep = 1.2pt] at (5.5,3.5) {{{\footnotesize \bf 3 }}}; 

  \draw [thick, rounded corners] (6.5,3.5) -| (8.5,4.5);
  \draw [color=black,fill=black,thick] (8.5,4.5) circle (0.8ex);
  \node [draw, circle, fill = white, inner sep = 1.2pt] at (6.5,3.5) {{{\footnotesize \bf 4 }}}; 

  \node [draw, circle, fill = white, inner sep = 1.2pt] at (9.5,4.5) {{{\footnotesize \bf 5 }}}; 

  \draw [thick, rounded corners] (1.5,0.5) |- (3.5,2.5);
  \draw [color=black,fill=black,thick] (3.5,2.5) circle (0.8ex);
  \node [draw, circle, fill = white, inner sep = 1.2pt] at (1.5,0.5) {{{\footnotesize \bf 6 }}};   

  \node at (4.5,-2) {$w = [1,2,5,4,6,3]$};

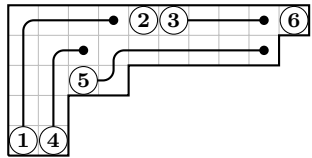
\begin{scope}
    \clip (12,0) -| (14,2) -| (16,3) -| (21,4) -| (22,5) -| (12,0);
    \draw [color=black!25] (12,0) grid (22,5);
\end{scope}

  \draw [thick] (12,0) -| (14,2) -| (16,3) -| (21,4) -| (22,5) -| (12,0);

  \draw [thick, rounded corners] (12.5,0.5) |- (15.5,4.5);
  \draw [color=black,fill=black,thick] (15.5,4.5) circle (0.8ex);
  \node [draw, circle, fill = white, inner sep = 1.2pt] at (12.5,0.5) {{{\footnotesize \bf 1 }}}; 

  \node [draw, circle, fill = white, inner sep = 1.2pt] at (16.5,4.5) {{{\footnotesize \bf 2 }}}; 

  \draw [thick, rounded corners] (17.5,4.5) |- (20.5,4.5);
  \draw [color=black,fill=black,thick] (20.5,4.5) circle (0.8ex);
  \node [draw, circle, fill = white, inner sep = 1.2pt] at (17.5,4.5) {{{\footnotesize \bf 3 }}}; 

  \draw [thick, rounded corners] (13.5,0.5) |- (14.5,3.5);
  \draw [color=black,fill=black,thick] (14.5,3.5) circle (0.8ex);
  \node [draw, circle, fill = white, inner sep = 1.2pt] at (13.5,0.5) {{{\footnotesize \bf 4 }}};

  \draw [thick, rounded corners] (14.5,2.5) -| (15.5,3.5) -- (20.5,3.5);
  \draw [color=black,fill=black,thick] (20.5,3.5) circle (0.8ex);
  \node [draw, circle, fill = white, inner sep = 1.2pt] at (14.5,2.5) {{{\footnotesize \bf 5 }}};

 \node [draw, circle, fill = white, inner sep = 1.2pt] at (21.5,4.5) {{{\footnotesize \bf 6 }}};

  \node at (16.5,-2) {$w = [2,5,4,3,1,6]$};

 \begin{scope}
    \clip (24,0) -| (26,2) -| (28,3) -| (33,4) -| (34,5) -| (24,0);
    \draw [color=black!25] (24,0) grid (34,5);
\end{scope}

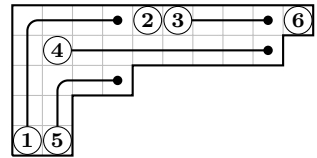
  \draw [thick] (24,0) -| (26,2) -| (28,3) -| (33,4) -| (34,5) -| (24,0);

    \draw [thick, rounded corners] (24.5,0.5) |- (27.5,4.5);
  \draw [color=black,fill=black,thick] (27.5,4.5) circle (0.8ex);
  \node [draw, circle, fill = white, inner sep = 1.2pt] at (24.5,0.5) {{{\footnotesize \bf 1 }}}; 

  \node [draw, circle, fill = white, inner sep = 1.2pt] at (28.5,4.5) {{{\footnotesize \bf 2 }}}; 

  \draw [thick, rounded corners] (29.5,4.5) |- (32.5,4.5);
  \draw [color=black,fill=black,thick] (32.5,4.5) circle (0.8ex);
  \node [draw, circle, fill = white, inner sep = 1.2pt] at (29.5,4.5) {{{\footnotesize \bf 3 }}}; 

  \draw [thick, rounded corners] (25.5,3.5) -- (32.5,3.5);
  \draw [color=black,fill=black,thick] (32.5,3.5) circle (0.8ex);
  \node [draw, circle, fill = white, inner sep = 1.2pt] at (25.5,3.5) {{{\footnotesize \bf 4 }}}; 

  \draw [thick, rounded corners] (25.5,0.5) |- (27.5,2.5);
  \draw [color=black,fill=black,thick] (27.5,2.5) circle (0.8ex);
  \node [draw, circle, fill = white, inner sep = 1.2pt] at (25.5,0.5) {{{\footnotesize \bf 5 }}};

  \node [draw, circle, fill = white, inner sep = 1.2pt] at (33.5,4.5) {{{\footnotesize \bf 6 }}};  

  \node at (28.5,-2) {$w = [2,5,4,1,3,6]$};
\end{tikzpicture}
\end{center}
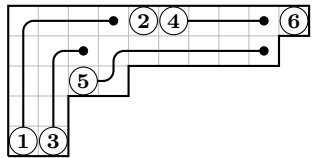
\begin{center}
\begin{tikzpicture}[scale = 0.4]

  \begin{scope}
    \clip (0,0) -| (2,2) -| (4,3) -| (9,4) -| (10,5) -| (0,0);
    \draw [color=black!25] (0,0) grid (10,5);
  \end{scope}

  \draw [thick] (0,0) -| (2,2) -| (4,3) -| (9,4) -| (10,5) -| (0,0);

  \draw [thick, rounded corners] (0.5,0.5) |- (3.5,4.5);
  \draw [color=black,fill=black,thick] (3.5,4.5) circle (0.8ex);
  \node [draw, circle, fill = white, inner sep = 1.2pt] at (0.5,0.5) {{{\footnotesize \bf 1 }}}; 

  \node [draw, circle, fill = white, inner sep = 1.2pt] at (4.5,4.5) {{{\footnotesize \bf 2 }}};  

  \draw [thick, rounded corners] (1.5,0.5) |- (2.5,3.5);
  \draw [color=black,fill=black,thick] (2.5,3.5) circle (0.8ex);
  \node [draw, circle, fill = white, inner sep = 1.2pt] at (1.5,0.5) {{{\footnotesize \bf 3 }}}; 

  \draw [thick, rounded corners] (5.5,4.5) |- (8.5,4.5);
  \draw [color=black,fill=black,thick] (8.5,4.5) circle (0.8ex);
  \node [draw, circle, fill = white, inner sep = 1.2pt] at (5.5,4.5) {{{\footnotesize \bf 4 }}};

  \draw [thick, rounded corners] (2.5,2.5) -| (3.5,3.5) -- (8.5,3.5);
  \draw [color=black,fill=black,thick] (8.5,3.5) circle (0.8ex);
  \node [draw, circle, fill = white, inner sep = 1.2pt] at (2.5,2.5) {{{\footnotesize \bf 5 }}}; 

  \node [draw, circle, fill = white, inner sep = 1.2pt] at (9.5,4.5) {{{\footnotesize \bf 6 }}};   

  \node at (4.5,-2) {$w = [2,5,3,4,1,6]$};

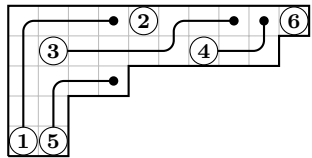
\begin{scope}
    \clip (12,0) -| (14,2) -| (16,3) -| (21,4) -| (22,5) -| (12,0);
    \draw [color=black!25] (12,0) grid (22,5);
\end{scope}

  \draw [thick] (12,0) -| (14,2) -| (16,3) -| (21,4) -| (22,5) -| (12,0);

  \draw [thick, rounded corners] (12.5,0.5) |- (15.5,4.5);
  \draw [color=black,fill=black,thick] (15.5,4.5) circle (0.8ex);
  \node [draw, circle, fill = white, inner sep = 1.2pt] at (12.5,0.5) {{{\footnotesize \bf 1 }}}; 

  \node [draw, circle, fill = white, inner sep = 1.2pt] at (16.5,4.5) {{{\footnotesize \bf 2 }}};

  \draw [thick, rounded corners] (13.5,3.5) -| (17.5,4.5) -- (19.5,4.5);
  \draw [color=black,fill=black,thick] (19.5,4.5) circle (0.8ex);
  \node [draw, circle, fill = white, inner sep = 1.2pt] at (13.5,3.5) {{{\footnotesize \bf 3 }}}; 

  \draw [thick, rounded corners] (18.5,3.5) -| (20.5,4.5);
  \draw [color=black,fill=black,thick] (20.5,4.5) circle (0.8ex);
  \node [draw, circle, fill = white, inner sep = 1.2pt] at (18.5,3.5) {{{\footnotesize \bf 4 }}};  

  \draw [thick, rounded corners] (13.5,0.5) |- (15.5,2.5);
  \draw [color=black,fill=black,thick] (15.5,2.5) circle (0.8ex);
  \node [draw, circle, fill = white, inner sep = 1.2pt] at (13.5,0.5) {{{\footnotesize \bf 5 }}}; 

  \node [draw, circle, fill = white, inner sep = 1.2pt] at (21.5,4.5) {{{\footnotesize \bf 6 }}};   

  \node at (16.5,-2) {$w = [2,5,1,4,3,6]$};

 \begin{scope}
    \clip (24,0) -| (26,2) -| (28,3) -| (33,4) -| (34,5) -| (24,0);
    \draw [color=black!25] (24,0) grid (34,5);
\end{scope}

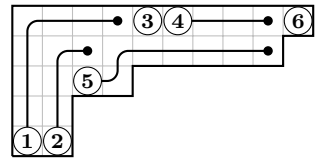
  \draw [thick] (24,0) -| (26,2) -| (28,3) -| (33,4) -| (34,5) -| (24,0);

  \draw [thick, rounded corners] (24.5,0.5) |- (27.5,4.5);
  \draw [color=black,fill=black,thick] (27.5,4.5) circle (0.8ex);
  \node [draw, circle, fill = white, inner sep = 1.2pt] at (24.5,0.5) {{{\footnotesize \bf 1 }}};

  \draw [thick, rounded corners] (25.5,0.5) |- (26.5,3.5);
  \draw [color=black,fill=black,thick] (26.5,3.5) circle (0.8ex);
  \node [draw, circle, fill = white, inner sep = 1.2pt] at (25.5,0.5) {{{\footnotesize \bf 2 }}};

  \node [draw, circle, fill = white, inner sep = 1.2pt] at (28.5,4.5) {{{\footnotesize \bf 3 }}}; 

  \draw [thick, rounded corners] (29.5,4.5) -- (32.5,4.5);
  \draw [color=black,fill=black,thick] (32.5,4.5) circle (0.8ex);
  \node [draw, circle, fill = white, inner sep = 1.2pt] at (29.5,4.5) {{{\footnotesize \bf 4 }}}; 

  \draw [thick, rounded corners] (26.5,2.5) -| (27.5,3.5) -- (32.5,3.5);
  \draw [color=black,fill=black,thick] (32.5,3.5) circle (0.8ex);
  \node [draw, circle, fill = white, inner sep = 1.2pt] at (26.5,2.5) {{{\footnotesize \bf 5 }}}; 

  \node [draw, circle, fill = white, inner sep = 1.2pt] at (33.5,4.5) {{{\footnotesize \bf 6 }}};   

  \node at (28.5,-2) {$w = [2,3,5,4,1,6]$};
\end{tikzpicture}
\end{center}
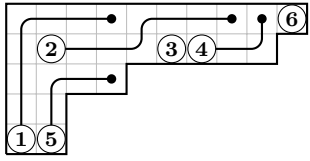
\begin{center}
\begin{tikzpicture}[scale = 0.4]

  \begin{scope}
    \clip (0,0) -| (2,2) -| (4,3) -| (9,4) -| (10,5) -| (0,0);
    \draw [color=black!25] (0,0) grid (10,5);
  \end{scope}

  \draw [thick] (0,0) -| (2,2) -| (4,3) -| (9,4) -| (10,5) -| (0,0);

  \draw [thick, rounded corners] (0.5,0.5) |- (3.5,4.5);
  \draw [color=black,fill=black,thick] (3.5,4.5) circle (0.8ex);
  \node [draw, circle, fill = white, inner sep = 1.2pt] at (0.5,0.5) {{{\footnotesize \bf 1 }}};  

  \draw [thick, rounded corners] (1.5,3.5) -| (4.5,4.5) -- (7.5,4.5);
  \draw [color=black,fill=black,thick] (7.5,4.5) circle (0.8ex);
  \node [draw, circle, fill = white, inner sep = 1.2pt] at (1.5,3.5) {{{\footnotesize \bf 2 }}}; 

  \node [draw, circle, fill = white, inner sep = 1.2pt] at (5.5,3.5) {{{\footnotesize \bf 3 }}}; 

 \draw [thick, rounded corners] (6.5,3.5) -| (8.5,4.5);
  \draw [color=black,fill=black,thick] (8.5,4.5) circle (0.8ex);
  \node [draw, circle, fill = white, inner sep = 1.2pt] at (6.5,3.5) {{{\footnotesize \bf 4 }}}; 

   \draw [thick, rounded corners] (1.5,0.5) |- (3.5,2.5);
  \draw [color=black,fill=black,thick] (3.5,2.5) circle (0.8ex);
  \node [draw, circle, fill = white, inner sep = 1.2pt] at (1.5,0.5) {{{\footnotesize \bf 5 }}}; 

  \node [draw, circle, fill = white, inner sep = 1.2pt] at (9.5,4.5) {{{\footnotesize \bf 6 }}};   

  \node at (4.5,-2) {$w = [2,1,5,4,3,6]$};

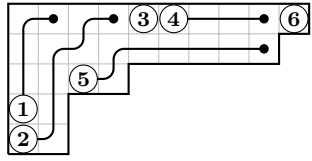
\begin{scope}
    \clip (12,0) -| (14,2) -| (16,3) -| (21,4) -| (22,5) -| (12,0);
    \draw [color=black!25] (12,0) grid (22,5);
\end{scope}

  \draw [thick] (12,0) -| (14,2) -| (16,3) -| (21,4) -| (22,5) -| (12,0);

  \draw [thick, rounded corners] (12.5,1.5) |- (13.5,4.5);
  \draw [color=black,fill=black,thick] (13.5,4.5) circle (0.8ex);
  \node [draw, circle, fill = white, inner sep = 1.2pt] at (12.5,1.5) {{{\footnotesize \bf 1 }}}; 

  \draw [thick, rounded corners] (12.5,0.5) -| (13.5,3.5) -| (14.5,4.5) -- (15.5,4.5);
  \draw [color=black,fill=black,thick] (15.5,4.5) circle (0.8ex);
  \node [draw, circle, fill = white, inner sep = 1.2pt] at (12.5,0.5) {{{\footnotesize \bf 2 }}};  

  \node [draw, circle, fill = white, inner sep = 1.2pt] at (16.5,4.5) {{{\footnotesize \bf 3 }}};    

  \draw [thick, rounded corners] (17.5,4.5) -- (20.5,4.5);
  \draw [color=black,fill=black,thick] (20.5,4.5) circle (0.8ex);
  \node [draw, circle, fill = white, inner sep = 1.2pt] at (17.5,4.5) {{{\footnotesize \bf 4 }}}; 

  \draw [thick, rounded corners] (14.5,2.5) -| (15.5,3.5) -- (20.5,3.5);
  \draw [color=black,fill=black,thick] (20.5,3.5) circle (0.8ex);
  \node [draw, circle, fill = white, inner sep = 1.2pt] at (14.5,2.5) {{{\footnotesize \bf 5 }}}; 

  \node [draw, circle, fill = white, inner sep = 1.2pt] at (21.5,4.5) {{{\footnotesize \bf 6 }}};   

  \node at (16.5,-2) {$w = [3,2,5,4,1,6]$};

 \begin{scope}
    \clip (24,0) -| (26,2) -| (28,3) -| (33,4) -| (34,5) -| (24,0);
    \draw [color=black!25] (24,0) grid (34,5);
\end{scope}

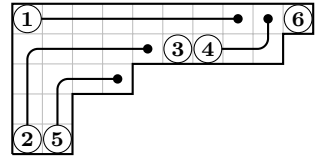
  \draw [thick] (24,0) -| (26,2) -| (28,3) -| (33,4) -| (34,5) -| (24,0);

  \draw [thick, rounded corners] (24.5,4.5) -- (31.5,4.5);
  \draw [color=black,fill=black,thick] (31.5,4.5) circle (0.8ex);
  \node [draw, circle, fill = white, inner sep = 1.2pt] at (24.5,4.5) {{{\footnotesize \bf 1 }}}; 

  \draw [thick, rounded corners] (24.5,0.5) |- (28.5,3.5);
  \draw [color=black,fill=black,thick] (28.5,3.5) circle (0.8ex);
  \node [draw, circle, fill = white, inner sep = 1.2pt] at (24.5,0.5) {{{\footnotesize \bf 2 }}}; 

  \node [draw, circle, fill = white, inner sep = 1.2pt] at (29.5,3.5) {{{\footnotesize \bf 3 }}}; 

  \draw [thick, rounded corners] (30.5,3.5) -| (32.5,4.5);
  \draw [color=black,fill=black,thick] (32.5,4.5) circle (0.8ex);
  \node [draw, circle, fill = white, inner sep = 1.2pt] at (30.5,3.5) {{{\footnotesize \bf 4 }}}; 

   \draw [thick, rounded corners] (25.5,0.5) |- (27.5,2.5);
  \draw [color=black,fill=black,thick] (27.5,2.5) circle (0.8ex);
  \node [draw, circle, fill = white, inner sep = 1.2pt] at (25.5,0.5) {{{\footnotesize \bf 5 }}}; 

  \node [draw, circle, fill = white, inner sep = 1.2pt] at (33.5,4.5) {{{\footnotesize \bf 6 }}};

  \node at (28.5,-2) {$w = [1,2,5,4,3,6]$};
\end{tikzpicture}
\end{center}
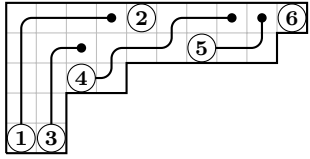
\begin{center}
\begin{tikzpicture}[scale = 0.4]

  \begin{scope}
    \clip (0,0) -| (2,2) -| (4,3) -| (9,4) -| (10,5) -| (0,0);
    \draw [color=black!25] (0,0) grid (10,5);
  \end{scope}

  \draw [thick] (0,0) -| (2,2) -| (4,3) -| (9,4) -| (10,5) -| (0,0);

  \draw [thick, rounded corners] (0.5,0.5) |- (3.5,4.5);
  \draw [color=black,fill=black,thick] (3.5,4.5) circle (0.8ex);
  \node [draw, circle, fill = white, inner sep = 1.2pt] at (0.5,0.5) {{{\footnotesize \bf 1 }}}; 

  \node [draw, circle, fill = white, inner sep = 1.2pt] at (4.5,4.5) {{{\footnotesize \bf 2 }}};

  \draw [thick, rounded corners] (1.5,0.5) |- (2.5,3.5);
  \draw [color=black,fill=black,thick] (2.5,3.5) circle (0.8ex);
  \node [draw, circle, fill = white, inner sep = 1.2pt] at (1.5,0.5) {{{\footnotesize \bf 3 }}}; 

  \draw [thick, rounded corners] (2.5,2.5) -| (3.5,3.5) -| (5.5,4.5) -- (7.5,4.5);
  \draw [color=black,fill=black,thick] (7.5,4.5) circle (0.8ex);
  \node [draw, circle, fill = white, inner sep = 1.2pt] at (2.5,2.5) {{{\footnotesize \bf 4 }}}; 

  \draw [thick, rounded corners] (6.5,3.5) -| (8.5,4.5);
  \draw [color=black,fill=black,thick] (8.5,4.5) circle (0.8ex);
  \node [draw, circle, fill = white, inner sep = 1.2pt] at (6.5,3.5) {{{\footnotesize \bf 5 }}}; 

  \node [draw, circle, fill = white, inner sep = 1.2pt] at (9.5,4.5) {{{\footnotesize \bf 6 }}};   

  \node at (4.5,-2) {$w = [2,5,3,1,4,6]$};

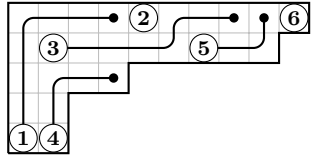
\begin{scope}
    \clip (12,0) -| (14,2) -| (16,3) -| (21,4) -| (22,5) -| (12,0);
    \draw [color=black!25] (12,0) grid (22,5);
\end{scope}

  \draw [thick] (12,0) -| (14,2) -| (16,3) -| (21,4) -| (22,5) -| (12,0);

  \draw [thick, rounded corners] (12.5,0.5) |- (15.5,4.5);
  \draw [color=black,fill=black,thick] (15.5,4.5) circle (0.8ex);
  \node [draw, circle, fill = white, inner sep = 1.2pt] at (12.5,0.5) {{{\footnotesize \bf 1 }}}; 

  \node [draw, circle, fill = white, inner sep = 1.2pt] at (16.5,4.5) {{{\footnotesize \bf 2 }}}; 

  \draw [thick, rounded corners] (13.5,3.5) -| (17.5,4.5) -- (19.5,4.5);
  \draw [color=black,fill=black,thick] (19.5,4.5) circle (0.8ex);
  \node [draw, circle, fill = white, inner sep = 1.2pt] at (13.5,3.5) {{{\footnotesize \bf 3 }}};

  \draw [thick, rounded corners] (13.5,0.5) |- (15.5,2.5);
  \draw [color=black,fill=black,thick] (15.5,2.5) circle (0.8ex);
  \node [draw, circle, fill = white, inner sep = 1.2pt] at (13.5,0.5) {{{\footnotesize \bf 4 }}};

  \draw [thick, rounded corners] (18.5,3.5) -| (20.5,4.5);
  \draw [color=black,fill=black,thick] (20.5,4.5) circle (0.8ex);
  \node [draw, circle, fill = white, inner sep = 1.2pt] at (18.5,3.5) {{{\footnotesize \bf 5 }}};

  \node [draw, circle, fill = white, inner sep = 1.2pt] at (21.5,4.5) {{{\footnotesize \bf 6 }}};  

  \node at (16.5,-2) {$w = [2,5,1,3,4,6]$};

 \begin{scope}
    \clip (24,0) -| (26,2) -| (28,3) -| (33,4) -| (34,5) -| (24,0);
    \draw [color=black!25] (24,0) grid (34,5);
\end{scope}

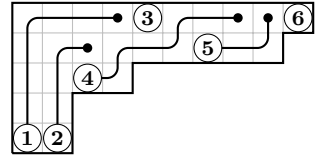
  \draw [thick] (24,0) -| (26,2) -| (28,3) -| (33,4) -| (34,5) -| (24,0);

  \draw [thick, rounded corners] (24.5,0.5) |- (27.5,4.5);
  \draw [color=black,fill=black,thick] (27.5,4.5) circle (0.8ex);
  \node [draw, circle, fill = white, inner sep = 1.2pt] at (24.5,0.5) {{{\footnotesize \bf 1 }}}; 

  \draw [thick, rounded corners] (25.5,0.5) |- (26.5,3.5);
  \draw [color=black,fill=black,thick] (26.5,3.5) circle (0.8ex);
  \node [draw, circle, fill = white, inner sep = 1.2pt] at (25.5,0.5) {{{\footnotesize \bf 2 }}}; 

  \node [draw, circle, fill = white, inner sep = 1.2pt] at (28.5,4.5) {{{\footnotesize \bf 3 }}}; 

  \draw [thick, rounded corners] (26.5,2.5) -| (27.5,3.5) -| (29.5,4.5) -- (31.5,4.5);
  \draw [color=black,fill=black,thick] (31.5,4.5) circle (0.8ex);
  \node [draw, circle, fill = white, inner sep = 1.2pt] at (26.5,2.5) {{{\footnotesize \bf 4 }}}; 

  \draw [thick, rounded corners] (30.5,3.5) -| (32.5,4.5);
  \draw [color=black,fill=black,thick] (32.5,4.5) circle (0.8ex);
  \node [draw, circle, fill = white, inner sep = 1.2pt] at (30.5,3.5) {{{\footnotesize \bf 5 }}};

  \node [draw, circle, fill = white, inner sep = 1.2pt] at (33.5,4.5) {{{\footnotesize \bf 6 }}};  

  \node at (28.5,-2) {$w = [2,3,5,1,4,6]$};
\end{tikzpicture}
\end{center}
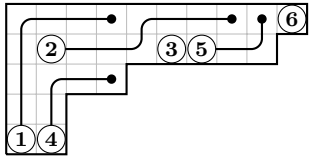
\begin{center}
\begin{tikzpicture}[scale = 0.4]

  \begin{scope}
    \clip (0,0) -| (2,2) -| (4,3) -| (9,4) -| (10,5) -| (0,0);
    \draw [color=black!25] (0,0) grid (10,5);
  \end{scope}

  \draw [thick] (0,0) -| (2,2) -| (4,3) -| (9,4) -| (10,5) -| (0,0);

  \draw [thick, rounded corners] (0.5,0.5) |- (3.5,4.5);
  \draw [color=black,fill=black,thick] (3.5,4.5) circle (0.8ex);
  \node [draw, circle, fill = white, inner sep = 1.2pt] at (0.5,0.5) {{{\footnotesize \bf 1 }}}; 

  \draw [thick, rounded corners] (1.5,3.5) -| (4.5,4.5) -- (7.5,4.5);
  \draw [color=black,fill=black,thick] (7.5,4.5) circle (0.8ex);
  \node [draw, circle, fill = white, inner sep = 1.2pt] at (1.5,3.5) {{{\footnotesize \bf 2 }}}; 

  \node [draw, circle, fill = white, inner sep = 1.2pt] at (5.5,3.5) {{{\footnotesize \bf 3 }}};  

  \draw [thick, rounded corners] (1.5,0.5) |- (3.5,2.5);
  \draw [color=black,fill=black,thick] (3.5,2.5) circle (0.8ex);
  \node [draw, circle, fill = white, inner sep = 1.2pt] at (1.5,0.5) {{{\footnotesize \bf 4 }}};  

  \draw [thick, rounded corners] (6.5,3.5) -| (8.5,4.5);
  \draw [color=black,fill=black,thick] (8.5,4.5) circle (0.8ex);
  \node [draw, circle, fill = white, inner sep = 1.2pt] at (6.5,3.5) {{{\footnotesize \bf 5 }}}; 

  \node [draw, circle, fill = white, inner sep = 1.2pt] at (9.5,4.5) {{{\footnotesize \bf 6 }}};   

  \node at (4.5,-2) {$w = [2,1,5,3,4,6]$};

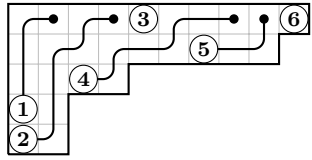
\begin{scope}
    \clip (12,0) -| (14,2) -| (16,3) -| (21,4) -| (22,5) -| (12,0);
    \draw [color=black!25] (12,0) grid (22,5);
\end{scope}

  \draw [thick, rounded corners] (12.5,1.5) |- (13.5,4.5);
  \draw [color=black,fill=black,thick] (13.5,4.5) circle (0.8ex);
  \node [draw, circle, fill = white, inner sep = 1.2pt] at (12.5,1.5) {{{\footnotesize \bf 1 }}};

  \draw [thick, rounded corners] (12.5,0.5) -| (13.5,3.5) -| (14.5,4.5) -- (15.5,4.5);
  \draw [color=black,fill=black,thick] (15.5,4.5) circle (0.8ex);
  \node [draw, circle, fill = white, inner sep = 1.2pt] at (12.5,0.5) {{{\footnotesize \bf 2 }}}; 

  \node [draw, circle, fill = white, inner sep = 1.2pt] at (16.5,4.5) {{{\footnotesize \bf 3 }}}; 

  \draw [thick, rounded corners] (14.5,2.5) -| (15.5,3.5) -| (17.5,4.5) -- (19.5,4.5);
  \draw [color=black,fill=black,thick] (19.5,4.5) circle (0.8ex);
  \node [draw, circle, fill = white, inner sep = 1.2pt] at (14.5,2.5) {{{\footnotesize \bf 4 }}};

  \draw [thick, rounded corners] (18.5,3.5) -| (20.5,4.5);
  \draw [color=black,fill=black,thick] (20.5,4.5) circle (0.8ex);
  \node [draw, circle, fill = white, inner sep = 1.2pt] at (18.5,3.5) {{{\footnotesize \bf 5 }}}; 

  \node [draw, circle, fill = white, inner sep = 1.2pt] at (21.5,4.5) {{{\footnotesize \bf 6 }}};   

  \draw [thick] (12,0) -| (14,2) -| (16,3) -| (21,4) -| (22,5) -| (12,0);

  \node at (16.5,-2) {$w = [3,2,5,1,4,6]$};

 \begin{scope}
    \clip (24,0) -| (26,2) -| (28,3) -| (33,4) -| (34,5) -| (24,0);
    \draw [color=black!25] (24,0) grid (34,5);
\end{scope}

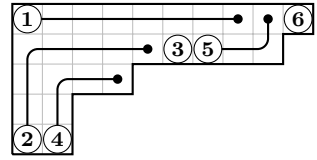
  \draw [thick] (24,0) -| (26,2) -| (28,3) -| (33,4) -| (34,5) -| (24,0);

  \draw [thick, rounded corners] (24.5,4.5) -- (31.5,4.5);
  \draw [color=black,fill=black,thick] (31.5,4.5) circle (0.8ex);
  \node [draw, circle, fill = white, inner sep = 1.2pt] at (24.5,4.5) {{{\footnotesize \bf 1 }}};

  \draw [thick, rounded corners] (24.5,0.5) |- (28.5,3.5);
  \draw [color=black,fill=black,thick] (28.5,3.5) circle (0.8ex);
  \node [draw, circle, fill = white, inner sep = 1.2pt] at (24.5,0.5) {{{\footnotesize \bf 2 }}};

  \node [draw, circle, fill = white, inner sep = 1.2pt] at (29.5,3.5) {{{\footnotesize \bf 3 }}};

  \draw [thick, rounded corners] (25.5,0.5) |- (27.5,2.5);
  \draw [color=black,fill=black,thick] (27.5,2.5) circle (0.8ex);
  \node [draw, circle, fill = white, inner sep = 1.2pt] at (25.5,0.5) {{{\footnotesize \bf 4 }}}; 

  \draw [thick, rounded corners] (30.5,3.5) -| (32.5,4.5);
  \draw [color=black,fill=black,thick] (32.5,4.5) circle (0.8ex);
  \node [draw, circle, fill = white, inner sep = 1.2pt] at (30.5,3.5) {{{\footnotesize \bf 5 }}};  

  \node [draw, circle, fill = white, inner sep = 1.2pt] at (33.5,4.5) {{{\footnotesize \bf 6 }}};    

  \node at (28.5,-2) {$w = [1,2,5,3,4,6]$};
\end{tikzpicture}
\end{center}
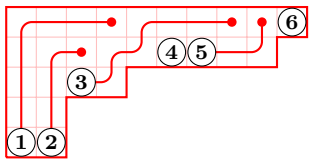
\begin{center}
\begin{tikzpicture}[scale = 0.4]

  \begin{scope}
    \clip (0,0) -| (2,2) -| (4,3) -| (9,4) -| (10,5) -| (0,0);
    \draw [color=red!25] (0,0) grid (10,5);
  \end{scope}

  \draw [thick, color = red] (0,0) -| (2,2) -| (4,3) -| (9,4) -| (10,5) -| (0,0);

  \draw [thick, rounded corners, color = red] (0.5,0.5) |- (3.5,4.5);
  \draw [color=red,fill=red,thick] (3.5,4.5) circle (0.8ex);
  \node [draw, circle, fill = white, inner sep = 1.2pt] at (0.5,0.5) {{{\footnotesize \bf 1 }}}; 

\draw [thick, rounded corners, color = red] (1.5,0.5) |- (2.5,3.5);
  \draw [color=red,fill=red,thick] (2.5,3.5) circle (0.8ex);
  \node [draw, circle, fill = white, inner sep = 1.2pt] at (1.5,0.5) {{{\footnotesize \bf 2 }}}; 

\draw [thick, rounded corners, color = red] (2.5,2.5) -| (3.5,3.5) -| (4.5,4.5) -- (7.5,4.5);
  \draw [color=red,fill=red,thick] (7.5,4.5) circle (0.8ex);
  \node [draw, circle, fill = white, inner sep = 1.2pt] at (2.5,2.5) {{{\footnotesize \bf 3 }}}; 

  \node [draw, circle, fill = white, inner sep = 1.2pt] at (5.5,3.5) {{{\footnotesize \bf 4 }}}; 

 \draw [thick, rounded corners, color = red] (6.5,3.5) -| (8.5,4.5);
  \draw [color=red,fill=red,thick] (8.5,4.5) circle (0.8ex);
  \node [draw, circle, fill = white, inner sep = 1.2pt] at (6.5,3.5) {{{\footnotesize \bf 5 }}}; 

  \node [draw, circle, fill = white, inner sep = 1.2pt] at (9.5,4.5) {{{\footnotesize \bf 6 }}};   

  \node at (4.5,-2) {$w = [2,3,1,5,4,6]$};

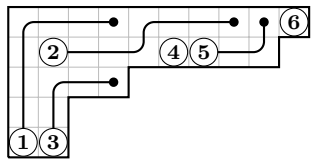
\begin{scope}
    \clip (12,0) -| (14,2) -| (16,3) -| (21,4) -| (22,5) -| (12,0);
    \draw [color=black!25] (12,0) grid (22,5);
\end{scope}

  \draw [thick] (12,0) -| (14,2) -| (16,3) -| (21,4) -| (22,5) -| (12,0);

 \draw [thick, rounded corners] (12.5,0.5) |- (15.5,4.5);
  \draw [color=black,fill=black,thick] (15.5,4.5) circle (0.8ex);
  \node [draw, circle, fill = white, inner sep = 1.2pt] at (12.5,0.5) {{{\footnotesize \bf 1 }}}; 

 \draw [thick, rounded corners] (13.5,3.5) -| (16.5,4.5) -- (19.5,4.5);
  \draw [color=black,fill=black,thick] (19.5,4.5) circle (0.8ex);
  \node [draw, circle, fill = white, inner sep = 1.2pt] at (13.5,3.5) {{{\footnotesize \bf 2 }}};

  \draw [thick, rounded corners] (13.5,0.5) |- (15.5,2.5);
  \draw [color=black,fill=black,thick] (15.5,2.5) circle (0.8ex);
  \node [draw, circle, fill = white, inner sep = 1.2pt] at (13.5,0.5) {{{\footnotesize \bf 3 }}}; 

  \node [draw, circle, fill = white, inner sep = 1.2pt] at (17.5,3.5) {{{\footnotesize \bf 4 }}};

  \draw [thick, rounded corners] (18.5,3.5) -| (20.5,4.5);
  \draw [color=black,fill=black,thick] (20.5,4.5) circle (0.8ex);
  \node [draw, circle, fill = white, inner sep = 1.2pt] at (18.5,3.5) {{{\footnotesize \bf 5 }}};  

  \node [draw, circle, fill = white, inner sep = 1.2pt] at (21.5,4.5) {{{\footnotesize \bf 6 }}};    

  \node at (16.5,-2) {$w = [2,1,3,5,4,6]$};

 \begin{scope}
    \clip (24,0) -| (26,2) -| (28,3) -| (33,4) -| (34,5) -| (24,0);
    \draw [color=black!25] (24,0) grid (34,5);
\end{scope}

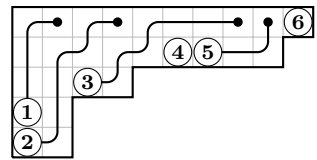
  \draw [thick] (24,0) -| (26,2) -| (28,3) -| (33,4) -| (34,5) -| (24,0);

 \draw [thick, rounded corners] (24.5,1.5) |- (25.5,4.5);
  \draw [color=black,fill=black,thick] (25.5,4.5) circle (0.8ex);
  \node [draw, circle, fill = white, inner sep = 1.2pt] at (24.5,1.5) {{{\footnotesize \bf 1 }}};

 \draw [thick, rounded corners] (24.5,0.5) -| (25.5,3.5) -| (26.5,4.5) -- (27.5,4.5);
  \draw [color=black,fill=black,thick] (27.5,4.5) circle (0.8ex);
  \node [draw, circle, fill = white, inner sep = 1.2pt] at (24.5,0.5) {{{\footnotesize \bf 2 }}};

 \draw [thick, rounded corners] (26.5,2.5) -| (27.5,3.5) -| (28.5,4.5) -- (31.5,4.5);
  \draw [color=black,fill=black,thick] (31.5,4.5) circle (0.8ex);
  \node [draw, circle, fill = white, inner sep = 1.2pt] at (26.5,2.5) {{{\footnotesize \bf 3 }}};  

  \node [draw, circle, fill = white, inner sep = 1.2pt] at (29.5,3.5) {{{\footnotesize \bf 4 }}}; 

   \draw [thick, rounded corners] (30.5,3.5) -| (32.5,4.5);
  \draw [color=black,fill=black,thick] (32.5,4.5) circle (0.8ex);
  \node [draw, circle, fill = white, inner sep = 1.2pt] at (30.5,3.5) {{{\footnotesize \bf 5 }}};

   \node [draw, circle, fill = white, inner sep = 1.2pt] at (33.5,4.5) {{{\footnotesize \bf 6 }}}; 

  \node at (28.5,-2) {$w = [3,2,1,5,4,6]$};
\end{tikzpicture}
\end{center}
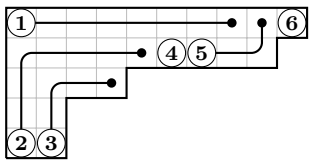
\begin{center}
\begin{tikzpicture}[scale = 0.4]

  \begin{scope}
    \clip (0,0) -| (2,2) -| (4,3) -| (9,4) -| (10,5) -| (0,0);
    \draw [color=black!25] (0,0) grid (10,5);
  \end{scope}

  \draw [thick] (0,0) -| (2,2) -| (4,3) -| (9,4) -| (10,5) -| (0,0);

  \draw [thick, rounded corners] (0.5,4.5) -- (7.5,4.5);
  \draw [color=black,fill=black,thick] (7.5,4.5) circle (0.8ex);
  \node [draw, circle, fill = white, inner sep = 1.2pt] at (0.5,4.5) {{{\footnotesize \bf 1 }}};

  \draw [thick, rounded corners] (0.5,0.5) |- (4.5,3.5);
  \draw [color=black,fill=black,thick] (4.5,3.5) circle (0.8ex);
  \node [draw, circle, fill = white, inner sep = 1.2pt] at (0.5,0.5) {{{\footnotesize \bf 2 }}}; 

  \draw [thick, rounded corners] (1.5,0.5) |- (3.5,2.5);
  \draw [color=black,fill=black,thick] (3.5,2.5) circle (0.8ex);
  \node [draw, circle, fill = white, inner sep = 1.2pt] at (1.5,0.5) {{{\footnotesize \bf 3 }}}; 

  \node [draw, circle, fill = white, inner sep = 1.2pt] at (5.5,3.5) {{{\footnotesize \bf 4 }}}; 

  \draw [thick, rounded corners] (6.5,3.5) -| (8.5,4.5);
  \draw [color=black,fill=black,thick] (8.5,4.5) circle (0.8ex);
  \node [draw, circle, fill = white, inner sep = 1.2pt] at (6.5,3.5) {{{\footnotesize \bf 5 }}}; 

  \node [draw, circle, fill = white, inner sep = 1.2pt] at (9.5,4.5) {{{\footnotesize \bf 6 }}};   

  \node at (4.5,-2) {$w = [1,2,3,5,4,6]$};

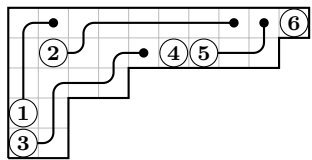
\begin{scope}
    \clip (12,0) -| (14,2) -| (16,3) -| (21,4) -| (22,5) -| (12,0);
    \draw [color=black!25] (12,0) grid (22,5);
\end{scope}

  \draw [thick] (12,0) -| (14,2) -| (16,3) -| (21,4) -| (22,5) -| (12,0);

  \draw [thick, rounded corners] (12.5,1.5) |- (13.5,4.5);
  \draw [color=black,fill=black,thick] (13.5,4.5) circle (0.8ex);
  \node [draw, circle, fill = white, inner sep = 1.2pt] at (12.5,1.5) {{{\footnotesize \bf 1 }}}; 

  \draw [thick, rounded corners] (13.5,3.5) -| (14.5,4.5) -- (19.5,4.5);
  \draw [color=black,fill=black,thick] (19.5,4.5) circle (0.8ex);
  \node [draw, circle, fill = white, inner sep = 1.2pt] at (13.5,3.5) {{{\footnotesize \bf 2 }}}; 

  \draw [thick, rounded corners] (12.5,0.5) -| (13.5,2.5) -| (15.5,3.5) -- (16.5,3.5);
  \draw [color=black,fill=black,thick] (16.5,3.5) circle (0.8ex);
  \node [draw, circle, fill = white, inner sep = 1.2pt] at (12.5,0.5) {{{\footnotesize \bf 3 }}};  

  \node [draw, circle, fill = white, inner sep = 1.2pt] at (17.5,3.5) {{{\footnotesize \bf 4 }}}; 

  \draw [thick, rounded corners] (18.5,3.5) -| (20.5,4.5);
  \draw [color=black,fill=black,thick] (20.5,4.5) circle (0.8ex);
  \node [draw, circle, fill = white, inner sep = 1.2pt] at (18.5,3.5) {{{\footnotesize \bf 5 }}};  

  \node [draw, circle, fill = white, inner sep = 1.2pt] at (21.5,4.5) {{{\footnotesize \bf 6 }}};    

  \node at (16.5,-2) {$w = [3,1,2,5,4,6]$};

 \begin{scope}
    \clip (24,0) -| (26,2) -| (28,3) -| (33,4) -| (34,5) -| (24,0);
    \draw [color=black!25] (24,0) grid (34,5);
\end{scope}

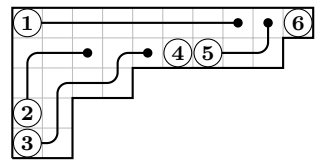
  \draw [thick] (24,0) -| (26,2) -| (28,3) -| (33,4) -| (34,5) -| (24,0);

    \draw [thick, rounded corners] (24.5,4.5) -- (31.5,4.5);
  \draw [color=black,fill=black,thick] (31.5,4.5) circle (0.8ex);
  \node [draw, circle, fill = white, inner sep = 1.2pt] at (24.5,4.5) {{{\footnotesize \bf 1 }}};

  \draw [thick, rounded corners] (24.5,1.5) |- (26.5,3.5);
  \draw [color=black,fill=black,thick] (26.5,3.5) circle (0.8ex);
  \node [draw, circle, fill = white, inner sep = 1.2pt] at (24.5,1.5) {{{\footnotesize \bf 2 }}}; 

  \draw [thick, rounded corners] (24.5,0.5) -| (25.5,2.5) -| (27.5,3.5) -- (28.5,3.5);
  \draw [color=black,fill=black,thick] (28.5,3.5) circle (0.8ex);
  \node [draw, circle, fill = white, inner sep = 1.2pt] at (24.5,0.5) {{{\footnotesize \bf 3 }}}; 

  \node [draw, circle, fill = white, inner sep = 1.2pt] at (29.5,3.5) {{{\footnotesize \bf 4 }}};

  \draw [thick, rounded corners] (30.5,3.5) -| (32.5,4.5);
  \draw [color=black,fill=black,thick] (32.5,4.5) circle (0.8ex);
  \node [draw, circle, fill = white, inner sep = 1.2pt] at (30.5,3.5) {{{\footnotesize \bf 5 }}}; 

  \node [draw, circle, fill = white, inner sep = 1.2pt] at (33.5,4.5) {{{\footnotesize \bf 6 }}};   

  \node at (28.5,-2) {$w = [1,3,2,5,4,6]$};
\end{tikzpicture}
\end{center}
\end{center}

\section{Irreducible characters on partial permutations}
Let $\lambda, \mu \vdash n$ be partitions with $\mu = (\mu_1, \dots, \mu_r)$.  We introduce the monotonic tiling enumerator
\begin{equation}
\vec{\chi}^{ \, \lambda}_{ \, \mu} := \sum_T \sign(T)
\end{equation}
where the sum is over monotonic ribbon tilings $T$ of shape $\lambda$ with ribbon sizes $\mu_1, \dots, \mu_r$.
With this notation, Theorem~\ref{path-murnaghan-nakayama} reads
\begin{equation}
\label{pmn}
\vec{p}_{\mu} = m(\mu)! \cdot \sum_{\lambda \vdash n} \vec{\chi}^{ \, \lambda}_{ \, \mu}  \cdot s_{\lambda}
\end{equation}
in parallel with the classical expansion $p_{\mu} = \sum_{\lambda \vdash n} \chi^{\lambda}_{\mu} \cdot s_{\lambda}.$

To state the Schur expansion of the atomic functions as cleanly as possible, we use one more piece of notation.
If $\lambda/\rho$ is a skew shape and $\nu$ is a composition, let $\chi^{\lambda/\rho}_{\nu}$ be the signed count of standard 
ribbon tableaux of skew shape $\lambda/\rho$ of type $\nu$.  If $\rho \not\subseteq \lambda$ we set
$\chi^{\lambda/\rho}_{\nu} := 0$.

\begin{corollary}
\label{atomic-schur-expansion}
Let $(I,J) \in \symm_{n,k}$ be a partial permutation whose graph $G_n(I,J)$ consists of paths of sizes
$\mu = (\mu_1 \geq \mu_2 \geq \cdots )$ and cycles of sizes
$\nu = (\nu_1 \geq \nu_2 \geq \cdots )$ where $\mu \vdash a$ and $\nu \vdash b$ so that $a + b = n$.  Then
\begin{equation}
\label{cor:atomic-schur-eqn}
A_{n,I,J} = m(\mu)! \cdot \sum_{\lambda \, \vdash \, n}  \left(
\sum_{\rho \, \vdash \, a}  \vec{\chi}_{ \, \mu}^{ \, \rho} \cdot \chi_{\nu}^{\lambda/\rho}
 \right) \cdot s_{\lambda}.
\end{equation}
\end{corollary}

\begin{proof}
We calculate
\begin{align}
A_{n,I,J} &= \vec{p}_{\mu} \cdot p_{\nu} \\
&= m(\mu)! \cdot \left( \sum_{\rho \, \vdash \, a} \vec{\chi}_{ \, \mu}^{ \, \rho} \cdot s_{\rho} \right) \cdot p_{\nu} \\
&= m(\mu)! \cdot \sum_{\lambda \, \vdash \, n}  \left(
\sum_{\rho \, \vdash \, a}  \vec{\chi}_{ \, \mu}^{ \, \rho} \cdot \chi_{\nu}^{\lambda/\rho}
 \right) \cdot s_{\lambda}
\end{align}
where the first equality uses Proposition~\ref{atomic-factorization}, the second uses the 
Path Murnaghan-Nakayama Rule as 
stated in
 Equation~\eqref{pmn}, 
and the third uses the classical Murnaghan-Nakayama Rule (Theorem~\ref{mn-rule}).
\end{proof}

\begin{figure}
\begin{center}
\begin{tikzpicture}[scale = 0.3]

\begin{scope}
   \clip (0,0) -| (7,1) -| (0,0);
    \draw [color=black!25] (0,0) grid (7,1);
\end{scope}

\draw [thick] (0,0) -| (7,1) -| (0,0);

\draw [thick, rounded corners]  (0.5,0.5) -- (1.5,0.5);
\draw [color=black,fill=black,thick] (1.5,0.5) circle (.4ex);
\node [draw, circle, fill = white, inner sep = 1.2pt] at (0.5,0.5) { };    
\node [draw, circle, fill = white, inner sep = 1.2pt] at (2.5,0.5) { };    
\draw [color=red, thick, rounded corners]  (4.5,0.5) -- (6.5,0.5);   
\draw [color=red,fill=black,thick] (6.5,0.5) circle (.4ex);
\node [color=red,draw, circle, fill = white, inner sep = 1.2pt] at (4.5,0.5) { }; 
\node [color=red,draw, circle, fill = white, inner sep = 1.2pt] at (3.5,0.5) { };   
    
\node at (3,-1) {$+1$};

\begin{scope}
   \clip (10,0) -| (17,1) -| (10,0);
    \draw [color=black!25] (10,0) grid (17,1);
\end{scope}

\draw [thick] (10,0) -| (17,1) -| (10,0);

\draw [thick, rounded corners]  (11.5,0.5) -- (12.5,0.5);
\draw [color=black,fill=black,thick] (12.5,0.5) circle (.4ex);
\node [draw, circle, fill = white, inner sep = 1.2pt] at (11.5,0.5) { };    
\node [draw, circle, fill = white, inner sep = 1.2pt] at (10.5,0.5) { };    
\draw [color=red, thick, rounded corners]  (14.5,0.5) -- (16.5,0.5);   
\draw [color=red,fill=black,thick] (16.5,0.5) circle (.4ex);
\node [color=red,draw, circle, fill = white, inner sep = 1.2pt] at (14.5,0.5) { }; 
\node [color=red,draw, circle, fill = white, inner sep = 1.2pt] at (13.5,0.5) { };   
    
\node at (13,-1) {$+1$};

\begin{scope}
   \clip (20,-1) -| (21,0) -| (26,1) -| (20,-1);
    \draw [color=black!25] (20,-1) grid (26,1);
\end{scope}

\draw [thick] (20,-1) -| (21,0) -| (26,1) -| (20,-1);

\draw [thick, rounded corners]  (20.5,0.5) -- (21.5,0.5);
\draw [color=black,fill=black,thick] (21.5,0.5) circle (.4ex);
\node [draw, circle, fill = white, inner sep = 1.2pt] at (20.5,0.5) { };    
\node [draw, circle, fill = white, inner sep = 1.2pt] at (22.5,0.5) { };    
\draw [color=red, thick, rounded corners]  (23.5,0.5) -- (25.5,0.5);   
\draw [color=red,fill=black,thick] (25.5,0.5) circle (.4ex);
\node [color=red,draw, circle, fill = white, inner sep = 1.2pt] at (23.5,0.5) { }; 
\node [color=red,draw, circle, fill = white, inner sep = 1.2pt] at (20.5,-0.5) { };   
    
\node at (23,-1) {$+1$};

\begin{scope}
   \clip (30,-1) -| (31,0) -| (36,1) -| (30,-1);
    \draw [color=black!25] (30,-1) grid (36,1);
\end{scope}

\draw [thick] (30,-1) -| (31,0) -| (36,1) -| (30,-1);

\draw [thick, rounded corners]  (31.5,0.5) -- (32.5,0.5);
\draw [color=black,fill=black,thick] (32.5,0.5) circle (.4ex);
\node [draw, circle, fill = white, inner sep = 1.2pt] at (31.5,0.5) { };    
\node [draw, circle, fill = white, inner sep = 1.2pt] at (30.5,0.5) { };    
\draw [color=red, thick, rounded corners]  (33.5,0.5) -- (35.5,0.5);   
\draw [color=red,fill=black,thick] (35.5,0.5) circle (.4ex);
\node [color=red,draw, circle, fill = white, inner sep = 1.2pt] at (33.5,0.5) { }; 
\node [color=red,draw, circle, fill = white, inner sep = 1.2pt] at (30.5,-0.5) { };   
    
\node at (33,-1) {$+1$};

\begin{scope}
   \clip (40,-1) -| (41,0) -| (46,1) -| (40,-1);
    \draw [color=black!25] (40,-1) grid (46,1);
\end{scope}

\draw [thick] (40,-1) -| (41,0) -| (46,1) -| (40,-1);

\draw [thick, rounded corners]  (40.5,-0.5) -- (40.5,0.5);
\draw [color=black,fill=black,thick] (40.5,0.5) circle (.4ex);
\node [draw, circle, fill = white, inner sep = 1.2pt] at (40.5,-0.5) { };    
\node [draw, circle, fill = white, inner sep = 1.2pt] at (41.5,0.5) { };    
\draw [color=red, thick, rounded corners]  (43.5,0.5) -- (45.5,0.5);   
\draw [color=red,fill=black,thick] (45.5,0.5) circle (.4ex);
\node [color=red,draw, circle, fill = white, inner sep = 1.2pt] at (43.5,0.5) { }; 
\node [color=red,draw, circle, fill = white, inner sep = 1.2pt] at (42.5,0.5) { };   
    
\node at (43,-1) {$-1$};

\end{tikzpicture}
\end{center}

\begin{center}
\begin{tikzpicture}[scale = 0.3]

\begin{scope}
   \clip (0,0) -| (2,1) -| (5,2) -| (0,0);
    \draw [color=black!25] (0,0) grid (5,2);
\end{scope}

\draw [thick] (0,0) -| (2,1) -| (5,2) -| (0,0);

\draw [thick, rounded corners]  (0.5,0.5) -- (0.5,1.5);
\draw [color=black,fill=black,thick] (0.5,1.5) circle (.4ex);
\node [draw, circle, fill = white, inner sep = 1.2pt] at (0.5,0.5) { };    
\node [draw, circle, fill = white, inner sep = 1.2pt] at (1.5,1.5) { };    
\draw [color=red, thick, rounded corners]  (2.5,1.5) -- (4.5,1.5);   
\draw [color=red,fill=black,thick] (4.5,1.5) circle (.4ex);
\node [color=red,draw, circle, fill = white, inner sep = 1.2pt] at (2.5,1.5) { }; 
\node [color=red,draw, circle, fill = white, inner sep = 1.2pt] at (1.5,0.5) { };   
    
\node at (3,0) {$-1$};

\begin{scope}
   \clip (8,-1) -| (9,1) -| (13,2) -| (8,-1);
    \draw [color=black!25] (8,-1) grid (13,2);
\end{scope}

\draw [thick] (8,-1) -| (9,1) -| (13,2) -| (8,-1);

\draw [thick, rounded corners]  (8.5,0.5) -- (8.5,1.5);
\draw [color=black,fill=black,thick] (8.5,1.5) circle (.4ex);
\node [draw, circle, fill = white, inner sep = 1.2pt] at (8.5,0.5) { };    
\node [draw, circle, fill = white, inner sep = 1.2pt] at (9.5,1.5) { };    
\draw [color=red, thick, rounded corners]  (10.5,1.5) -- (12.5,1.5);   
\draw [color=red,fill=black,thick] (12.5,1.5) circle (.4ex);
\node [color=red,draw, circle, fill = white, inner sep = 1.2pt] at (10.5,1.5) { }; 
\node [color=red,draw, circle, fill = white, inner sep = 1.2pt] at (8.5,-0.5) { };   
    
\node at (11,0) {$-1$};

\begin{scope}
   \clip (16,0) -| (19,1) -| (20,2) -| (16,0);
    \draw [color=black!25] (16,0) grid (20,2);
\end{scope}

\draw [thick] (16,0) -| (19,1) -| (20,2) -| (16,0);

\draw [thick, rounded corners]  (16.5,1.5) -- (17.5,1.5);
\draw [color=black,fill=black,thick] (17.5,1.5) circle (.4ex);
\node [draw, circle, fill = white, inner sep = 1.2pt] at (16.5,1.5) { };    
\node [draw, circle, fill = white, inner sep = 1.2pt] at (18.5,1.5) { };    
\draw [color=red, thick, rounded corners]  (16.5,0.5) -- (18.5,0.5);   
\draw [color=red,fill=black,thick] (18.5,0.5) circle (.4ex);
\node [color=red,draw, circle, fill = white, inner sep = 1.2pt] at (16.5,0.5) { }; 
\node [color=red,draw, circle, fill = white, inner sep = 1.2pt] at (19.5,1.5) { };   
    
\node at (18,-1) {$+1$};

\begin{scope}
   \clip (22,0) -| (25,1) -| (26,2) -| (22,0);
    \draw [color=black!25] (22,0) grid (26,2);
\end{scope}

\draw [thick]  (22,0) -| (25,1) -| (26,2) -| (22,0);

\draw [thick, rounded corners]  (23.5,1.5) -- (24.5,1.5);
\draw [color=black,fill=black,thick] (24.5,1.5) circle (.4ex);
\node [draw, circle, fill = white, inner sep = 1.2pt] at (23.5,1.5) { };    
\node [draw, circle, fill = white, inner sep = 1.2pt] at (22.5,1.5) { };    
\draw [color=red, thick, rounded corners]  (22.5,0.5) -- (24.5,0.5);   
\draw [color=red,fill=black,thick] (24.5,0.5) circle (.4ex);
\node [color=red,draw, circle, fill = white, inner sep = 1.2pt] at (22.5,0.5) { }; 
\node [color=red,draw, circle, fill = white, inner sep = 1.2pt] at (25.5,1.5) { };   
    
\node at (24,-1) {$+1$};

\begin{scope}
   \clip (28,0) -| (31,1) -| (32,2) -| (28,0);
    \draw [color=black!25] (28,0) grid (32,2);
\end{scope}

\draw [thick]  (28,0) -| (31,1) -| (32,2) -| (28,0);

\draw [thick, rounded corners]  (28.5,0.5) -- (28.5,1.5);
\draw [color=black,fill=black,thick] (28.5,1.5) circle (.4ex);
\node [draw, circle, fill = white, inner sep = 1.2pt] at (28.5,0.5) { };    
\node [draw, circle, fill = white, inner sep = 1.2pt] at (29.5,1.5) { };    
\draw [color=red, thick, rounded corners]  (30.5,0.5) |- (31.5,1.5);   
\draw [color=red,fill=black,thick] (31.5,1.5) circle (.4ex);
\node [color=red,draw, circle, fill = white, inner sep = 1.2pt] at (30.5,0.5) { }; 
\node [color=red,draw, circle, fill = white, inner sep = 1.2pt] at (29.5,0.5) { };   
    
\node at (30,-1) {$+1$};

\begin{scope}
   \clip (36,-1) -| (37,0) -| (38,1) -| (40,2) -| (36,-1);
    \draw [color=black!25] (36,-1) grid (40,2);
\end{scope}

\draw [thick]  (36,-1) -| (37,0) -| (38,1) -| (40,2) -| (36,-1);

\node at (39,-0.5) {$-1$};

\draw [thick, rounded corners]  (36.5,1.5) -- (37.5,1.5);
\draw [color=black,fill=black,thick] (37.5,1.5) circle (.4ex);
\node [draw, circle, fill = white, inner sep = 1.2pt] at (36.5,1.5) { };    
\node [draw, circle, fill = white, inner sep = 1.2pt] at (38.5,1.5) { };  
\draw [color=red, thick, rounded corners]  (36.5,-0.5) |- (37.5,0.5);   
\draw [color=red,fill=black,thick] (37.5,0.5) circle (.4ex);
\node [color=red,draw, circle, fill = white, inner sep = 1.2pt] at (36.5,-0.5) { }; 
\node [color=red,draw, circle, fill = white, inner sep = 1.2pt] at (39.5,1.5) { };

\begin{scope}
   \clip (42,-1) -| (43,0) -| (44,1) -| (46,2) -| (42,-1);
    \draw [color=black!25] (42,-1) grid (46,2);
\end{scope}

\draw [thick] (42,-1) -| (43,0) -| (44,1) -| (46,2) -| (42,-1);

\draw [thick, rounded corners]  (43.5,1.5) -- (44.5,1.5);
\draw [color=black,fill=black,thick] (44.5,1.5) circle (.4ex);
\node [draw, circle, fill = white, inner sep = 1.2pt] at (43.5,1.5) { };    
\node [draw, circle, fill = white, inner sep = 1.2pt] at (42.5,1.5) { };  
\draw [color=red, thick, rounded corners]  (42.5,-0.5) |- (43.5,0.5);   
\draw [color=red,fill=black,thick] (43.5,0.5) circle (.4ex);
\node [color=red,draw, circle, fill = white, inner sep = 1.2pt] at (42.5,-0.5) { }; 
\node [color=red,draw, circle, fill = white, inner sep = 1.2pt] at (45.5,1.5) { };

\node at (45,-0.5) {$-1$};

\end{tikzpicture}
\end{center}

\begin{center}
\begin{tikzpicture}[scale = 0.3]

\begin{scope}
   \clip (0,0) -| (1,3) -| (4,4) -| (0,0);
    \draw [color=black!25] (0,0) grid (4,4);
\end{scope}

\draw [thick] (0,0) -| (1,3) -| (4,4) -| (0,0);

\draw [thick, rounded corners]  (0.5,3.5) -- (1.5,3.5);
\draw [color=black,fill=black,thick] (1.5,3.5) circle (.4ex);
\node [draw, circle, fill = white, inner sep = 1.2pt] at (0.5,3.5) { };    
\node [draw, circle, fill = white, inner sep = 1.2pt] at (2.5,3.5) { };    
\draw [color=red, thick, rounded corners]  (0.5,0.5) -- (0.5,2.5);   
\draw [color=red,fill=black,thick] (0.5,2.5) circle (.4ex);
\node [color=red,draw, circle, fill = white, inner sep = 1.2pt] at (0.5,0.5) { }; 
\node [color=red,draw, circle, fill = white, inner sep = 1.2pt] at (3.5,3.5) { };   
    
\node at (2.5,1.5) {$+1$};

\begin{scope}
   \clip (6,0) -| (7,3) -| (10,4) -| (6,0);
    \draw [color=black!25] (6,0) grid (10,4);
\end{scope}

\draw [thick] (6,0) -| (7,3) -| (10,4) -| (6,0);

\draw [thick, rounded corners]  (7.5,3.5) -- (8.5,3.5);
\draw [color=black,fill=black,thick] (8.5,3.5) circle (.4ex);
\node [draw, circle, fill = white, inner sep = 1.2pt] at (7.5,3.5) { };    
\node [draw, circle, fill = white, inner sep = 1.2pt] at (6.5,3.5) { };    
\draw [color=red, thick, rounded corners]  (6.5,0.5) -- (6.5,2.5);   
\draw [color=red,fill=black,thick] (6.5,2.5) circle (.4ex);
\node [color=red,draw, circle, fill = white, inner sep = 1.2pt] at (6.5,0.5) { }; 
\node [color=red,draw, circle, fill = white, inner sep = 1.2pt] at (9.5,3.5) { };   
    
\node at (8.5,1.5) {$+1$};

\begin{scope}
   \clip (14,1) -| (15,2) -| (17,4) -| (14,1);
    \draw [color=black!25] (14,1) grid (17,4);
\end{scope}

\draw [thick] (14,1) -| (15,2) -| (17,4) -| (14,1);

\draw [thick, rounded corners]  (14.5,2.5) -- (14.5,3.5);
\draw [color=black,fill=black,thick] (14.5,3.5) circle (.4ex);
\node [draw, circle, fill = white, inner sep = 1.2pt] at (14.5,2.5) { };  
\node [draw, circle, fill = white, inner sep = 1.2pt] at (15.5,3.5) { };  
\draw [color=red, thick, rounded corners]  (15.5,2.5) -| (16.5,3.5);   
\draw [color=red,fill=black,thick] (16.5,3.5) circle (.4ex);
\node [color=red,draw, circle, fill = white, inner sep = 1.2pt] at (15.5,2.5) { }; 
\node [color=red,draw, circle, fill = white, inner sep = 1.2pt] at (14.5,1.5) { }; 
\node at (16,0.5) {$+1$};

\begin{scope}
   \clip (22,1) -| (24,3) -| (25,4) -| (22,1);
    \draw [color=black!25] (22,1) grid (25,4);
\end{scope}

\draw [thick] (22,1) -| (24,3) -| (25,4) -| (22,1);

\draw [thick, rounded corners]  (22.5,3.5) -- (23.5,3.5);
\draw [color=black,fill=black,thick] (23.5,3.5) circle (.4ex);
\node [draw, circle, fill = white, inner sep = 1.2pt] at (22.5,3.5) { };  
\node [draw, circle, fill = white, inner sep = 1.2pt] at (24.5,3.5) { };  
\node [color=red,draw, circle, fill = white, inner sep = 1.2pt] at (22.5,2.5) { }; 
\draw [color=red, thick, rounded corners]  (22.5,1.5) -| (23.5,2.5);  
\draw [color=red,fill=black,thick] (23.5,2.5) circle (.4ex);
\node [color=red,draw, circle, fill = white, inner sep = 1.2pt] at (22.5,1.5) { }; 
\node at (24.5,0.5) {$-1$};

\begin{scope}
   \clip (27,1) -| (29,3) -| (30,4) -| (27,1);
    \draw [color=black!25] (27,1) grid (30,4);
\end{scope}

\draw [thick] (27,1) -| (29,3) -| (30,4) -| (27,1);

\draw [thick, rounded corners]  (28.5,3.5) -- (29.5,3.5);
\draw [color=black,fill=black,thick] (29.5,3.5) circle (.4ex);
\node [draw, circle, fill = white, inner sep = 1.2pt] at (28.5,3.5) { };  
\node [draw, circle, fill = white, inner sep = 1.2pt] at (27.5,3.5) { };  
\node [color=red,draw, circle, fill = white, inner sep = 1.2pt] at (27.5,2.5) { }; 
\draw [color=red, thick, rounded corners]  (27.5,1.5) -| (28.5,2.5);  
\draw [color=red,fill=black,thick] (28.5,2.5) circle (.4ex);
\node [color=red,draw, circle, fill = white, inner sep = 1.2pt] at (27.5,1.5) { }; 
\node at (29.5,0.5) {$-1$};

\begin{scope}
   \clip (32,1) -| (34,3) -| (35,4) -| (32,1);
    \draw [color=black!25] (32,1) grid (35,4);
\end{scope}

\draw [thick] (32,1) -| (34,3) -| (35,4) -| (32,1);

\draw [thick, rounded corners]  (32.5,2.5) -- (32.5,3.5);
\draw [color=black,fill=black,thick] (32.5,3.5) circle (.4ex);
\node [draw, circle, fill = white, inner sep = 1.2pt] at (32.5,2.5) { };  
\node [draw, circle, fill = white, inner sep = 1.2pt] at (33.5,3.5) { };  
\node [color=red,draw, circle, fill = white, inner sep = 1.2pt] at (34.5,3.5) { }; 
\draw [color=red, thick, rounded corners]  (32.5,1.5) -| (33.5,2.5);  
\draw [color=red,fill=black,thick] (33.5,2.5) circle (.4ex);
\node [color=red,draw, circle, fill = white, inner sep = 1.2pt] at (32.5,1.5) { }; 
\node at (34.5,0.5) {$+1$};

\begin{scope}
   \clip (38,-1) -| (39,3) -| (41,4) -| (38,-1);
    \draw [color=black!25] (38,-1) grid (41,4);
\end{scope} 

\draw [thick] (38,-1) -| (39,3) -| (41,4) -| (38,-1);

\draw [thick, rounded corners]  (38.5,3.5) -- (39.5,3.5);
\draw [color=black,fill=black,thick] (39.5,3.5) circle (.4ex);
\node [draw, circle, fill = white, inner sep = 1.2pt] at (38.5,3.5) { };  
\node [draw, circle, fill = white, inner sep = 1.2pt] at (40.5,3.5) { };  
\node [color=red,draw, circle, fill = white, inner sep = 1.2pt] at (38.5,2.5) { }; 
\draw [color=red, thick, rounded corners]  (38.5,-0.5) -| (38.5,1.5);  
\draw [color=red,fill=black,thick] (38.5,1.5) circle (.4ex);
\node [color=red,draw, circle, fill = white, inner sep = 1.2pt] at (38.5,-0.5) { }; 
\node at (40.5,0.5) {$+1$};

\begin{scope}
   \clip (43,-1) -| (44,3) -| (46,4) -| (43,-1);
    \draw [color=black!25] (43,-1) grid (46,4);
\end{scope} 

\draw [thick] (43,-1) -| (44,3) -| (46,4) -| (43,-1);

\draw [thick, rounded corners]  (44.5,3.5) -- (45.5,3.5);
\draw [color=black,fill=black,thick] (45.5,3.5) circle (.4ex);
\node [draw, circle, fill = white, inner sep = 1.2pt] at (44.5,3.5) { };  
\node [draw, circle, fill = white, inner sep = 1.2pt] at (43.5,3.5) { };  
\node [color=red,draw, circle, fill = white, inner sep = 1.2pt] at (43.5,2.5) { }; 
\draw [color=red, thick, rounded corners]  (43.5,-0.5) -| (43.5,1.5);  
\draw [color=red,fill=black,thick] (43.5,1.5) circle (.4ex);
\node [color=red,draw, circle, fill = white, inner sep = 1.2pt] at (43.5,-0.5) { }; 
\node at (45.5,0.5) {$+1$};

\begin{scope}
   \clip (48,-1) -| (49,3) -| (51,4) -| (48,-1);
    \draw [color=black!25] (48,-1) grid (51,4);
\end{scope} 

\draw [thick] (48,-1) -| (49,3) -| (51,4) -| (48,-1);

\draw [thick, rounded corners]  (48.5,2.5) -- (48.5,3.5);
\draw [color=black,fill=black,thick] (48.5,3.5) circle (.4ex);
\node [draw, circle, fill = white, inner sep = 1.2pt] at (48.5,2.5) { };  
\node [draw, circle, fill = white, inner sep = 1.2pt] at (49.5,3.5) { };  
\node [color=red,draw, circle, fill = white, inner sep = 1.2pt] at (50.5,3.5) { }; 
\draw [color=red, thick, rounded corners]  (48.5,-0.5) -| (48.5,1.5);  
\draw [color=red,fill=black,thick] (48.5,1.5) circle (.4ex);
\node [color=red,draw, circle, fill = white, inner sep = 1.2pt] at (48.5,-0.5) { }; 
\node at (50.5,0.5) {$-1$};

\end{tikzpicture}
\end{center}

\begin{center}
\begin{tikzpicture}[scale = 0.3]

\begin{scope}
   \clip (0,0) -| (1,1) -| (2,4) -| (0,0);
    \draw [color=black!25] (0,0) grid (2,4);
\end{scope}

\draw [thick] (0,0) -| (1,1) -| (2,4) -| (0,0);

\draw [thick, rounded corners]  (0.5,2.5) -- (0.5,3.5);
\draw [color=black,fill=black,thick] (0.5,3.5) circle (.4ex);
\node [draw, circle, fill = white, inner sep = 1.2pt] at (0.5,2.5) { };    
\node [draw, circle, fill = white, inner sep = 1.2pt] at (1.5,3.5) { };    
\node [color=red,draw, circle, fill = white, inner sep = 1.2pt] at (1.5,2.5) { };  
\draw [color=red, thick, rounded corners]  (0.5,0.5) |- (1.5,1.5);   
\draw [color=red,fill=black,thick] (1.5,1.5) circle (.4ex);
\node [color=red,draw, circle, fill = white, inner sep = 1.2pt] at (0.5,0.5) { }; 
\node at (2.5,0) {$+1$};

\begin{scope}
   \clip (10,-1) -| (11,2) -| (12,4) -| (10,-1);
    \draw [color=black!25] (10,-1) grid (12,4);
\end{scope}

\draw [thick]  (10,-1) -| (11,2) -| (12,4) -| (10,-1);
\draw [thick, rounded corners]  (10.5,2.5) -- (10.5,3.5);
\draw [color=black,fill=black,thick] (10.5,3.5) circle (.4ex);
\node [draw, circle, fill = white, inner sep = 1.2pt] at (10.5,2.5) { };    
\node [draw, circle, fill = white, inner sep = 1.2pt] at (11.5,3.5) { }; 
\node [color=red,draw, circle, fill = white, inner sep = 1.2pt] at (11.5,2.5) { };   
\draw [color=red, thick, rounded corners]  (10.5,-0.5) -- (10.5,1.5);   
\draw [color=red,fill=black,thick] (10.5,1.5) circle (.4ex);
\node [color=red,draw, circle, fill = white, inner sep = 1.2pt] at (10.5,-0.5) { };  
\node at (12.5,0) {$-1$};

\begin{scope}
   \clip (20,-2) -| (21,3) -| (22,4) -| (20,-2);
    \draw [color=black!25] (20,-2) grid (22,4);
\end{scope}

\draw [thick]  (20,-2) -| (21,3) -| (22,4) -| (20,-2);
\draw [thick, rounded corners]  (20.5,2.5) -- (20.5,3.5);
\draw [color=black,fill=black,thick] (20.5,3.5) circle (.4ex);
\node [draw, circle, fill = white, inner sep = 1.2pt] at (20.5,2.5) { };    
\node [draw, circle, fill = white, inner sep = 1.2pt] at (21.5,3.5) { }; 
\node [color=red,draw, circle, fill = white, inner sep = 1.2pt] at (20.5,1.5) { };  
\draw [color=red, thick, rounded corners]  (20.5,-1.5) -- (20.5,0.5);  
\draw [color=red,fill=black,thick] (20.5,0.5) circle (.4ex);
\node [color=red,draw, circle, fill = white, inner sep = 1.2pt] at (20.5,-1.5) { };    
\node at (22.5,0) {$-1$};

\end{tikzpicture}
\end{center}

\caption{The Schur expansion in Example~\ref{atomic-expansion-example}.}
\label{fig:atomic-expansion}
\end{figure}
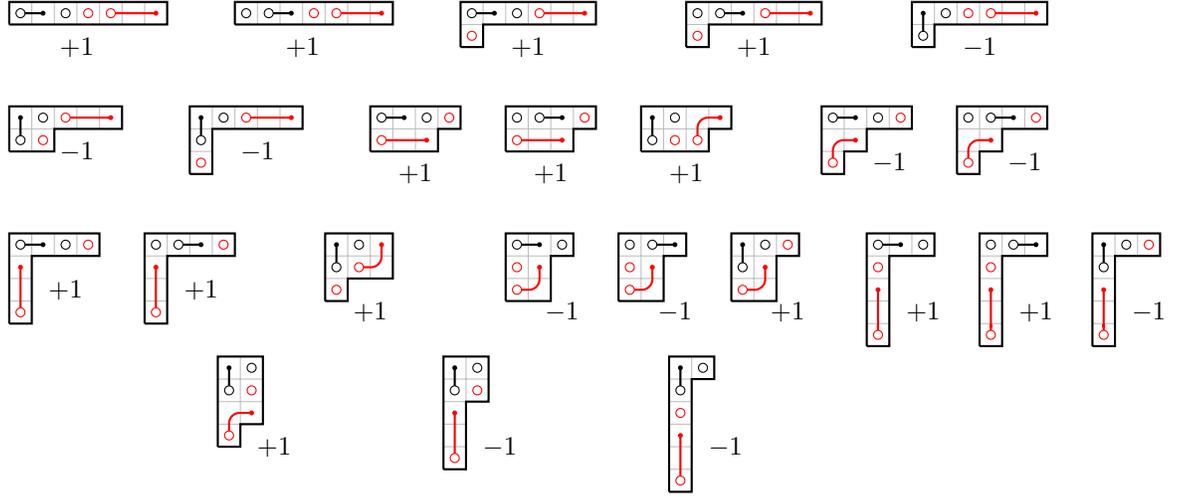

Before proceeding, we give an example of Corollary~\ref{atomic-schur-expansion}. Its combinatorics are a hybrid of those of the Path and Classical 
Murnaghan-Nakayama Rules.

\begin{example}
\label{atomic-expansion-example}
Let $(I,J) \in \symm_{7,5}$ with $I = [1,4,5,6,7]$ and $J = [2,5,6,4,7]$.
The path partition of $(I,J)$ is $\mu = (2,1)$ while the cycle partition is $\nu = (3,1)$.  
Figure~\ref{fig:atomic-expansion} illustrates how Corollary~\ref{atomic-schur-expansion} calculates the Schur expansion
\begin{equation*}
A_{7,I,J} = 2 s_7 + s_{61} - s_{52} - s_{511} + 3 s_{43} - 2 s_{421} + 2 s_{41^3} + s_{331} - s_{322} + s_{31^4} + s_{2^3 1} - s_{221^3} - s_{21^5}
\end{equation*} 
of $A_{7,I,J}$.  One first adds a 1-ribbon and a 2-ribbon (in either order) in a monotonic fashion; this is shown in black.
Once this is done, one adds a 1-ribbon and a 3-ribbon (in that order) in a fashion which may not be monotonic; this is shown in red.
The signs of both the black and red ribbons contribute to the expansion. In this case $m(\mu)! = 1$, so it does not make an appearance.
\end{example}

The coefficient of $s_{\lambda}$ in Equation~\eqref{cor:atomic-schur-eqn} is the trace of a natural linear operator.

\begin{corollary}
\label{trace-interpretation}
Let $(I,J) \in \symm_{n,k}$ be a partial permutation whose graph $G_n(I,J)$ consists of paths of sizes
$\mu = (\mu_1 \geq \mu_2 \geq \cdots )$ and cycles of sizes
$\nu = (\nu_1 \geq \nu_2 \geq \cdots )$ where $\mu \vdash a$ and $\nu \vdash b$. If $n = a  + b$ and $\lambda \vdash n$, we have
\begin{equation}
\chi^{\lambda}([I,J]) = m(\mu)! \cdot \sum_{\rho \, \vdash \, a}
 \vec{\chi}_{ \, \mu}^{ \, \rho} \cdot \chi_{\nu}^{\lambda/\rho}
\end{equation}
where $\chi^{\lambda}: \CC[\symm_n] \rightarrow \CC$ is the linearly extended irreducible character of $\symm_n$ and
$[I,J] \in \CC[\symm_n]$ is the group algebra element 
$[I,J] = \sum_{w(I) = J} w$.
\end{corollary}

Informally, the expression in Corollary~\ref{trace-interpretation} gives the value $\chi^{\lambda}[I,J]$ of the irreducible character
$\chi^{\lambda}$ of $\symm_n$ on the partial permutation $(I,J) \in \symm_{n,k}$.

\begin{proof}
Apply the definition of $A_{n,I,J}$ as a $p$-expansion, the classical Murnaghan-Nakayama Rule, and 
Corollary~\ref{atomic-schur-expansion}.
\end{proof}

When $(I,J) \in \symm_{n,0}$ is the partial permutation with $I = J = \varnothing$, the group algebra element $[I,J] = [\symm_n]_+$
is the sum over all permutations $w \in \symm_n$.  For any $\symm_n$-module $V$, the operator 
$[\symm_n]_+$ projects onto the trivial submodule $V^{\symm_n} \subseteq V$ and scales by $n!$.
On the other hand, Corollary~\ref{atomic-schur-expansion} gives $A_{n,I,J} = n! \cdot s_{n}$ since the only monotonic ribbon tiling consisting
of $n$ ribbons of size 1 is 
\begin{center}
\begin{tikzpicture}[scale = 0.3]
\begin{scope}
   \clip (0,0) -| (10,1) -|  (0,0);
    \draw [color=black!25] (0,0) grid (10,1);
\end{scope}
\draw [thick]  (0,0) -| (10,1) -| (0,0);
\node [draw, circle, fill = white, inner sep = 1.2pt] at (0.5,0.5) { };   
\node [draw, circle, fill = white, inner sep = 1.2pt] at (1.5,0.5) { };   
\node [draw, circle, fill = white, inner sep = 1.2pt] at (2.5,0.5) { };   
\node [draw, circle, fill = white, inner sep = 1.2pt] at (3.5,0.5) { };   
\node [draw, circle, fill = white, inner sep = 1.2pt] at (4.5,0.5) { };   
\node [draw, circle, fill = white, inner sep = 1.2pt] at (5.5,0.5) { };   
\node [draw, circle, fill = white, inner sep = 1.2pt] at (6.5,0.5) { };   
\node [draw, circle, fill = white, inner sep = 1.2pt] at (7.5,0.5) { };   
\node [draw, circle, fill = white, inner sep = 1.2pt] at (9.5,0.5) { };   
\node [draw, circle, fill = white, inner sep = 1.2pt] at (8.5,0.5) { };   
\end{tikzpicture}
\end{center}
which has shape $(n)$.
Since 
\begin{equation}
[I,J] \cdot V^{\lambda} = [\symm_n]_+ \cdot V^{\lambda} = 
(V^{\lambda})^{\symm_n} =
\begin{cases}
V^{\lambda} & \lambda = (n) \\
0 & \text{otherwise}
\end{cases}
\end{equation}
we have
\begin{equation}
\chi^{\lambda} ([I,J]) = \sum_{w \in \symm_n} \chi^{\lambda}(w) = \begin{cases}
n! & \lambda = (n) \\
0 & \text{otherwise}
\end{cases}
\end{equation}
which is in agreement with Corollary~\ref{trace-interpretation}.
At the other extreme, if $(I,J) \in \symm_{n,n} = \symm_n$ is a genuine (rather than merely partial) permutation, 
Corollary~\ref{trace-interpretation} reduces to the classical fact that the character table of $\symm_n$ is the transition matrix from the power sum to 
Schur bases of $\Lambda_n$.

The character evaluation in Corollary~\ref{trace-interpretation} can be na\"ively computed using the classical
Murhaghan-Nakayama rule via
\begin{equation}
\label{stupid-implementation}
\chi^{\lambda} \left( [I,J] \right) = \sum_{\substack{w \in \symm_n \\ w(I) = J}} \chi^{\lambda}(w).
\end{equation}
For $(I,J) \in \symm_{n,k}$ this sum contains $(n-k)!$ terms.
Grouping terms of \eqref{stupid-implementation} according to the cycle type of $w$ improves matters,
 but Corollary~\ref{trace-interpretation} still eliminates much (but not all) of the cancellation in \eqref{stupid-implementation}
 and is substantially more efficient.
 
The Littlewood-Richardson coefficients (and in particular their interpretation in terms of skew Schur functions)
yield a version
\begin{equation}
\label{less-stupid-implementation}
\chi^{\lambda} \left( [I,J]  \right) = m(\mu)! \cdot \sum_{\substack{\rho \, \vdash \, a \\ \tau \, \vdash \, b}} c_{\rho, \tau}^{\lambda} \cdot
 \vec{\chi}_{ \, \mu}^{ \, \rho} \cdot \chi_{\nu}^{\tau}
\end{equation}
of Corollary~\ref{trace-interpretation} which only involves straight shapes. However, the skew shape formula
in Corollary~\ref{trace-interpretation} is easier to apply.


\begin{remark}
\label{support-remark}
Theorem~\ref{path-murnaghan-nakayama}
implies Theorem~\ref{atomic-support} on the Schur support of the $A_{n,I,J}$.
Indeed, if $\mu \vdash s$ then any monotonic ribbon tiling whose ribbons are of size $\mu$ has at lest $\ell(\mu)$ boxes in its 
 first row.  Thus, if the coefficient of 
 $s_{\lambda}$ in the product $\vec{p}_{\mu} \cdot p_{\nu}$ is nonzero, the partition $\lambda$ must also have at least
 $\ell(\mu)$ boxes in its first row.  Now observe that if $(I,J) \in \symm_{n,k}$ is a partial permutation, the graph $G_n(I,J)$ will have at least
 $n-k$ paths so that in the factorization $A_{n,I,J} = \vec{p}_{\mu} \cdot p_{\nu}$ we have $\ell(\mu) \geq n-k$.
\end{remark}

\section{The $n \rightarrow \infty$ limit}
We will use the atomic functions $A_{n,I,J}$ to study
the asymptotic behavior of functions $f: \symm_n \rightarrow \CC$
 as $n \rightarrow \infty$.
These results are based on asymptotics of the atomic functions.
We examine the case of path power sums first.

For a partition $\lambda = (\lambda_1, \lambda_2, \dots )$ recall the padded partition
$\lambda[n] = ( n - |\lambda|, \lambda_1, \lambda_2, \dots )$ which is defined whenever 
$n \geq |\lambda| + \lambda_1$.  We also let $\lambda(n) \vdash n$ be the partition 
$\lambda(n) = (\lambda_1, \lambda_2, \dots , 1^{n-|\lambda|})$ given by appending $n-|\lambda|$ copies of 1 to the end of $\lambda$.
The partition $\lambda(n)$ is defined whenever $n \geq |\lambda|$.

\begin{corollary}
\label{degree-corollary}
Let $\mu = (\mu_1, \dots, \mu_r)$ be a partition whose $r$ parts are all 
of size $> 1$.   We have the Schur expansion
\begin{equation}
\frac{1}{(n-|\mu|)!} \times \vec{p}_{\mu(n)} = \sum_{|\lambda| \leq |\mu| - r} f_{\lambda}(n) \cdot s_{\lambda[n]}
\end{equation}
where each $f_{\lambda}(n)$ is a polynomial function of $n$ on the domain $\{n \geq 2  |\mu| - 2r \}$. 
The degree of $f_{\lambda}(n)$ is $\leq r$, with this maximum
degree uniquely achieved when $\lambda = \varnothing$.
\end{corollary}

\begin{proof}
Consider a monotonic
ribbon tiling $T$ of the padded partition $\lambda[n]$ with ribbon sizes $\mu(n)$.
We decompose $T$ into two pieces
$$
T = T_0 \sqcup T_1
$$
by letting $T_1$ be the union of all ribbons which are entirely contained in the first row of $T$ and letting $T_0$ be the union of the remaining
ribbons in $T$.  
Since $T$ is monotonic, the ribbons in $T_1$ will be at the east end of the first row of $T$ and the remaining ribbons in $T_1$ 
will be a monotonic ribbon tiling of a shape contained in $\lambda[n]$.

Schematically, the decomposition $T = T_0 \sqcup T_1$ 
is shown in Figure~\ref{fig:tiling-decomposition}, where the length of the first row of $\lambda[n]$,
and the number of size 1 ribbons coming from $\mu(n)$, both increase with $n$.
Ribbons only partially contained in the first row are not included in $T_1$.
Ribbons contained entirely within $T_0$ are not shown.

\begin{figure}
\begin{center}
\begin{tikzpicture}[scale = 0.3]

\begin{scope}

\clip (0,0) -| (2,3) -| (5,4) -| (7,6) -| (40,7) -| (0,0) ;
  \draw [color=black!25] (0,0) grid (40,7);

\end{scope}

\draw [thick]  (0,0) -| (2,3) -| (5,4) -| (7,6) -| (40,7) -| (0,0) ;

\draw [dashed] (10,3.5) -- (10,8);

  \draw [thick, rounded corners]  (9.5,6.5) -| (6.5,4.5) -- (5.5,4.5) ;
  \draw [color=black,fill=black,thick] (9.5,6.5) circle (.4ex);
  \node [draw, circle, fill = white, inner sep = 1.2pt] at (5.5,4.5) { }; 
  
  \node at (7,2) {$T_0$}; 
  
   \draw [thick, rounded corners]  (10.5,6.5) -- (12.5,6.5) ;
  \draw [color=black,fill=black,thick] (12.5,6.5) circle (.4ex);
  \node [draw, circle, fill = white, inner sep = 1.2pt] at (10.5,6.5) { }; 
  \node [draw, circle, fill = white, inner sep = 1.2pt] at (13.5,6.5) { }; 
  \node [draw, circle, fill = white, inner sep = 1.2pt] at (14.5,6.5) { }; 
  \node [draw, circle, fill = white, inner sep = 1.2pt] at (15.5,6.5) { }; 
  \node [draw, circle, fill = white, inner sep = 1.2pt] at (16.5,6.5) { }; 
     \draw [thick, rounded corners]  (17.5,6.5) -- (18.5,6.5) ;
  \draw [color=black,fill=black,thick] (18.5,6.5) circle (.4ex);
  \node [draw, circle, fill = white, inner sep = 1.2pt] at (17.5,6.5) { }; 
  \node [draw, circle, fill = white, inner sep = 1.2pt] at (19.5,6.5) { }; 
  \node [draw, circle, fill = white, inner sep = 1.2pt] at (20.5,6.5) { }; 
   \node [draw, circle, fill = white, inner sep = 1.2pt] at (21.5,6.5) { }; 
  \node [draw, circle, fill = white, inner sep = 1.2pt] at (22.5,6.5) { }; 
   \node [draw, circle, fill = white, inner sep = 1.2pt] at (23.5,6.5) { }; 
     \node [draw, circle, fill = white, inner sep = 1.2pt] at (24.5,6.5) { }; 
  \draw [thick, rounded corners]  (25.5,6.5) -- (27.5,6.5) ;
  \draw [color=black,fill=black,thick] (27.5,6.5) circle (.4ex);
  \node [draw, circle, fill = white, inner sep = 1.2pt] at (25.5,6.5) { }; 
  \draw [thick, rounded corners]  (28.5,6.5) -- (29.5,6.5) ;
  \draw [color=black,fill=black,thick] (29.5,6.5) circle (.4ex);
  \node [draw, circle, fill = white, inner sep = 1.2pt] at (28.5,6.5) { }; 
    \node [draw, circle, fill = white, inner sep = 1.2pt] at (30.5,6.5) { }; 
    \node [draw, circle, fill = white, inner sep = 1.2pt] at (31.5,6.5) { }; 
 \node [draw, circle, fill = white, inner sep = 1.2pt] at (32.5,6.5) { }; 
 \node [draw, circle, fill = white, inner sep = 1.2pt] at (33.5,6.5) { }; 
  \node [draw, circle, fill = white, inner sep = 1.2pt] at (34.5,6.5) { }; 
  \node [draw, circle, fill = white, inner sep = 1.2pt] at (35.5,6.5) { }; 
   \node [draw, circle, fill = white, inner sep = 1.2pt] at (36.5,6.5) { }; 
     \draw [thick, rounded corners]  (37.5,6.5) -- (38.5,6.5) ;
  \draw [color=black,fill=black,thick] (38.5,6.5) circle (.4ex);
  \node [draw, circle, fill = white, inner sep = 1.2pt] at (37.5,6.5) { }; 
  \node [draw, circle, fill = white, inner sep = 1.2pt] at (39.5,6.5) { };

\node at (25,4.5) {$T_1$}; 

\end{tikzpicture}
\end{center}
\caption{The decomposition $T = T_0 \sqcup T_1$ of a monotonic tiling used in the proof of
Proposition~\ref{degree-corollary}.}
\label{fig:tiling-decomposition}
\end{figure}
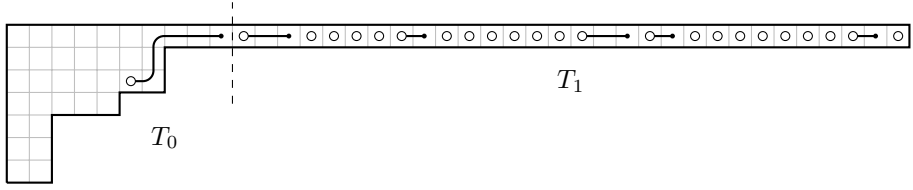

The monotonicity condition implies that only a finite set of tilings $T_0$ can appear among the family of all monotonic
tilings $T = T_0 \sqcup T_1$ of all partition shapes $\{ \lambda[n] \,:\, n \geq 0 \}$. 
We fix a `frozen tiling' $T_0$ and consider the possibilities for $T_1$ as a function of $n$.
Observe (as in Figure~\ref{fig:tiling-decomposition}) that a ribbon of size $> 1$ can extend strictly east of the portion of the 
Young diagram $\lambda[n]$ occupied by $\lambda$.
We assume that $n \geq 2 |\mu| - 2 r$, 
which guarantees that all possible frozen tilings $T_0$ actually appear among monotonic tilings with ribbon sizes $\mu(n)$.

For fixed $T_0$, we form a tiling of $T_1$ such extending $T_0$ such that the overall shape of $T = T_0 \sqcup T_1$ is 
$\lambda[n]$ by selecting an order on the ribbons of sizes $\mu(n)$ not contained in $T_0$ and placing these ribbons in $T_1$ from left to right.
The overall tiling $T$ contains ribbons of sizes $m_i(\mu)$ ribbons of size $i$ for each $i > 1$ together with $n - |\mu|$ singleton ribbons.
Each of these ribbons is in $T_0$ or in $T_1$.  Writing $a_i$ for the number of size $i$ ribbons in $T_0$, 
\begin{multline}
\label{multinomial-coefficient}
\# \{ \text{ribbon tilings $T = T_0 \sqcup T_1$ of $\lambda[n]$ extending $T_0$} \}  \\ = 
\binom{n - |\mu| + r - a_1 - a_2 - \cdots}{ n - |\mu| - a_1, m_2(\mu) - a_2, m_3(\mu) - a_3, \dots }
\end{multline}
where the multinomial coefficient counts ways order ribbons within $T_1$.
Since $\sign(T) = \sign(T_0)$,
Theorem~\ref{path-murnaghan-nakayama} says that $T_0$ will contribute 
\begin{equation}
\label{polynomial-expression}
\sign(T_0) \cdot m_2(\mu)! m_3(\mu)! \cdots \times
\binom{ n - |\mu| + r - a_1 - a_2 - \cdots}{ n - |\mu| - a_1, m_2(\mu) - a_2, m_3(\mu) - a_3, \dots }
\end{equation}
to the coefficient of $s_{\lambda[n]}$ in $\frac{1}{(n - |\mu|)!} \times \vec{p}_{\mu(n)}$.
The expression \eqref{polynomial-expression}
is a polynomial in $n$ of degree equal to
\begin{equation}
\label{difference-degree}
r - a_2 - a_3 - \cdots = 
r - \text{(number of ribbons in $T_0$ of size $> 1$)}.
\end{equation}
Summing over $T_0$,
 the overall coefficient $f_{\lambda}(n)$ of $s_{\lambda[n]}$ in $\frac{1}{(n - |\mu|)!} \times \vec{p}_{\mu(n)}$ is a polynomial in $n$.
The degree \eqref{difference-degree}
is uniquely maximized (and equal to $r$) when $\lambda = \varnothing$ is the empty partition and $T_0 = \varnothing$ is the empty tiling.
\end{proof}

We make two remarks on Corollary~\ref{degree-corollary}. The first is on the domain of polynomiality of the
coefficients $f_{\lambda}(n)$ and the second is on their degree bound.

\begin{remark}
\label{polynomiality-domain-remark}
The coefficients $f_{\lambda}(n)$ appearing in Corollary~\ref{degree-corollary} are often polynomials in $n$ on a larger domain than the set 
 $\{ n \geq 2 |\mu| - 2 r \}$. If $N$ is the maximum integer such that there is a monotonic ribbon tiling $T$ of $\lambda[N]$ with ribbons 
 of sizes $\mu(N)$ whose decomposition $T = T_0 \sqcup T_1$ satisfies $T_1 = \varnothing$, then $f_{\lambda}(n)$ 
 is a polynomial on the domain $\{ n \geq N \}$.
 The integer $N$ depends on both $\lambda$ and $\mu$. Since any monotonic tiling $T$ with $T_1 = \varnothing$ will have at most one singleton
 ribbon in any column, and at least $r = \ell(\mu)$ columns free of singleton ribbons, we have $N \leq 2 |\mu| - 2 r$, but this upper bound on $N$
 is not tight.
 For example, when $\lambda = \varnothing$ (so that $\lambda[n]$ is a single row for all
 $n$) any monotonic ribbon tiling $T = T_0 \sqcup T_1$ of $\lambda[n]$ satisfies
 $T = T_1$ which means $N = 0$ and $f_{\varnothing}(n)$ is a polynomial for all values of $n$.
 
 Determining the integer $N = N(\lambda,\mu)$ that extends the domain of polynomiality of $f_{\lambda}(n)$ seems to be a tricky 
 combinatorial problem on monotonic tilings.  Corollary~\ref{degree-corollary} is stated with the bound $N \leq 2 |\mu| - 2 r$ which is 
 independent of $\lambda$ for simplicity and to concord with a result of Gaetz-Ryba \cite[Thm 1.1 (b)]{GR} which we will generalize.
\end{remark}

\begin{remark}
\label{path-degree-remark}
Finding the degree of $f_{\lambda}(n)$ in Corollary~\ref{degree-corollary} involves minimizing the number of size $> 1$ ribbons 
in the $T_0$ portion of a monotonic tiling of $\lambda[n]$ with ribbons of lengths $\mu(n)$ as $n \rightarrow \infty$.
Equivalently, we aim to maximize the number of size $> 1$ ribbons in the $T_1$ portion.
While the $\mu$-dependence of this combinatorial problem seems to be subtle, we have the bound
\begin{equation}
\label{durfee-bound}
\deg f_{\lambda}(n) \leq r - \text{side length of the largest Durfee square contained in $\lambda$}
\end{equation}
which is independent of $\mu$.
Recall that a {\em Durfee square} of side length $a$ contained in $\lambda$ is an $a \times a$ grid
contained in the northwest corner of its Young diagram.

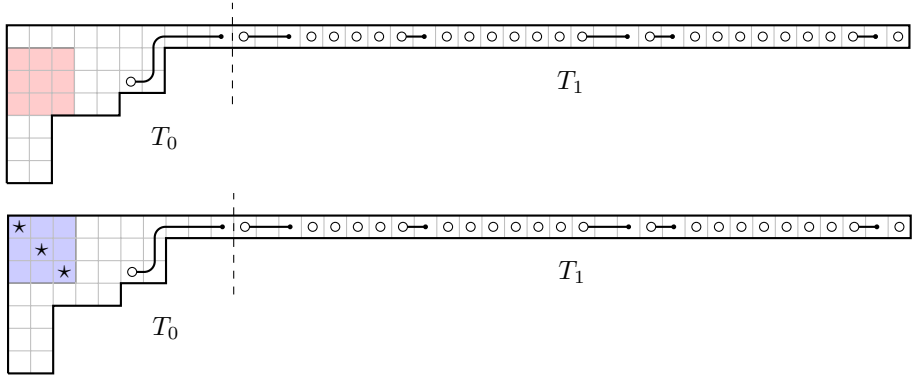
\begin{figure}
\begin{center}
\begin{tikzpicture}[scale = 0.3]

\draw [fill = red!20] (0,3) -- (0,6) -- (3,6) -- (3,3) -- (0,3);

\begin{scope}

\clip (0,0) -| (2,3) -| (5,4) -| (7,6) -| (40,7) -| (0,0) ;
  \draw [color=black!25] (0,0) grid (40,7);

\end{scope}

\draw [thick]  (0,0) -| (2,3) -| (5,4) -| (7,6) -| (40,7) -| (0,0) ;

\draw [dashed] (10,3.5) -- (10,8);

  \draw [thick, rounded corners]  (9.5,6.5) -| (6.5,4.5) -- (5.5,4.5) ;
  \draw [color=black,fill=black,thick] (9.5,6.5) circle (.4ex);
  \node [draw, circle, fill = white, inner sep = 1.2pt] at (5.5,4.5) { }; 
  
  \node at (7,2) {$T_0$}; 
  
   \draw [thick, rounded corners]  (10.5,6.5) -- (12.5,6.5) ;
  \draw [color=black,fill=black,thick] (12.5,6.5) circle (.4ex);
  \node [draw, circle, fill = white, inner sep = 1.2pt] at (10.5,6.5) { }; 
  \node [draw, circle, fill = white, inner sep = 1.2pt] at (13.5,6.5) { }; 
  \node [draw, circle, fill = white, inner sep = 1.2pt] at (14.5,6.5) { }; 
  \node [draw, circle, fill = white, inner sep = 1.2pt] at (15.5,6.5) { }; 
  \node [draw, circle, fill = white, inner sep = 1.2pt] at (16.5,6.5) { }; 
     \draw [thick, rounded corners]  (17.5,6.5) -- (18.5,6.5) ;
  \draw [color=black,fill=black,thick] (18.5,6.5) circle (.4ex);
  \node [draw, circle, fill = white, inner sep = 1.2pt] at (17.5,6.5) { }; 
  \node [draw, circle, fill = white, inner sep = 1.2pt] at (19.5,6.5) { }; 
  \node [draw, circle, fill = white, inner sep = 1.2pt] at (20.5,6.5) { }; 
   \node [draw, circle, fill = white, inner sep = 1.2pt] at (21.5,6.5) { }; 
  \node [draw, circle, fill = white, inner sep = 1.2pt] at (22.5,6.5) { }; 
   \node [draw, circle, fill = white, inner sep = 1.2pt] at (23.5,6.5) { }; 
     \node [draw, circle, fill = white, inner sep = 1.2pt] at (24.5,6.5) { }; 
  \draw [thick, rounded corners]  (25.5,6.5) -- (27.5,6.5) ;
  \draw [color=black,fill=black,thick] (27.5,6.5) circle (.4ex);
  \node [draw, circle, fill = white, inner sep = 1.2pt] at (25.5,6.5) { }; 
  \draw [thick, rounded corners]  (28.5,6.5) -- (29.5,6.5) ;
  \draw [color=black,fill=black,thick] (29.5,6.5) circle (.4ex);
  \node [draw, circle, fill = white, inner sep = 1.2pt] at (28.5,6.5) { }; 
    \node [draw, circle, fill = white, inner sep = 1.2pt] at (30.5,6.5) { }; 
    \node [draw, circle, fill = white, inner sep = 1.2pt] at (31.5,6.5) { }; 
 \node [draw, circle, fill = white, inner sep = 1.2pt] at (32.5,6.5) { }; 
 \node [draw, circle, fill = white, inner sep = 1.2pt] at (33.5,6.5) { }; 
  \node [draw, circle, fill = white, inner sep = 1.2pt] at (34.5,6.5) { }; 
  \node [draw, circle, fill = white, inner sep = 1.2pt] at (35.5,6.5) { }; 
   \node [draw, circle, fill = white, inner sep = 1.2pt] at (36.5,6.5) { }; 
     \draw [thick, rounded corners]  (37.5,6.5) -- (38.5,6.5) ;
  \draw [color=black,fill=black,thick] (38.5,6.5) circle (.4ex);
  \node [draw, circle, fill = white, inner sep = 1.2pt] at (37.5,6.5) { }; 
  \node [draw, circle, fill = white, inner sep = 1.2pt] at (39.5,6.5) { };

\node at (25,4.5) {$T_1$}; 

\end{tikzpicture}
\end{center}

\begin{center}
\begin{tikzpicture}[scale = 0.3]

\draw [fill = blue!20] (0,4) -- (0,7) -- (3,7) -- (3,4) -- (0,4);

\node at (0.5,6.5) {$\star$};
\node at (1.5,5.5) {$\star$};
\node at (2.5,4.5) {$\star$};

\begin{scope}

\clip (0,0) -| (2,3) -| (5,4) -| (7,6) -| (40,7) -| (0,0) ;
  \draw [color=black!25] (0,0) grid (40,7);

\end{scope}

\draw [thick]  (0,0) -| (2,3) -| (5,4) -| (7,6) -| (40,7) -| (0,0) ;

\draw [dashed] (10,3.5) -- (10,8);

  \draw [thick, rounded corners]  (9.5,6.5) -| (6.5,4.5) -- (5.5,4.5) ;
  \draw [color=black,fill=black,thick] (9.5,6.5) circle (.4ex);
  \node [draw, circle, fill = white, inner sep = 1.2pt] at (5.5,4.5) { }; 
  
  \node at (7,2) {$T_0$}; 
  
   \draw [thick, rounded corners]  (10.5,6.5) -- (12.5,6.5) ;
  \draw [color=black,fill=black,thick] (12.5,6.5) circle (.4ex);
  \node [draw, circle, fill = white, inner sep = 1.2pt] at (10.5,6.5) { }; 
  \node [draw, circle, fill = white, inner sep = 1.2pt] at (13.5,6.5) { }; 
  \node [draw, circle, fill = white, inner sep = 1.2pt] at (14.5,6.5) { }; 
  \node [draw, circle, fill = white, inner sep = 1.2pt] at (15.5,6.5) { }; 
  \node [draw, circle, fill = white, inner sep = 1.2pt] at (16.5,6.5) { }; 
     \draw [thick, rounded corners]  (17.5,6.5) -- (18.5,6.5) ;
  \draw [color=black,fill=black,thick] (18.5,6.5) circle (.4ex);
  \node [draw, circle, fill = white, inner sep = 1.2pt] at (17.5,6.5) { }; 
  \node [draw, circle, fill = white, inner sep = 1.2pt] at (19.5,6.5) { }; 
  \node [draw, circle, fill = white, inner sep = 1.2pt] at (20.5,6.5) { }; 
   \node [draw, circle, fill = white, inner sep = 1.2pt] at (21.5,6.5) { }; 
  \node [draw, circle, fill = white, inner sep = 1.2pt] at (22.5,6.5) { }; 
   \node [draw, circle, fill = white, inner sep = 1.2pt] at (23.5,6.5) { }; 
     \node [draw, circle, fill = white, inner sep = 1.2pt] at (24.5,6.5) { }; 
  \draw [thick, rounded corners]  (25.5,6.5) -- (27.5,6.5) ;
  \draw [color=black,fill=black,thick] (27.5,6.5) circle (.4ex);
  \node [draw, circle, fill = white, inner sep = 1.2pt] at (25.5,6.5) { }; 
  \draw [thick, rounded corners]  (28.5,6.5) -- (29.5,6.5) ;
  \draw [color=black,fill=black,thick] (29.5,6.5) circle (.4ex);
  \node [draw, circle, fill = white, inner sep = 1.2pt] at (28.5,6.5) { }; 
    \node [draw, circle, fill = white, inner sep = 1.2pt] at (30.5,6.5) { }; 
    \node [draw, circle, fill = white, inner sep = 1.2pt] at (31.5,6.5) { }; 
 \node [draw, circle, fill = white, inner sep = 1.2pt] at (32.5,6.5) { }; 
 \node [draw, circle, fill = white, inner sep = 1.2pt] at (33.5,6.5) { }; 
  \node [draw, circle, fill = white, inner sep = 1.2pt] at (34.5,6.5) { }; 
  \node [draw, circle, fill = white, inner sep = 1.2pt] at (35.5,6.5) { }; 
   \node [draw, circle, fill = white, inner sep = 1.2pt] at (36.5,6.5) { }; 
     \draw [thick, rounded corners]  (37.5,6.5) -- (38.5,6.5) ;
  \draw [color=black,fill=black,thick] (38.5,6.5) circle (.4ex);
  \node [draw, circle, fill = white, inner sep = 1.2pt] at (37.5,6.5) { }; 
  \node [draw, circle, fill = white, inner sep = 1.2pt] at (39.5,6.5) { };

\node at (25,4.5) {$T_1$}; 

\end{tikzpicture}
\end{center}

\caption{The bound \eqref{durfee-bound} on the degree of $f^{\lambda}(n)$ from Corollary~\ref{degree-corollary}.}
\label{fig:durfee}
\end{figure}

The justification for the bound \eqref{durfee-bound} is shown in Figure~\ref{fig:durfee}.
The Durfee square of $\lambda$ (which excludes the first row of $\lambda[n]$) is shown in red.
The blue square is the top-justification of the red square and has the same dimensions.
In any monotonic tiling of $T_0$, each box on the starred blue diagonal will be occupied
by distinct ribbons, all of size $> 1$.
Equation~\eqref{difference-degree} applies to prove \eqref{durfee-bound}.
\end{remark}

For the final result of this section, we bootstrap Corollary~\ref{degree-corollary} from path power sums to atomic symmetric functions.
If $(I,J) \in \symm_n$ is a partial permutation, we abuse notation to write $I \cup J$ for the {\bf set} of letters which are contained in
either $I$ or $J$, e.g. if $(I,J) = (253,423)$ then $I \cup J =  \{ 2,3,4,5\}$.

\begin{figure}
\begin{center}
\begin{tikzpicture}[scale = 0.3]

\draw [fill = green!20] (0,0) -- (2,0) -- (2,3) -- (5,3) -- (5,4) -- (7,4) -- (7,6) -- (0,6) -- (0,0);

\begin{scope}

\clip (0,0) -| (2,3) -| (5,4) -| (7,6) -| (40,7) -| (0,0) ;
  \draw [color=black!25] (0,0) grid (40,7);

\end{scope}

\draw [thick]  (0,0) -| (2,3) -| (5,4) -| (7,6) -| (40,7) -| (0,0) ;

\draw [dashed] (7,0) -- (7,8);

  \draw [thick, rounded corners]  (0.5,0.5) -- (0.5,6.5) ;
  \draw [color=black,fill=black,thick] (0.5,6.5) circle (.4ex);
  \node [draw, circle, fill = white, inner sep = 1.2pt] at (0.5,0.5) { }; 
  
  \draw [thick, rounded corners]  (1.5,0.5) -- (1.5,6.5) ;
  \draw [color=black,fill=black,thick] (1.5,6.5) circle (.4ex);
  \node [draw, circle, fill = white, inner sep = 1.2pt] at (1.5,0.5) { }; 
  
    \draw [thick, rounded corners]  (2.5,3.5) -- (2.5,6.5) ;
  \draw [color=black,fill=black,thick] (2.5,6.5) circle (.4ex);
  \node [draw, circle, fill = white, inner sep = 1.2pt] at (2.5,3.5) { }; 
  
  \draw [thick, rounded corners]  (3.5,3.5) -- (3.5,6.5) ;
  \draw [color=black,fill=black,thick] (3.5,6.5) circle (.4ex);
  \node [draw, circle, fill = white, inner sep = 1.2pt] at (3.5,3.5) { }; 
  
  \draw [thick, rounded corners]  (4.5,3.5) -- (4.5,6.5) ;
  \draw [color=black,fill=black,thick] (4.5,6.5) circle (.4ex);
  \node [draw, circle, fill = white, inner sep = 1.2pt] at (4.5,3.5) { }; 
  
  \draw [thick, rounded corners]  (5.5,4.5) -- (5.5,6.5) ;
  \draw [color=black,fill=black,thick] (5.5,6.5) circle (.4ex);
  \node [draw, circle, fill = white, inner sep = 1.2pt] at (5.5,4.5) { }; 
  
  \draw [thick, rounded corners]  (6.5,4.5) -- (6.5,6.5) ;
  \draw [color=black,fill=black,thick] (6.5,6.5) circle (.4ex);
  \node [draw, circle, fill = white, inner sep = 1.2pt] at (6.5,4.5) { }; 
  
  
  \node at (5,1) {$T_0$};

\node at (25,4.5) {$T_1$, all ribbons have sizes 1 or 2}; 

\end{tikzpicture}
\end{center}
\caption{The greedy tiling.}
\label{fig:greedy}
\end{figure}
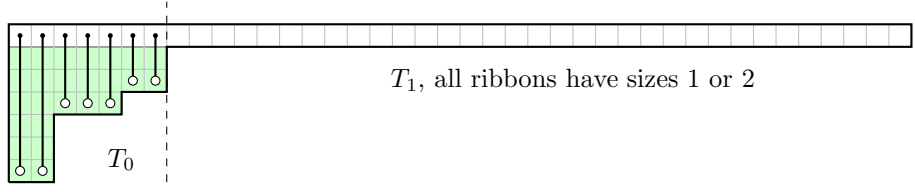

\begin{corollary}
\label{atomic-asymptotics}
Let $(I,J) \in \symm_{n,k}$ be a partial permutation. We have
\begin{equation}
\label{atomic-asymptotic-expansion-degrees}
\frac{1}{(n-|I \cup J|)!} \times A_{n,I,J} = \sum_{|\lambda| \leq k} g_{\lambda}(n) \cdot s_{\lambda[n]}
\end{equation}
where $g_{\lambda}(n)$ is a polynomial in $n$ on the domain $\{ n \geq 2 k\}$.
The degree of $g_{\lambda}(n)$ is at most $k - |\lambda|$.
\end{corollary}

\begin{proof}
Let $\nu$ be the cycle partition of $(I,J)$ and $\mu$ be the {\em reduced} path partition of $(I,J)$, so that 
the parts of $\mu$ record the path lengths in $G_n(I,J)$ which are $> 1$ in weakly decreasing order.
The path partition of $(I,J)$ is therefore $\mu(n - |\nu|)$.

By Proposition~\ref{atomic-factorization} we have the factorization
 $A_{n,I,J} = \vec{p}_{\mu(n-|\nu|)} \cdot p_{\nu}$.
Corollary~\ref{degree-corollary} states that the Schur expansion of $\vec{p}_{\mu(n - |\nu|)}$ has the form
\begin{equation}
\label{atomic-asymptotics-one}
\frac{1}{m_1( \mu(n-|\nu|) )!} \times \vec{p}_{\mu(n-|\nu|)} = \sum_{|\lambda| \, \leq \, |\mu| - \ell(\mu)} 
f_{\lambda}(n) \cdot s_{\lambda[n - |\nu|]}
\end{equation}
where the coefficients $f_{\lambda}(n)$ are polynomials in the regime $n-|\nu| \geq 2 \cdot |\mu| - 2 \cdot \ell(\mu)$.
We need to analyze what happens 
when we multiply both size of Equation~\eqref{atomic-asymptotics-one} 
by the classical power sum $p_{\nu}$.

Given an integer $n$ such that 
\begin{equation*}
n-|\nu| \geq 2 \cdot |\mu| - 2 \cdot \ell(\mu) \quad \Leftrightarrow \quad n  \geq 2 \cdot |\mu| - 2 \cdot \ell(\mu) + |\nu|
\end{equation*}
consider what happens to Equation~\eqref{atomic-asymptotics-one}
under the transition $n \rightsquigarrow n+1$. The coefficients $f_{\lambda}(n)$ are polynomials in this regime and the first rows
$\lambda[n - |\nu|]_1$ of the partitions 
indexing the Schur functions $s_{\lambda[n - |\nu|]}$ increase in length by 1, i.e. $\lambda[n-|\nu|+1]_1 = \lambda[n-|\nu|]_1 + 1$. Iterating  $|\nu|$
times, when
\begin{equation*}
n-|\nu| \geq 2 \cdot |\mu| - 2 \cdot \ell(\mu) + |\nu| \quad \Leftrightarrow \quad n  \geq 2 \cdot |\mu| - 2 \cdot \ell(\mu) + 2 \cdot |\nu|
\end{equation*} 
every partition $\lambda[n-|\nu|]$ indexing a Schur function in Equation~\eqref{atomic-asymptotics-one} satisfies 
\begin{equation*}
\lambda[n-|\nu]]_1 \geq \lambda[n-|\nu|]_2 + |\nu| 
\end{equation*}
so that the Young diagram of $\lambda[n - |\nu|]$ has the form below. 
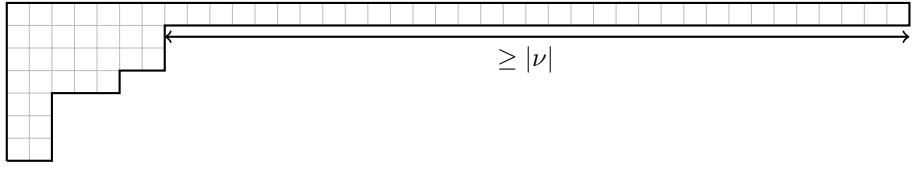
\begin{center}
\begin{tikzpicture}[scale = 0.3]

\begin{scope}

\clip (0,0) -| (2,3) -| (5,4) -| (7,6) -| (40,7) -| (0,0) ;
  \draw [color=black!25] (0,0) grid (40,7);

\end{scope}

\draw [thick]  (0,0) -| (2,3) -| (5,4) -| (7,6) -| (40,7) -| (0,0) ;

\draw [<->, thick] (7,5.5) -- (40,5.5);

\node at (23,4.5) {$\geq |\nu|$}; 
\end{tikzpicture}

\end{center}
By the classical Murnaghan-Nakayama Rule, calculating the Schur expansion of $s_{\lambda[n-|\nu|]} \cdot p_{\nu}$
involves the signed addition of ribbons of sizes $\nu_1, \nu_2, \dots $ to the Young diagram $\lambda[n-|\nu|]$, resulting in
\begin{equation}
\label{atomic-asymptotics-two}
s_{\lambda[n-|\nu|]} \cdot p_{\nu} = \sum_{\rho} c_{n,\lambda,\rho} \cdot s_{\rho[n]}
\end{equation}
where the sum is over partitions $\rho$ satisfying $|\rho| \leq \min(n - \rho_1, |\lambda| + |\nu|)$.
When $n$ belongs to the range $\{n  \geq 2 \cdot |\mu| - 2 \cdot \ell(\mu) + 2 \cdot |\nu|\}$ so that 
the first two rows in the Young diagram of $\lambda[n-|\nu|]$ differ by at least $|\nu$ boxes  as above, 
the coefficients $c_{n,\lambda,\rho}$ 
in Equation~\eqref{atomic-asymptotics-two} satisfy
\begin{equation}
c_{n,\lambda,\rho} = c_{n+1,\lambda,\rho}
\end{equation}
since the total number of boxes added by the classical Murnaghan-Nakayama rule is $|\nu|$.
Since
\begin{equation}
n - |I \cup J| = m_1(\mu(n-|\nu|))
\end{equation}
and
\begin{equation}
2 \cdot |\mu| - 2 \cdot \ell(\mu) + 2 \cdot |\nu| = 2 \cdot ( |\mu| - \ell(\mu) + |\nu| ) = 2 \cdot k,
\end{equation}
multiplying Equation~\eqref{atomic-asymptotics-one} by $p_{\nu}$ yields
 the  expansion 
\begin{equation}
\frac{1}{(n-|I \cup J|)!} \times A_{n,I,J} = \sum_{|\lambda| \leq k} g_{\lambda}(n) \cdot s_{\lambda[n]}
\end{equation}
where the coefficients $g_{\lambda}(n)$ are polynomials for $n \geq 2k$.

It remains to prove the degree bound.  Fix a partition $\lambda$ with $|\lambda| < k$.
By Equation~\eqref{difference-degree} in the proof of Corollary~\ref{degree-corollary}, the maximum possible 
degree of $g_{\lambda}(n)$ over all $(I,J) \in \symm_{n,k}$ occurs when $G_n(I,J)$ has no cycles.
If $G_n(I,J)$ has no cycles, this degree is the maximum number of size $> 1$ paths in the $T_1$ portion of a monotonic ribbon tiling
of $\lambda[n]$. This number is maximized by the {\em greedy tiling} shown in Figure~\ref{fig:greedy} (the partition $\lambda$ is shown in green)
where all ribbons in the $T_0$ portion are vertical and all ribbons in the $T_1$ portion have sizes 1 or 2. 
Since $(I,J) \in \symm_{n,k}$, the number of size $> 1$ ribbons in $T_1$ is $k - |\lambda|$.
\end{proof}

When $\lambda = \varnothing$ in Corollary~\ref{atomic-asymptotics} so that $\lambda[n] = (n)$ is a single row,
the conclusion can be sharpened.
Suppose the graph $G_n(I,J)$ has path partition $\mu$ and cycle partition $\nu$ so that 
$A_{n,I,J} = \vec{p}_{\mu} \cdot p_{\nu}$.
For any symmetric function $F$ of degree $n - |\nu|$ we have
$\langle F, s_{n-|\nu|} \rangle = \langle F p_{|\nu|}, s_n \rangle$ by the classical Murnaghan-Nakayama Rule.
By taking $F = \vec{p}_{\mu}$,
 Remark~\ref{polynomiality-domain-remark} extends to show that $g_{\varnothing}(n)$ is a polynomial function for all values of $n$
 and Remark~\ref{path-degree-remark} implies that the degree of $g_{\varnothing}(n)$ precisely equals the number of paths in 
 $G_n(I,J)$ of length $> 1$.

\section{A polynomiality result}
A partial permutation $(I,J)$ is {\em packed} if $I \cup J = [r]$ for some $r \geq 0$.
Let $(I,J)$ be a packed partial permutation of size $k$ such that $I \cup J = [r]$.
The character polynomial result Theorem~\ref{character-polynomial-theorem} and Corollary~\ref{atomic-asymptotics}
imply that the class function $(n)_r \cdot R \, \one_{I,J}: \symm_n \rightarrow \CC$ is a polynomial in 
$\CC[n,m_1,m_2,\dots,m_k]$ on the domain $\{n \geq 2k\}$, namely
\begin{equation}
\label{indicator-character-polynomial-first}
(n)_r \cdot R \, \one_{I,J} = \sum_{|\lambda| \leq k} g_{\lambda}(n) \cdot q_{\lambda}(m_1, \dots, m_k)  \quad \quad
\text{for $n \geq 2k$}
\end{equation}
where 
the $g_{\lambda}(n)$ are as in 
Corollary~\ref{atomic-asymptotics} and  $q_{\lambda}(m_1, \dots, m_k)$ is the character polynomial corresponding to $\lambda$
defined after Theorem~\ref{character-polynomial-theorem}.
We will see that the class function $(n)_r \cdot R \, \one_{I,J}$ is a polynomial in $\CC[n,m_1,\dots,m_k]$ on the full disjoint union
$\symm = \bigsqcup_{n \geq 1} \symm_n$ of symmetric groups and (more specifically) that the restriction
$n \geq 2k$ in Equation~\eqref{indicator-character-polynomial-first} can be dropped.
Note that the packed condition is necessary, as otherwise $\one_{I,J}$ vanishes on $\symm_r$.

We begin with a polynomial identity involving falling factorials. 
This will be used to show that $R \one_{I,J}$ is independent of cycle multiplicities $m_i$ for $i$ large.

\begin{lemma}
\label{falling-factorial-lemma}
Let $a,b, \mu_1, \dots, \mu_t \geq 1$ be integers with $\sum_i \mu_i < \min(a,b)$. There holds the identity
\begin{multline}
\label{falling-factorial-lemma-formula}
(a + b) \cdot \left(  a+b - \mu_{[t]} + t - 1  \right)_{t-1} =
a \cdot \left(  a - \mu_{[t]} + t - 1  \right)_{t-1}  + 
b \cdot \left( b - \mu_{[t]} + t - 1  \right)_{t-1}   \\ +
\sum_{[t] \, = \, A \sqcup B} ab \left( a - \mu_A + |A| - 1 \right)_{|A| - 1} \cdot 
\left( b - \mu_B + |B| - 1 \right)_{|B| - 1}
\end{multline}
where the sum is over all ordered partitions $A \sqcup B$ of $[t]$ into two nonempty sets.
\end{lemma}

\begin{proof}
We will prove Equation~\eqref{falling-factorial-lemma-formula} with the $t+2$ quantities $a, b, \mu_1, \dots, \mu_t$ therein regarded as variables.
Recall that the {\em (signless) Stirling number of the first kind} $c(n,k)$ counts permutations $w \in \symm_n$ with $\cyc(w) = k$.
If $e_k(x_1, x_2, \dots, x_n)$ is the elementary symmetric polynomial of degree $k$ in $n$ variables, there holds the identity
\begin{equation}
\label{stirling-number-identity}
e_k(1, 2, \dots, n) = c(n+1,n+1-k).
\end{equation}
Expanding the LHS of Equation~\eqref{falling-factorial-lemma-formula} and applying Equation~\eqref{stirling-number-identity} yields
\begin{align}
(a + b) \cdot \left(  a+b - \mu_{[t]} + t - 1  \right)_{t-1} &= (a+b) \cdot \sum_{i = 1}^{t} (a+b - \mu_{[t]})^{i-1} e_{t-i}(1, 2, \dots, t-1) \\
\label{new-lhs-falling}
&= (a+b) \sum_{w \in \symm_t} (a+b-\mu_{[t]})^{\cyc(w)-1}.
\end{align}
Similarly, the LHS of Equation~\eqref{falling-factorial-lemma-formula} expands as 
\begin{multline}
\label{new-rhs-falling}
a \cdot \sum_{w \in \symm_t} (a - \mu_{[t]})^{\cyc(w)-1}  + 
b \cdot \sum_{w \in \symm_t} (b - \mu_{[t]})^{\cyc(w)-1}  + \\
a b \cdot
\sum_{\substack{[t] \, = \, A \sqcup B \\ A,B  \, \neq  \, \varnothing}} 
\sum_{\substack{u \in \symm_A  \\  v \in \symm_B}} (a - \mu_A)^{\cyc(u)-1} \cdot (b - \mu_B)^{\cyc(v)-1}
\end{multline}
where $\symm_A$ and $\symm_B$ denote the groups of permutations of $A$ and $B$, respectively.

We aim to prove the equality of expressions \eqref{new-lhs-falling} =  \eqref{new-rhs-falling}. 
To do this, fix a permutation $w \in \symm_t$ with $k$ cycles $C_1, \dots, C_k$.
We have $2^k$ partitions $[t] = A \sqcup B$ obtained by letting the elements of each cycle $C_i$ be contained in $A$ or $B$.
Any such partition also gives rise to permutations $w \mid_A \in \symm_A$ and $w \mid_B \in \symm_B$ which satisfy
 $\cyc(w \mid_A) + \cyc(w \mid_B) = k = \cyc(w)$.
We show that the contribution of $w$ to \eqref{new-lhs-falling} equals the contribution of $w$ to \eqref{new-rhs-falling}.

If $x, y, z_1, \dots, z_k$ are variables, the following equation
\begin{multline}
\label{hurwitz-identity}
(x + y) (x + y + z_{[k]})^{k-1} = \\
x  (x + z_{[k]} )^{k-1} + y  (y + z_{[k]})^{k-1} +  xy
\sum_{\substack{[k] \, = \, S \sqcup T \\ S, T  \, \neq  \, \varnothing}} 
(x + z_S)^{|S| - 1} (y + z_T)^{|T|-1}
\end{multline}
is a version of {\em Hurwitz's Binomial Theorem}  \cite{Hurwitz} where we use our usual shorthand
$z_{R} = \sum_{i \in R} z_i$ for any subset $R \subseteq [k]$.  Equation~\eqref{hurwitz-identity} is a generalization of Abel's Binomial
Theorem; its most enlightening proof uses tree enumeration.
Making the substitutions $x \rightarrow a, y \rightarrow b,$ and $z_i \rightarrow - \mu_{C_j} = - \sum_{j \in C_i}  \mu_j$ into 
Equation~\eqref{hurwitz-identity} yields
\begin{multline}
\label{w-hurwitz-identity}
(a+b)(a+b-\mu_{[t]})^{\cyc(w) - 1} = 
a (a - \mu_{[t]})^{\cyc(w) - 1} + b (b - \mu_{[t]})^{\cyc(w) - 1} \\
+ ab \sum_{\substack{[k] \, = \, S \sqcup T \\ S, T \neq \varnothing}}
(a - \sum_{i \in S} \mu_{C_i})^{\cyc(w \mid_A) - 1}
(b - \sum_{j \in T} \mu_{C_j})^{\cyc(w \mid_B) - 1}
\end{multline}
which shows that the contributions of $w$ to \eqref{new-lhs-falling} and \eqref{new-rhs-falling} coincide.
\end{proof}

Lemma~\ref{falling-factorial-lemma} controls the behavior of $R \, \one_{I,J}$ on conjugacy classes with only long cycles.
The next result bootstraps  to all cycle types.

\begin{lemma}
\label{indicator-polynomiality-lemma}
Let $(I,J) \in \symm_{r,k}$ be a packed partial permutation size $k$ such that $I \cup J = [r]$. There is a polynomial
 $p_{I,J}(n,m_1,\dots,m_r) \in \CC[n,m_1,\dots,m_r]$ such that 
 \begin{equation}
 \label{indicator-character-polynomial-second}
 (n)_r \cdot R \, \one_{I,J} = p_{I,J}(n,m_1, \dots, m_r)
 \end{equation}
 as functions on $\symm = \bigsqcup_{n \geq 1} \symm_n$.  The LHS of Equation~\eqref{indicator-character-polynomial-second} is regarded
 as the zero map $\symm_n \rightarrow \CC$ for $n < r$.
\end{lemma}

\begin{proof}
For any $w \in \symm_n$, the evaluation of $(n)_r \cdot R \, \one_{I,J}$ on $w$ is given by
\begin{equation}
\label{indicator-character-polynomial-third}
 (n)_r \cdot R \, \one_{I,J}  = \sum_{v \in T} \one_{vI,vJ}(w)
\end{equation}
where the sum ranges over a fixed left transversal $T$ of coset representatives $v$ in $\symm_n/\symm_{n-r}$.
Since $r$ is a constant, we need only show that $\sum_{v \in T} \one_{vI,vJ}$  is a polynomial in $n, m_1, \dots, m_k$.
The following special case is pivotal.

{\bf Case 1:}
{\em The graph $G_n(I,J)$ is a disjoint union of paths of sizes 
$\mu = (\mu_1, \dots, \mu_t) \vdash r$ together with $n-r$ singleton paths.}

For a given permutation $w \in \symm_n$, we give a combinatorial model for the sum $\sum_{v \in T} \one_{vI,vJ}(w)$
appearing in Equation~\eqref{indicator-character-polynomial-third}.
A {\em $(\mu,w)$-cyclic sequence} is a $t$-tuple
$(S_1, \dots, S_t)$ of sequences drawn from $[n]$ with disjoint entries
such that every set $S_i$ is a list of $\mu_i$ consecutive elements in the cycle decomposition of $w$.
For example, if $w$ has cycle decomposition
\begin{equation*}
w = ( 3, 7, 2, 10, 5, 14 ) ( 1, 18, 17, 6, 4, 8) ( 12 , 13, 15) (9, 11, 16)
\end{equation*}
and $\mu = (\mu_1, \mu_2, \mu_3, \mu_4) =  (4,3,3,1)$ then one possible $(\mu,w)$-cyclic sequence is 
$(S_1, S_2, S_3, S_4)$ where
\begin{equation*}
S_1 = (10, 5, 14, 3), \quad S_2 = (13, 15, 12), \quad S_3 = (18, 17, 6), \quad S_4 = (7).
\end{equation*}
Interchanging $S_2$ and $S_3$ or `rotating' $S_2$ (obtaining $(15, 12, 13)$ or $(12, 13, 15)$) results in different 
$(\mu,w)$-cyclic sequences.
Since $G_n(I,J)$ consists only of paths, it follows that 
\begin{equation}
R \, \one_{I,J}(w) = \sum_{v \in T} \one_{vI,vJ}(w) = \text{number of $(\mu,w)$-cyclic sequences $(S_1, \dots, S_t)$}.
\end{equation}
We enumerate $(\mu,w)$-cyclic sequences for permutations $w$ with cycle types of increasing complexity.

{\bf Subcase 1.1:}
{\em $w \in \symm_n$ is an $n$-cycle.}

In this subcase, we form a $(\mu,w)$-cyclic sequence $(S_1, \dots, S_t)$ by selecting a cyclic order on the set $[t]$ in $(t-1)!$ ways corresponding to the 
cyclic order these sets appear in $w$, 
then placing $n-r$ singletons
(which will not appear among the sequences $S_1, \dots, S_r$) between the $t$ members of this cyclic order in
$\binom{n-r+t-1}{t-1}$ ways, and finally choosing one of $n$ possible cyclic rotations of the relevant figure. We conclude that 
\begin{equation}
\label{subcase-1.1-formula}
\sum_{v \in T} \one_{vI, vJ} w = n \cdot (t-1)! \cdot \binom{n-r+t-1}{t-1} = n \cdot (n-r+t-1)_{t-1}
\end{equation}
in this case. Observe that \eqref{subcase-1.1-formula} is a polynomial in $n$.

{\bf Subcase 1.2:}
{\em Every cycle of $w \in \symm_n$ has length $> r$.}

We claim that Equation~\eqref{subcase-1.1-formula} is true for $w$.
We induct on the number $\cyc(w)$ of cycles in $w$.  If $\cyc(w) = 1$, we are done by Subcase 1.1. If $\cyc(w) \geq 2$,
Lemma~\ref{falling-factorial-lemma} and induction show that
Equation~\eqref{subcase-1.1-formula} holds for $w$.
Observe that \eqref{subcase-1.1-formula} is independent of the cycle structure of $w$.

{\bf Subcase 1.3:}
{\em The permutation $w \in \symm_n$ is arbitrary.}

Given a $(\mu,w)$-cyclic sequence $(S_1, \dots, S_t)$, we form $r+1$ equivalence relations $\sim_1, \sim_2, \dots, \sim_r, \sim_{\infty}$ 
on the set $[t]$ as follows.  For $1 \leq i \leq r$, we write $a \sim_i b$ if $S_a$ and $S_b$ are in of the same cycle of $w$,
and this cycle has size $i$.  We also write $a \sim_{\infty} b$ if $S_a$ and $S_b$ are  in (possibly different) ``big" cycles of $w$ of sizes $> r$ .
If $\{a_1, \dots, a_p \}$ is an equivalence class of $\sim_i$, we must have $\mu_{a_1} + \cdots + \mu_{a_b} \leq i$.

Turning this process around, call a pair $(\Pi,f)$ {\em valid} if  
\begin{itemize}
\item
$\Pi$ is a set partition of $[t]$, 
\item $f$ is a labeling $f: \Pi \rightarrow \{1,2,\dots,r,\infty\}$ 
of the blocks of $\Pi$ with $1, 2, \dots, r, \infty$ such that $f^{-1}(\infty)$ has at most one element, and
\item for any block $B$ of $\Pi$ with $f(B) = i < \infty$ we have $\sum_{j \in B} \mu_j \leq i$.
\end{itemize}
Observe that a pair $(\Pi,f)$ being valid is a condition independent of $n$ and $w$.
Applying Subcase 1.2, 
the number of $(\mu,w)$-cyclic sequences $(S_1, \dots, S_t)$ giving rise to a fixed valid pair $(\Pi,f)$ is 
\begin{multline}
\label{subcase-1.3-equation}
\prod_{i = 1}^r  i^{|f^{-1}(i)|}  (m_i(w))_{|f^{-1}(i)|} \left(  \prod_{f(B) = i} (i - \sum_{j \in B} \mu_j + |B| - 1)_{|B| - 1} \right) \\
\times (n - \sum_{i = 1}^r i m_i(w) \cdot 
(n - \sum_{i = 1}^r i m_i(w) ) - \sum_{j \in B_{\infty}} \mu_j + |B_{\infty}| - 1)_{|B_{\infty}| - 1}
\end{multline}
where $f^{-1}(\infty) = \{ B_{\infty} \}$ when $f^{-1}(\infty) \neq \varnothing$ and the second line is interpreted as 1 when 
$f^{-1}(\infty) = \varnothing$. The expression \eqref{subcase-1.3-equation} is a (complicated)
polynomial in $\CC[n,m_1, \dots, m_r]$.  Summing \eqref{subcase-1.3-equation} over all valid pairs $(\Pi,f)$, we conclude that 
\begin{quote}
{\em the function $\sum_{v \in T} \one_{vI,vJ}$ is a polynomial in $\CC[n,m_1, \dots, m_r]$ whenever $I \cup J = [r]$ and $G_n(I,J)$ consists 
entirely of disjoint paths.}
\end{quote}
This completes our analysis of Case 1.
The next case proves the lemma.

{\bf Case 2:} {\em The partial permutation $(I,J)$ is arbitrary with $I \cup J = [r]$.}

Let $\mu = (\mu_1, \dots, \mu_t, 1, \dots, 1)$ be the path partition of $(I,J)$ where $\mu_t > 1$ and let $\nu = (\nu_1, \dots, \nu_s)$ 
be the cycle partition of $(I,J)$. Then $\sum_{v \in T} \one_{vI,vJ}(w)$ counts $(t+s)$-tuples $(S_1, \dots, S_t, C_1, \dots, C_s)$ 
such that 
\begin{itemize}
\item every $S_i$ is a sequence consisting of $\mu_i$ cyclically consecutive elements in the cycle notation of $w$, 
\item every $C_i$ is a cycle of $w$ of size $\nu_i$, and
\item none of the $S_1, \dots, S_t, C_1, \dots, C_s$ have any letters in common.
\end{itemize}
The number of such tuples $(S_1, \dots, S_t,C_1, \dots, C_s)$ equals
\begin{multline}
\label{case-two-equation} 
 r! \cdot (n)_r \cdot R \, \one_{I,J}(w) = \sum_{v \in T} \one_{vI,vJ}(w)  \\ = 
(m_1(w))_{m_1(\nu)} \cdot (m_2(w))_{m_2(\nu)} \cdot \cdots \cdot (m_r(w))_{m_r(\nu)} \times 
\begin{array}{c}
\text{number of $( \mu', w')$-cyclic} \\ 
\text{ sequences }  (S'_1, \dots, S'_t)
\end{array}
\end{multline}
where $\mu' = (\mu_1, \dots, \mu_t)$ and $w' \in \symm_{n-|\nu|}$ is any permutation satisfying 
$m_i(w') = m_i(w) - m_i(\nu)$.
Case 1 shows that the second factor on the second line \eqref{case-two-equation} is a polynomial in $\CC[n,m_1,\dots,m_r]$.
The entire second line of \eqref{case-two-equation} is therefore a polynomial  
$p_{I,J}(n,m_1,\dots,m_r) \in \CC[n,m_1, \dots, m_r]$, completing the proof.
\end{proof}

Lemma~\ref{indicator-polynomiality-lemma} may appear weaker  than our assertion that $R \, \one_{I,J}$ is a polynomial
in $\CC[n,m_1, \dots, m_k]$ when $(I,J) \in \symm_{n,k}$ is a partial permutation with $I \cup J = [r]$.
In particular, we need to show the polynomial $p_{I,J}(n,m_1, \dots, m_r) \in \CC[n,m_1, \dots, m_r]$ does not depend on 
$m_{k+1}, m_{k+2}, \dots, m_r$.
This may be derived from the following general fact.
Given a polynomial $f = f(n, m_1, \dots, m_r) \in \CC[n,m_1, \dots, m_r]$ and $w \in \symm_n$, we define
$f(w) := f(n,m_1(w), \dots, m_r(w)) \in \CC$.

\begin{lemma}
\label{cycle-identification-lemma}
Let $f, g \in \CC[n,m_1, \dots, m_r]$ be two polynomials.  Suppose there exists $n_0$ such that for all $n \geq n_0$ we have 
$f(w) = g(w)$ for all $w \in \symm_n$.
Then $f = g$ as elements of $\CC[n,m_1,\dots,m_r]$.
\end{lemma}

\begin{proof}
Suppose first that $f, g \in \CC[m_1, \dots, m_r]$ do not involve the variable $n$.  By assumption, the polynomials 
$f, g$ agree on the infinite set of lattice points
\begin{equation*}
\{ (a_1, \dots, a_r) \in \ZZ^r \,:\, a_i \geq n_0 \text{ for all $i$ } \} 
\end{equation*}
which forces $f = g$ as polynomials in $\CC[m_1, \dots, m_n]$.

For the general case of $f, g \in \CC[m_1, \dots, m_r]$, write
\begin{multline}
f - g = \\ h_d(m_1, \dots, m_r) \cdot n^d + h_{d-1}(m_1, \dots, m_r) \cdot n^{d-1} + \cdots + h_0(m_1, \dots, m_r) \cdot n^0
\end{multline}
for some polynomials $h_d, h_{d-1}, \dots, h_0 \in \CC[m_1, \dots, m_r]$.  Suppose $h_d \neq 0$ as an element of $\CC[m_1, \dots, m_r]$. 
If $a = (a_1, \dots, a_r) \in \ZZ^r$ is an $r$-tuple with $a_i \geq n_0$ for all $i$, consider a sequence $( w^{(s)} )_{s > r}$ of permutations 
with cycles type $(1^{a_1} \cdots r^{a_r} s^1)$.  
For all $s > r$, we have $f(w^{(s)}) = g(w^{(s)})$ and $h_i(w^{(s)}) = h_i(a_1, \dots, a_r)$. Letting $s \rightarrow \infty$, we see that 
$h_d(a_1, \dots, a_r) = 0$.  The reasoning of the first paragraph yields the contradiction $h_d = 0$ as an element of $\CC[m_1, \dots, m_r]$. 
\end{proof}

Lemma~\ref{cycle-identification-lemma} strengthens Lemma~\ref{indicator-polynomiality-lemma} as follows.

\begin{proposition}
\label{indicator-polynomiality-proposition}
Let $(I,J) \in \symm_{n,k}$ be a partial permutation such that $I \cup J = [r]$ for some $r \geq 0$.
There exists a polynomial $p_{I,J}(n,m_1, \dots, m_k) \in \CC[n,m_1, \dots, m_k]$ such that
\begin{equation}
(n)_r \cdot R \, \one_{I,J} = p_{I,J}(n,m_1, \dots, m_k)
\end{equation} 
as functions $\symm \rightarrow \CC$. Furthermore, if we put $n$ in degree $1$ and $m_i$ in degree $i$, the polynomial
$p_{I,J}(n,m_1, \dots, m_k)$ has degree at most $k$.
\end{proposition}

\begin{proof}
Combine Equation~\eqref{indicator-character-polynomial-first}, 
Lemma~\ref{indicator-polynomiality-lemma}, and Lemma~\ref{cycle-identification-lemma}.
\end{proof}

\chapter{Regular Statistics}
\label{Regular}

The support bound of Theorem~\ref{local-support-theorem} notwithstanding, we will need to impose a stronger condition than locality to understand
the symmetric function $\ch_n(R \, f)$ for a statistic $f: \symm_n \rightarrow \CC$ as $n$ varies.
Mere locality does not guarantee any kind of coherence for the statistic $f$ among different values of $n$: the statistic
 \begin{equation}
 f(w) = \begin{cases} \inv(w) & \text{$n$ odd} \\ \des(w) & \text{$n$ even} \end{cases}
 \end{equation}
 is 2-local where
 \begin{equation}
 \inv(w) := | \{ 1 \leq i < j \leq n \,:\, w(i) < w(j) \} |
 \end{equation}
 is the {\em inversion number}.
 Furthermore, for any function $G: \ZZ_{\geq 0} \rightarrow \CC$ the statistic
 \begin{equation}
 g(w) = G(n) \cdot \inv(w)
 \end{equation}
is 2-local, but when $G$ grows too quickly or erratically very little can be said about the asymptotics of coefficients in the Schur expansion $\ch_n(R \, g)$.
In this chapter, we introduce
 a more restrictive class of {\em regular statistics} that excludes such pathologies while still being broad enough to 
 contain a wide class of statistics.

Regular statistics are best thought of as `reasonable' weighted pattern counting statistics.
One can think of the statistic $\inv$, defined above, as counting the pattern `out of order pair'.
In the statistic $\des$, we insist a descent is an `out of order pair' satisfying the constraint that the pair is consecutive: $w(i) > w(i+1)$ in a permutation $w \in \symm_n$.
The statistic $\maj = \sum_i i \cdot \delta_{w(i) > w(i+1)}$ weights each of the patterns in $\des$ by its position.
Regular statistics are flexible enough to include all of these behaviors.

We define regular statistics as linear combinations of a new family of permutation statistics we call \emph{constrained translates}, which are weighted sums of $1_{I,J}$'s with the same cycle-path type.
This last property guarantees that the image of a constrained translate under the Reynolds operator $R_n$ has a particularly simple formula.
To demonstrate the ubiquity of regular statistics and familiarize the reader, we give explicit expansions of several classical permutation statistics into constrained translates.

Regular statistics form a subalgebra of $\Fun(\symm_n, \CC)$ that respects the grading by locality, but also carries an additional notion of `degree' we call \emph{power}.
They inherit the relatively simple description  of constrained translates under the Reynolds operator, which is the key technical tool for our future asymptotic analyses.
Lastly, we introduce and study a distinguished subalgebra of regular statistics we call \emph{closed statistics}, whose asymptotic properties are even easier to analyze.

\section{Constrained translates}
Constraints on permutation patterns that require positions or values to be consecutive are known as  {\em vincular conditions}
(or {\em bivincular conditions} since we are considering values as well as positions).
To handle these conditions, we introduce the following notation.
Recall that $\binom{[n]}{m}$ is the family of $m$-element subsets of $[n]$.  For $C \subseteq [m-1]$, define
\begin{equation}
\label{constraint-sets}
\binom{[n]}{m}_C := \left\{  I = \{ i_1 < \cdots < i_m \} \in \binom{[n]}{m} \,:\, i_{c+1} = i_c + 1 \text{ for $c \in C$} \right\}.
\end{equation}
Notice that
\begin{equation}
\left| \binom{[n]}{m}_C \right| = \binom{ n - |C|}{k - |C| }.
\end{equation}

A partial permutation $(I,J)$ is \emph{packed} if $I \cup J = [m]$ for some value $m$.
For any $m$-element set  $L = \{ \ell_1 < \cdots < \ell_m \}$ of integers and $I \subseteq [m]$, let $L(I) := \{ \ell_i \,:\, i \in I \}$.
The following observation should be thought of as a `compression' result on packed partial permutations in the spirit of ideas in~\cite{CDKR1,CDKR2}.

\begin{lemma}
\label{l:unique-pack}
Let $(I,J)$ be a partial permutation and $L = I \cup J$ with $|L| = m$.
Then there is a unique packed partial permutation $(U,V)$ with 
$U \cup V = [m]$ such that $(L(U), L(V))  = (I,J)$.
\end{lemma}

\begin{proof}
To obtain $(U,V)$ from $(I,J)$, replace all instances of the smallest letter in $L$ with 1, all instances of the next smallest letter in 
$L$ with 2, and so on.
\end{proof}

Motivated by Lemma~\ref{l:unique-pack}, we consider a class of building block statistics coming from packed partial permutations 
$(U,V)$.
In order to incorporate vincular conditions in this setup, we include a set $C$ of indices which must be consecutive.
We will also want to consider statistics such as the {\em major index}
\begin{equation}
\maj(w) := \sum_{w(i) > w(i+1)} i
\end{equation}
which assigns a linear weight $i$ to a descent at $i$.
To achieve this, we include a polynomial weight function in our building blocks.

\begin{defn}
\label{vincular-translate-definition}
Let $(U,V)$ be a packed partial permutation with $U \cup V = [m]$, let $C \subseteq [m-1]$, and let 
$f \in \CC[x_1, \dots, x_m]$ be a polynomial.  We call $((U,V),C,f)$ a {\em packed triple}.
The {\em constrained translate} of  $((U,V),C,f)$ is the 
statistic $T^f_{(U,V),C}: \symm_n \rightarrow \CC$ given by
\begin{equation}
\label{eq:translate}
T^f_{(U,V),C} = \sum_{L \in \binom{[n]}{m}_C} f(\ell_1, \dots, \ell_m) \cdot \one_{L(U),L(V)}
\end{equation}
where $L = \{ \ell_1 < \cdots < \ell_m \}$. 
\end{defn}

We want to understand the symmetric functions $\ch_n (R \, T^f_{(U,V),C})$ as functions of $n$.
To handle the polynomial weights $f$, we state a simple lemma about polynomial evaluations.

\begin{lemma}
\label{l:polynomial}
	Let $f \in \CC[x_1,\dots,x_m]$ be a polynomial in $m$ variables.
	Then there exists $\overline{f} \in \CC[n]$ with $\deg f = \deg \overline{f}$ so that
	\begin{equation}
	\sum_{\{i_1<\dots<i_m\} \in \binom{[n]}{m}} f(i_1,\dots,i_m) = \overline{f}(n) \cdot \binom{n}{m}.
	\end{equation}
	In particular, this expression has degree $m + \deg f$ as an element of $\CC[n]$.
\end{lemma}

\begin{proof}
By linearity we may assume that $f = x_1^{a_1}  \cdots x_m^{a_m}$ is a monomial in $x_1, \dots, x_m$. 
The result is clear for $m = 0$. Given $m > 0$
the function
\begin{equation}
F(n) := \sum_{\{i_1<\dots<i_m\} \in \binom{[n]}{m}} f(i_1,\dots,i_m) = 
\sum_{\{i_1<\dots<i_m\} \in \binom{[n]}{m}} (i_1)^{a_1} \cdots (i_m)^{a_m}
\end{equation}
satisfies 
\begin{equation}
\label{F-recursion}
F(n) - F(n-1) = n^{a_m} \cdot \sum_{\{i_1<\dots<i_{m-1}\} \in \binom{[n-1]}{m-1}} (i_1)^{a_1} \cdots (i_{m-1})^{a_{m-1}}
\end{equation}
where the summation $n^{a_m} \cdot \sum_{\{i_1<\dots<i_{m-1}\} \in \binom{[n-1]}{m-1}} (i_1)^{a_1} \cdots (i_{m-1})^{a_{m-1}}$ is inductively a polynomial
in $n$ of degree $\deg f + m - 1$.  Since Equation~\eqref{F-recursion}
 holds for all $n > 0$, we see that $F(n)$ is a polynomial in $n$ of degree $\deg f + m$.
 It is evident that $F(i) = 0$ for $i = 0, 1, \dots, m-1$ so that $F(n)$ is divisible by $\binom{n}{m}$ in the ring $\CC[n]$.
\end{proof}

The  parameter $C$ in Definition~\ref{vincular-translate-definition} necessitates an enhancement of
Lemma~\ref{l:polynomial} in which only summands corresponding to sets $\{ i_1 < \cdots < i_m \}$ satisfying 
a vincular constraint are allowed.

\begin{lemma}
\label{l:vincular-poly}
	Let $f \in \CC[x_1,\dots,x_k]$ and $C \subseteq [k-1]$ with $|C| = q$.
	Then there exists $\overline{f}_C \in \mathbb{C}[n]$ with $\deg \overline{f}_C= \deg f$ so that
	\begin{equation}
	\label{eq:vincular-poly}
	\sum_{\{i_1< \dots < i_k\} \in \binom{[n]}{k}_C} f(i_1,\dots,i_k) = \overline{f}_C(n) \binom{n-q}{k-q}.
	\end{equation}
\end{lemma}

\begin{proof}
	For $j \in [k]$, let $c_j = |\{c \in C: c < j\}|$.
	The shift map $\mathrm{sh}_C:\binom{[n]}{k}_C \to \binom{[n-p]}{k-p}$ defined by
	\begin{equation}
	\mathrm{sh}_C(\{i_1< \dots < i_k\}) = \{i_j - c_j: j-1 \notin C\}
	\end{equation}
	is a bijection.
	Viewing $f$ as a function of $\mathrm{sh}_C(I)$, the result follows from Lemma~\ref{l:polynomial}.
\end{proof}

Lemma~\ref{l:vincular-poly} lets us expand the symmetric functions of the vincular translate statistics in terms of atomic symmetric functions.

\begin{corollary}
\label{c:vincular-trans}
	For $(U,V)$ a packed partial permutation with 
	$U \cup V = [m]$, $C \subseteq [m-1]$ and $f \in \CC[x_1,\dots,x_m]$, we have
\begin{equation}
\ch_n(R \, T^f_{(U,V),C}) = \frac{1}{n!} \binom{n-|C|}{m-|C|} \overline{f}(n) \cdot A_{n,U,V}.
\end{equation}
where $\overline{f} \in \CC[n]$ with $\deg \overline{f} = \deg f$ is as in Lemma~\ref{l:vincular-poly}.
\end{corollary}

\begin{proof}
	Since the $\one_{I,J}$ terms in Equation~\eqref{eq:translate} have the same cycle-path type, the result follows from 
	Lemma~\ref{l:vincular-poly} and Proposition~\ref{atomic-g-invariance}.
\end{proof}

Corollaries~\ref{atomic-asymptotics} and \ref{c:vincular-trans} imply that the Schur expansion of $ \ch_n(R \, T^f_{(U,V),C})  $ 
has coefficients which
are rational functions of $n$ with denominators $n(n-1) \cdots (n - |C| + 1)$. In the case of the single-row
Schur function $s_{(n)}$, more can be said about this coefficient.

\begin{corollary}
\label{constrained-one-row}
Let $(U,V)$ be a packed partial permutation with $U \cup V = [m]$, let $C \subseteq [m-1]$, and let $f \in \CC[x_1, \dots, x_m]$ be a polynomial.
Writing 
\begin{equation}
\ch_n(R \, T^f_{(U,V),C})  = \sum_{\lambda \vdash n} d_{\lambda}(n) \cdot s_{\lambda},
\end{equation}
when $\lambda = (n)$ the expression
\begin{equation}
\label{eq:rational-constraint}
(n)_{|C|} \cdot d_{(n)}(n) \in \CC[n]
\end{equation}
is a polynomial in $n$ of degree exactly $\ell + m + \deg f$ where $\ell$ is the number of `long' paths of length $> 1$ in $G(U,V)$.
\end{corollary}

\begin{proof}
This is a consequence of Corollary~\ref{c:vincular-trans} and
 the discussion following Corollary~\ref{atomic-asymptotics}.
\end{proof}

Determining the degrees of $n(n-1) \cdots (n - |C| + 1) \cdot d_{\lambda}(n) \in \CC[n]$ for general partitions $\lambda$ appears to be a difficult
problem.  Corollaries~\ref{atomic-asymptotics} and \eqref{c:vincular-trans} give an upper bound of 
$\ell + m + \deg f - (n - \lambda_1)$, which is strictly less than the degree in Corollary~\ref{constrained-one-row}.

\begin{remark}
	\label{large-constraints}
The adjacency constraint appearing in our definition of constrained translates can be relaxed.
A \emph{generalized constaint} of size $k$ is a sequence
\[
\mathcal{D} = (d_1,\dots,d_{k-1}) \in \{1,2,\dots\} \cup \{\infty\}.
\]
Analogous to Equation~\eqref{constraint-sets}, we define
\begin{equation}
\binom{[n]}{k}_{\mathcal{D}} := \left\{ I = \{i_1 < \dots < i_m\} \in \binom{[n]}{m} : i_{j+1} = i_j + d_j\ \mbox{if}\ d_j \neq \infty \right\}.
\end{equation}
For $C \subseteq [k-1]$ and $\mathcal{D} = ( d_1,\dots,d_{k-1})$ with 
\[
d_i = \begin{cases}
1 & i \in C\\
\infty & i \notin C	
 \end{cases},
\]
note $\binom{[n]}{k}_C = \binom{[n]}{k}_\mathcal{D}$.
Generalized constraints are considered in~\cite{Janson}, where they are called \emph{exact constraints}.

The proof of Lemma~\ref{l:vincular-poly} can be adapted to the generalized constraints setting.
Let 
\[
Q_\mathcal{D} = \{i: d_i \neq \infty\},\quad |\mathcal{D}| = \sum_{i \in Q_\mathcal{D}} i,\quad q_\mathcal{D} = |Q_\mathcal{D}|.
\]
The only change to Lemma~\ref{l:vincular-poly} is that the factor of $\binom{n-q}{k-q}$ in Equation~\ref{eq:vincular-poly} becomes $\binom{n-|\mathcal{D}|}{k-q_\mathcal{D}}$.
Working with generalized constraints makes denominators like the term $(n)_{|C|}$ in Equation~\eqref{eq:rational-constraint} and its consequences more difficult to state, which is why we only state our results in terms of vincular constraints.

\end{remark}

\section{Examples: $\exc, \inv, \des,$ and $\maj$}
\label{ss:examples}

We are ready to define our class of regular permutation statistics.
Just as the indicator statistics $\one_{I,J}$ for $(I,J) \in \symm_{n,k}$ are the building blocks of local statistics,
the constrained translates are the building blocks of regular statistics. 

\begin{defn}
\label{regular-statistics-definition}
The permutation statistic $\stat: \symm \to \CC$ is \emph{regular} if it is a linear combination of constrained translates, that is
\begin{equation}
\stat = \sum_{((U,V),C,f) \in \Pi} c_{(U,V),C,f} \cdot T^f_{(U,V),C}
\end{equation}
for some collection $\Pi$ of packed triples
and some constants $c_{(U,V),C,f)} \in \CC$.
\end{defn}

The term {\em regular} was inspired from regular functions in algebraic geometry. 
Regular statistics allow for polynomial weights $f$ in their constituent parts. 
We will see (Proposition~\ref{p:simple-prod}) that, like regular functions, regular statistics behave well under multiplication.
Before exploring their general properties, we how show four classical permutation statistics into our framework.

We have met three of these ($\inv, \des,$ and $\maj$) before.
The {\em excedance number} $\exc: \symm_n \rightarrow \CC$ is given by
\begin{equation}
\exc(w) := | \{ 1 \leq i \leq n \,:\, w(i) > i \} |.
\end{equation}
All of the permutation statistics statistics $\exc, \inv, \des,$ and $\maj$ are regular.
The excedance statistic is given by
\begin{equation}
\exc = \sum_{i < j} \one_{(i),(j)} = T^f_{(1,2),\varnothing} = T_{(1,2)}
\end{equation}
where we drop the notation $f$ and $C$ from $T^f_{(U,V),C}$ when $f \equiv 1$ and $C = \varnothing$.
The inversion statistic expands in terms of $T$'s as 
\begin{multline}
\inv = \sum_{\substack{a < b \\ c < d}} \one_{(ab),(dc)} = T_{(12,21)} + T_{(12,31)}  +  T_{(12,32)} + T_{(13,21)} + T_{(13,32)}  + 
T_{(23,21)} + T_{(23,31)} \\ + T_{(12,43)}  + T_{(13,42)} + T_{(14,32)} + T_{(23,41)} + T_{(24,31)} + T_{(34,21)}.
\end{multline}
The descent number satisfies
\begin{multline}
\des = \sum_{\substack{1 \leq a \leq n-1 \\ c < d}} \one_{(a,a+1),(d,c)} = 
T_{(12,21),\{1\}} + T_{(12,31),\{1\}} + T_{(12,32),\{1\}} + T_{(23,21),\{2\}} \\ + T_{(23,31),\{2\}}
+ T_{(12,43),\{1\}} + T_{(23,41),\{2\}} + T_{(34,21),\{3\}}.
\end{multline}
 The major index expands as
\begin{multline}
\maj = \sum_{\substack{1 \leq a \leq n-1 \\ c < d}} a \cdot \one_{(a,a+1),(d,c)} = 
T_{(12,21),\{1\}}^{\, x_1} + T_{(12,31),\{1\}}^{ \, x_1} + T_{(12,32),\{1\}}^{\, x_1} + T_{(23,21),\{2\}}^{\, x_2} \\ + T_{(23,31),\{2\}}^{\, x_2}
+ T_{(12,43),\{1\}}^{\, x_1} + T_{(23,41),\{2\}}^{\, x_2} + T_{(34,21),\{3\}}^{\, x_3}.
\end{multline}
The $T$-expansions of $k$-local regular statistics become more complicated as $k$ grows. 
For example, expanding the 3-local statistic $\peak$ in terms of $T$'s requires 76 terms.
With $T$-expansions as above, it is a simple matter to write $\ch_n(R \, f)$ in the Schur basis where $f = \exc, \inv, \des, \maj$.
In doing so, we recover results of Hultman.

\begin{proposition}
\label{basic-expansions}
{\em (Hultman \cite[Thm. 5.1, Thm. 6.2]{Hultman})}
We have the Schur expansions
\begin{align}
\ch_n(R \, \exc) &= \frac{n-1}{2} \cdot s_n - \frac{1}{2} \cdot s_{n-1,1} \\
\ch_n(R \, \inv) &= \frac{n(n-1)}{4} \cdot s_n - \frac{n+1}{6} \cdot s_{n-1,1} - \frac{1}{6} \cdot s_{n-2,1,1} \\
\ch_n(R \, \des) &= \frac{n-1}{2} \cdot s_n - \frac{1}{n} \cdot s_{n-1,1} - \frac{1}{n} \cdot s_{n-2,1,1} \\
\ch_n(R \, \maj) &= \frac{n(n-1)}{4} \cdot s_n - \frac{1}{2} \cdot s_{n-1,1} - \frac{1}{2} \cdot s_{n-2,1,1}
\end{align}
for all $n \geq 1$.
\end{proposition}

\begin{proof}
Apply Corollary~\ref{c:vincular-trans} and the $T$-expansions above to write these symmetric functions in terms of atomics.
Then apply Corollary~\ref{atomic-schur-expansion} to convert to the Schur basis.
\end{proof}

Hultman proved Proposition~\ref{basic-expansions} by comparing explicit formulas 
for the expectations of the statistics $\exc, \inv, \des, \maj$ on conjugacy classes $K_{\lambda}$
 with closed formulas for the irreducible characters $\chi^n, \chi^{n-1,1}, \chi^{n-2,1,1}$.
 See Section~\ref{subsection:exc} for the computation of $\ch_n (R \exc)$.

Recall that the statistics $\inv$ and $\maj$ are equidistributed on $\symm_n$
(any statistic with this distribution is {\em Mahonian}). The symmetric functions
$\ch_n(R \, \inv)$ and $\ch_n(R \, \maj)$ differ because $\inv$ and $\maj$ are not equidistributed on every conjugacy class $K_{\lambda} \subseteq \symm_n$.
Similarly, although $\exc$ and $\des$ are equidistributed on the full symmetric group $\symm_n$ (such statistics are called {\em Eulerian}),
we have $\ch_n(R \, \exc) \neq \ch_n(R \, \des)$ since $\exc$ and $\des$ are not equidistributed on conjugacy classes.
Note that the $s_n$ coefficients of the equidistributed pairs of statistics \emph{are} equal.
We will prove this is a general phenomenon in Chapter~\ref{Pattern}.

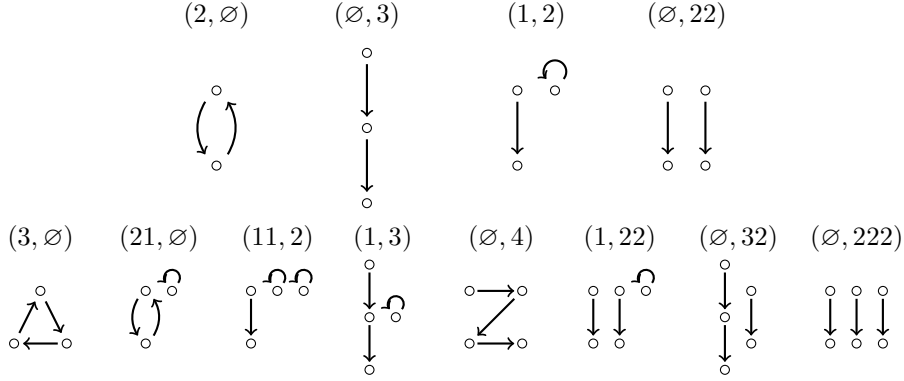
\begin{figure}
\begin{center}
\begin{tikzpicture}[scale = 0.5]

\node at (0,2) {$(2,\varnothing)$};
\node at (4,2) {$(\varnothing,3)$};
\node at (8.5,2) {$(1,2)$};
\node at (12.5,2) {$(\varnothing,22)$};

 \node [draw, circle, fill = white, inner sep = 1.2pt] at (0,0) { }; 
 \node [draw, circle, fill = white, inner sep = 1.2pt] at (0,-2) { }; 
 \draw [->, thick]  (-0.3,-0.3) to[bend right]  (-0.3,-1.7);
 \draw [->, thick]  (0.3,-1.7) to[bend right]  (0.3,-0.3);

 \node [draw, circle, fill = white, inner sep = 1.2pt] at (4,1) { }; 
 \node [draw, circle, fill = white, inner sep = 1.2pt] at (4,-1) { }; 
 \node [draw, circle, fill = white, inner sep = 1.2pt] at (4,-3) { }; 
 \draw [->, thick]  (4,0.7) -- (4,-0.7);
 \draw [->, thick]  (4,-1.3) -- (4,-2.7);

  \node [draw, circle, fill = white, inner sep = 1.2pt] at (8,0) { }; 
 \node [draw, circle, fill = white, inner sep = 1.2pt] at (8,-2) { }; 
  \draw [->, thick]  (8,-0.3) -- (8,-1.7);
   \node [draw, circle, fill = white, inner sep = 1.2pt] at (9,0) { }; 
   \draw[->, thick]  ($(9,0.5) + (-40:3mm)$) arc (-40:220:3mm);
   
   \node [draw, circle, fill = white, inner sep = 1.2pt] at (12,0) { }; 
 \node [draw, circle, fill = white, inner sep = 1.2pt] at (12,-2) { }; 
   \draw [->, thick]  (12,-0.3) -- (12,-1.7);
   \node [draw, circle, fill = white, inner sep = 1.2pt] at (13,0) { }; 
 \node [draw, circle, fill = white, inner sep = 1.2pt] at (13,-2) { }; 
  \draw [->, thick]  (13,-0.3) -- (13,-1.7);

\end{tikzpicture}
\end{center}

\begin{center}
\begin{tikzpicture}[scale = 0.35]

\node at (0,2) {$(3,\varnothing)$};
\node at (4.5,2) {$(21,\varnothing)$};
\node at (9,2) {$(11,2)$};
\node at (13,2) {$(1,3)$};
\node at (17.5,2) {$(\varnothing,4)$};
\node at (22,2) {$(1,22)$};
\node at (26.5,2) {$(\varnothing,32)$};
\node at (31,2) {$(\varnothing,222)$};

 \node [draw, circle, fill = white, inner sep = 1.2pt] at (0,0) { }; 
 \node [draw, circle, fill = white, inner sep = 1.2pt] at (1,-2) { }; 
 \node [draw, circle, fill = white, inner sep = 1.2pt] at (-1,-2) { }; 
 \draw [->, thick]  (0.2,-0.4) -- (0.8,-1.6);
\draw [->, thick]  (0.66,-2) -- (-0.66,-2);
 \draw [->, thick]  (-0.8,-1.6) -- (-0.2,-0.4);
 
 \node [draw, circle, fill = white, inner sep = 1.2pt] at (4,0) { }; 
 \node [draw, circle, fill = white, inner sep = 1.2pt] at (4,-2) { }; 
 \node [draw, circle, fill = white, inner sep = 1.2pt] at (5,0) { }; 
   \draw[->, thick]  ($(5,0.5) + (-40:3mm)$) arc (-40:220:3mm);
 \draw [->, thick]  (3.7,-0.3) to[bend right]  (3.7,-1.7);
 \draw [->, thick]  (4.3,-1.7) to[bend right]  (4.3,-0.3);

  \node [draw, circle, fill = white, inner sep = 1.2pt] at (8,0) { }; 
 \node [draw, circle, fill = white, inner sep = 1.2pt] at (8,-2) { }; 
 \node [draw, circle, fill = white, inner sep = 1.2pt] at (9,0) { }; 
    \draw[->, thick]  ($(9,0.5) + (-40:3mm)$) arc (-40:220:3mm);
  \node [draw, circle, fill = white, inner sep = 1.2pt] at (10,0) { }; 
     \draw[->, thick]  ($(10,0.5) + (-40:3mm)$) arc (-40:220:3mm);
      \draw [->, thick]  (8,-0.3) -- (8,-1.7);
      
\node [draw, circle, fill = white, inner sep = 1.2pt] at (12.5,1) { }; 
  \draw [->, thick]  (12.5,0.7) -- (12.5,-0.7);
 \node [draw, circle, fill = white, inner sep = 1.2pt] at (12.5,-1) { }; 
   \draw [->, thick]  (12.5,-1.3) -- (12.5,-2.7);
 \node [draw, circle, fill = white, inner sep = 1.2pt] at (12.5,-3) { }; 
  \node [draw, circle, fill = white, inner sep = 1.2pt] at (13.5,-1) { }; 
     \draw[->, thick]  ($(13.5,-0.5) + (-40:3mm)$) arc (-40:220:3mm);

     \node [draw, circle, fill = white, inner sep = 1.2pt] at (16.3,0) { }; 
       \draw [->, thick]  (16.6,0) -- (18,0);
     \node [draw, circle, fill = white, inner sep = 1.2pt] at (18.3,0) { }; 
      \draw [->, thick]  (18,-0.3) -- (16.6,-1.7);
     \node [draw, circle, fill = white, inner sep = 1.2pt] at (16.3,-2) { }; 
      \draw [->, thick]  (16.6,-2) -- (18,-2);
     \node [draw, circle, fill = white, inner sep = 1.2pt] at (18.3,-2) { };

 \node [draw, circle, fill = white, inner sep = 1.2pt] at (21,0) { }; 
 \draw [->, thick]  (21,-0.3) -- (21,-1.7);
 \node [draw, circle, fill = white, inner sep = 1.2pt] at (21,-2) { }; 
 \draw [->, thick]  (22,-0.3) -- (22,-1.7);
 \node [draw, circle, fill = white, inner sep = 1.2pt] at (22,0) { }; 
 \node [draw, circle, fill = white, inner sep = 1.2pt] at (22,-2) { }; 
 \node [draw, circle, fill = white, inner sep = 1.2pt] at (23,0) { };
 \draw[->, thick]  ($(23,0.5) + (-40:3mm)$) arc (-40:220:3mm);
 
  \node [draw, circle, fill = white, inner sep = 1.2pt] at (27,0) { }; 
 \draw [->, thick]  (27,-0.3) -- (27,-1.7);
 \node [draw, circle, fill = white, inner sep = 1.2pt] at (27,-2) { }; 
 \node [draw, circle, fill = white, inner sep = 1.2pt] at (26,1) { }; 
\node [draw, circle, fill = white, inner sep = 1.2pt] at (26,-1) { }; 
 \node [draw, circle, fill = white, inner sep = 1.2pt] at (26,-3) { }; 
 \draw [->, thick]  (26,0.7) -- (26,-0.7);
 \draw [->, thick]  (26,-1.3) -- (26,-2.7);

   \node [draw, circle, fill = white, inner sep = 1.2pt] at (30,0) { }; 
 \draw [->, thick]  (30,-0.3) -- (30,-1.7);
 \node [draw, circle, fill = white, inner sep = 1.2pt] at (30,-2) { }; 
 
   \node [draw, circle, fill = white, inner sep = 1.2pt] at (31,0) { }; 
 \draw [->, thick]  (31,-0.3) -- (31,-1.7);
 \node [draw, circle, fill = white, inner sep = 1.2pt] at (31,-2) { }; 
 
   \node [draw, circle, fill = white, inner sep = 1.2pt] at (32,0) { }; 
 \draw [->, thick]  (32,-0.3) -- (32,-1.7);
 \node [draw, circle, fill = white, inner sep = 1.2pt] at (32,-2) { }; 

\end{tikzpicture}
\end{center}
\caption{Top: The cycle-path types and graphs involved in calculating $\ch_n(R \, f)$ for $f = \inv, \des, \maj$. Bottom: The cycle-path types and graphs involved in 
calculating $\ch_n(R \, f)$ for $f = \peak$.}
\label{fig:graph-iso}
\end{figure}

In the case of $f = \exc$, calculating $\ch_n(R \, f)$ as in
Proposition~\ref{basic-expansions} involved finding the Schur expansion of a single atomic function: that corresponding to the graph
$\circ \rightarrow \circ$ together with $n-2$ paths of size 1. 
For $f = \inv, \des, \maj$, the nontrivial components of the graphs $G(I,J)$ involved are shown up to isomorphism
 in the top of Figure~\ref{fig:graph-iso}.
 The bottom of Figure~\ref{fig:graph-iso} shows the graphs involved for the 3-local  peak statistic $f = \peak$ given by
 \begin{equation}
 \peak(w) := | \{ 2 \leq i \leq n-1 \,:\, w(i-1) < w(i) > w(i+1) \} |
 \end{equation}
 for $w \in \symm_n$.
 In particular, although the $T$-expansion of $\peak$ has 76 terms, to calculate $\ch_n(R \, \peak)$ we need only find the Schur expansions
 of 8 atomic functions; the result  is
 \begin{equation}
 \label{eq:peak-expansion}
 \ch_n(R \, \peak) =  \frac{n-2}{3}  \cdot s_n - \frac{1}{n(n-1)} \cdot \left[ s_{n-1,1} +  s_{n-2,2} + s_{n-2,1,1} +  s_{n-3,2,1} \right].
 \end{equation}
 The above expansion of $\ch_n(R \, \peak)$ is equivalent to a result of Gill \cite[Ch. 4, Cor. 4.6]{Gill}.
 In general, finding the Schur expansion of $\ch_n(R f)$ for a $k$-local regular statistic $f$ involves cycle-path types of size $k$.

The support constraints in Proposition~\ref{basic-expansions} are consistent with locality.
As guaranteed by Theorem~\ref{local-support-theorem}, the only Schur functions appearing in $\ch_n(R \, \exc)$ for the 1-local
statistic $\exc$ are $s_n$ and $s_{n-1,1}$.
Theorem~\ref{local-support-theorem} is also consistent with the fact that the Schur expansions coming from the 2-local statistics
$\inv, \des,$ and $\maj$ only contain $s_n, s_{n-1,1},$ and $s_{n-2,1,1}$.
We do not have a conceptual reason why these 
Schur expansions do not contain $s_{n-2,2}$.

We expand the statistics from Proposition~\ref{basic-expansions} into character polynomials.
\begin{corollary}
	We have the character polynomials
	\begin{align}
		R \exc &= \frac{n-m_1}{2}\\
		R \inv &= \frac{1}{12}\left(3n^2 -2 nm_1 - m_1^2 +2 m_2 - n + m_1 \right)\\
		R \des &= \frac{1}{2n}\left(n^2 + m_1^2 + 2m_2 - n + 3 m_1 - 4 \right)\\
		R \maj &= \frac{1}{4}\left(n^2 -m_1^2 + 2m_2 -n - m_1 \right)
	\end{align}
\end{corollary}

\section{Multiplication and Schur expansion of regular statistics}
Examining the coefficients in Proposition~\ref{basic-expansions} shows  additional structure.  These coefficients are polynomials in $\CC[n]$
except in the case of $\des$, in which they become polynomials upon multiplication by $n$.
Furthermore, the degree of the coefficient of $s_{\lambda[n]}$ decreases with $|\lambda|$.
We will see that this can be read off from the $T$-expansions of the statistics in Proposition~\ref{basic-expansions}.
To that end, we introduce three parameters associated to a regular statistic.

\begin{defn}
\label{simple-parameters}
For $((U,V),C,f)$ a packed triple, its \emph{size} is $|U| = |V|$, its \emph{shift} is $|C|$ and its \emph{power} is $|U| + \deg f - |C|$.
By analogy, for  $\Pi$ a set of packed triples $((U,V),C,f)$, if we can expand the regular statistic $\stat$ as
\begin{equation}
\label{stat-expansion-parameter}
\stat = \sum_{((U,V),C,f) \in \Pi} c_{(U,V),C,f} \cdot T^f_{(U,V),C}
\end{equation}
with the coefficients $ c_{(U,V),C,f} $ are all nonzero:
\begin{itemize}
\item  The {\em size} of the expansion
\eqref{stat-expansion-parameter} is the largest size of a packed
triple in $\Pi$.  
\item  The {\em shift} of the expansion
\eqref{stat-expansion-parameter} is the largest shift of a packed triple
in $\Pi$.
\item  The {\em power} of the expansion
\eqref{stat-expansion-parameter} is the largest power of a packed triple
in $\Pi$.
\end{itemize}
\end{defn}

The $T$-functions are linearly dependent, so the parameters size, shift, and power  can depend on the
expansion \eqref{stat-expansion-parameter} of the statistic $\stat$.
We define {\em size, shift,} and {\em power} of  $\stat$ itself  to be the minimum size, shift, and power of $\stat$ over all possible $T$-expansions
\eqref{stat-expansion-parameter} of $\stat$.

Every size $k$ regular statistic is $k$-local.
As with local statistics, regular statistics behave well under taking products.
The parameters size, shift, and power transform nicely under multiplication.

\begin{proposition}
	\label{p:simple-prod}
	Let $\mathrm{stat}_1$ and $\mathrm{stat}_2$ be regular statistics with sizes $k_1,k_2$, shifts $q_1,q_2$, and powers $p_1,p_2$ respectively.
	Then $\mathrm{stat}_1 \cdot \mathrm{stat}_2$ is a regular statistic of size  at most $k_1 + k_2$, shift at most  $q_1 + q_2$ and power  at most 
	$p_1 + p_2$.
\end{proposition}

\begin{proof}
Let $((U,V),C,f)$ and $((X,Y),D,g)$ be packed triples with $U \cup V = [m], X \cup Y = [\ell]$.
Applying linearity, it is enough to show this result when  $\stat_1 = T^f_{(U,V),C}$ and $\stat_2 =  T^g_{(X,Y),D}$.

We  expand the product $\stat_1 \cdot \stat_2 = T^f_{(U,V),C} \cdot T^g_{(X,Y),D}$ as a linear combination of indicator functions $\one_{I,J}$.
A typical term in this expansion  looks like
\begin{equation}
\label{fg-product}
(f(M) \cdot \one_{M(U),M(V)}) \cdot (g(L) \cdot \one_{L(X),L(Y)}) = f(M)g(L)\cdot \one_{M(U),M(V)} \cdot \one_{L(X),L(Y)}
\end{equation}
for a pair of sets $M \in \binom{[n]}{m}_C$ and $L \in \binom{[n]}{\ell}_D$.
Here we adopt the shorthand $f(M) := f(a_1, \dots, a_m)$ for $M = \{a_1 < \cdots < a_m \}$ and similarly for $g(L)$.

By Lemma~\ref{compatible-lemma}, if the expression \eqref{fg-product} does not vanish 
it equals  $f(M)g(L) \cdot \one_{I,J}$ for some partial permutation $(I,J)$ with $I \cup J = M \cup L := S$ and $|S| = r \leq \ell+m$.
Applying Lemma~\ref{l:unique-pack}, we write $(I,J) = (S(U'),S(V'))$ for a unique packed partial permutation $(U',V')$.
The vincular constraints given by $C$ and $D$ on $S$ are given by the set $I(C) \cup J(D) \subseteq S$. 
The compression  $C' = [r](I(C)\cup J(D))$ of the set $C$ gives equivalent vincular constraints on the packed partial permutation
 $(U',V')$.
Furthermore, the product
 $f(M) \cdot g(L) = h(M\cup L)$ is a polynomial with $\deg h = \deg f + \deg g$ 
 so $T^{h}_{(U,'V'),C'}$ occurs as a summand in $T^f_{(U,V),C} \cdot T^g_{(X,Y),D}$.
Repeated subtraction shows that $T^f_{(U,V),C} \cdot T^g_{(X,Y),D}$ is regular.
The largest possible value of the size $|U'| = |V'|$ is $k_1 + k_2$; this occurs when $M \cap L = \varnothing$.
Similarly, the power $|U'| + \deg h - |C'| = |V'| + \deg h - |C'|$ is $p_1 + p_2$; this also occurs when $M \cap L = \varnothing$.
Finally, the shift achieves maximum value $|I(C) \cup J(D)| = q_1+q_2$ when $I(C) \cap J(D) = \varnothing$.
While these values are all attained for a given product, cancellation can occur, hence the inequality.
\end{proof}

We are ready to state our main result on regular statistics.
Theorem~\ref{local-support-theorem} implies that the Schur expansions of the symmetric functions attached to $k$-local
statistics have bounded support. 
For the subclass of regular statistics, we can give substantial asymptotic information about the coefficients appearing in these expansions.
If $f$ and $g$ are polynomials, the {\em rational degree} of $f/g$ is given by $\deg(f/g) := \deg f - \deg g$.

\begin{theorem}
\label{t:simple}
	Let $\mathrm{stat}$ be regular statistic of size $k$, shift $q$ and power $p$.
	Then
	\begin{equation}
	\ch_n(R \, \stat) = \sum_{|\lambda| \leq k} c_\lambda(n) \cdot s_{\lambda[n]}
	\end{equation}
	for some coefficients $c_\lambda(n)$ indexed by partitions $\lambda$ of size at most $k$.
	On the domain $\{n \geq 2k \}$, the coefficients
	$c_\lambda(n) \in \CC(n)$ are rational functions of $n$ of rational degree at most $p - |\lambda|$. 
	Furthermore, the product
$
(n)_q \cdot c_\lambda(n) \in \CC[n]
$
is a polynomial in $n$ on the domain $\{ n \geq 2k \}$.

On the full domain $\{ n \geq 1 \}$ of nonnegative integers, the function
\begin{equation}
(n)_q \cdot R\, \stat:\symm_{n} \to \CC
\end{equation}
can be viewed as an element of $\CC[n,m_1,\dots,m_k]$ of degree at most $p$ with the grading where $\deg n = 1$ and $\deg m_i = i$.
\end{theorem}

\begin{proof}
Let $(U,V)$ be a packed partial permutation of
 size $k$ satisfying $U \cup V = [m]$ and let $C \subseteq [m-1]$ and $f \in \CC[x_1,\dots,x_m]$.
If $T^f_{(U,V),C}$ appears in the vincular translate expansion of $\stat$,
we have
	\begin{align}
	\ch_n R \, T^f_{(U,V),C} &= \frac{1}{n!} \binom{n-|C|}{m-|C|} \overline{f}(n) \cdot A_{n,U,V} &(\mathrm{Corollary~\ref{c:vincular-trans}})\\
	&= \frac{\overline{f}(n)}{(n)_{|C|} (m-|C|)!} \cdot  \frac{A_{n,U,V}}{(n-m)!}\\
	&= 	\frac{\overline{f}(n)}{(n)_{|C|} (m-|C|)!} \sum_{|\lambda| \leq |U|} g_\lambda(n) \cdot s_{\lambda[n]}			&(\mathrm{Corollary}~\ref{atomic-asymptotics})
	\end{align}
	where $\deg \overline{f} = \deg f$ and the coefficient
	$g_\lambda(n)$ is a polynomial of degree at most $|U| - |\lambda|$ on the domain $\{n \geq 2 |U| \} \supseteq \{n \geq 2k \}$.
	Extracting the coefficient of $s_{\lambda[n]}$ gives  
	\begin{equation}
	\label{c-coefficient-expansion}
	c_\lambda(n) = \frac{\overline{f}(n)g_\lambda(n)}{(n)_{|C|}(m-|C|)!}.
	\end{equation}
	Since $q \geq |C|$ and $(m - |C|)!$ does not depend on $n$,
	the product $(n)_q \cdot c_\lambda(n)$ is a polynomial in $\CC[n]$  on the domain $\{n \geq 2k \}$. 
	Furthermore, the rational degree of $c_{\lambda}(n)$ is 
	\begin{equation} 
	\deg c_{\lambda} = \deg f + \deg g_{\lambda} - |C| \leq \deg f + |U| - |\lambda| - |C| \leq p - |\lambda|
	\end{equation}
	as required.
	
	For the second statement, Theorem~\ref{character-polynomial-theorem} guarantees that $(n)_q \cdot R \, \stat$ coincides with a polynomial
	$f \in \CC[n,m_1,\dots,m_k]$ of the required form {\bf provided} that $n \geq 2k$ so that the first part of the theorem applies.
	By Proposition~\ref{indicator-polynomiality-proposition} and Lemma~\ref{l:vincular-poly}, the product $(n)_q \cdot R \, \stat$ 
	is given by some polynomial $g \in \CC[n,m_1, \dots, m_k]$ {\bf for all $n \geq 0$} (which {\em a priori} does not satisfy the required degree
	bound).
	Lemma~\ref{cycle-identification-lemma} applies to show $f = g$ as elements of $\CC[n,m_1, \dots, m_k]$.
\end{proof}

%
%
%
%
%
%

\section{Closed statistics}
Proposition~\ref{p:simple-prod} implies that the regular permutation statistics $\symm_n \rightarrow \CC$ form a subalgebra of 
the family of all statistics $\symm_n \rightarrow \CC$.
In this section we define a further subalgebra of ``closed statistics".
In Theorem~\ref{matching-moments} we will apply closed statistics to recover a result on patterns in perfect matchings due to 
Khare, Lorentz, and Yan \cite{KLY}.

A constrained translate statistic $T^f_{(U,V),C}: \symm_n \rightarrow \CC$ 
indexed by a packed triple $((U,V),C,f)$ is {\em closed} if the packed partial permutation 
$(U,V) \in \symm_{n,k}$
is in fact a permutation in $\symm_k$,  that is, the letters appearing in $U$ are the same as those appearing in $V$.  We write $v \in \symm_k$
for the permutation $(U,V)$ and $T^f_{v,C} := T^f_{(U,V),C}$ and refer to the triple $(v,C,f)$ as a {\em permutation triple}.
The word ``closed" refers to the fact that the graph $G(U,V)$ consists entirely of (closed) cycles and size 1 paths.
Closed statistics are defined in a fashion analogous to regular statistics.

\begin{defn}
\label{closed-statistic-definition}
A permutation statistic $\stat$ is {\em closed} if there is a finite set $\Pi$ of permutation triples such that 
\begin{equation}
\stat = \sum_{(v,C,f) \in \Pi}  c_{v,C,f} \cdot T_{v,C}^f
\end{equation}
for some scalars $c_{v,C,f} \in \CC$.
\end{defn}

Every closed statistic is regular. As with regular statistics, closed statistics are closed under taking products.

\begin{proposition}
\label{closed-product}
If $\stat_1, \stat_2: \symm_n \rightarrow \CC$ are closed statistics, the pointwise product $\stat_1 \cdot \stat_2$ is also closed.
\end{proposition}

\begin{proof}
Let $(v,C,f)$ and $(u,D,g)$ be two packed permutation triples. It is enough to show that the product $T^f_{v,C} \cdot T^g_{u,D}$ is closed.
Lemma~\ref{compatible-lemma} implies that $T^f_{v,C} \cdot T^g_{u,D}$ expands in terms of indicator functions $\one_{I,J}$ for $(I,J) \in \symm_{n,k}$
where the sets of letters appearing in $I$ and $J$ coincide.
Therefore, the $T$-expansion of $T^f_{v,C} \cdot T^g_{u,D}$ derived in the proof of Proposition~\ref{p:simple-prod}  only involves closed $T$-functions.
\end{proof}

By Proposition~\ref{closed-product}, we have a chain of subalgebras
\begin{center}
$\left\{
\begin{array}{c}
\text{closed statistics} \\
\symm \rightarrow \CC
\end{array} \right\}  \subseteq
\left\{
\begin{array}{c}
\text{regular statistics} \\
\symm \rightarrow \CC
\end{array} \right\}  \subseteq
\left\{
\begin{array}{c}
\text{all permutation statistics} \\
\symm \rightarrow \CC
\end{array} \right\}$
\end{center}
of permutation statistics under pointwise product.
We can say a little more about the expansion of $T^f_{v,C} \cdot T^g_{u,D}$ appearing in the proof of Proposition~\ref{closed-product}.
If $T^f_{w,D}$ appears in this expansion, the number $m_i(w)$ of $i$-cycles in $w$ is related to the numbers of $i$-cycles in $v$ and $u$ by the bounds
\begin{equation}
\max( m_i(v), m_i(u) ) \leq m_i(w) \leq m_i(v) + m_i(u).
\end{equation}
These bounds hold because the $i$-cycles in $w$ will come from cycles in $u$ or $v$ (or both).
The following three examples show that the containments are nontrivial.

\begin{example}
\label{not-regular}
If $\stat: \symm_n \rightarrow \CC$ is any regular statistic, it is not hard to see that $\stat = o(n!)$ as $n \rightarrow \infty$.
Therefore $n!$ is not a regular statistic.
\end{example}

\begin{example}
\label{regular-not-closed}
The statistic $n: \symm \rightarrow \CC$ is regular but not closed.  Indeed, as functions on $\symm$ we have
$n = T_{(1,1)} + T_{(1,2)} + T_{(2,1)}$ since for any $w \in \symm_n$ and any $1 \leq i \leq n$ precisely one of $w(i) < i, w(i) = i,$ or $w(i) > i$ is true.
We conclude that $w$ is regular.

On the other hand, if $n$ were closed then $n = \sum_{(v,C,f) \in \Pi} c_{v,C,f} \cdot T_{v,C}^f$ for some set $\Pi$ of permutation triples and some 
constants $c_{v,C,f}$.  If $n_0$ is maximal such that there exists $(v,C,f) \in \Pi$ with $v \in \symm_{n_0}$ then 
for any $n > n_0$ and any $n$-cycle $w \in \symm_n$ we have $\sum_{(v,C,f) \in \Pi} c_{v,C,f} \cdot T_{v,C}^f(w) = 0$. This contradiction shows 
$w$ is not closed.
\end{example}

\begin{example}
\label{m-are-closed}
For any $i \geq 1$, the statistic $m_i: \symm \rightarrow \CC$ is closed. In fact, we have $m_i = \sum_{v \in \symm_i} T_v$.
By Theorem~\ref{character-polynomial-theorem},
for any partition $\mu \vdash k$ the class function $\chi^{\mu[n]}$ is a polynomial in $\CC[m_1, \dots, m_k]$. The statistic 
$\chi^{\mu[-]}: \symm \rightarrow \CC$ is therefore closed.
\end{example}

Combining Examples~\ref{regular-not-closed} and Example~\ref{m-are-closed}, we see that any polynomial
in $\CC[n,m_1,m_2, \dots ]$ is a regular statistic.
An advantage of closed  over regular statistics is that their symmetric functions are relatively easy to compute.

\begin{proposition}
\label{closed-symmetric-function}
Let $(v,C,f)$ be a permutation triple with $v \in \symm_k$, $f \in \CC[x_1, \dots, x_k]$, and $C \subseteq [k-1]$ with $|C| = q$.
Let $\overline{f} \in \CC[n]$ be the polynomial of degree $\deg f$ determined by Lemma~\ref{l:vincular-poly}.  If $v$ has cycle type $\nu \vdash k$, then 
\begin{equation}
\ch_n \left( 
T^f_{v,C}
\right) =
\overline{f}(n) \cdot \frac{(n-k)!}{n!} \cdot \binom{n-q}{k-q} \cdot s_{n-k} \cdot p_{\nu}.
\end{equation}
\end{proposition}

\begin{proof}
This follows from Proposition~\ref{atomic-factorization} and  Corollary~\ref{c:vincular-trans} together with the path power sum expansion 
$\vec{p}_{1^{n-k}} = (n-k)! \cdot s_{n-k}$
which follows from the Path Murnaghan-Nakayama Rule (Theorem~\ref{path-murnaghan-nakayama}).
\end{proof}

\chapter{Moments  of Pattern Count}
\label{Pattern}

The Reynolds operator has a simple probabilistic interpretation. Viewing a statistic $\stat: \symm_n \rightarrow \RR$ 
as a random variable on $\symm_n$ with respect to the uniform distribution, for any $w \in \symm_n$ of cycle type $\lambda$ we have
\begin{equation}
\label{eq:cond-exp}
R f(w) = \EE(\stat \mid K_{\lambda})
\end{equation}
which is the expectation of $\stat$ with respect to the uniform distribution on the conjugacy class $K_{\lambda}$.
When $\stat$ is regular, we can reinterpret Theorem~\ref{t:simple} to give a  structural characterization of moments of $f$ applied to uniformly random permutation on a fixed conjugacy class.

\begin{corollary}
	\label{cycle-type-asymptotics}
	Let $\stat$ be a regular permutation statistic with size $k$, shift $q$ and power $p$, and let $d \geq 1$.
	Then for $\Sigma_n$ a uniformly random permutation in $K_\lambda$, 
	\begin{equation}
		\label{eq:regular-moment}
		\EE\left(\stat^d(\Sigma_n) \right) = \frac{f(n,m_1(\lambda),\dots,m_{dk}(\lambda))}{(n)_{dq}}
	\end{equation}
	where $f$ is a polynomial of degree $dp+dq$ with the grading where $\deg(n) = 1$ and $\deg(m_i) =i$.
\end{corollary}

\begin{proof}
	The $d=1$ case follows by reinterpreting the second statement in Theorem~\ref{t:simple} as a conditional expectation as done in Equation~\eqref{eq:cond-exp}.
	By Proposition~\ref{p:simple-prod}, $\stat^d$ is a regular permutation statistic with size at most $dk$, shift at most $dq$ and power at most $dp$ so the case $d > 1$ follows.
\end{proof}

Corollary~\ref{cycle-type-asymptotics} recovers a litany of results on moments of permutation statistics by viewing them as regular statistics.
The first result of this form we are aware of is due to Zeilberger~\cite{Zeilberger}, who showed the $d$th moment of the random variable counting  permutation pattern occurrences in a uniformly random permutation is a polynomial in $n$.
This result has been generalized in several ways.
The most general family of permutation statistics we are aware where an analogous result was known are weighted bivincular pattern counts~\cite{DK}, whose moments  are rational functions in $n$ with predictable denominators.

When studying expectations for an arbitrary but fixed conjugacy class, the examples in Chapter~\ref{ss:examples} from~\cite{Hultman} are the first results we are aware of.
Gaetz and Ryba were the first to apply the theory of character polynomials to this setting, proving that moments of classical pattern counts for a uniform permutation  in a fixed conjugacy class are polynomials in $n$ and the short cycle counts~\cite{GR}.
Their work was later extended to linear combinations of classical pattern counts~\cite{GP}.
Both of these results rely on the partition algebra, and do not extend in an obvious way to the weighted or constrained settings.

We recover the results from~\cite{DK}   by showing their statistics are regular and that the expectation of any statistic on a uniformly random permutation is the $s_{(n)}$ coefficient computed in~Theorem~\ref{t:simple}.
Since linear combinations of classical pattern counts are regular statistics and the denominator in follow from Corollary~\ref{cycle-type-asymptotics} is always 1 for such functions, we also recover~\cite{GR} and~\cite{GP}.
Lastly, since perfect matchings of $[2n]$ can be identified with fixed point-free involutions (a single conjugacy class), we also recover results from~\cite{KLY} for weighted pattern counts in perfect matchings.

\section{Bivincular pattern counting}
Pattern enumeration gives a rich class of regular statistics.
If $\Upsilon \subset \symm$ is a finite set of patterns, the counting statistic $N_{\Upsilon}: \symm_n \rightarrow \CC$ is regular.
There is a more general notion of permutation patterns which gives a broader family of regular statistics.
If $f \in \CC[x_1, \dots, x_k]$ is a polynomial in $k$ variables and $I = (i_1, \dots, i_k)$ is a length $k$ sequence of integers, we abbreviate
$f(I) := f(i_1, \dots, i_k)$.

\begin{defn} {\em(\cite[Def. 2.1]{DK})}
\label{bivincular-statistics}
	A \emph{bivincular pattern}  is a triple $(v,A,B)$ where $v \in \symm_k$ and $A,B \subseteq [k-1]$.
	The partial permutation $(I,J)$ \emph{matches} $(v,A,B)$ 
	if 
	\begin{itemize}
	\item the sequences $I$ and $v(J)$ are increasing,
	\item
	 $a \in A$ implies $i_{a+1} = i_a + 1$, and 
	 \item
	 $b \in B$ implies $j_{v(b+1)} = j_{v(b) + 1}$.
	 \end{itemize}
	Given polynomials $f,g \in \RR[x_1,\dots,x_k]$, define a statistic $N^{f,g}_{v,A,B}: \symm_n \rightarrow \RR$ 
	associated to the quintuple $(v,A,B,f,g)$
	to be the weighted 
	enumeration
	\begin{equation}
	N^{f,g}_{v,A,B} := \sum_{(I,J) \ \mathrm{ matches}\ (v,A,B)} f(I)g(J)\cdot \one_{I,J}.
	\end{equation}
	We call the quintuple $(v,A,B,f,g)$ a {\em weighted bivincular pattern}.
	If $\Upsilon$ is a finite set of weighted bivincular patterns, we let $N_{\Upsilon}: \symm_n \rightarrow \RR$ be the sum
	\begin{equation}
	N_{\Upsilon} := \sum_{(v,A,B,f,g) \, \in \, \Upsilon} N^{f,g}_{v,A,B}
	\end{equation}
	of the corresponding statistics.
\end{defn}

When $A = B = \varnothing$ and $f \equiv g \equiv 1$ in Definition~\ref{bivincular-statistics}, we recover the pattern enumeration statistics
$N_{\Upsilon}$ from before. 
The $B = \varnothing$ case of Definition~\ref{bivincular-statistics} is known as {\em (weighted) vincular} pattern counting.
When $f \equiv g \equiv 1$ and $B = \varnothing$ we are in the situation of {\em unweighted vincular} pattern counting and write
$N_{v,A}$ instead of the more cumbersome $N^{f,g}_{v,A,\varnothing}$.
The statistics $\inv, \des,$ and $\maj$ are  instances of a statistic $N_{\Upsilon}$ as in Definition~\ref{bivincular-statistics}, but the 
statistic $\exc$ is not.

\begin{proposition}
\label{bivincular-simple}
Let $\Upsilon$ be a finite and nonempty set of weighted bivincular patterns.  The statistic $N_{\Upsilon}: \symm_n \rightarrow \RR$
is regular.
\begin{itemize}
\item
The size of $N_{\Upsilon}$ is $\leq$ the largest value of $k$ such that there exists $(v,A,B,f,g) \in \Upsilon$ with $v \in \symm_k$.
\item
The shift of $N_{\Upsilon}$ is $\leq$ the maximum value of $|A| + |B|$ among all $(v,A,B,f,g) \in \Upsilon$.
\item
The power of $N_{\Upsilon}$ is $\leq$ the maximum value of $k + \deg f + \deg g - |A| - |B|$ among all $(v,A,B,f,g) \in \Upsilon$ with $v \in \symm_k$.
\end{itemize}
\end{proposition}

\begin{proof}
It suffices to consider the case where $\Upsilon = \{ (v,A,B,f,g) \}$ is a single weighted bivincular pattern. The statistic
$N^{f,g}_{v,A,B}$ is certainly regular and the claims about size is clear.
For the claim about shift, when an indicator function $\one_{I,J}$ indexed by disjoint sets 
$I \cap J = \varnothing$ appears in the indicator expansion of $N^{f,g}_{v,A,B}$, the vincular conditions imposed by $A$ on $I$ and $B$ on $J$
are disjoint.
The claim about power can be seen in a similar way.
\end{proof}

Our results apply to weighted bivincular pattern enumeration as follows.
The following  is a generalization of a theorem of Gaetz and Ryba \cite[Thm. 1.1]{GR}.

\begin{corollary}
\label{bivincular-corollary}
Let $\Upsilon$ be a finite and nonempty set of weighted bivincular patterns $(v,A,B,f,g)$ with counting statistic
$N_{\Upsilon}: \symm_n \rightarrow \RR$.
Let $k$ be the largest size of a pattern $v$ with $(v,A,B,f,g) \in \Upsilon$, let
 $q$ be the maximum size of $|A| + |B|$ over all $(v,A,B,f,g) \in \Upsilon$, and let
 $p$ be the largest value of $r + \deg f + \deg g - |A| - |B|$ where $v \in \symm_r$ and $(v,A,B,f,g) \in \Upsilon$.
 \begin{itemize}
 \item
 For any $d \geq 1$, we have the Schur expansion
\begin{equation}
\ch_n( R N_{\Upsilon}^d ) = \sum_{|\lambda| \leq dk} c_{\lambda}(n) \cdot s_{\lambda[n]}
\end{equation}
where on the domain $\{ n \geq 2 d k \}$ the coefficient $c_{\lambda}(n)$ is a  rational function of $n$ of rational degree $\leq p - |\lambda|$.
Furthermore, on the domain $\{ n \geq 2 d k \}$ the product $(n)_{dq} \cdot c_{\lambda}(n) \in \RR[n]$ is a polynomial in $n$. 
\item 
For any $w \in \symm_n$, the number $R \, N_{\Upsilon}^d$ is a polynomial of $n, m_1, m_2, \dots, m_{dk}$ divided by $(n)_{dp}$.
Here the grading sets $\deg n = 1$ and $\deg m_i = i$. 
In particular, if $A = B = \varnothing$ for all $(v,A,B,f,g) \in \Upsilon$ then $p = 0$ and $R \, N_{\Upsilon}^d$ is a polynomial in
$n, m_1(w), m_2(w), \dots, m_{dk}(w)$ .
\end{itemize}
\end{corollary}

\begin{proof}
Combine Proposition~\ref{p:simple-prod}, Theorem~\ref{t:simple}, and Proposition~\ref{bivincular-simple} for the first claim.
The second claim is a consequence of Proposition~\ref{p:simple-prod} and Corollary~\ref{cycle-type-asymptotics}.
\end{proof}

Gaetz and Ryba \cite{GR} proved Corollary~\ref{bivincular-corollary} in the special case where 
$\Upsilon = \{ (v,\varnothing,\varnothing,1,1) \}$ consists of a single unweighted pattern $v \in \symm_k$ free of bivincular conditions.
The method of proof in \cite{GR} was representation-theoretic.
Writing  $M_n = \RR^n$ for the defining representation of $\symm_n$, the tensor power $M_n^{\otimes k} = M_n \otimes \cdots \otimes M_n$
carries a diagonal action of $\symm_n$. 
There is an epimorphism from the $k^{th}$ partition algebra $\PPP_k(x)$ at parameter $x = n$ to the endomorphism ring 
$\End_{\symm_n}(M_n^{\otimes k})$, and the actions of $\RR[\symm_n]$ and $\PPP_k(n)$ on $M_n^{\otimes k}$ generate each others
commutator subalgebras.
In \cite{GR} the main result is proven by examining the action of $\PPP_k(n)$ on $\End_{\symm_n}(M_n^{\otimes k})$.
It is unclear how to apply double commutant theory to the case where
$\Upsilon$ contains more than one pattern or has nontrivial bivincular conditions.
Gaetz and Pierson \cite[Thm. 1.3]{GP} built on the results of \cite{GR}  to prove the case of Corollary~\ref{bivincular-corollary}
where $\Upsilon$ is a finite list of patterns (without bivincular conditions).
The result \cite[Thm. 1.3]{GP} is only stated in the case $d = 1$, but the result for $d \geq 1$ follows by expanding
$N_{\Upsilon}^d = \left( \sum_{v \in \Upsilon} N_v \right)^d$.

Theorem~\ref{t:simple} also extends work of Dimitrov and Khare \cite{DK}. To explain how, we need a general fact about 
the symmetric function $\ch_n(R \, f)$ attached to a statistic $f: \symm_n \rightarrow \RR$.
We can ask for the average value of a statistic $f$ on the entire symmetric group rather than a given conjugacy class.
At the level of symmetric functions, this corresponds to extracting the coefficient of $s_n$.

\begin{proposition}
\label{average-value-extraction}
For any statistic $f: \symm_n \rightarrow \RR$, the expectation $\EE(f) = \frac{1}{n!} \sum_{w \in \symm_n} f(w)$ is given by
\begin{equation}
\EE(f) = \langle \ch_n(R \, f), s_n \rangle = \text{{\em coefficient of $s_n$ in $\ch_n(R \, f)$}}.
\end{equation}
\end{proposition}

\begin{proof}
Recall that the power sums form an orthogonal basis of $\Lambda_n$ with
$\langle p_{\lambda}, p_{\mu} \rangle = z_{\mu}$.  We calculate
\begin{equation}
\EE(f) = \frac{1}{n!} \sum_{w \in \symm_n} f(w) = \frac{1}{n!} \sum_{\lambda \vdash n} |K_{\lambda}| \cdot Rf(\lambda) =
\sum_{\lambda \vdash n} \frac{Rf(\lambda)}{z_{\lambda}} = \langle \ch_n(Rf), s_n \rangle
\end{equation}
where we write $Rf(\lambda)$ for the value of $Rf: \symm_n \rightarrow \RR$ on any permutation of cycle type $\lambda$
and use the expansion $s_n = \sum_{\lambda \vdash n} \frac{1}{z_{\lambda}} p_{\lambda}$.
\end{proof}

It would be interesting to find probabilistic interpretations of other parts of the Schur expansion of $\ch_n(R \, f)$.
 The following application of Proposition~\ref{average-value-extraction} is equivalent to a result \cite[Thm. 4.5]{DK} of Dimitrov and Khare.

\begin{corollary}
\label{average-value-corollary}
Let $\Upsilon, p,$ and $q$ be as in Corollary~\ref{bivincular-corollary}. For any $d \geq 1$, the $d^{th}$ moment of the statistic
$N_{\Upsilon}: \symm_n \rightarrow \RR$ is of the form
\begin{equation}
\EE(N_{\Upsilon}^d) = \frac{f(n)}{(n)_{dq}}
\end{equation}
where $f(n) \in \RR[n]$ is a polynomial in $n$ of degree $\leq dp$.
\end{corollary}

\begin{proof} 
Applying Proposition~\ref{p:simple-prod},
Proposition~\ref{bivincular-simple},
and Proposition~\ref{average-value-extraction}, 
this amounts to applying Theorem~\ref{t:simple} in the case $\stat = N_{\Upsilon}^d$
to the coefficient $c_{\varnothing}(n)$.
Corollary~\ref{constrained-one-row} shows that $c_{\varnothing}(n)$ has the required form.
\end{proof}

\section{The coefficient of $s_n$} The Schur expansions coming from the statistics
$\exc, \inv, \des,$ and $\maj$  in
Proposition~\ref{basic-expansions} share a common feature: the degree of the coefficient of $s_n$
(as a rational function of $n$) is strictly greater than the degrees of all other coefficients.  Not all regular statistics have this property.
For example, let $\fix: \symm_n \rightarrow \RR$ count the fixed points $w(i) = i$ of a permutation $w$.
Then
$\fix = T_{(1,1)}$
is a regular class function satisfying
\begin{equation}
\ch_n(R \, \fix) = \ch_n(\fix) = s_n + s_{n-1,1}
\end{equation}
where the coefficients both have degree 0.
This does not happen for weighted bivincular pattern counts.

\begin{figure}
\begin{center}
\begin{tikzpicture}[scale = 0.5]

 \node [draw, circle, fill = white, inner sep = 1.2pt] at (0,0) { }; 
 \node [draw, circle, fill = white, inner sep = 1.2pt] at (0,-2) { }; 
 \draw [->, thick]  (0,-0.3) -- (0,-1.7);
 
  \node [draw, circle, fill = white, inner sep = 1.2pt] at (1,0) { }; 
 \node [draw, circle, fill = white, inner sep = 1.2pt] at (1,-2) { }; 
 \draw [->, thick]  (1,-0.3) -- (1,-1.7);

 \node at (3,-1) {$\cdots$};
 \node at (3,-3) {$dr$ paths};
 \node at (13,-1) {$\cdots$};
   \node at (13,-3) {$n-dr$ vertices};
 
  \node [draw, circle, fill = white, inner sep = 1.2pt] at (5,0) { }; 
 \node [draw, circle, fill = white, inner sep = 1.2pt] at (5,-2) { }; 
 \draw [->, thick]  (5,-0.3) -- (5,-1.7);

  \node [draw, circle, fill = white, inner sep = 1.2pt] at (6,0) { }; 
 \node [draw, circle, fill = white, inner sep = 1.2pt] at (6,-2) { }; 
 \draw [->, thick]  (6,-0.3) -- (6,-1.7);

  \node [draw, circle, fill = white, inner sep = 1.2pt] at (10,-1) { }; 
  \node [draw, circle, fill = white, inner sep = 1.2pt] at (11,-1) { }; 
  \node [draw, circle, fill = white, inner sep = 1.2pt] at (15,-1) { }; 
  \node [draw, circle, fill = white, inner sep = 1.2pt] at (16,-1) { };

\end{tikzpicture}
\end{center}
\caption{The graph appearing in the proof of Corollary~\ref{degree-max-corollary} which contributes the maximum degree.}
\label{fig:max-degree}
\end{figure}
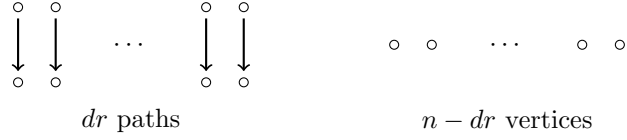

\begin{corollary}
\label{degree-max-corollary}
Preserve the notation of Corollary~\ref{bivincular-corollary}. Assume the polynomials $f, g$ for weighted
bivincular patterns $(v,A,B,f,g) \in \Upsilon$ all have positive  coefficients.
 The rational degree of  the coefficients $c_{\lambda}(n)$ 
 as $\lambda$ varies over partitions with $|\lambda| \leq dk$
is uniquely maximized when $\lambda = \varnothing$.  Furthermore, the coefficient
$(n)_{dq} \cdot c_{\varnothing}(n)$ is a polynomial of degree precisely $dp + dq$.
\end{corollary}

\begin{proof}
In order to compute the coefficient $c_{\lambda}(n)$ of $s_{\lambda[n]}$, we first expand $N_{\Upsilon}^d$ in terms of the $T$-statistics,
then use Corollary~\ref{c:vincular-trans} to pass to atomics, and finally use Corollary~\ref{atomic-schur-expansion} to get the Schur expansion.
The crucial observation is that the $T$-expansion of $N_{\Upsilon}^d$ contains a term on which the coefficient of $s_n$ achieves the required
degree.

As in Corollary~\ref{bivincular-corollary}, let $(v,A,B,f,g) \in \Upsilon$ be so that $v \in \symm_r$ and 
$v - |A| - |B| + \deg f + \deg g = p$.
There is a term $T^h_{(U,V),C}$ in the $T$-expansion 
of $N_{\Upsilon}^d$ corresponding to a packed graph $G(U,V)$ consisting of $dr$ paths of size 2 and 
$n-dr$ singleton paths; see Figure~\ref{fig:max-degree}. The graph $G(U,V)$ corresponds to $d$ instances of the size $r$ pattern $v$ which share no 
positions or values.
The sets $A$ and $B$ impose a total of $d \cdot (|A| + |B|)$  bivincular conditions $C$, 
and the polynomial $h \in \RR[x_1, \dots, x_{2dr}]$ is of degree
 $d \cdot (\deg f + \deg g)$.
By Corollary~\ref{constrained-one-row}, the coefficient of $s_n$
in the expansion of 
\begin{equation*}
(n)_{dq} \cdot \ch_n(R \, T^h_{(U,V),C}) 
\end{equation*}
is a polynomial in $n$ of degree $dp + dq$.
The positivity assumptions guarantee that there will be no cancellation in the coefficient of $s_n$ coming from other terms in the $T$-expansion,
so Corollary~\ref{constrained-one-row} implies that the degree of $(n)_{dq} \cdot c_{\varnothing}(n)$ is precisely $dp + dq$.
Corollary~\ref{bivincular-corollary} also implies that the degrees of the polynomials  
 $(n)_{dq} \cdot c_{\lambda}(n)$ for $\lambda \neq \varnothing$ are $\leq dp + dq - |\lambda| < dp + dq$.
\end{proof}

\section{Patterns in perfect matchings}
Given a local statistic $f: \symm_n \rightarrow \RR$ and a partition $\lambda \vdash n$, the $d^{th}$ moment 
$\EE( f^d \mid K_{\lambda} )$ of $f$ restricted to the conjugacy class $K_{\lambda}$ may be obtained 
from the power sum expansion of $\ch_n(R \, f^d)$ by
\begin{equation}
\EE( f^d \mid K_{\lambda} ) = z_{\lambda} \times \text{coefficient of $p_{\lambda}$ in $\ch_n(R \, f^d)$}.
\end{equation}
In practice,
these moments are often difficult to calculate since the number of nonisomorphic graphs $G \in \GGG_n$ involved in the atomic expansion
of $\ch_n(R \, f^d)$, and hence the number of terms in the atomic expansion of $\ch_n(R \, f^d)$, can grow dramatically with $d$.
In this section we apply closed statistics to say more about these moments in the case of matchings, reproving a result
of Khare, Lorentz, and Yan \cite{KLY}.

A {\em perfect matching}  of size $n$ is a partition of the set $[2n]$ into $n$ blocks of size $2$.
Perfect matchings of size $n$ coincide with permutations $t \in \symm_{2n}$ of cycle type $2^n$.
We write $\MMM_{2n} := K_{2^n}$ for the set of size $n$ perfect matchings.
We have  $|\MMM_{2n}| = (2n-1)!! := (2n-1) \cdot (2n-3) \cdot \cdots \cdot 3 \cdot 1$.

We aim to study patterns in perfect matchings. To this end, we define
a {\em partial matching} on $[k]$ to be a collection of pairs in $\binom{[k]}{2}$ whose elements are mutually disjoint. 
A partial matching $A = \{ \{ a_1, b_1\}, \dots, \{ a_j, b_j \} \}$ with $j$ pairs may be identified with the partial permutation $ (I_A, J_A)$ of size $2j$ where
$I_A = (a_1, \dots, a_j, b_1, \dots, b_j)$ and $J_A = (b_1, \dots, b_j, a_1, \dots, a_j)$.

Partial matchings play the role of patterns in perfect matchings. To state our results in greater generality, we introduce a vincular condition and consider 
polynomial weights.
A {\em vincular $\MMM$-pattern of size $k$} is a pair $(A,C)$ where $v$ is a partial matching on $[k]$ and 
$C \subseteq [k-1]$ is such that either $k$ is matched in $A$ or $k-1 \in C$.
We say $t \in \MMM_{2n}$ {\em contains $(A,C)$} at $L = \{ \ell_1 < \cdots < \ell_k \} \subseteq [2 n]$ if
\begin{itemize}
\item
 for all pairs $\{i,j\}$ in $A$, the pair $(\ell_i, \ell_j)$ is a 2-cycle in $t$, and
 \item
 for all $i \in C$ we have $\ell_{i+1} = \ell_i + 1$, or equivalently $L \in \binom{[2n]}{k}_C$.
 \end{itemize}
This is similar to the notion of vincular patterns in permutations, but we allow vincular conditions on unmatched values in $A$.
Given $t \in \MMM_{2n}$ and $L \in \binom{[2n]}{k}_C$, we let 
\begin{equation}
\one^{\MMM}_{A,L}(t) = \begin{cases}
1 & \text{$t$ contains $(A,C)$ at $L$} \\
0 & \text{otherwise}
\end{cases}
\end{equation}
be the corresponding indicator statistic.
Given a polynomial weight $f \in \RR[x_1, \dots, x_k]$ is a polynomial, we define a statistic 
\begin{equation}
N_{A,C}^f = \sum_{L} f(\ell_1, \dots, \ell_k) \cdot \one^{\MMM}_{A,L}
\end{equation}
where the index ranges over all subsets $L = \{\ell_1 < \cdots < \ell_k \} \in \binom{[2n]}{k}_C$.
The matching statistics $N_{A,C}^f$ are closed.

\begin{lemma}
\label{matching-t-expansion}
Let $f \in \RR[x_1, \dots, x_k]$ and let $(A,C)$ be a vincular $\MMM$-pattern of size $k$ with $|A| = j$.
The statistic $N^f_{A,C}$ is a sum of closed constrained translates $T^f_{v,C}$ where $v \in \symm_{2k-2j}$ has cycle type $2^{k-j}$.
In particular, the statistic $N^f_{A,C}$ is closed.
\end{lemma}

\begin{proof}
If we set $\Gamma := \{ v \in \MMM_{2k-2j} \,:\, \text{every pair $\{i,j\} \in A$ is a 2-cycle $(i,j)$ in $v$} \}$, we have
\begin{equation}
N^f_{A,C} = \sum_{v \in \Gamma} T^f_{v,C}
\end{equation}
and the result follows.
\end{proof}

The $T$-expansion in the proof of Proposition~\ref{matching-t-expansion} is not necessarily as simple as possible since unmatched entries in $A$ can
pair with each other.
If $\Phi$ is a nonempty and finite set of triples $(A,C,f)$ where $(A,C)$ is a vincular $\MMM$-pattern of size $k$ and $f \in \RR[x_1, \dots, x_k]$ 
is a polynomial, we write
\begin{equation}
N^{\MMM}_{\Phi} := \sum_{(A,C,f) \in \Phi} N^f_{A,C}
\end{equation}
for the sum of the pattern counting statistics. Motivated by Proposition~\ref{matching-t-expansion}, we define
\begin{itemize}
\item the {\em $\MMM$-size} of $\Phi$ to be the maximum value of $k$ such that $(A,C,f) \in \Phi$ and $A$ is a partial matching on $[k]$,
\item the 
{\em $\MMM$-power} of $\Phi$ be the maximum value of $k + \deg f$ where $(A,C,f) \in \Phi$ and $A$ is a partial matching on $[k]$, and
\item the 
{\em $\MMM$-shift} of $\Phi$ to be the maximum value of $|C|$ where $(A,C,f) \in \Phi$.
\end{itemize}
The moments of $N^{\MMM}_{\Phi}$ have a nice form which depends on these parameters.
The following result is equivalent to \cite[Thm. 3.3]{KLY}.

\begin{theorem}
\label{matching-moments}
{\em (Khare-Lorentz-Yan \cite{KLY})}
Let $\Phi$ be as above with $\MMM$-size $k$, $\MMM$-power $p$, and $\MMM$-shift $q$. Let $d \geq 0$.  There exists a polynomial $g$ of degree 
at most $dp$ such that the $d^{th}$ moment of $N_{\Phi}^{\MMM}$ is given by
\begin{equation}
\EE_{\MMM_{2n}} \left(  (N_{\Phi}^{\MMM})^d \right) = 
\frac{(2n-dk-1)!!}{(2n-1)!!} \cdot \binom{2n-dq}{dk-dq} \cdot g(n)
\end{equation}
for $n$ sufficiently large.
\end{theorem}

\begin{proof}
By the discussion following Proposition~\ref{closed-product} and
 Lemma~\ref{matching-t-expansion}, for any $d \geq 0$ the statistic $(N^{\MMM}_{\Phi})^d$ is a linear combination of statistics
 of the form $N^{\MMM}_{\Phi'}$
 of $\MMM$-size at most $dk$, $\MMM$-power at most $dp$, and $\MMM$-shift at most $dq$.
  This reduces us to the case where $d = 1$ and $\Phi = \{ (A,C,f) \}$ is a singleton with $A$ a partial matching on $[k]$ with $|A| = j$
  and $|C| = q$.
 
If we let
 $\Gamma := \{ v \in \MMM_{2k-2j} \,:\, \text{every pair $\{i,j\} \in A$ is a 2-cycle $(i,j)$ in $v$} \}$ 
 be the set of perfect matchings appearing in the proof of Lemma~\ref{matching-t-expansion}, the reasoning of that proof gives
 \begin{align}
 \ch_{2n} \left( R \, N_{A,C}^f \right) &= \sum_{v \in \Gamma} \ch_{2n} \left(  R \, T_{v,C}^f  \right) \\
 &= \frac{1}{(2n)!}  \sum_{v \in \Gamma} \binom{2n - q}{k-q} \overline{f}(n)  (2n-k)! \cdot s_{2n-k}  \cdot p_{2^{k/2}}  \\
  &= \frac{|\Gamma|}{(2n)!}   \binom{2n - q}{k-q} \overline{f}(n)  (2n-k)! \cdot s_{2n-k} \cdot  p_{2^{k/2}} 
 \end{align}
 where we applied Proposition~\ref{closed-symmetric-function} for the second equality. Since the coefficient of $p_{2^{(2n-k)/2}}$ in $(2n-k)! \cdot s_{2n-k}$
 is $(2n-k-1)!!$ for $k$ even, we have
 \begin{align}
 \EE_{\MMM_{2n}} \left( N_{A,C}^f   \right) &= z_{2^n} \times \text{coefficient of $p_{2^n}$ in $\ch_{2n}\left( R \, N_{A,C}^f \right)$} \\
 &= \frac{|\Gamma|}{(2n-1)!!} \binom{2n-q}{k-q} \overline{f}(n) (2n-k-1)!!
 \end{align}
 where $\deg \overline{f} = \deg f$.  Bringing the constant $|\Gamma|$ into $\overline{f}(n)$ completes the proof.
\end{proof}

A similar result should be true for partial matchings, which correspond to involutions.
We omit the details.

\chapter{Convergence for Local and Regular Statistics}
\label{Convergence}

In this section we analyze limiting behavior of regular and local permutation statistics using properties of their moments.
For local statistics, we present a general method to translate the limiting behavior on a uniformly random permutation to permutations drawn from certain non-uniform distributions.
This allows us to improve on a variety of previous results in this vein.
Most notably, we extend  the asymptotic normality of vincular statistics~\cite{Hofer} to conjugacy classes with all long cycles, generalizing a classic result of Fulman for $\Des$~\cite{Fulman}.

We also study properties of normalized regular statistics applied to a permutations drawn from conjugacy invariant distributions.
Specifically, we show their expectations depend only on the limiting proportion of fixed points and their variances depends only on the limiting proportion of fixed points and two-cycles.
One applications of these theorems is a weak law of large numbers for these normalized regular statistics.
Additionally, we present a proof that uniformly random permutations drawn from a sequence conjugacy classes whose proportion of fixed points converges converges to a permuton\footnote{We thank Valentin F\'eray for showing us how to characterize this permuton}.

\section{Prior work}
\label{ss:prior-work}
Let $v \in \symm_k$ be a pattern, let $A \subseteq [k-1]$, and consider the statistic $N_{v,A}: \symm_n \rightarrow \RR$ that counts 
instances of the vincular pattern $(v,A)$.
The limiting behavior of the statistics $N_{v,A}$ has received extensive study. By linearity of expectations we have
\begin{equation}
\EE(N_{v,A}) = \frac{n^{k-q}}{k! (k-q)!}
\end{equation}
when $|A| = q$.
Hofer proved that 
the random variables $N_{v,A}: \symm_n \rightarrow \RR$ converge to a normal distribution as $n \rightarrow \infty$.

\begin{theorem}
\label{hofer-convergence}
{\em (Hofer \cite[Thm. 3.1]{Hofer})}
Let $v \in \symm_k$, $A \subsetneq [k-1]$ with $|A|=q$ and $X_n = N_{v,A}(\Sigma_n)$ wher $\Sigma_n \sim Unif(\symm_n)$.
Then there exists a constant $\sigma_{v,A}>0$ so that as $n \to \infty$ we have
\begin{equation}
\left(  \frac{X_n - \frac{n^{k-q}}{k! (k-q)!}}{n^{k-q-\frac{1}{2}}}
 \right) \xrightarrow{d} \NNN(0, \sigma^2_{v,A}).
\end{equation}
\end{theorem}

The analogue of this result for generalized constraints appears as \cite[Thm. 14.1]{Janson}.
F\'eray \cite{Feray} had previously shown convergence to the normal distribution assuming $\VV(X_n) = \Theta(n^{2k-2q-1})$.

We want to consider the restriction of the statistics $N_{v,A}$ to a conjugacy class $K_{\lambda}$ in the limit as $n \rightarrow \infty$.
To do this, we must impose conditions on how a sequence $\{\lambda^{(n)}\}$ of partitions of $n$ will grow.
Recall that $m_i(\lambda)$ denotes the multiplicity of $i$ as a part of the partition $\lambda$. We say that a sequence
$\{ \lambda^{(n)} \}$ of partitions of $n$ has {\em all long cycles} if 
\begin{equation}
\lim_{n \rightarrow \infty} m_i(\lambda^{(n)}) = 0
\end{equation}
for all $i$.
Fulman proved the following result for the vincular statistic $\des$.

\begin{theorem}
\label{fulman-theorem}
{\em (\cite[Cor. 4]{Fulman})}
Let $\{ \lambda^{(n)} \}$ be a sequence of partitions of $n$ with all long cycles. Then
\begin{equation}
\left(  
\frac{\des - n/2}{\sqrt{n/12}} \, \, \middle\vert \, \,  K_{\lambda^{(n)}} \right) \xrightarrow{d} \NNN(0,1)
\end{equation}
as $n \rightarrow \infty$.
\end{theorem}

Kim \cite{Kim} proved an analogue of Theorem~\ref{fulman-theorem} for the sequence of cycle types $(2^n)$.
The following version of Theorem~\ref{fulman-theorem} for all cycle types is due to Kim and Lee.

\begin{theorem}
\label{kl-theorem}
{\em (\cite[Thm. 1.4]{KL1})}
Let $\{ \lambda^{(n)} \}$ be a sequence of partitions of $n$ and suppose $m_1(\lambda^{(n)}) \rightarrow \alpha \cdot n$ for some $0 < \alpha < 1$.
Then
\begin{equation}
\left(  
\frac{\des - n(1-\alpha^2)/2}{\sqrt{n(1- 4 \alpha^3 + 3 \alpha^4)/12}} \, \, \middle\vert \, \, K_{\lambda^{(n)}} \right) \xrightarrow{d} \NNN(0,1)
\end{equation}
as $n \rightarrow \infty$.
\end{theorem}

Fulman, Kim, and Lee proved \cite{FKL} a version of Theorem~\ref{kl-theorem} for the statistic $\peak$ that also depends on the proportion $\alpha$
of fixed points.
For the general vincular statistics $N_{v,A}$, the state of the art prior to our work was as follows\footnote{Building on our work, Feray and Slim have since  extended Theorem~\ref{kl-theorem} to all vincular statistics~\cite{FK}.}.
A probability distribution $\mu_n: \symm_n \rightarrow [0,1]$ is {\em conjugacy invariant} if $\mu_n(vwv^{-1}) = \mu_n(w)$ for all $w,v \in \symm_n$.
Recall $\cyc(w)$ is the total number of cycles in a permutation $w$.

\begin{theorem}
\label{kammoun-theorem}
{(\cite[Prop. 31]{Kammoun})}
For all $n$, let $\mu_n$ be a conjugacy invariant distribution on $\symm_n$ so that for $\Sigma_n \sim \mu_n$ we have
\begin{equation}
\PP \left(  \lim_{n \rightarrow \infty} \cyc(\Sigma_n)/\sqrt{n} = 0 \right) = 1.
\end{equation}
For $v \in \symm_k$ and $A \subseteq [k-1]$ with $|A| = q$, let $X_n = N_{v,A}(\Sigma_n)$ where $\Sigma_n \sim \mu_n$. 
With $\sigma_{v,A} > 0$ as in Theorem~\ref{hofer-convergence}, we have 
\begin{equation}
\left(  \frac{X_n - \frac{n^{k-q}}{k! (k-q)!}}{n^{k-q-\frac{1}{2}}}
 \right) \xrightarrow{d} \NNN(0, \sigma^2_{v,A})
\end{equation}
as $n \rightarrow \infty$.
\end{theorem}

If $\{ \lambda^{(n)} \}$ is a sequence of partitions of $n$, a sequence of conjugacy invariant distributions $\mu_n$ on $\symm_n$
may be obtained by concentrating $\mu_n$ on the permutations of cycle type $\lambda^{(n)}$, viz.
\begin{equation}
\mu_n(w) = \begin{cases}
z_{\lambda^{(n)}}/n! & w \in K_{\lambda^{(n)}} \\
0 & \text{otherwise.}
\end{cases}
\end{equation}
Theorem~\ref{kammoun-theorem} applies to such a sequence $\mu_n$ if and only if $\frac{\ell( \lambda^{(n)} )}{\sqrt{n}} \rightarrow 0$ 
as $n \rightarrow \infty$.
The special case of Theorem~\ref{kammoun-theorem} where $\mu_n$ is concentrated on the conjugacy class of a single cycle $(n)$ will
be especially important. 
To be self-contained, we give a proof below depending only on Hofer's Theorem~\ref{hofer-convergence}.

\begin{corollary}
\label{kammoun-one-cycle}
Let $v \in \symm_k$, $A \subsetneq [k-1]$ with $|A| = q$ and $X_n = N_{v,A}(\Sigma_n)$ with $\Sigma_n \sim Unif(K_{(n)})$.
With $\sigma_{v,A} > 0$ as in Theorem~\ref{hofer-convergence}, we have
\begin{equation}
\left(  \frac{X_n - \frac{n^{k-q}}{k! (k-q)!}}{n^{k-q-\frac{1}{2}}}  \,\, \middle\vert \,\, K_{(n)}
 \right) \xrightarrow{d} \NNN(0, \sigma^2_{v,A})
\end{equation}
as $n \rightarrow \infty$.
\end{corollary}

\begin{proof}
Let $\Pi_n$ be a uniformly random permutation in $\symm_n$ and let $c_n = \cyc(\Pi_n)$.
We construct a cyclic permutation in $\symm_n$ by picking two cycles in $\Pi_n$ uniformly at random and concatenating them,
repeating this process until the resulting permutation $\Sigma_n$ is cyclic.  As discussed in \cite{Kammoun}, $\Sigma_n$ is uniformly random on $K_{(n)}$.

Define $S = \{ 1 \leq i \leq n \,:\, \Sigma_n(i) \neq \Pi_n(i) \}$.  We observe that $|S| = c_n$ and $\Sigma_n(j) = \Pi_n(j)$ for $j \notin S$.
Then
\begin{align}
N_{v,A}(\Sigma_n) - N_{v,A}(\Pi_n) &= \sum_{I \in \binom{[n]}{k}_A} \sum_{J \in \binom{[n]}{k}}
 \left(  
 \one_{I,v(J)}(\Sigma_n) - \one_{I,v(J)}(\Pi_n) \right) \\
 &= \sum_{\substack{I \in \binom{[n]}{k}_A \\ I \cap S \neq \varnothing}} \sum_{J \in \binom{[n]}{k}} 
 \left(  
 \one_{I,v(J)}(\Sigma_n) - \one_{I,v(J)}(\Pi_n)
 \right).\nonumber
\end{align}
For fixed $I$, notice that 
\begin{equation}
\left|   
\sum_{J \in \binom{[n]}{k}}
 \left(  
 \one_{I,v(J)}(\Sigma_n) - \one_{I,v(J)}(\Pi_n)
 \right)
\right|  \leq 1
\end{equation}
since there is at most one instance of the pattern $v$ at the indices $I$. Furthermore, there are fewer than 
$c_n \cdot \binom{n}{k - q - 1}$ choices of $I \in \binom{[n]}{k}_A$ for which $I \cap S \neq \varnothing$ so that 
\begin{equation}
\left|
\frac{N_{v,A}(\Sigma_n) - N_{v,A}(\Pi_n)}{n^{k-q}}
\right| \leq \frac{c_n}{n}.
\end{equation}
The number of cycles in a uniformly random permutation in $\symm_n$ satisfies
\begin{equation}
\PP\left( \lim_{n \rightarrow \infty} \frac{c_n}{n} = 0 \right) = 1 
\end{equation}
(this is discussed in \cite[Appendix A]{Kammoun}), so 
\begin{equation}
\PP \left( 
\lim_{n \rightarrow \infty}
\left|
X_n -  
\frac{N_{v,A} - \frac{n^{k-q}}{k!(k-q)!}}{n^{k-q-\frac{1}{2}}} 
\right| = 0
\right) = 1.
\end{equation}
Since almost sure convergence implies convergence in distribution, the result follows from
Theorem~\ref{hofer-convergence}.
\end{proof}

Kammoun's methods also give a law of large numbers so long as the number of cycles is $o(n)$.

\begin{theorem}
\label{kammoun-large-numbers}
{\em (Corollary of \cite[Prop. 31]{Kammoun})}
For all $n$, let $\mu_n$ be a conjugacy invariant distribution on $\symm_n$ so that for $\Sigma_n \sim \mu_n$ we have
\begin{equation}
\PP \left(  \lim_{n \rightarrow \infty} \cyc(\Sigma_n)/n = 0 \right) = 1.
\end{equation}
For $v \in \symm_k$ and $A \subseteq [k-1]$ with $|A| = q$, let $X_n = N_{v,A}(\Sigma_n)$ where $\Sigma_n \sim \mu_n$.
Then
\begin{equation}
\left(  \frac{X_n}{n^{k-q}}
 \right) \xrightarrow{p} \frac{1}{k!(k-q)!}
\end{equation}
as $n \rightarrow \infty$.
\end{theorem}

\section{Local statistics}

Let $f$ be a permutation statistic.
For any sequence $\{ \lambda^{(n)} \}$ of partitions of $n$, we have a sequence 
of expectations $\EE(f \mid K_{\lambda^{(n)}} )$.  Changing the partition sequence $\{ \lambda^{(n)} \}$ will usually result in different expectation
sequences.
However, if $f$ is local and $\{ \lambda^{(n)} \}$ has all
long cycles, the tail of the expectation sequence is determined.

\begin{proposition}
\label{local-statistic-long} 
Fix $d,k \geq 0$ and
let $f:\symm_n \to \mathbb{R}$ be $k$-local.
For $\lambda \vdash n$ so that
$
m_1(\lambda) = m_2(\lambda) = \dots = m_{dk}(\lambda) = 0$,
we have
\begin{equation}
\EE( f^d \mid K_{\lambda} ) = \EE( f^d \mid K_{(n)} ).
\end{equation}

\end{proposition}

\begin{proof}
Since locality is submultiplicative, we see $f^d$ is $(d k)$--local.
Then Theorem~\ref{local-support-theorem} gives a class function expansion
\begin{equation}
R  f^d = \sum_{|\nu| \, \leq \, dk}  c_{\nu} \cdot \chi^{\nu[n]}.
\end{equation}
for some coefficients $c_{\nu}$.
For $w \in K_{\lambda}$ and $c^{(n)} = (1,2, \dots , n) \in \symm_n$ the long cycle,
Theorem~\ref{character-polynomial-theorem} implies that
\begin{equation}
\EE( f \mid K_{\lambda} )  = R f(w) = R f(c) = \EE( f \mid K_{(n)} )
\end{equation}
for any $n$ such that $m_i(\lambda) = 0$ for $i = 1, 2, \dots, dk$.
\end{proof}

We remark that Proposition~\ref{local-statistic-long} is the main result of the preprint~\cite{LLLSSY}, where it is proved in a purely combinatorial fashion.
We use Proposition~\ref{local-statistic-long} and the Method of Moments (Theorem~\ref{method-of-moments}) 
to generalize Fulman's Theorem~\ref{fulman-theorem} to all vincular statistics.

\begin{theorem}
\label{our-long-cycles}
Let $v \in \symm_k$ be a pattern and let $A \subsetneq [k-1]$ satisfy $|A| = q$.  Consider the random variable $X_n = N_{v,A}$ on $\symm_n$ and let 
$\{ \lambda^{(n)} \}$ be a sequence of partitions of $n$ with all long cycles. There exists $\sigma_{v,A} > 0$ so that
\begin{equation}
\left(
\frac{X_n - \frac{n^{k-q}}{k! (k-q)!}}{n^{k-q-\frac{1}{2}}} \, \, \middle\vert \, \, K_{\lambda^{(n)} }
\right)  \xrightarrow{d} \NNN(0, \sigma_{v,A}^2)
\end{equation}
as $n \rightarrow \infty$.
\end{theorem}

\begin{proof}
When $\lambda = (n)$ this is Corollary~\ref{kammoun-one-cycle}. 
In general, for any fixed $d$
the $d^{th}$ moment $X_n^d$ is a $dk$-local function.
Since $\{ \lambda^{(n)} \}$ has all long cycles, Proposition~\ref{local-statistic-long} guarantees that 
\begin{equation}
\EE(X_n^d \mid K_{\lambda^{(n)}}) = \EE(X_n^d \mid K_{(n)} ) 
\end{equation}
for $n$ sufficiently large.
Since this is true for all $d$, after standardizing $(X_n \mid K_{\lambda^{(n)}})$ the result follows from Theorem~\ref{method-of-moments}.
\end{proof}

Kammoun's Theorem~\ref{kammoun-theorem} and  our Theorem~\ref{our-long-cycles} are incomparable results.
For example, only Theorem~\ref{kammoun-theorem} applies to the sequence
$\lambda^{(n)} = (n - \lfloor \sqrt{n} \rfloor, 1^{\lfloor \sqrt{n} \rfloor} )$ whereas only Theorem~\ref{our-long-cycles}
applies to a sequence
$\lambda^{(n)}$ where $\lambda^{(n)}$ has approximately $n/\log n$ parts of size approximately $\log n / n$.

\begin{remark}
The conditions of Theorem~\ref{our-long-cycles} may be relaxed to the setting of conjugacy invariant probability measures 
$\mu_n$ on $\symm_n$ where the number of ``short'' cycles is almost surely $0$.
We chose not to pursue a precise statement since we suspected it would be much weaker than best possible for this family of statistics, a suspicion that has been vindicated by subsequent work in~\cite{FK}.
\end{remark}

\section{Regular statistics}
Proposition~\ref{local-statistic-long} on local statistics assumes that the partition $\lambda$ has no cycles below a certain size.
For regular statistics we can work with small cycles, though our results are less broad.
In particular, for $\{\lambda^{(n)}\}$ a sequence of partitions and $\stat$ a sufficiently generic regular statistic we show the limiting behavior of
$
\EE\left(\stat \mid K_{\lambda^{(n)}}\right)
$
only depends on the sequence $\{m_1(\lambda^{(n)})\}$ while
$
\VV \left(\stat \mid K_{\lambda^{(n)}}\right)
$
only depends on $\{m_1(\lambda^{(n)})\}$ and $\{m_2(\lambda^{(n)})\}$.
This allows us to prove a law of large numbers much stronger than Theorem~\ref{kammoun-large-numbers}.

Our main tool for understanding expected values is the following asymptotic result.

\begin{theorem}
\label{regular-class-asymptotics}
Let $\stat$ be a regular permutation statistic of size $k$ and power $p$.
There exists a polynomial
$g \in \RR[x_1, \dots, x_n]$ of degree at most $p$ such that
\begin{equation}
R \left( \frac{\stat}{n^p} \right) = g \left( \frac{m_1}{n}, \frac{m_2}{n^2}, \dots, \frac{m_k}{n^k}  \right) + O(n^{-1})
\end{equation}
where $m_i(w)$ is the number of cycles in $w$ of length $i$.
\end{theorem}

\begin{proof}
Theorem~\ref{t:simple}
 yields the equality 
 \begin{equation}
R \left( \frac{\stat}{n^p} \right) = \frac{f(n,m_1, \dots, m_k)}{n^{p+q}} \cdot \frac{n^q}{(n)_q}
 \end{equation} 
 of class functions on $\symm_n$ where $f$ is a polynomial of degree at most $p+q$ under the degree conventions $\deg n = 1$ and $\deg m_i = i$.
 Every monomial in $f$ is of the form $n^{a_0} m_1^{a_1} \cdots m_k^{a_k}$ where the exponents $a_0, a_1, \dots, a_k$ satisfy
 $a_0 + \sum_{i = 1}^k i \cdot a_i = \ell \leq p + q$. We write
 \begin{equation}
 \frac{n^{a_0} m_1^{a_1} \cdots m_k^{a_k}}{n^{p+q}} = \left(  \frac{n}{n}  \right) \left(  \frac{x_1}{n}  \right) \left(  \frac{x_2}{n^2}  \right) \cdots 
 \left(  \frac{x_k}{n^k}  \right) \cdot n^{-(p+q-\ell)}.
 \end{equation}
 Grouping terms of degree $\ell$ into the polynomial $g_\ell$, we have 
 \begin{equation}
 R \left( \frac{\stat}{n^p} \right) = \frac{n^q}{(n)_q} \cdot \sum_{\ell = 0}^{k+q} n^{-(p+q-\ell) } \cdot
  g_{\ell} \left(\frac{m_1}{n}, \frac{m_2}{n^2}, \dots, \frac{m_k}{n^k} \right)
 \end{equation}
 and the result follows by taking $g = g_{p+q}$.
\end{proof}

\begin{corollary}
\label{expectation-fixed-point}
	Let $\stat$ be a regular permutation statistic of power $p$ and $\{\lambda^{(n)}\}$ be a sequence of cycle types so that 
\begin{equation}
\lim_{n \to \infty}m_1(\lambda^{(n)})/n = \alpha \in [0,1].
\end{equation}
Then exists a polynomial $f \in \mathbb{R}[x]$ so that viewing $\stat$ as a random variable with respect to the uniform distribution
\begin{equation}
	\lim_{n \to \infty} \EE \left( \frac{\stat}{n^p}\mid K_{\lambda^{(n)}}\right) = f(\alpha).
\end{equation}
\end{corollary}

\begin{proof}
	Let $k$ be the size of $\stat$, define $g$ as in Theorem~\ref{regular-class-asymptotics} and observe that $m_i(\lambda^{(n)}) \leq n/i$ for all $i$.
	For $w \in K_{\lambda^{(n)}}$, we then see
\[
\lim_{n \to \infty} \EE \left( \frac{\stat}{n^p}\mid K_{\lambda^{(n)}}\right) = \lim_{n \to \infty} g \left(\frac{m_1}{n}, \frac{m_2}{n^2}, \dots, \frac{m_k}{n^k}  \right) + O(n^{-1}) = g(\alpha,0,\dots,0).
\]
Setting $f(x) = g(x,0,\dots,0)$, the result follows.
\end{proof}

We note that $n^p$ in Corollary~\ref{expectation-fixed-point} need not be the correct normalizing coefficient.

\begin{example}
The $2$-cycle counting statistic $m_2$ has power $2$ but expected value $1/2$, so with $f$ as in Corollary~\ref{expectation-fixed-point} we have $f(\alpha) = 0$ for all $\alpha$.
However, $m_2$ is maximized on the cycle type $(2^{n/2})$ ($n$ even) with the value $n/2$, we see $\EE(m_2 \mid K_{(2^{n/2})}) = n/2$, we see $\EE(m_2/n^p \mid K_\lambda) = O(n^{-1})$ for all $\lambda \vdash n$ as desired.
\end{example}

For a large class of regular statistics including vincular statistics, the normalization factor $n^p$ in Corollary~\ref{expectation-fixed-point} is the correct normalizing coefficient.
For such statistics, we now extend Corollary~\ref{expectation-fixed-point} to a law of large numbers.
To do so, we will show structural properties for the variance of a regular statistic.
First, we show a technical lemma about vincular translates.

\begin{lemma}
	\label{translate-product}
Let $((I,J),A,f)$ and $((K,L),B,g)$ be packed triples with respective sizes $k$, $\ell$, shifts $a$, $b$, powers $p=k-a+\deg(f)$, $q=\ell - b +\deg(g)$ and let $G_1, G_2$ be the graphs of $(I,J)$ and $(K,L)$.
Then the sum of coefficients for indicators with a fixed graph  $G$ with $v$ vertices in
\begin{equation}
\label{eq:translate-covariance}
	T_{(I,J),A}^f \cdot T_{(K,L),B}^g - T_{(I,J),A}^f \cdot R (T_{(K,L),B}^g)
\end{equation}
is $O(n^v)$  and $O(n^{p+q-1})$ when $G = G_1 \sqcup G_2$.
\end{lemma}

Implicitly,  Lemma~\ref{translate-product} is an assertion about the naive expansion of Equation~\eqref{eq:translate-covariance} into $1_{IJ}$'s. 

\begin{proof}
We treat the case where $G = G_1 \sqcup G_2$ first.
We will see this is the only case of real interest.

Let $\overline{f}$ and $\overline{g}$ be associated to $f$ and $g$ as in Lemma~\ref{l:vincular-poly}.
We first compute the sum of coefficients for indicators with graph $G_1 \sqcup G_2$ in $T_{(I,J),A}^f \cdot T_{(K,L),B}^g$, which have the form
	\[
	\one_{(S[I] \sqcup T[K], S[J] \sqcup T[L])}
	\]
	where $S \in \binom{[n]}{k}_A$ and $T \in \binom{[n]}{\ell}_B$.
	Note 
	\begin{equation}
\overline{g}_S(n) := \sum_{T = \{t_1 < \cdots < t_\ell \}\in \binom{[n]}{\ell}_B \cap \binom{[n]-S}{\ell}}  g(t_1,\dots,t_\ell)		
	\end{equation}
depends on the set $S$ in a non-trivial way.
However, one can see $\overline{g} - \overline{g}_S$ is of degree at most $q-1$ since the difference necessarily involves one of the $k$ elements of $S$, and for each of these fewer elements is available when selecting $T$.
Therefore, the sum of coefficients for such indicators will be
	\begin{equation}
\sum_{S = \{s_1 < \cdots < s_k\} \in \binom{[n]}{k}_A} f(s_1,\dots,s_k) \overline{g}_S = \sum_{S = \{s_1 < \cdots < s_k \} \in \binom{[n]}{k}_A} f(s_1,\dots,s_k) (\overline{g} + O(n^{q-1})),
\end{equation}
which is $ (\overline{f} \cdot \overline{g})(n) + O(n^{p+q-1})$.
	
	We now compute the sum of coefficients for indicator with graph $G_1 \sqcup G_2$ in $T_{(I,J),A}^1 \cdot R (T_{(K,L),B}^1)$, which expands as
	\begin{align}
T_{(I,J),A}^1 \cdot R (T_{(K,L),B}^1) &= T_{(I,J),A}^1 \cdot \overline{g}(n) R( \one_{(K,L)}) \\
& = \overline{g}(n) \sum_{S \in \binom{[n]}{k}} \one_{(S[I],S[J])} \cdot  \frac{1}{n!} \sum_{v \in S_n} \one_{(v \cdot K, v \cdot L)}.
	\end{align}
	For each partial permutation $(S[I],S[J])$ and a permutation $v$ so that the set $\{v(1), \dots, v(\ell)\} \cap S$ is empty, 
	we obtain an indicator with graph $G_1 \sqcup G_2$, while all other indicators have graphs with fewer vertices.
	There are $(n-k)_\ell \cdot (n-\ell)!$ such permutations $v$ for each $S$.
	Therefore the desired sum of coefficients is
	\begin{equation}
	\overline{f}(n) \cdot \overline{g}(n) \cdot \frac{(n-k)_\ell}{(n)_\ell} = (\overline{f}\cdot \overline{g})(n) + O(n^{p+q-1}),	
	\end{equation}
	and the result follows for the graph $G_1 \sqcup G_2$.
	
To obtain a graph $G \neq G_1 \sqcup G_2$ with $v$ vertices, note $k+\ell-v$ vertices must be merged with other vertices, which can only happen in finitely many ways. 
Therefore, a similar argument shows the sum of coefficients for indicators with graph $G$ in either product is at most $O(n^{p+q-v})$ and the result follows.
\end{proof}

Note the Reynolds operator applied to Equation~\eqref{eq:translate-covariance} is the covariance of $T_{(I,J),A}^f$ and $T_{(K,L),B}^g$.

We need an additional result about coefficients in character polynomials.
Recall for $\mu \vdash k$ that $\mu(n) = (\mu, 1^{n-k})$.

\begin{proposition}
\label{fixed-point-coefficient}
	Let  $\mu = (\mu_1,\dots,\mu_k)$ be a partition of $m$ with $\mu_k > 1$.
	Viewing $(n-m)!\ch^{-1}(\vec{p}_{\mu(n)})$ as a polynomial in $\RR[n,m_1,m_2,\dots,m_k]$, maximal degree terms of the form $n^a m_1^b$ have degree $k$.
\end{proposition}

The proof of Proposition~\ref{fixed-point-coefficient} is closely related to the proof of Lemma~\ref{indicator-polynomiality-lemma}.

\begin{proof}
	Let $f = \ch^{-1}(\vec{p}_{\mu(n)})$ and $w \in K_{(p,1^{n-p})}$.
	To compute $f(w)$, observe each of the fixed points must be filled by a $1$-path, while the $p$-cycle must be filled with the parts of $\mu$ and the remaining $1$-paths.
	To place the parts of $\mu$ on the cycle, first endow them with a cyclic order, then place the remaining $(p-m)$ singletons.
	There are $(k-1)!$ cyclic orders, $\binom{p-m+(k-1)}{k-1}$ ways to place the $1$-paths and $p$ locations to begin the path $\mu_1$ (after which all other path locations are determined).
	This gives a total of 
	\begin{equation}
		p(p-|\mu|+(k-1))_{k-1}
	\end{equation}
	ways to place the paths in $\mu$, which is a polynomial in $p$ of degree $k$.
	Since $p = n - m_1$, we see this is a polynomial in $n$ and $m_1$ of maximal degree $k$.
\end{proof}

As a consequence, we can show the limiting variance of a regular statistic depends only on the proportion fixed points and 2-cycles.

\begin{theorem}
\label{regular-variance}
	Let $\stat$ be a regular statistic of size $k$ and power $p$.
	Also, let $\{\lambda^{(n)}\}$ be a sequence of integer partitions so that
	\begin{equation}
		\lim_{n \to \infty} \frac{m_1(\lambda^{(n)})}{n} = \alpha \in [0,1], \quad \lim_{n \to \infty} \frac{m_2(\lambda^{(n)})}{n} = \beta \in [0,1].
	\end{equation}
	Viewing $\stat$ as a random variable with respect to the uniform distribution, there exists polynomials $f,g \in \RR[x]$ independent
	of $\{ \lambda^{(n)} \}$ so that
	\begin{equation}
	\label{eq:beta-variance}
		\lim_{n \to \infty} \frac{\VV(\stat \mid K_{\lambda^{(n)}})}{n^{2p-1}} = f(\alpha) + \beta g(\alpha).
	\end{equation}
	
\end{theorem}

\begin{proof}
For $n \geq 1$, let $w^{(n)} \in K_{\lambda^{(n)}}$.
By definition 
\[
\VV\left(\stat \mid K_{\lambda^{(n)}}\right) = R(\stat^2)(w^{(n)}) - R(\stat)^2(w^{(n)}).
\]
Therefore by Proposition~\ref{p:simple-prod} we can view $\VV(\stat)$ as a regular statistic with power at most $2p$.
Following notation in the proof of Theorem~\ref{regular-class-asymptotics}, there exist polynomials $g_{2p}$ and $g_{2p-1}$ so that
\begin{align*}
	R\left(\frac{\VV(\stat)}{n^{2p-1}}\right) &= n R\left(\frac{\VV(\stat)}{n^{2p}}\right) \\
	&= n \cdot g_{2p} \left(\frac{m_1}{n},\frac{m_2}{n^2}, \dots, \frac{m_k}{n^k}\right) + g_{2p-1}\left(\frac{m_1}{n},\frac{m_2}{n^2}, \dots, \frac{m_k}{n^k}\right) + O(n^{-1}).
\end{align*}
Therefore
\begin{align}
\label{regular-variance-one}
\begin{split}
	R\left(\frac{\VV(\stat)}{n^{2p-1}}\right)(w^{(n)}) = \ & n \cdot g_{2p}\left(\alpha, \frac{\beta}{n}, \frac{m_3(\lambda^{(n)})}{n^3},\dots , \frac{m_k(\lambda^{(n)})}{n^k}\right) \\
	&+ g_{2p-1}\left(\alpha, \frac{\beta}{n}, \frac{m_3(\lambda^{(n)})}{n^3},\dots , \frac{m_k(\lambda^{(n)})}{n^k}\right).
	\end{split}
\end{align}
Assume the first term on the RHS of Equation~\eqref{regular-variance-one}
has no terms involving only $\alpha$ and has linear term $\beta g(\alpha)$ (and possibly higher order terms as well).
Then taking the limit as $n \to \infty$, the result will follow.

We now show the first term on the RHS of Equation~\eqref{regular-variance-one}
 has no terms involving only $\alpha$.
Since $\stat$ is a regular statistic, there is a collection of packed triples $\Upsilon$ so that
	\begin{equation}
		\stat = \sum_{((I,J),A,f) \in \Upsilon} c_{(I,J),A,f} T^f_{(I,J),A}.
	\end{equation}
	We then have
	\begin{align}
	\VV\left(\stat \mid \lambda^{(n)}\right) = \sum_{((I,J),A,f) \in \Upsilon} \sum_{((K,L),B,g) \in \Upsilon} c_{(I,J),A,f} c_{(K,L),B,g} \mbox{Cov}\left(T^f_{(I,J),A} T^g_{(K,L),B}\right).
	\end{align}
	where for $X,Y$ random variables $\mbox{Cov}(X,Y) = \EE(X^2) -2\EE(XY) + \EE(Y^2)$ is the covariance.
	We obtain $\mbox{Cov}\left(T^f_{(I,J),A}, T^g_{(K,L),B}\right)$ by applying $R$ to Equation~\eqref{eq:translate-covariance}.
	
	Let $G_1$ be the graph of $(I,J)$ with $u$ vertices and $\ell$ connected components and $G_2$ be the graph of $(K,L)$ with $v$ vertices and  $m$ connected components.
	We observe the maximum number of components for an indicator occurring in the products
	\begin{equation}
		T^f_{(I,J),A} T^g_{(K,L),B} \quad \mbox{and} \quad T^f_{(I,J),A} R(T^g_{(K,L),B})
	\end{equation}
	is $\ell + m$, and that this is uniquely attained by indicators with the graph $G_1 \sqcup G_2$.
	By Lemma~\ref{translate-product} there are at most $O(n^{u+v-1})$ such terms.
	Recall from Corollary~\ref{atomic-asymptotics} that the rational degree of $R(\one_{G})$ is the number of edges minus the number of vertices.
	Therefore, applying Proposition~\ref{fixed-point-coefficient} we see the degree contributed by terms $\one_{G_1 \sqcup G_2}$ is of order $\ell+m-1$.
	Meanwhile, for graphs $G \neq G_1 \sqcup G_2$, Lemma~\ref{translate-product} and Proposition~\ref{fixed-point-coefficient} show the degree of these terms is at most the number of connected components, i.e, at most $\ell + m - 1$.
	Since $\ell, m \leq p$, we see $g_{2p}$ must vanish on $\alpha$, and the result follows.
	\end{proof}
	
As a consequence, we derive a weak law of large numbers for regular statistics.
	
\begin{corollary}
	\label{regular-weak-law}
	Let $\stat$ be a regular statistic with power $p$, $\alpha \in [0,1]$ and $\{\lambda^{(n)}\}$ be a sequence of partitions so that
	\begin{equation}
	\label{eq:fixed-point}
	\lim_{n \to \infty} \frac{m_1(\lambda^{(n)})}{n} = \alpha.
	\end{equation}
Viewing $\stat$ as a random variable with respect to the uniform measure, 
	\begin{equation}
	\frac{1}{n^p}\left((\stat\mid K_{\lambda^{(n)}}) - R(\stat \mid K_{\lambda^{(n)}})\right) \xrightarrow{p} 0.
	\end{equation}
\end{corollary}

\begin{proof}
Combining Theorem~\ref{regular-class-asymptotics} and Theorem~\ref{regular-variance} with Chebyshev's inequality we have
	\begin{equation}
\PP\left( \frac{1}{n^p}\left|(\stat \mid K_{\lambda^{(n)}}) - R(\stat \mid K_{\lambda^{(n)}})\right| >  n^{1/3} \sqrt{\VV\left(\frac{1}{n^p}\stat\right)} \right) \leq n^{-1/3}.
	\end{equation}
	Since 
	\[
	\lim_{n \to \infty} \VV\left(\frac{1}{n^p}\stat\right)  = O(n^{-1}),
	\]
	the result follows.
\end{proof}

We can extend Corollary~\ref{regular-weak-law} to more general conjugacy invariant distributions.
\begin{corollary}
\label{c:conj-invariant-fixed}
Let $\stat$ be a regular statistic with power $p$, $\alpha \in [0,1]$ and $\{\mu_n\}$ be a sequence of conjugacy invariant probability measures so that for $\Sigma_n \sim \mu_n$
\begin{equation}
\label{conj-invariant-fixed}
	\PP\left(\lim_{n \to \infty} \frac{m_1(\Sigma_n)}{n} = \alpha\right) = 1.
\end{equation}
Then
\begin{equation}
\label{eq:conj-invariant-fixed}
	\frac{1}{n^p}\left(\stat(\Sigma_n) - R\left(\stat \mid K_{(n-\lfloor\alpha n\rfloor,1^{\lfloor\alpha n\rfloor})}\right)\right) \xrightarrow{p} 0.
	\end{equation}
\end{corollary}

\begin{proof}
	First, we must set an explicit rate of convergence in Equation~\eqref{conj-invariant-fixed} and rule out the contribution coming when the limit is not attained.
	This can be done using Markov's inequality since 
	\[
	\frac{1}{n^p}\EE(\stat(\Sigma_n)) \leq f(\alpha_0)
	\]
	where $f$ is the polynomial from Corollary~\ref{expectation-fixed-point} and $\alpha_0 \in [0,1]$ maximizes $f$ on this interval.
	
	By the triangle inequality, Equation~\eqref{eq:conj-invariant-fixed} is bounded above by
\begin{equation}
\frac{1}{n^p}\left|\stat(\Sigma_n) - R(\stat(\Sigma_n)\right| + \frac{1}{n^p}\left|R(\stat(\Sigma_n) - R\left(\stat \mid K_{(n-\lfloor\alpha n\rfloor,1^{\lfloor\alpha n\rfloor})}\right)\right|.
\label{eq:conj-invariant-proof}
\end{equation}
By conditioning on cycle type and applying Corollary~\ref{regular-weak-law}, the first term in Equation~\eqref{eq:conj-invariant-proof} converges in probability to 0.
Similarly, by conditioning on cycle type and applying Corollary~\ref{expectation-fixed-point} the second term in Equation~\eqref{eq:conj-invariant-proof} also converges in probability to 0, so the result follows.
\end{proof}

\section{The statistic $\mathrm{exc}$}
\label{subsection:exc}

We give a worked example demonstrating the dependence on $m_2$ occurring in Theorem~\ref{regular-variance}.
Recall $\exc$ is the excedance statistic
\[\exc = \sum_{1 \leq i < j < \infty } \one_{(i,j)} = T_{(1,2)}.
\]
For fixed $n$ we have
\[
R \, \exc = \binom{n}{2} \cdot R \, \one_{(1,2)}.
\]
Taking the Frobenius characteristic yields
\begin{equation}
	\label{exc-three}
\ch_n \, R \, \exc = \binom{n}{2} \frac{1}{n!} \cdot A_{n,(1,2)} = \frac{1}{2 \cdot (n-2)!} \cdot \vec{p}_{2 1^{n-2}}.
\end{equation}
The Path Murnaghan-Nakayama Rule Theorem~\ref{path-murnaghan-nakayama} implies 
\begin{equation}
\label{exc-four}
\vec{p}_{2 1^{n-2}} = (n-2)! \cdot \left(  (n-1) \cdot s_n - s_{n-1,1} \right)
\end{equation}
and plugging \eqref{exc-four} into \eqref{exc-three} gives
\begin{equation}
\label{exc-five}
\ch_n \, R \, \exc = \frac{1}{2} \cdot \left(  (n-1) \cdot s_n - s_{n-1,1} \right)
\end{equation}
The character polynomial associated to $s_n$ is the constant polynomial $1$ whereas the character polynomial associated to $s_{n-1,1}$ is $m_1 - 1$, so that
\begin{equation}
\label{exc-six}
R \, \exc = \frac{1}{2} \cdot \left(  (n-1) - (m_1 - 1) \right) = \frac{n - m_1}{2}.
\end{equation}
Equivalently, the average value of $\exc$ on a permutation of cycle type $\lambda$ is half the number of non-fixed points as can be seen directly from \eqref{exc-des} below and linearity of expectation.

Following Theorem~\ref{translate-product}, the statistic $\exc^2 = T_{(1,2)}^2$ on $\symm_n$ can be written as
\begin{equation}
\label{exc-seven}
\exc^2 =  T_{(1,2)} + 2 T_{(12,23)}  +
2  \left(T_{(12,34)} + T_{(12,43)} + T_{(13,24)}\right).
\end{equation}
Applying the Reynolds operator for fixed $n$ to \eqref{exc-seven} gives
\begin{equation}
\label{exc-eight}
R \, \exc^2 = \binom{n}{2} \cdot R \, \one_{(1,2)} + 2 \binom{n}{3} \cdot R \, \one_{(12,23)} + 6 \cdot \binom{n}{4} \cdot R \, \one_{(12,34)}
\end{equation}
and applying $\ch_n$ to \eqref{exc-eight} gives
\begin{equation}
\label{exc-nine}
\ch_n \, R \, \exc^2 = \frac{1}{2 \cdot (n{-}2)!} \cdot \vec{p}_{21^{n-2}} + \frac{2}{3! \cdot (n{-}3)!} \cdot \vec{p}_{31^{n-3}} + \frac{6}{4! \cdot (n{-}4)!} \cdot \vec{p}_{2^21^{n-4}}.
\end{equation}
We calculated that $s$-expansion of $\vec{p}_{2 1^{n-2}}$ in \eqref{exc-five}.
The other path power sums in \eqref{exc-nine} have $s$-expansions
\begin{equation}
\label{exc-ten}
\vec{p}_{3 1^{n-3}} = (n{-}3)! \cdot \left( (n{-}2) \cdot s_n -  s_{n-1,1} - s_{n-2,2} + s_{n-2,1,1}  \right)
\end{equation}
and
\begin{equation}
\label{exc-eleven}
\vec{p}_{2^2,1^{n-4}} = 2! \cdot (n{-}4)! \cdot \left(  \binom{n{-}2}{2} \cdot s_n - (n{-}3) \cdot s_{n-1,1} + s_{n-2,2}  \right).
\end{equation}
As can be found in e.g.~\cite[p. 323]{Specht}, 
the character polynomials of $s_{n-2,2}$ and 
$s_{n-2,1,1}$ are given by
\begin{equation}
   \frac{m_1 (m_1 - 1)}{2} + m_2 - m_1 \quad \text{and}
   \quad
   \frac{m_1(m_1 - 1)}{2} - m_2 - m_1 + 1,
\end{equation}
respectively.
Combining with \eqref{exc-nine}, \eqref{exc-ten}, and \eqref{exc-eleven} gives
\begin{equation}
\label{exc-twelve}
R \, \exc^2 = \frac{1}{12} \cdot \left( 3 m_1^2 - 6 m_1 n + 3 n^2 - 2 m_2 - m_1 + n    \right).
\end{equation}
From \eqref{exc-six} we have $(R \exc)^2 = \left(\frac{n - m_1}{2}\right)^2$ so by \eqref{exc-twelve} the variance of $\exc$ on a given conjugacy class is
\begin{equation}
\label{exc-thirteen}
R \, \exc^2 - \left( R \exc \right)^2 = \frac{n -m_1 - 2m_2}{12}.
\end{equation}
This quantity vanishes for conjugacy classes that are all fixed points and two cycles, as should be expected since a fixed point contributes nothing and a two cycle always contributes one to $\exc$.

The results here can also be derived directly since $\exc$ has an alternate combinatorial description.
Every cyclic permutation $w \in K_{(n)}$ has the form $(1, v_1, \dots , v_{n-1})$ for some $v \in \symm_{n-1}$, and $\exc(w) = 1 + \asc(v)$ unless $n = 1$, in which case $\exc(w) = 0$.
Here $\asc = n - 1 - \des$ is the ascent statistic.
Therefore, for $\lambda = (\lambda_1,\dots,\lambda_k)$ a partition of $n$ and $\Sigma$ uniformly random in $K_\lambda$, we have
\begin{equation}
\label{exc-des}
		\exc(\Sigma) = \sum_{i=1}^k \lambda_i - \des(\Sigma_i) = n - \sum_{i=1}^k \des(\Sigma_i).
\end{equation}
where the $\Sigma_i$ mutually independent and uniformly random in $\symm_{(\lambda_i-1)}$.
For fixed $n$ recall $\EE(\asc) = \EE(\des) = \frac{n-1}{2}$.
Then for $\Sigma$ uniformly random in $K_\lambda$, by \eqref{exc-des}
\begin{equation}
	\label{exc-exp}
\EE(\exc(\Sigma)) = \sum_{i=1}^k 1 + \frac{\lambda_1-2}{2} - \frac{\delta_{\lambda_i = 1}}{2} = \frac{n - m_1}{2}
\end{equation}
as in \eqref{exc-six}.
Recall variance of independent random variables is additive and for fixed $n \geq 3$ that $\VV(\asc) = \VV(\des) = \frac{n+1}{12}$ and $\VV(\asc) = 0$ for $n = 1,2$.
Therefore by \eqref{exc-des} we compute
\begin{equation}
	\label{exc-var}
	\VV(\exc(\Sigma)) = \sum_{i=1}^k \frac{\lambda_i \cdot \delta_{\lambda_i \notin \{1,2\}}}{12} = \frac{n - m_1 - 2m_2}{12}
\end{equation}
as in \eqref{exc-thirteen}.

Let $\alpha,\beta \in [0,1]$ and $\{\lambda^{(n)}\}$ be a sequence of cycle types with
\begin{equation}
	\lim_{n\to \infty} \frac{m_1(\lambda^{(n)})}{n} = \alpha \quad \mbox{and} \quad \lim_{n\to \infty} \frac{m_2(\lambda^{(n)})}{n} = \beta.
\end{equation}
Using Equation~\eqref{exc-thirteen}, we see for $\Sigma^{(n)}$ uniformly random on $K_{\lambda^{(n)}}$ that $\exc(\Sigma^{(n)})$ has variance $\frac{n - \alpha - 2\beta}{12}$.
Moreover, the decomposition from \eqref{exc-des} shows $\exc(\Sigma^{(n)})$ is a sum of independent random variables.
The short cycles in $\lambda^{(n)}$ that are not fixed points or two cycles contribute independent random variables with bounded positive variance to $\exc(\Sigma^{(n)})$.
The long cycles contribute approximately normal random variables.
Therefore, the mixture will be asymptotically normal.
Following Hofer's work, we believe the asymptotic normality of $\exc$ should be a general phenomenon for a wide variety of regular statistics. 
See Section~\ref{ss:normality} for further discussion.

\section{Quasi-random permutations}

A sequence $\{ w^{(n)} \}$ of permutations with $w^{(n)} \in \symm_n$ is called {\em quasi-random} if for all $k \geq 1$ and for all $v \in \symm_k$ we have
\begin{equation}
\label{quasi-random-def} \frac{N_v(w^{(n)})}{\binom{n}{k}} = \frac{1}{k!} + o(1).
\end{equation}
This pattern-theoretic definition of randomness asserts that $\{ w^{(n)} \}$, in the limit, contains equal proportion of patterns of size $k$
for every $k \geq 1$.
This property shows that the small-scale behavior of a quasi-random sequence is close to that for a sequence of uniformly random permutations.
Note when $w^{(n)}$ is a sequence of random permutations that Equation~\eqref{quasi-random-def} becomes
\[
\PP\left(\lim_{n \to \infty} \frac{N_v(w^{(n)})}{\binom{n}{k}} = \frac{1}{k!} \right) = 1.
\]
Here, $N_v(w^{(n)})/\binom{n}{k}$ would be the probability a uniformly random size $k$ partial permutation in $w^{(n)}$ has the pattern $v$.
For this reason, it may be more appropriate to call quasi-random sequences \emph{quasi-uniform}, since not every random sequence of permutations has this property.

Quasi-random sequences of permutations were studied by Cooper \cite{Cooper} in his PhD thesis at UC San Diego.
Answering a question of Graham (see \cite{Cooper}), Kr\'al' and Pikhurko proved \cite[Thm. 1]{KP} the surprising result that 
Equation~\eqref{quasi-random-def} need only hold for $v \in \symm_4$, but that checking patterns of size up to 3 is insufficient.

\begin{corollary}
\label{quasirandom-permutation-generation}
Let $\mu_n$ be a sequence of congugacy-invariant probability distributions on $\symm_n$.
For $\Sigma_n \sim \mu_n$, if
\begin{equation}
\PP \left(  \lim_{n \to \infty} \frac{m_1(\Sigma_n)}{n} = 0 \right) = 1,
\end{equation}
then the sequence $\Sigma_n$ is quasi-random.
\end{corollary}

\begin{proof}
This follows immediately from Corollary~\ref{c:conj-invariant-fixed}.
\end{proof}

A result analogous to Corollary~\ref{quasirandom-permutation-generation} is possible when the limiting proportion of fixed points tends to any $\alpha \in [0,1]$.
To state this, we need the following definition, introduced in~\cite{HKMRS}.
A \emph{permuton} is a measure $\mu$ on the Borel $\sigma$-algebra of $[0,1] \times [0,1]$ so that for every measurable $A \subseteq [0,1]$
\[
\mu(A \times [0,1]) = \mu([0,1] \times A) = \mathcal{L}_{[0,1]}(A)
\]
where $\mathcal{L}$ is Lebesgue measure.
Given a permuton $\mu$, sample $(X_1,Y_1),\dots,(X_k,Y_k) \sim \mu$ i.i.d., reordereed so that $X_1 < \dots < X_k$, and define $\Pi^\mu_k$ to be the permutation with the same relative order as $(Y_1,\dots,Y_k)$.
A sequence $\{\mu_n\}$ of measures on $\symm_n$ \emph{converges} to $\mu$, denoted $\mu_n \to \mu$, if given $\Sigma_n \sim \mu_n$, for every $k$ and $v \in \symm_k$ we have
\[
\lim_{n \to \infty} \frac{N_v(\Sigma_n)}{\binom{n}{k}} = \PP(\Pi^\mu_k = v).
\]

\begin{theorem}
	\label{alpha-permuton}
Let $\mu_n$ be a sequence of congugacy-invariant probability distributions on $\symm_n$ and $\alpha \in [0,1]$ so that for $\Sigma_n \sim \mu_n$, we have 
\[
\PP\left(\lim_{n \to \infty} \frac{m_1(\Sigma_n)}{n} = \alpha \right) = 1.
\]
Then $\mu_\alpha$ so $\mu_n \to \mu_\alpha$ with $\mu_\alpha = (1-\alpha) U + \alpha I$ where $U$ is Lebesgue measure on $[0,1,] \times [0,1]$ and $I$ is Lebesgue measure on $\{(x,x): x\in [0,1]\}$. 
\end{theorem}

\begin{proof}
	The existence of $\mu_\alpha$ follows from Corollary~\ref{c:conj-invariant-fixed} and~\cite[Thm. 1.6]{HKMRS}.
	The explicit description of $\mu_\alpha$ can be seen as follows.
	Let $\mu_{\alpha,n}$ be the measure on $\symm_n$ given by sampling $\lfloor \alpha n \rfloor$ values in $[n]$ uniformly at random, setting those as fixed points, then permuting the remaining values uniformly at random.
	It is clear both that $\mu_{\alpha,n} \to \mu_\alpha$ and that $\mu_{\alpha,n}$ satisfies the assumptions of our theorem.
\end{proof}

The explicit description of $\mu_\alpha$ in Theorem~\ref{alpha-permuton} is due to Valentin F\'eray, and we thank him for allowing us to include his proof.

\chapter{Concluding Remarks}
\label{Conclusion}

 \section{Describing low degree and related functions}
 \label{ss:local-approx}

Theorem~\ref{shadow-basis} gives a natural basis for the space $\Loc_k(\symm_n,\CC)$ of local functions with many desirable properties.
However, this basis is not homogeneous as it is comprised of indicator functions for partial permutations of varying size.
Theorem~\ref{local-support-theorem} guarantees the existence of a basis of indicators for partial permutations of the same size.

\begin{problem}
Find a uniform description for given $n$ and $k$ of a basis for $\Loc_k(\symm_n,\CC)$ whose elements are of the for $1_{IJ}$ with $(I,J) \in \symm_{n,k}$.
\end{problem}

Recall the $k$-local approximation $L_k f$ for any map $f: \symm_n \rightarrow \CC$.
Huang, Guestrin, and Guibas \cite{HGG} considered a generalization of $L_k$ defined as follows.
Let $\mu$ be a partition of $n$ and let $f: \symm_n \rightarrow \CC$ be regarded as an element of the group algebra 
$\CC[\symm_n]$ in the usual way. 
The Artin-Wedderburn theorem gives an isomorphism $\Psi: \CC[\symm_n] \xrightarrow{\, \, \sim \, \, } \bigoplus_{\lambda \vdash n} \End_\CC(V^\lambda)$.
The {\em band-limited approximation} $L_\mu f \in \CC[\symm_n]$ corresponding to $\mu$ is characterized by 
\begin{equation}
\Psi(L_\mu f) = \bigoplus_{\lambda \vdash n} \psi_\lambda,
\end{equation}
where $\psi_\lambda: V^\lambda \rightarrow V^\lambda$ is given by
\begin{equation}
\psi_\lambda(v) = \begin{cases}
f \cdot v & \mu \leq \lambda \text{ in dominance order,} \\
0 & {otherwise.}
\end{cases}
\end{equation}
The special case of $L_k f$ corresponds to when $\mu = (n-k,1^k)$ is a hook.

 
 \begin{problem}
 \label{band-limit-problem}
 Let $\mu \vdash n$ be a partition.
 Clarify the relationship between the functions $f: \symm_n \rightarrow \CC$ and $L_{\mu} f: \symm_n \rightarrow \CC$.
 In particular, describe a natural basis of the subspace
 \begin{equation*}
 \{ L_{\mu} f \,:\, f: \symm_n \rightarrow \CC \} \subseteq \Fun(\symm_n,\CC)
 \end{equation*}
 of the vector space $\Fun(\symm_n,\CC)$ of maps $\symm_n \rightarrow \CC$.
 \end{problem}

 An argument similar to that in the proof of Theorem~\ref{aw-local} gives a combinatorial {\em spanning set} of the space in
 Problem~\ref{band-limit-problem}.
 Given
 ordered set partitions $\mathcal{S} = (S_1,\dots,S_r)$ and $\mathcal{T} = (T_1, \dots,T_r)$ of $[n]$ with $|S_i| = |T_i|$ for all $i$, consider the indicator function $\chi_{\mathcal{S},\mathcal{T}}: \symm_n \rightarrow \CC$ defined  by
\begin{equation}
\label{set-partition-indicator}
\chi_{\mathcal{S},\mathcal{T}} = \prod_{i=1}^r \chi_{S_i,T_i}.	
\end{equation}
Such functions appear in another guise in the representation-theoretic perspective on voting theory pioneered
by Diaconis \cite{Diaconis2} and generalize our functions $\one_{I,J}$ corresponding to partial permutations $(I,J)$.
We claim that
the set of $\chi_{\mathcal{S},\mathcal{T}}$ functions where $\mathcal{S}$ and $\mathcal{T}$ are ordered set partitions with $|S_i| = |T_i| = \mu_i$ form a spanning set as requested in Problem~\ref{band-limit-problem}.
(When $r = 2$, Huang, Guestrin, and Guibas \cite{HGG} proved that these functions lie in the relevant subspace of $\Fun(\symm_n,\CC)$.)

Indeed, let $\symm_\mu \subseteq \symm_n$ be the parabolic subgroup corresponding to $\mu$, and let $[\symm_\mu]_+ = \sum_{w \in \symm_\mu} w \in \CC[\symm_n]$
be the group algebra element summing over $\symm_\mu$.
Let $\III_\mu \subseteq \CC[\symm_n]$ be the two-sided ideal in the group algebra generated by $[\symm_\mu]_+$.  Lemma~\ref{product-action-on-indicators}
extends in a straightforward fashion to shows that
\begin{equation}
\III_\mu = \mathrm{span}_\CC \{ \chi_{\mathcal{S},\mathcal{T}}  \,:\, \mathcal{S} = (S_1, \dots, S_r), \, \mathcal{T} = (T_1, \dots, T_r), \, |S_i| = |T_i| = \mu_i \}
\end{equation}
as subspaces of $\CC[\symm_n] \cong \Fun(\symm_n,\CC)$.
On the other hand, the same reasoning as in the proof of Theorem~\ref{aw-local} implies that
under the Artin-Wedderburn isomorphism $\Psi: \CC[\symm_n] \xrightarrow{ \, \, \sim \, \, } \bigoplus_{\lambda \vdash n} \End_\CC(V^\lambda)$
the ideal $\III_\mu$ maps to
\begin{equation}
\Psi(\III_\mu) = \bigoplus_{\substack{ \lambda \vdash n \\ [\symm_\mu]_+ \cdot V^\lambda \neq 0}} \End_\CC(V^\lambda).
\end{equation}
For $\lambda \vdash n$, Frobenius Reciprocity implies that
\begin{align}
[\symm_\mu]_+ \cdot V^{\lambda} \neq 0  &\Leftrightarrow  \langle \Res^{\symm_n}_{\symm_\mu} V^\lambda, \mathrm{triv}_{\symm_{\mu}} \rangle_{\symm_\mu} \neq 0  \\
&\Leftrightarrow  
\langle V^\lambda, \mathrm{Ind}_{\symm_\mu}^{\symm_n} \mathrm{triv}_{\symm_\mu} \rangle_{\symm_n} \neq 0.
\end{align}
Young's Rule (see e.g. \cite{Sagan}) implies that the induction
$\mathrm{Ind}_{\symm_\mu}^{\symm_n} \mathrm{triv}_{\symm_\mu}$ contains a copy of $V^\lambda$ if and only if $\mu \leq \lambda$ in dominance order.
One way to solve Problem~\ref{band-limit-problem} would be to find a basis generalizing that of Theorem~\ref{shadow-basis}
from hooks to arbitrary partitions $\mu$.

\section{Coefficients of $s_{\lambda[n]}$ for regular statistics}
 For any real-valued $f: \symm_n \rightarrow \RR$ we studied the symmetric function $\ch_n(R f)$ and, in particular, its Schur expansion
 \begin{equation*}
 \ch_n(Rf) = \sum_{\lambda \vdash n} d_{\lambda}(n) \cdot s_{\lambda}.
 \end{equation*}
 We proved (Proposition~\ref{average-value-extraction}) that for $\lambda = (n)$, the coefficient $d_{(n)}$ is the average value
 $\EE(f)$ of $f$ on $\symm_n$.
  In the context of voting theory, Diaconis \cite{Diaconis2} interpreted more general character inner products in terms of coalitions within an electorate.
  In our setting, we have the following problem.
 
 \begin{problem}
 \label{other-values-problem}
Interpret the other coefficients $d_{\lambda}$ appearing in the Schur expansion of $\ch_n(Rf)$.
 \end{problem}
 
For $v \in \symm_k$, Gaetz and Pierson give an explicit formula for the coefficient of $s_{(1)[n]}$ in $\ch_n(R N_v)$~\cite[Prop. 4.1]{GP}.
Additionally, for the special case where $v = 12\dots k$, they show the coefficients of $s_{(1,1)[n]}$ and $s_{(2)[n]}$ are non-negative.
This is partial progress towards their conjecture:

\begin{conjecture} \em{(\cite[Conj. 1.4]{GP})}
\label{GaetzPierson}
	For $\lambda$ a fixed partition, the coefficient of $s_{\lambda[n]}$ in $\ch_n(R N_{12\dots k})$ is non-negative.
\end{conjecture}

Recently, Iskander made significant progress on Conjecture~\ref{GaetzPierson}, demonstrating that the coefficient of $s_{\lambda[n]}$ for $\ch_n(R N_{12\dots k})$ can be viewed as a bivariate polynomial in $n$ and $k$.
He has verified the conjecture up for $|\lambda| \leq 8$ and disproved the stronger conjecture that such coefficients are real rooted.

More generally, for $f$ a regular function it would be interesting to identify when the coefficient of $s_{\lambda[n]}$ in $\ch_n (R f)$ is non-negative, or when its numerator is real rooted.
 
 \section{Path power sums as character polynomials}
 Proposition~\ref{indicator-polynomiality-proposition} shows that for any partial permutation $(I,J) \in \symm_{n,k}$, the function $R \, \one_{I,J}$
 is (up to a normalizing factor) a polynomial in the variables $n, m_1, \dots, m_k$.
 We have made extensive use of this polynomiality and the relevant degree bounds in this paper.

 \begin{problem}
 \label{characteristic-poly-problem}
 	For $(I,J) \in \symm_{n,k}$, give an explicit description of or an
	algorithm for computing the polynomial $R \, \one_{I,J}$ in terms of $n,m_1,m_2,\dots,m_k$.
 \end{problem}

Note that Proposition~\ref{fixed-point-coefficient} characterizes all terms of the form $n^a m_1^b$.
For $\lambda \vdash k$, a solution Problem~\ref{characteristic-poly-problem} would 
also give a description for the polynomial $\ch_n^{-1}(p_\lambda s_{(n-k)})$.
Garsia and Goupil \cite{GG} solved Problem~\ref{characteristic-poly-problem} for the character polynomials attached to partitions.
Their result~\cite[Prop. 6.2 (a)]{GG} shows $\ch_n^{-1}(p_{(k)} s_{(n-k)}) = k \cdot m_k$.

 \section{Other groups}
 In this paper we consider functions on the symmetric group.
One can also consider functions on other groups $G$ from a 
 Fourier theoretic perspective (the cases $G = S^1$ and $G = (\ZZ/2\ZZ)^n$ were mentioned in the introduction).
 For representation-theoretic purposes, it is best to consider complex-valued functions on $G$.
 
 \begin{problem}
 \label{other-groups-problem}
 Extend the results of this paper to (complex-valued) functions $f: G \rightarrow \CC$
 on a wider class of groups $G$.
 \end{problem}
 
One class of groups $G$ to which Problem~\ref{other-groups-problem} could be applied are the complex reflection groups.
Given parameters $n, r,$ and $p$ with $p \mid r$, let $G(r,p,n)$ be the group of $n \times n$ complex matrices $A$ with a unique nonzero entry in every 
row and column such that 
\begin{itemize}
\item every nonzero entry of $A$ is a $r^{th}$ root-of-unity, and
\item the product of the nonzero entries of $A$ is a $(r/p)^{th}$ root-of-unity.
\end{itemize}
We have $G(1,1,n) = \symm_n$ and we may identify $G(r,1,n) = (\ZZ/r\ZZ) \wr \symm_n$ with a wreath product, which have been previously studied from a Fourier-theoretic perspective~\cite{Rockmore} with applications to image processing~\cite{FMRHO}.
A result analogous to Corollary~\ref{cycle-type-asymptotics} has been proved for these groups in~\cite{LLLSY2}.
The groups $G(r,p,n)$ constitute the single infinite family in the classification of irreducible complex reflection groups (with the proviso that 
$G(1,1,n)$ is viewed as acting on the subspace $x_1 + \cdots + x_n = 0$ of $\CC^n$).
The representation theory of $G(r,p,n)$ is analogous to that of $\symm_n$ (see \cite{Stembridge}), so Problem~\ref{other-groups-problem}
could have a productive solution in this setting.

Another class of groups to which Problem~\ref{other-groups-problem} could be applied are compact simple complex Lie groups $G$.
For example, the character theory of the unitary group $U(n)$ has many formal similarities with that of the symmetric group $\symm_n$.
One would likely need conditions on $f: G \rightarrow \CC$ such as $f \in L^2(G)$ or $f$ being polynomial to prove interesting results in this setting.
Alternatively, one could consider Lie groups over finite fields.
For these groups, a notion of ``level'' analogous to degree appears in~\cite{GLT}, where it is used to compute bounds on characters.
 
\section{Character evaluations}
For any $k \leq n$ and $w \in \symm_n$,
Corollary~\ref{trace-interpretation} gives a combinatorial interpretation of the irreducible character evaluations $\chi^{\lambda}([w \cdot \symm_k]_+)$
on the sum over the coset $w \cdot \symm_k$ for any partition $\lambda \vdash n$.
This leads to the following problem, which we will sharpen and motivate after its statement.

\begin{problem}
\label{trace-evaluation-problem}
Let $\psi: \CC[\symm_n] \rightarrow \CC$ be a linearly extended class function $\symm_n \rightarrow \CC$. Give a combinatorial rule to evaluate 
$\psi(g)$ for certain group algebra elements $g \in \CC[\symm_n]$.
\end{problem}

In this paper we solved Problem~\ref{trace-evaluation-problem} when $\psi = \chi^{\lambda}$ is an irreducible character and 
$g = [w \cdot \symm_k]_+$ is conjugate to a partial permutation.
One way to generalize would be to replace $\symm_k$ with a more general parabolic subgroup $\symm_{\mu}$.
Cosets in $\symm_n/\symm_{\mu}$ are in bijection with {\em ordered set partitions} \cite{HRS}
of $[n]$ into blocks $(B_1 \mid B_2 \mid \cdots )$
where $B_i$ has size $\mu_i$. 
It could be interesting to find an interpretation of $\chi^{\lambda}([w \cdot \symm_{\mu}]_+)$, the irreducible character value `on an ordered set partition'.

The class function $\psi$ in Problem~\ref{trace-evaluation-problem} can also be varied in interesting ways.
If $\varepsilon^{\lambda}: \CC[\symm_n] \rightarrow \CC$ is the {\em induced sign character} satisfying $\ch_n(\varepsilon^{\lambda}) = e_{\lambda}$
and $\{ C'_w(1) \,:\, w \in \symm_n \}$ is the {\em Kazhdan-Lusztig basis} of $\CC[\symm_n]$,
Clearman, Hyatt, Shelton, and Skandera \cite{CHSS} proved that if 
 $\varepsilon^{\lambda}(C'_w(1)) > 0$ whenever $w \in \symm_n$ is 312-avoiding and $\lambda \vdash n$, the 
 Stanley-Stembridge $e$-positivity conjecture \cite{SS} would follow.
 Since $[\symm_k]_+ = C'_w(1)$ where $w = [k, k-1, \dots, 2, 1, k+1, k+2, \dots, n-1,n]$, there is some overlap between evaluations on
 the KL basis and evaluations on $[w \cdot \symm_k]_+$.
 
 Problem~\ref{trace-evaluation-problem} has a natural $q$-analog.  Let $\HHH_n(q)$ be the Hecke algebra attached to $\symm_n$ over the field
 $\CC(q)$.  A $\CC(q)$-linear function $\psi: \HHH_n(q) \rightarrow \CC(q)$ is a {\em trace} if $\psi(ab) = \psi(ba)$ for 
 all $a, b \in \HHH_n(q)$. Hecke algebra traces reduce to symmetric group class functions in the limit $q \rightarrow 1$.
 The irreducible and induced sign characters extend to traces $\chi^{\lambda}_q, \varepsilon^{\lambda}_q: \HHH_n(q) \rightarrow \CC(q)$.
 
 Ram \cite{Ram} extended the Murnaghan-Nakayama Rule to $\HHH_n(q)$ by giving a combinatorial formula for 
 $\chi^{\lambda}_q(T_w)$ where $\{ T_w \,:\, w \in \symm_n \}$ is the natural basis of $\HHH_n(q)$ lifting the basis $\{ w \,:\, w \in \symm_n \}$ 
 of $\CC[\symm_n]$. 
 It may be interesting to give an appropriate lift $[w \cdot \symm_k]^{(q)}_+$ of $[w \cdot \symm_k]_+$ to $\HHH_n(q)$ such that 
 $\chi^{\lambda}_q( [w \cdot \symm_k]^{(q)}_+ )$ has a nice form.
 
 In a different direction, there is an extension $\{ C'_w(q) \,:\, w \in \symm_n \}$ of the KL-basis to $\HHH_n(q)$.
 Clearman et. al. \cite{CHSS} gave combinatorial interpretations for 
 $\chi^{\lambda}_q(C'_w(q)), \varepsilon^{\lambda}_q(C'_w(q)) \in \CC(q)$ when $w$ is 3412,4231-avoiding.
 The same authors proved \cite{CHSS} that if $\varepsilon^{\lambda}_q(C'_w(q)) \in \ZZ_{\geq 0}[q]$ for 
 all $\lambda \vdash n$ and 312-avoiding $w \in \symm_n$, the Shareshian-Wachs refinement \cite{SW} of the Stanley-Stembridge Conjecture holds.
 
 \section{Joint distributions}
 In this paper we exclusively considered real-valued permutation statistics of the form $\stat:\symm_n \rightarrow \RR$.
It is natural to ask whether our methods can extend to functions of the form $f:\symm_n \rightarrow \RR^k$.
Since $f = (f_1,\dots,f_k)$ with $f_i : \symm_n \to \RR$, this would correspond to understanding the joint distribution of $f_1,\dots,f_k$.

\begin{problem}
\label{multivariate-problem}
	Use our methods to understand the joint distribution of regular functions.
\end{problem}

In the uniform case, the joint distribution of ordinary pattern counting statistics is relatively well understood.
Building on work of Burstein and H\"ast\"o~ for the bivariate case\cite{BH}, Janson, Nakamura and Zeilberger showed $(N_\sigma)_{\sigma \in \symm_k}$ is jointly normal and explicitly computed its covariance matrix~\cite{JNZ}.
Even-Zohar gives refinements of their work by describing some geometric  properties for the image of this $k!$-variate function~\cite{EZ}.

One can also consider non-numeric functions of permutations.
For example, there are important set-valued permutation statistics such as the 
 {\em descent set} 
 \begin{equation}
 \Des(w) := \{ 1 \leq i \leq n-1 \,:\, w(i) > w(i+1) \}
 \end{equation}
 of $w \in \symm_n$.
 If $D: \symm_n \rightarrow 2^{[n-1]}$ is a statistic (such as $\Des$) which attaches a subset $D(w) \subseteq [n-1]$
 to a permutation $w \in \symm_n$, we have a formal power series
 \begin{equation}
 \label{eq:quasi}
 \ch_n(R \, D) = \frac{1}{n!} \cdot \sum_{w \in \symm_n} F_{D(w)} \cdot p_{\lambda(w)}
 \end{equation}
 where $F_{D(w)} = F_{D(w)}(y_1, y_2, \dots )$ is the {\em fundamental quasisymmetric function} over a variable set  $(y_1, y_2, \dots )$
 disjoint from the variables $(x_1, x_2, \dots)$ of $p_{\lambda(w)}$.
Here $\lambda(w) \vdash n$ is the cycle type of $w$.

In the case $D(w) = \Des(w)$, the statistic $\ch_n(R \, D)$ has a nice formula. 
For $n \geq 1$, consider the symmetric function 
$L_n := \frac{1}{n} \sum_{d \mid n} \mu(d) p_d^{n/d} \in \Lambda_n$ where $\mu(d)$ is the number-theoretic M\"obius function.
Given a partition $\lambda = (1^{m_1} 2^{m_2} \cdots )$, we extend this definition by setting
$L_{\lambda} := h_{m_1}[L_1] h_{m_2}[L_2] \cdots $ where $h_m[ - ]$ denotes  plethystic substitution.
The $L_{\lambda}$ are the {\em Lyndon symmetric functions}. Gessel and Reutenauer proved \cite[Thm. 2.1]{GesReu} that
\begin{equation}
\ch_n(R \, \Des) = \frac{1}{n!} \cdot \sum_{\lambda \vdash n} L_{\lambda} \cdot p_{\lambda},
\end{equation}
motivating the following problem.

\begin{problem}
\label{set-valued-problem}
Extend the techniques of this paper to set-valued permutation statistics.
\end{problem}

Problems~\ref{multivariate-problem} and~\ref{set-valued-problem}  can also be combined in the context of tuples of statistics $(\stat_1, \stat_2, \dots )$ on $\symm_n$.
Gessel and Zhuang \cite{GZ} used plethystic techniques to compute the functions defined in Equation~\eqref{eq:quasi} attached to 
 various (tuples of) statistics
 over sets $\Pi \subseteq \symm_n$ when such functions are in fact symmetric.
 By the work of Gessel-Reutenauer \cite{GR}, this includes all conjugacy classes $K_{\lambda} \subseteq \symm_n$.

\section{Limiting distributions}
\label{ss:normality}

In Section~\ref{Convergence} we applied our results to obtain asymptotic properties of certain regular statistics restricted to sequences of conjugacy classes, 
 most notably for vincular pattern enumerators.
 While our results on the asymptotic mean and variance are quite general, they do not extend to higher moments.
The sequences $\{ \lambda^{(n)} \}$ of partitions of $n$ for which we establish convergence are much more limited.
We suspect convergence should be much more general.

\begin{problem}
	Given a sequence $\{\lambda^{(n)}\}$ of cycle types, identify conditions on small cycle counts for which regular statistics are asymptotically normal.
\end{problem} 

Using our work, Feray and Kammoun have solved this problem for classicaly statistics.
Specifically, they demonstrate for $\{\lambda^{(n)}\}$ a sequence of partitions with $\alpha = \lim_{n \to \infty} m_1(\lambda^{(n)})/n$ and $\beta = \lim_{n \to \infty} m_2(\lambda^{(n)})/n$ with $\alpha < 1$ that any vincular statistics is asympotically normal or degenerate~\cite[Thm~1.3]{FK}.
For classical patterns, they establish non-degeneracy~\cite[Thm~1.4]{FK}.
We remark that their methods extend to joint distributions.
Further work by Dubach established several of these results independently of our work~\cite{Dubach}.

Our proofs of convergence to normality ultimately relied on Hofer's Theorem~\ref{hofer-convergence} showing that the statistics $N_{v,A}$ 
  converge to normality on all of $\symm_n$, or alternatively on Janson's~\cite{Janson}.
The Feray-Slim proofs rely critically on Hofer's methods as well.

  \begin{problem}
  \label{hofer-problem}
  Use the techniques of this paper to prove Hofer's Theorem~\ref{hofer-convergence}.
  More generally, use the methods of this paper to describe further families of regular permutation statistics that are asymptotically normal.
  \end{problem}
  
Hofer's proof relies on dependency graphs, and ultimately reduces to showing the variance of a regular statistic, as normalized in Theorem~\ref{regular-variance}, does not vanish.
The Feray--Kammoun work also relies on this method.
Alternatively, a na\"ive approach to Problem~\ref{hofer-problem} via the Method of Moments would use the expression
  \begin{equation}
  \EE(N_{v,A}^d) = \text{coefficient of $s_n$ in } \ch_n(R \, N_{v,A}^d)
  \end{equation}
  coming from Proposition~\ref{average-value-extraction}. Unfortunately, the $T$-expansions of the statistics $N_{v,A}^d$ become 
  very complicated as $d$ grows, making the coefficient of $s_n$ in $\ch_n(R \, N_{v,A}^d)$ difficult to access.

It would also be interesting to find the limiting distributions of regular statistics that are not asymptotically normal on a sequence of cycle types.
For example, a bivincular pattern $(v,A,B)$ is {\em very tight} if 
 $v \in \symm_k$ and $A = B = [k-1]$.
The statistic $N_{v,A,B}$ of a very tight pattern counts instances of $v$ in which all positions
 and all values must be consecutive.
Such statistics studied in \cite{CLP}.
The very tight random variables $N_{v,A,B}: \symm_n \rightarrow \RR$ converge to a Poisson distribution for patterns of size $k = 2$.
When $k > 2$, the probability to find a very tight pattern of size $k$ in a permutation in $\symm_n$ tends to zero.
Characterizing the cycle types where this holds would be an interesting challenge.
  
The building blocks of regular permutation statistics are the constrained translate statistics $T^f_{(U,V),C}: \symm_n \rightarrow \RR$
  indexed by packed triples $((U,V),C,f)$.
 The following would constitute a general approach to Problem~\ref{hofer-problem}
   
   \begin{problem}
   \label{constrained-problem}
   Let $((U,V),C,f)$ be a packed triple and let $T_{(U,V),C}^f: \symm_n \rightarrow \RR$ be the associated constrained translate statistic.
   Characterize the convergence properties of $T_{(U,V),C}^f$ as $n \rightarrow \infty$, both on the full symmetric group $\symm_n$ and 
   on sequences $\{ \lambda^{(n)} \}$ of cycle types.
   \end{problem}
   In unpublished work, Coopman uses similar methods to identify the unweighted constrained translates that are asymptotically normal~\cite{Coopman}.
   
   Since regular statistics are linear combinations of $T$-statistics, Problem~\ref{constrained-problem} could be studied in conjunction 
   with the following to determine their moments.
   
   \begin{problem}
   \label{dependence-problem}
   Characterize the dependency between the constrained translate statistics $T_{(U,V),C}^f$ for various packed triples 
   $((U,V),C,f)$.
   \end{problem}
   
   Although the random variables $T_{(U,V),C}^f$ are not independent, one  hopes that Problem~\ref{dependence-problem} could be solved to
   show that they are ``almost" independent in many cases of interest.
   This would allow one to apply the method of dependency graphs.

 \section{Other probability measures}
   Kammoun's Theorem~\ref{kammoun-theorem} and our Theorem~\ref{our-long-cycles} consider permutation statistics 
   with respect to conjugacy-invariant measures 
   $\mu_n: \symm_n \rightarrow \RR_{\geq 0}$ which are not necessarily concentrated on a single conjugacy class $K_{\lambda}$.
   Conjugacy-invariant measures on $\symm_n$ bear the same relationship to arbitrary measures as class functions $\symm_n \rightarrow \CC$ do to 
   arbitrary functions $\symm_n \rightarrow \CC$.
   
   \begin{problem}
   \label{measures-problem}
   Extend the results of Section~\ref{Convergence} to measures $\mu_n: \symm_n \rightarrow \RR_{\geq 0}$ that are not necessarily conjugacy-invariant.
   \end{problem}

  Problem~\ref{measures-problem} features challenges similar to those present when studying functions $\symm_n \rightarrow \CC$
  that are not necessarily class functions.
  One loses the machinery of symmetric functions and often needs to consider specific bases of 
  $\symm_n$-irreducibles $V^{\lambda}$.
  Nevertheless, Hultman has demonstrated a method for applying $R(\stat)$ to understand its behavior for certain distributions induced by random walks~\cite{Hultman}.

\section{Nonlocal functions}  
There are many interesting permutation statistics which do not possess nontrivial locality. One famous example is
  $\lis: \symm_n \rightarrow \CC$ where $\lis(w)$ is the longest increasing subsequence of a permutation $w \in \symm_n$.
   More formally, we have
   \begin{equation}
   \lis(w) = \max \{\, |I| \,:\, I = \{i_1 < \cdots < i_k\}\  \mbox{with}\ w(i_1) < \cdots < w(i_k) \}.
   \end{equation}
   Baik, Deift, and Johansson proved \cite{BDJ} that the statistic $\lis$ converges 
    to the Tracy-Widom distribution as $n \rightarrow \infty$, after shifting and rescaling.
  
  \begin{problem}
  \label{nonlocal-problem}
  Apply the methods of this paper to nonlocal permutation statistics $\symm_n \rightarrow \CC$ such as the statistic $\lis$.
  \end{problem}
  
  One approach to Problem~\ref{nonlocal-problem} could involve `cutting off' permutation statistics to lower their locality.
  If a statistic $f$ is well-approximated by $L_k f$ for some value $k$, it is reasonable to expect our methods to still apply.
  For example, it is known that $\EE(\lis)/\sqrt{n} \rightarrow 2$ as $n \rightarrow \infty$ and that the distribution of $\lis$ is strongly 
  concentrated around its expectation.
  This suggests introducing a new statistic $\overline{\lis}$ given by
  \begin{equation}
  \overline{\lis}(w) := \min( \lis(w), n^{\beta} )
  \end{equation}
  where $1/2 < \beta < 1$ is a parameter.
  Although not $k$-local for any fixed $k$, it is conceivable that 
  $\overline{\lis}$ can be analyzed by methods that apply to local statistics and admit profitable comparison to the statistic $\lis$.
  
In a related vein, Dubail has recently given a concentration inequality for certain functions with bounds based on locality~\cite[Thm.~1.2]{Dubail}.
His application to Markov chain mixing uses the inequality for functions that are $O(\sqrt{n})$--local.
  
  \section{Other combinatorial objects}
 This paper was exclusively concerned with functions $f: \symm_n \rightarrow \CC$ defined on permutations.
 
 \begin{problem}
 \label{other-objects-problem}
 Extend this paper to the study of statistics $f: \Omega_n \rightarrow \CC$ on other classes $\Omega_n$ of combinatorial objects.
 \end{problem}
 
 For example, in Section~\ref{Pattern} we studied the case  $\Omega_n = \MMM_n$ of perfect matchings by means of the identification 
 $\MMM_n = K_{2^n}$ with fixed-point-free involutions.
 We could also study the class $\Omega_n = \AAA^n$ of all length $n$ words over a given alphabet $\AAA$ or
 (following Chern-Diaconis-Kane-Rhoades \cite{CDKR1, CDKR2}) the class $\Omega_n = \Pi_n$ of set partitions of $[n]$.

 For $r$ fixed, another setting to which Problem~\ref{other-objects-problem} can be applied is the class
 $\Omega_n = \VVV_{n,r}$ of $r$-uniform hypergraphs on $[n]$.  Hypergraphs $\gamma \in \VVV_{n,r}$ may be thought
 of as Boolean functions $\gamma: \{0,1\}^{\binom{n}{r}} \rightarrow \{0,1\}$ on the vertex set of a hypercube of dimension $\binom{n}{r}$.
The  space $\CC[\VVV_{n,r}]$ of functions $\VVV_{n,r} \rightarrow \CC$
 as the polynomial ring quotient
 \begin{equation}
 \CC[\VVV_{n,r}] = \CC \left[x_S \,:\, S \in \binom{[n]}{r} \middle] \right/ \III
 \end{equation}
 where the $x_S$ are variables indexed by $r$-element subsets $S \subseteq [n]$ and the ideal $\III$ is generated by 
 the quadratic relations $x_S^2 - x_S$.
  Since the ideal $\III$ is inhomogeneous, the quotient ring $\CC[\VVV_{n,k}]$ is not graded.
However, for any degree $d$  one can still consider the space 
\begin{equation}
\CC[\VVV_{n,r}]_{\leq d} = \text{image of $\CC[x_S]_{\leq d}$ in $\CC[\VVV_{n,r}]$}
\end{equation}
of functions $\VVV_{n,r} \rightarrow \CC$ of degree $\leq d$.  
The space $\CC[\VVV_{n,r}]_{\leq d}$ carries an action of $\symm_n$.

In the important case $r = 2$ of graphs on $[n]$, Raymond, Saunderson, Singh, and Thomas proved \cite[Thm. 2.4]{RSST}
a support result on the $\symm_n$-structure of 
$\CC[\VVV_{n,2}]_{\leq d}$ analogous to those in Section~\ref{Local}.
They also determined \cite[Thm. 3.18]{RSST} a spanning set of $\CC[\VVV_{n,2}]_{\leq d}$ in which the indicator functions $\one_{I,J}$
on permutation are replaced with indicator functions involving subgraph containment; this is a graph-theoretic analogue
of Corollary~\ref{aw-local}.
As sketched at the conclusion of \cite{RSST}, the results of Raymond et. al. extend to the hypergraph setting of $\VVV_{n,r}$ with minor modifications.
The authors of \cite{RSST} applied their methods to the computation of sum of squares certificates which establish the nonnegativity of polynomial functions
on the vertex set $\{ 0,1 \}^{\binom{n}{k}}$.

The results of Raymond et. al. in \cite{RSST} may be thought of (at a high level)
as graph-theoretic versions of the results in Section~\ref{Local} on local permutation
statistics.
These results concern the {\em support} of certain $\symm_n$-modules.
It would likely be interesting to study the actual {\em multiplicities} in these modules by developing a machinery analogous 
to the Path Murnaghan-Nakayama Rule in Section~\ref{Path} and defining a `regular' statistic on hypergraphs as in Section~\ref{Regular}.
The deeper representation theory of Section~\ref{Path} and more powerful convergence results of Section~\ref{Convergence} suggest that 
strong results on (hyper)graph statistics can be obtained in this way.

\backmatter
\bibliographystyle{amsalpha}
\bibliography{local-bib}

\printindex

\end{document}